\documentclass[10pt]{amsart}

\usepackage[utf8]{inputenc}
\usepackage{amsthm}
\usepackage{amsmath}
\usepackage{amscd}
\usepackage[mathscr]{euscript}
\usepackage{tikz-cd}
\usepackage{url}
\usepackage{hyperref}
\usepackage{amssymb}

\newtheorem{theorem}{Theorem}
\numberwithin{theorem}{section}
\newtheorem{proposition}[theorem]{Proposition}
\newtheorem{lemma}[theorem]{Lemma}
\newtheorem{corollary}[theorem]{Corollary}
\theoremstyle{definition} 
\newtheorem{definition}[theorem]{Definition}
\theoremstyle{remark}
\newtheorem{example}[theorem]{Example}

\numberwithin{theorem}{section}
\numberwithin{equation}{section}
\providecommand{\Q}{\mathbb Q}
\providecommand{\RR}{\mathbb R}

\providecommand{\ZZ}{\mathbb Z}

\providecommand{\charac}{\operatorname{char}}

\providecommand{\supp}{\operatorname{Supp}}
\providecommand{\Val}{\operatorname{Val}}
\providecommand{\val}{\operatorname{val}}
\providecommand{\Ray}{\operatorname{Ray}}
\providecommand{\Spec}{\operatorname{Spec}}

\providecommand{\A}{\mathbb A}

\providecommand{\PP}{\mathbb P}
\providecommand{\Gm}{\mathbb{G}_m}

\providecommand{\pr}{\mathrm{pr}}

\providecommand{\HH}{\mathscr{H}}
\providecommand{\WW}{\mathscr{W}}
\providecommand{\id}{\mathrm{id}}

\providecommand{\Rt}{\mathscr{R}}
\providecommand{\Kt}{\mathscr{K}}
\providecommand{\spe}{\operatorname{sp}}
\providecommand{\bdd}{\operatorname{bdd}}

\providecommand{\OO}{\mathscr O}

\providecommand{\rank}{\operatorname{rank}}

\providecommand{\ker}{\operatorname{ker}}

\title{Stable rationality of hypersurfaces of mock toric variety I}
\author{Taro Yoshino}
\date{\today}
\address{Graduate School of Mathematical Sciences, The University of Tokyo, 3-8-1 Komaba,
Meguro-ku, Tokyo, 153-8914, Japan}
\email{yotaro@ms.u-tokyo.ac.jp}

\begin{document}

\begin{abstract}
    We introduce a mock toric variety, a generalization of a toric variety. 
    For a non-toric example, Del-Pezzo surfaces are mock toric varieties. 
    These new varieties inherit some properties of mock toric varieties. 
    In application, we give sufficient conditions for the concrete construction of a strictly toroidal model of a hypersurface in a mock toric variety. 
\end{abstract}

\maketitle
\tableofcontents
\section{Introduction}
    \subsection{Toric variety and its application for construction of a strictly toroidal model}
    Toric varieties are "simple" objects. 
    Hence, they are suitable for concretely constructing various abstract concepts. 
    For example, we give examples of compactification and a resolution of singularities. 
    In contexts of toric varieties, a complete fan induces a compactification of an algebraic tori, and an unimodular subdivision induces the resolution of a toric variety. 

    Let $k$ be an algebraically closed field with $\charac(k) = 0$. 
    Let $\Rt$ be a valuation ring defined as follows:
            \[
                \Rt = \bigcup_{n\in\ZZ_{>0}} k[[t^{\frac{1}{n}}]]
            \]
    Let $\Kt$ be a fraction field of $\Rt$. 
            We remark that $\Kt$ is written as follows:
            \[
                \Kt = \bigcup_{n\in\ZZ_{>0}} k((t^{\frac{1}{n}}))
            \]
    In this article, we consider the following definition: 
    \begin{definition}[\cite{NO21}]\label{NONO21} 
        Let $\mathscr{X}$ be a flat and separated $\Rt$-scheme of finite type, and let $x\in\mathscr{X}_k$ be a point. 
        We say that $\mathscr{X}$ is strictly toroidal at $x$ if there exist a toric monoid $S$, an open neighbourhood $U$ of $x$ in $\mathscr{X}$, and a smooth morphism of $\Rt$-schemes $U \rightarrow \Spec(\Rt[S]/(\chi^\omega - t^q))$, where $q$ is a positive rational number, $\omega$ is an element of $S$, and $\chi^\omega$ is a torus invariant monomial in $k[S]$ associated with $\omega$ such that $k[S]/(\chi^\omega)$ is reduced. 
        We call that $\mathscr{X}$ is \textbf{strictly toroidal} if it is so at any $x\in \mathscr{X}_k$. 
        \end{definition}
    For a proper and smooth $\Kt$-variety $X$, a strictly toroidal model of $X$ is essential to compute a stably birational model of $X$.(cf. \cite[Theorem 3.3.5]{NO21})
    Although its construction is generally challenging, it is easy for toric varieties and Newton non-degenerate hypersurfaces of toric varieties over $\Kt$(cf. \cite{NO22}). 

    This article gives new examples of strictly toroidal models of some varieties. 
    In particular, we would like to mention the following two points in this article: 
    \begin{itemize}
        \item An introduction to mock toric varieties, which are the generalization of toric varieties
        \item Sufficient conditions for the concrete (NOT abstract) construction of a strictly toroidal model of a hypersurface in a mock toric variety(See Proposition \ref{prop: irreducible computation})
    \end{itemize}
    These new examples are applied to check the stable rationality of some varieties. 
    For further details, please refer to the preprint by the same author, which will be released on arXiv soon.
    
    In the following subsections, we summarize this article's key concepts and outline. 
    \subsection{Mock toric varieties}
    Section 3 defines mock toric varieties(cf. Definition \ref{def:mock-toric}). 
    In short, a mock toric variety $Z$ is a closed subscheme of a toric variety $X$, and it is covered by open subschemes of toric varieties. 
    In particular, it is a toroidal rational variety. 
    For example, toric varieties are mock toric varieties. 
    Moreover, the "base point free" hyperplane arrangement induces a mock toric variety(See Proposition \ref{prop: hyper-example}). 
    As a corollary of it, we show that a Del Pezzo surface is a mock toric variety(See Corollary \ref{cor: Del-Pezzo-model}). 
    
    While the definition is complicated, mock toric varieties inherit various properties with toric varieties as follows: 
    \begin{itemize}
        \item The torus orbit decomposition of the ambient space $X$ induces the stratification of $Z$. 
        Similar to toric varieties, there is a correspondence between the inclusion relationships of the closure of each stratum and the inclusion relationships of cones(cf. Corollary\ref{cor:stratification-structure2}). 
        \item Each closure of the stratum is also a mock toric variety(cf. Proposition\ref{prop:orbit-structure}). 
        \item Let $\mu$ be a dominant toric morphism $X'\rightarrow X$. 
        Then $Z\times_X X'$ is a mock toric variety(cf. Proposition\ref{prop:relative-induced-structure}). 
        \item There exists discrete valuations on $Z$ trivial on $k$ corresponding to lattice points in a toric fan of $X$(cf. Proposition\ref{prop:property-of-valuations}). 
        \item If $Z$ is proper over $k$, then a closed immersion $Z\hookrightarrow X$ is a tropical compactification of $Z\cap T$ in $X$, where $T$ is a dense open torus orbit of $X$(cf. Proposition\ref{prop:tropical}). 
    \end{itemize}
    \subsection{Hypersurfaces of mock toric varieties}
        Let $X$ be a toric variety, $T$ be a dense open torus orbit of $X$, $H^\circ$ be a hypersurface in $T$, and $H$ be a scheme theoretic closure of $H^\circ$ in $X$. 
        For each torus orbit $O$, we can check whether $O\subset H$ by using torus invariant valuations of the defining Laurent polynomial of $H^\circ$(cf. Lemma\ref{lem: orbit and function}). 
        In general, $H$ contains some torus orbits of $X$.
        However, there exists a toric resolution $X'\rightarrow X$ such that $H'$ does not contain any torus orbits of $X'$, where $H'$ denotes a scheme theoretic closure of $H^\circ$ in $X'$(cf.Propositon\ref{lem: property of val function}(g)).
        The condition that the closure of $H^\circ$ does not contain any torus orbits of the toric variety is important for constructing a good model. 
        
        The same argument works in the case of mock toric varieties. 
        The details of the above are discussed in Section 6. 
    \subsection{Strictly toroidal models of hypersurfaces of mock toric varieties}
        One of the goals of this article is to give sufficient conditions for the concrete construction of a strictly toroidal model of a hypersurface in a mock toric variety. 
        Because mock toric varieties inherit most properties from toric varieties, sufficient conditions for them can be given, similarly to those of a strict toroidal model of a hypersurface in a toric variety. 
        
        As a note, \cite{NO22} focused on an irreducible hypersurface over $\Kt$ in an algebraic tori. 
        In this article, we start the argument from irreducible hypersurfaces in mock toric varieties over $k(t)$ and then change it to those over $\Kt$ by a base change.
        In general, the irreducibility might be compromised due to the field extension. 
        Instead, we consider all irreducible components. 
        Although the argument is complicated, it is still applicable for practical purposes. 
    \subsection{Outline of the paper}
        This paper is organized as follows.
        In Section 2, we organize the notation used in this article. 
        Section 3 introduces the definition of mock toric varieties and considers their basic properties. 
        Section 4 considers some mock toric varieties constructed from a mock toric variety. 
        We also discuss the set of discrete valuations on mock toric varieties corresponding to the lattice points of the ambient toric fan. 
        In Section 5, we give some examples of mock toric varieties. 
        In particular, we show that the "base point free" hyperplane arrangements induce mock toric varieties. 
        In Section 6, we consider hypersurfaces of mock toric varieties. 
        Section 7 gives sufficient conditions for the concrete construction of a strictly toroidal model of a hypersurface in a mock toric variety. 
        In Section 8, we prove the lemmas needed in this article. 
    \subsection{Acknowledgment}
        The author is grateful to his supervisor, Yoshinori Gongyo, for his encouragement.
\section{Notation}
    In this paper, we use the following notations according to \cite{CLS11} and \cite{Ful93}:
    \begin{itemize}
        \item Let $k$ be a field. 
        \item For an integral scheme $X$, 
        let $K(X)$ denote the function field of $X$. 
        Let $\Val(X)$ denote a set of all discrete valuations on $K(X)$. 
        \item We call an integral separated scheme of finite type over $k$ a variety over $k$. 
        Let $X$ be a variety over $k$ and $\Val_k(X)$ denote a set of all discrete valuations on $K(X)$, which is trivial on $k$. 
        \item Let $f\colon X\dashrightarrow Y$ be a dominant rational map of integral schemes, $f^*$ denote a morphism of fields $K(Y)\rightarrow K(X)$ induced by $f$, and $f_{\natural}$ denote the map $\Val(X)\rightarrow \Val(Y)$ defined by $f_\natural(v) = v\circ f^*$ for $v\in \Val(X)$. 
        \item Let $V$ be an $\RR$-linear vector space of finite dimension, $W$ be the dual space of $V$, and $B$ be a subset of $V$. 
        Let $B^\perp$ and $B^\vee$ denote as the following subsets of $W$: 
        \[
            B^\perp = \{w\in W\mid w(x) = 0\quad(\forall x\in B)\}
        \]
        \[
            B^\vee = \{w\in W\mid w(x) \geq 0\quad(\forall x\in B)\}
        \]
        \item Let $N$ denote a lattice of finite rank and $M$ denote the dual lattice of $N$. 
        \item Let $\sigma$ be a convex cone in $N_\RR$ and $\tau$ be a face of $\sigma$.   
        Then we write $\tau\preceq \sigma$. 
        \item Let $\sigma$ be a convex cone in $N_\RR$. 
        Then $\sigma^\circ$ denotes the relative interior of $\sigma$. 
        \item Let $\Delta$ be a strongly convex rational polyhedral fan in $N_\RR$.  
        Then $X_k(\Delta)$ denote a toric variety corresponding to $\Delta$ over $k$. 
        We sometimes write $X(\Delta)$ insted of $X_k(\Delta)$. 
        \item Let $k[M]$ denote a $k$-algebra induced by a monoid $M$. 
        For $\omega\in M$, let $\chi^\omega$ denote a monomial in $k[M]$ associated with $\omega$. 
        \item Let $T_N$ denote $\Spec(k[M])$. 
        we remark that $X(\{0_N\}) = T_N$. 
        \item Let $\Delta$ be a strongly convex rational polyhedral fan and $\sigma\in\Delta$ be a cone. 
        Let $O_\sigma$ denote the orbit of the torus action of $X(\Delta)$ corresponding to $\sigma$. 
        \item Let $\sigma$ be a convex cone in an $\RR$-linear space $V$. 
        Let $\langle\sigma\rangle$ denote an $\RR$-linear subspace of $V$ which is spanned by $\sigma$. 
        \item 
        For a strongly convex rational fan $\Delta$ in $N_\RR$ and $\sigma\in\Delta$, let $N[\sigma]$ denote $N/(\langle\sigma\rangle\cap N)$, $\pi^\sigma$ denote the quotient map from $N$ to $N[\sigma]$, $k[\sigma^\vee\cap M]$ denote $k$-algebra defined by a monoid $\sigma^\vee\cap M$, $M^\sigma$ denote the dual lattice of $N[\sigma]$. 
        We remark that $k[\sigma^\vee\cap M] \cong \Gamma(X(\sigma), \OO_{X(\sigma)})$ and $k[M^\sigma]\cong \Gamma(O_\sigma, \OO_{O_\sigma})$. 
        \item For the notation above, let $p^\sigma$ denote a quotient map from $k[\sigma^\vee\cap M]$ to $k[M^\sigma]$ associated with the closed immersion $O_\sigma\hookrightarrow X(\sigma)$ and $(\pi^\sigma)^*$ denote a lattice morphism $M^\sigma\rightarrow M$ induced by $\pi^\sigma$. 
        We remark that $(\pi^\sigma)^*$ is injective, $(\pi^\sigma)^*(M^\sigma) = \sigma^\perp\cap M$, and an ideal $\ker(p^\sigma)$ is generated by $\{\chi^\omega\}_{\omega\in (\sigma^\vee\cap M)\setminus \sigma^\perp}$. 
        \item For the notation above, let $\Delta[\sigma]$ denote the following set:
        \[
            \Delta[\sigma] = \{\pi^\sigma_\RR(\tau)\mid\sigma\preceq\tau\in\Delta\}
        \]
        From Proposition \ref{prop:orbit-toric}, $\Delta[\sigma]$ is a strongly convex rational polyhedral fan in $N[\sigma]_\RR$. 
        \item Let $N$ and $N'$ be lattices of finite rank and $f\colon N'\rightarrow N$ be a homomorphism, $f_\RR$ denote an $\RR$-linear map $N'_\RR\rightarrow N_\RR$ induced by $f$,  $M'$ denote the dual lattice of $N'$, and $f^*$ denote a lattice morphism $M\rightarrow M'$ induced by $f$. 
        With a slight abuse of notation, We also use $f^*$ for a morphism of $k$-algebras $k[M]\rightarrow k[M']$ defined by $\chi^\omega\mapsto\chi^{f^*(\omega)}$ for $\omega\in M$. 
        \item On the convention above, let $\Delta$ and $\Delta'$ be a strongly convex rational polyhedral fan in $N_\RR$ and $N'_\RR$ respectively. 
        If for any $\tau\in\Delta'$, there exists  $\sigma\in\Delta$ such that $f_\RR(\tau)\subset\sigma$, then we call that $f$ is compatible with the fans $\Delta'$ and $\Delta$. 
        \item On the notation above, the map $f_*$ denote the toric morphism from $X(\Delta')$ to $X(\Delta)$ induced by $f$. 
        \item Let $\Delta_!$ denote a convex fan $\{\{0\}, [0, \infty)\}$ in $\RR$. 
        \item Let $\Delta_1, \Delta_2$ be fans in $\RR^n$ and $\RR^m$ respectively and $\Delta_1\times\Delta_2$ denote the following fan in $\RR^{n+m}$: 
        \[
            \{\sigma_1\times\sigma_2\mid\sigma_1\in \Delta_1, \sigma_2\in \Delta_2\}
        \]
        \item Let $\Delta$ be a strongly convex rational polyhedral fan in $N_\RR$, $f\in k[M]$ denote a polynomial.  
        Let $H^\circ_{X(\Delta), f}$ denote a hypersurface of $T_N$ defined by $f = 0$, and $H_{X(\Delta), f}$ denote a scheme theoretic closure of $H^\circ_{X(\Delta), f}$ in $X(\Delta)$. 
        \item Let $N$ be a lattice of finite rank, $M$ be the dual lattice, and $\langle\cdot,\cdot\rangle$ be a pairing of $N$ and $M$. 
        For $v\in N$, we can identify $v$ as a $T_N$-invariant valuation on $T_N$ with trivial on $k$, and let $v(f)$ denote a value of $f\in k(M)$ by $v$, where $k(M)$ denote the fraction field of $k[M]$. 
        On this identification, $v(\chi^\omega) = \langle v,\omega\rangle$ for any $v\in N$ and $\omega\in M$. 
        \item Let $S$ be a monoid. 
        If there exists a lattice $N$ of finite rank and a full and strongly convex rational polyhedral cone $\sigma$ in $N_\RR$ such that $S$ is isomorphic to  $\sigma^\vee\cap M$, then we call $S$ is a \textbf{toric monoid}.
    \end{itemize}
\section{Mock toric variety}
    In this section, $k$ denotes a field. 
    We define a structure of mock toric varieties for $k$-schemes. 
    Subsequently, we examine the properties of the mock toric varieties. 
    Surprisingly, these new varieties inherit certain properties of the toric varieties.
    \subsection{Definition of the mock toric variety}
    In this subsection, we introduce mock toric varieties. 
    We remark that this variety has no torus action in general. 
    \begin{definition}\label{def:mock-toric}
        Let $Z$ be a scheme over $k$. 
        The mock toric structure of $Z$ consists of the following 6 data, which satisfy the following conditions from (1) to (5):
        \begin{itemize}
            \item A lattice $N$ of finite rank 
            \item A strongly convex rational polyhedral fan $\Delta$ in $N_\RR$
            \item A closed immersion $\iota\colon Z\hookrightarrow X(\Delta)$
            \item A finite set $\Phi$
            \item A family $\{N_\varphi\}_{\varphi\in\Phi}$ which consists of sublattices of $N$
            \item A family $\{\Delta_\varphi\}_{\varphi\in\Phi}$ which consists of sub fans of $\Delta$
        \end{itemize}
        \begin{itemize}
            \item[(1)] 
            For any $\varphi\in\Phi$ and $\sigma\in\Delta_\varphi$, $N/((\langle\sigma\rangle\cap N)+N_\varphi)$ is torsion-free. 
            In particular, from this condition, $N/N_\varphi$ is torsion-free. (see proposition Proposition \ref{prop: first prop}(a))
            \item[(2)] The following equation holds:
            \[
                \Delta = \bigcup_{\varphi\in\Phi}\Delta_\varphi.
            \]
            \item[(3)] For $\varphi\in \Phi$, let $q_\varphi$ denote a quotient map  $N\rightarrow N/N_{\varphi}$. 
            Then the map $(q_\varphi)_\RR|_{\supp(\Delta_\varphi)}$ is injective. 
            \item[(4)]For any $\sigma\in\Delta$, 
            $Z\cap O_\sigma \neq \emptyset$. 
            \item[(5)] From (3), The following set $\Delta(\varphi)$ is a strongly convex rational polyhedral fan in $(N/N_\varphi)\otimes\RR$\ (see Proposition \ref{prop: first prop}(c)): 
            \[
                \Delta(\varphi) = \{(q_\varphi)_\RR(\sigma)\mid\sigma\in\Delta_{\varphi}\}
            \]
            Then the following composition $\iota_\varphi$ is an open immersion: 
            \begin{equation*}
                \begin{tikzcd} 
                    Z\cap X(\Delta_\varphi)\ar[r, "\iota"]& X(\Delta_\varphi)\ar[r, "(q_\varphi)_*"]& X(\Delta(\varphi)), 
                \end{tikzcd}
            \end{equation*}
            where the former one is a closed immersion restricted to $Z\cap X(\Delta_\varphi)$.
        \end{itemize}
        We call that $(N, \Delta, \iota, \Phi, \{N_\varphi\}_{\varphi\in\Phi},\{\Delta_\varphi\}_{\varphi\in\Phi})$ is a mock toric structure of $Z$. 
    \end{definition}
    The following lemma summarizes fundamental facts that immediately hold from Definition\ref{def:mock-toric}. 
    \begin{proposition}\label{prop: first prop}
        In the definition Definition \ref{def:mock-toric}, the following statements follow:
        \begin{enumerate}
            \item[(a)] For any $\varphi\in\Phi$, $N/N_\varphi$ is a torsion-free module. 
            \item[(b)] For any $\varphi\in\Phi$ and any $\sigma\in\Delta_\varphi$, $\langle\sigma\rangle\cap N_\varphi = \{0\}$. 
            \item[(c)] A set $\Delta(\varphi)$ is a strongly convex rational polyhedral fan in $(N/N_\varphi)\otimes\RR$. 
            \item[(d)] Let $\varphi\in\Phi$ be an element and $\sigma\in\Delta\varphi$ be a cone. Let $\sigma_\varphi$ denote $(q_\varphi)_\RR(\sigma)$.   
            Then $q_\varphi|_{\langle\sigma\rangle\cap N}\colon \langle\sigma\rangle\cap N\rightarrow \langle\sigma_\varphi\rangle\cap N/N_\varphi$ is isomorphic. 
        \end{enumerate}
    \end{proposition}
    \begin{proof}
        We prove these statements from (a) to (d) in order: 
        \begin{enumerate}
            \item[(a)] Because $\Delta_\varphi$ is a strongly convex fan, we have $\{0\}\in \Delta_\varphi$. 
            Thus, $N/N_\varphi$ is torsion-free from the condition (1). 
            \item[(b)] Let $v\in \langle\sigma\rangle\cap N_\varphi$ be an element. 
            Then there exists $v_1$ and $v_2\in \sigma$ such that $v = v_1 - v_2$. 
            Because $v\in N_\varphi$, we have $(q_\varphi)_\RR(v_1) = (q_\varphi)_\RR(v_2)$. 
            Thus, $v_1 = v_2$ from the condition (3). 
            Therefore, $v = 0$. 
            \item[(c)] 
            We remark that $N/N_\varphi$ is a lattice from (a). 
            From the condition (3) and Lemma \ref{lem: injective fan}(c), $\Delta(\varphi)$ is a polyhedral convex fan in $(N/N_\varphi)\otimes_\ZZ\RR$. 
            Because $q_\varphi$ is a lattice morphism, $\Delta(\varphi)$ is a rational polyhedral convex fan. 
            Moreover, because $\Delta_\varphi$ is strongly convex, $\Delta(\varphi)$ is strongly convex from \ref{lem: injective fan}(d). 
            \item[(d)] From the condition (1) in \ref{def:mock-toric}, (b), and Lemma\ref{lem:take-section}, there exists a section $\theta$ of $q_\varphi$ such that $\langle\sigma\rangle\cap N\subset \theta(N/N_\varphi)$. 
            In particular, $\langle\sigma\rangle\subset \theta_\RR((N/N_\varphi)_\RR)$. 
            Let $v'\in\langle\sigma_\varphi\rangle\cap N/N_\varphi$ be an element. 
            Then there exists $v_1\in\langle\sigma\rangle$ and $v_2\in N$ such that $(q_\varphi)_\RR(v_1) = v'$ and $q_\varphi(v_2) = v'$. 
            From the condition of $\theta$ and the rational polyhedral convexity of $\sigma$, $v_1 = \theta_\RR(v')$ and $v_2 = \theta(v')$. 
            Thus, $v_1 = v_2$ because $\theta_\RR$ is injective. 
            This shows that $v'\in q_\varphi(\langle\sigma\rangle\cap N)$. 
            
            On the other hand, $q_\varphi|_{\langle\sigma\rangle\cap N}$ is injective form (b). Therefore, the statement holds. 
        \end{enumerate}
    \end{proof}
    \begin{example}
        We show some examples of mock toric varieties.   
        \begin{itemize}
            \item[(1)] 
            Let $N$ be a lattice of finite rank and $\Delta$ be a strongly convex rational polyhedral fan in $N_\RR$. 
            Then, $Z = X(\Delta)$ has a natural mock toric structure. 
            In fact, $\mathrm{id}_Z$ is a closed immersion and let $\Phi$ denote $\{0\}$, $N_0$ denote $\{0\}$, and $\Delta_0$ denote $\Delta$.
            Then one can check easily that $(N$, $\Delta$, $\mathrm{id}_Z$, $\{0\}$, $\{N_0\}$, $\{\Delta_0\})$ is a mock toric structure of $Z$. 
            \item[(2)] 
            Let $Z$ be a projective line $\PP^1$. 
            A scheme $Z$ has a natural toric structure, but it has another mock toric structure as follows: 
            Let $N$ be a lattice of rank $2$. 
            Then there exists $e_0, e_1, e_2\in N$ as follows:
            \begin{itemize}
                \item These are generators of $N$.
                \item $e_0+e_1+e_2 = 0$.
            \end{itemize}
            Let $\Delta$ denote the following fan: 
            \[
                \Delta=\{\{0\}, \RR_{\geq0}e_{0}, \RR_{\geq0}e_{1}, \RR_{\geq0}e_{2}\}
            \]
            This toric variety $X(\Delta)$ is an open subscheme of projective plane $\PP^2$. 
            Let $\iota\colon Z\rightarrow X(\Delta)$ be a closed immersion in which $Z$ intersects transversely at one point on each torus invariant divisor of $X(\Delta)$. 
            Let $\Phi$ denote $\{0, 1, 2\}$, and for $i\in\Phi$, we define sublattice $N_i$ of $N$ as $\ZZ\cdot e_i$ and a sub fan $\Delta_i$ of $\Delta$ as $\Delta\setminus\{\RR_{\geq0}e_{i}\}$.
            Then we can check easily that $(N$, $\Delta$, $\iota$, $\Phi$, $\{N_i\}_{i\in\Phi}$, $\{\Delta_i\}_{i\in\Phi})$ is a mock toric structure of $Z$.  
            \item[(3)] For example (2), we can regard three points on $\PP^1_k$ as a "base point free" hyperplane arrangement on $\PP^1_k$. 
            In general, we show that a "base point free" hyperplane arrangement induces a mock toric variety in Section 5. 
        \end{itemize}
    \end{example}
    From the definition of a mock toric structure, we can check easily the properties of mock toric varieties as follows:  
    \begin{proposition}\label{prop:basic-property1}
        Let $Z$ be a scheme over $k$ and $(N, \Delta, \iota, \Phi, \{N_\varphi\}_{\varphi\in\Phi},\{\Delta_\varphi\}_{\varphi\in\Phi})$ be a mock toric structure of $Z$. 
        The following statements follow:
        \begin{itemize}
            \item[(a)] The scheme $Z$ is a separated normal integral scheme of finite type over $k$.
            \item[(b)] For any $\varphi\in\Phi$, the rank of $N/N_\varphi$ is $\dim(Z)$.
            \item[(c)] For any $\varphi\in\Phi$, the image of $\iota_\varphi$ intersects generic points of all torus orbits of $X(\Delta(\varphi))$. 
            \item[(d)] For any $\sigma\in\Delta$, we have $\dim(Z\cap O_\sigma) + \dim(\langle\sigma\rangle) = \dim(Z)$.  
        \end{itemize}
    \end{proposition}
    \begin{proof}
        We prove these statements from (a) to (d) in order.
        \begin{itemize}
            \item[(a)] From the condition (2) and (5), $\{Z\cap X(\Delta_\varphi)\}_{\varphi\in\Phi}$ is an open covering of $Z$. 
            From condition (5), each $Z\cap X(\Delta_\varphi)$ is an integral normal scheme, and each of them contains a non-empty subset $Z\cap T_N$ from condition (4). 
            Thus, $Z$ is an integral normal scheme. 
            A scheme $Z$ is a closed subscheme of toric variety over $k$, so $Z$ is a separated scheme of finite type over $k$. 
            \item[(b)] From the condition (5) and the above argument, $\dim(Z) = \dim(X(\Delta(\varphi))$ for any $\varphi\in\Phi$. 
            Because $\Delta(\varphi)$ is strongly convex, $\rank(N/N_\varphi) = \dim(X(\Delta(\varphi)))$. 
            Thus, the statement holds.
            \item[(c)]
            From condition (3) and Lemma \ref{lem: injective fan}(b), there exists a one-to-one correspondence of cones in $\Delta_\varphi$ and those in $\Delta(\varphi)$ by $(q_\varphi)_\RR$. 
            Let $\varphi\in\Phi$ be an element and $\sigma\in\Delta_\varphi$ be a cone. 
            Let $\sigma_\varphi$ denote $(q_\varphi)_\RR(\sigma)$. 
            Then the following diagram of toric morphisms is a Cartesian diagram: 
            \begin{equation*}
                \begin{tikzcd} 
                    X(\sigma)\ar[d, hook]\ar[r, "(q_\varphi)_*"]& X(\sigma_\varphi)\ar[d, hook]\\
                    X(\Delta_\varphi)\ar[r, "(q_\varphi)_*"]& X(\Delta(\varphi)), 
                \end{tikzcd}
            \end{equation*}
            where vertical morphisms are open immersions. 
            Moreover, from the condition (1), Proposition \ref{prop: first prop}(b), Lemma \ref{lem:take-section}, and Lemma \ref{lem:torus-fibration}(e), the following diagram is a Cartesian diagram too. 
            \begin{equation*}
                \begin{tikzcd} 
                    O_\sigma\ar[d, hook]\ar[r, "(q_\varphi)_*"]& O_{\sigma_\varphi}\ar[d, hook]\\
                    X(\sigma)\ar[r, "(q_\varphi)_*"] & X(\sigma_\varphi).
                \end{tikzcd}
            \end{equation*}
            Thus, there exists the following diagram whose two small squares are Cartesian squares:
            \begin{equation*}
                \begin{tikzcd} 
                    Z\cap O_\sigma\ar[r, hook, "\iota"]\ar[d, hook]& O_\sigma\ar[d, hook]\ar[r, "(q_\varphi)_*"]& O_{\sigma_\varphi}\ar[d, hook]\\
                    Z\cap X(\Delta_\varphi)\ar[r, hook, "\iota"]& X(\Delta_\varphi)\ar[r, "(q_\varphi)_*"] & X(\Delta(\varphi))
                \end{tikzcd}
            \end{equation*}
            Hence, from the condition (5) and the above diagram, $\iota_\varphi(Z\cap O_\sigma)$ is an 
            open subset of $O_{\sigma_\varphi}$. 
            From the condition (4), $\iota_\varphi(Z\cap O_\sigma)$ is non-empty, so the statement follows. 
            \item[(d)] We use the notation in the proof of (c). 
            We have that $\dim(\langle\sigma_\varphi\rangle) + 
            \dim(O_{\sigma_\varphi}) = \rank(N/N_\varphi)$ from the argument of the toric varieties. 
            Because $(q_\varphi)_\RR|_{\langle\sigma\rangle}$ is injective, $\dim(\langle\sigma_\varphi\rangle) = \dim(\langle\sigma\rangle)$. 
            Moreover, from the above argument of proof of (c), $\dim(O_{\sigma_\varphi}) = \dim(Z\cap O_\sigma)$. 
            Thus, from (b), $\dim(\langle\sigma\rangle) + \dim(Z\cap O_\sigma) = \dim(Z)$. 
        \end{itemize}
    \end{proof}
    A mock toric variety has a natural stratification and properties similar to a toric variety. 
    After defining the notation of the stratification at Definition \ref{def: stratification}, we show the properties of this stratification in Proposition\ref{prop:stratification-structure}.
    \begin{definition}\label{def: stratification}
        Let $Z$ be a scheme over $k$, $(N, \Delta, \iota, \Phi, \{N_\varphi\}_{\varphi\in\Phi}, \{\Delta_\varphi\}_{\varphi\in\Phi})$ be a mock toric structure of $Z$, and $\sigma\in\Delta$ be a cone. 
        \begin{itemize}
            \item[(1)] 
            Let $Z(\sigma)$ denote an open subscheme of $Z$ defined as the following Cartesian diagram: 
            \begin{equation*}
                \begin{tikzcd} 
                    Z(\sigma) \ar[d]\ar[r, hook, "\iota"]& X(\sigma)\ar[d]\\
                    Z\ar[r, hook,"\iota"] & X(\Delta), 
                \end{tikzcd}
            \end{equation*}
            where vertical morphisms are open immersions. 
            \item[(2)] 
            Let $Z_\sigma$ denote a scheme over $k$ defined as the following Cartesian diagram: 
            \begin{equation*}
                \begin{tikzcd} 
                    Z_\sigma \ar[d]\ar[r, hook, "\iota"]& O_\sigma\ar[d]\\
                    Z(\sigma)\ar[r, hook,"\iota"] & X(\sigma), 
                \end{tikzcd}
            \end{equation*}
            where vertical morphisms are closed immersions. 
            \item[(3)] Let $Z^\circ$ denote $Z\cap T_N$. 
            \item[(4)] Let $k[Z^\circ]$ denote the section ring $\Gamma(Z^\circ, \OO_Z)$. 
        \end{itemize}
    \end{definition}
    \begin{proposition}\label{prop:stratification-structure}
        Let $Z$ be a scheme over $k$ and $(N, \Delta, \iota, \Phi, \{N_\varphi\}_{\varphi\in\Phi},\{\Delta_\varphi\}_{\varphi\in\Phi})$ be a mock toric structure of $Z$. 
        The following statements follow. 
        \begin{itemize}
            \item[(a)] For any $\sigma$ and $\tau\in\Delta$, the following statements are equivalent: 
            \begin{itemize}
                \item[(1)] $\tau\subset\sigma$.
                \item[(2)] $Z_\sigma\subset\overline{Z_\tau}$.
                \item[(3)] $Z(\tau)\subset Z(\sigma)$. 
            \end{itemize}
            \item[(b)] For any $\sigma\in\Delta$, $Z(\sigma)$ is an integral normal affine scheme of finite type over $k$.
            \item[(c)] For any $\sigma\in\Delta$, $Z_\sigma$ is an integral smooth affine scheme of finite type over $k$. 
            \item[(d)] A ring $k[Z^\circ]$ is a UFD. 
        \end{itemize}
    \end{proposition}
    \begin{proof}
        We prove these statements from (a) to (c) in order. 
        \begin{itemize}
            \item[(a)] 
            First, when the statement (1) holds, we show that the statement (2) holds. 
            From the condition (2) in Definition \ref{def:mock-toric}, there exists $\varphi\in\Phi$ such that $\sigma\in\Delta_\varphi$. 
            Because $\Delta_\varphi$ is a sub fan of $\Delta$, $\tau\in\Delta_\varphi$. 
            Let $\sigma_\varphi$ denote $(q_\varphi)_\RR(\sigma)$ and $\tau_\varphi$ denote $(q_\varphi)_\RR(\tau)$. 
            From the proof of Proposition \ref{prop:basic-property1}(c), $\iota_\varphi(Z_\sigma)$ and $\iota_\varphi(Z_\tau)$ are dense open subset of $O_{\sigma_\varphi}$ and $O_{\tau_\varphi}$ respectively. 
            In $X(\Delta(\varphi))$, $O_{\sigma_\varphi}\subset\overline{O_{\tau_\varphi}}$ because $X(\Delta(\varphi))$ is a toric variety. 
            Thus, statement (2) holds. 

            Second, when the statement (2) holds, we show that the statement (3) holds. 
            From the statement (2) and the condition (2) in Definition \ref{def:mock-toric}, $O_\sigma\cap \overline{O_\tau}\neq\emptyset$. 
            Thus, from the property of the torus orbit decomposition of toric varieties, $X(\tau)\subset X(\sigma)$. 

            Finally, when the statement (3) holds, we show that the statement (1) holds. 
            From statement (3), we can check that $O_\tau\cap X(\sigma)\neq\emptyset$. 
            Thus, $\tau\subset\sigma$ from the property of the torus orbit decomposition of toric varieties. 
            \item[(b)] From the definition of $Z(\sigma)$, $Z(\sigma)$ is an open subscheme of a normal integral scheme of $Z$ from Proposition \ref{prop:basic-property1} (a). 
            Moreover, $Z(\sigma)$ is a closed subscheme of an affine scheme $X(\sigma)$. 
            Thus, the statement follows. 
            \item[(c)] From the definition of $Z_\sigma$, $Z_\sigma$ is a closed subscheme of an affine variety of $O_\sigma$, so that $Z_\sigma$ is an affine scheme. 
            Let $\varphi\in\Phi$ be an element such that $\sigma\in\Delta_\varphi$.  
            we have already shown  that $Z_\sigma$ is an open subscheme of $O_{\sigma_\varphi}$ in the proof of Proposition \ref{prop:basic-property1}(c), where $\sigma_\varphi$ denote $(q_\varphi)_\RR(\sigma)$. 
            Thus, $Z_\sigma$ is smooth over $k$. 
            \item[(d)] Because $Z^\circ$ is a closed subscheme of $T_N$, $k[Z^\circ]$ is of finite type over $k$. 
            In particular, $k[Z^\circ]$ is a Noether ring. 
            On the other hand, for any $\varphi\in\Phi$, $Z^\circ$ is an open subscheme of  $T_{N/N_\varphi}$. 
            Thus, it is enough to show the following basic statement of the ring theory:
            \begin{itemize}
                \item Let $A$ be a ring of finite type over $k$. Let $V$ be an affine open subscheme of $\Spec(A)$ and $B$ denote $\Gamma(V, \OO_{\Spec(A)})$. 
                We assume that $B$ is of finite type over $k$ and $A$ is a UFD. 
                Then $B$ is a UFD. 
            \end{itemize}
            Indeed, it is enough to show that every hight 1 prime ideal of $B$ is principal. 
            Let $\mathfrak{p}\in V$ be a hight 1 prime of $B$ and $\mathfrak{q}$ denote $A\cap \mathfrak{p}$. 
            We remark that $\mathfrak{q}$ is a hight 1 prime ideal of $A$ because $V\rightarrow \Spec(A)$ is an open immersion. 
            Thus, there exists $f\in A$ such that $f$ generates $\mathfrak{q}$. 
            We claim that $f$ generate $\mathfrak{p}$. 
            Let $g\in\mathfrak{p}$ be an element. 
            Then there exists $u, h\in A$ such that $u$ and $h$ are coprime and $g = u^{-1}h$ . 
            Hence, we can check that $\supp(A/(u))\cap V=\emptyset$ because $g\in\Gamma(V, \OO_{\Spec(A)})$ and $A$ is a UFD. 
            In particular, $u\in B^*$ and $h\in \mathfrak{q}$. 
            Thus, there exists $l\in A$ such that $h = lf$. 
            Therefore, $g = u^{-1}lf$ in $B$. 
        \end{itemize}
    \end{proof}
    The following proposition shows that the scheme theoretic closure of the torus orbit of the ambient toric variety gives the scheme theoretic closure of each stratum of the mock toric variety.  
    \begin{proposition}\label{prop: closure of stratum}
        Let $Z$ be a scheme over $k$ and $(N, \Delta, \iota, \Phi, \{N_\varphi\}_{\varphi\in\Phi}, \{\Delta_\varphi\}_{\varphi\in\Phi})$ be a mock toric structure of $Z$. 
        Then the scheme theoretic closure of $Z_\sigma$ in $Z$ is equal to $\overline{O_\sigma}\times_{X(\Delta)} Z$ for any $\sigma\in\Delta$. 
    \end{proposition}
    \begin{proof}
        A family of open subset $\{Z(\tau)\}_{\tau\in\Delta}$ is an open covering of $Z$. 
        From Lemma \ref{lem:covering-image}, it is enough to show that the scheme theoretic closure of $Z_\sigma$ in $Z(\tau)$ is $\overline{O_\sigma}\times_{X(\Delta)} Z(\tau) = (\overline{O_\sigma}\cap X(\tau))\times_{X(\tau)} Z(\tau)$ for any $\tau\in \Delta$. 
        Moreover, it is enough to check this claim for all $\tau\in\Delta$ such that $\sigma\subset\tau$. 
        Indeed, for any $\tau\in\Delta$ such that $\sigma\not\subset\tau$, we have $\overline{O_\sigma}\cap X(\tau) = \emptyset$ and $Z_\sigma\cap Z(\tau) = \emptyset$. 
        There exists $\varphi\in\Phi$ such that $\tau\in\Delta_\varphi$.  
        let $\sigma_\varphi$ denote $(q_\varphi)_\RR(\sigma)$ and $\tau_\varphi$ denote $(q_\varphi)_\RR(\tau)$. 
        From  Lemma \ref{lem:torus-fibration}(d) and (e), all small squares of the following diagram are 
        Cartesian squares and the lower morphism from $Z(\tau)$ to $X(\tau_\varphi)$ is an open immersion:
        \begin{equation*}
            \begin{tikzcd} 
                Z_\sigma\ar[r, "\iota"]\ar[d]& O_\sigma\ar[r, "(q_\varphi)_*"]\ar[d]&O_{\sigma_\varphi}\ar[d]\\
                (\overline{O_\sigma}\cap X(\tau))\times_{X(\tau)} Z(\tau)\ar[r, "\iota"]\ar[d]& \overline{O_\sigma}\cap X(\tau)\ar[r, "(q_\varphi)_*"]\ar[d]&\overline{O_{\sigma_\varphi}}\cap X(\tau_\varphi)\ar[d]\\
                Z(\tau)\ar[r, "\iota"]&X(\tau)\ar[r, "(q_\varphi)_*"]&X(\tau_\varphi)\\
            \end{tikzcd}
        \end{equation*}
        By the way, $\overline{O_{\sigma_\varphi}}\cap X(\tau_\varphi)$ is a scheme theoretic closure of $O_{\sigma_\varphi}$ in $X(\tau_\varphi)$. 
        Thus, from Lemma \ref{lem:open-image}, $(\overline{O_\sigma}\cap X(\tau))\times_{X(\tau)} Z(\tau)$ is a scheme theoretic closure of $Z_\sigma$ in $Z(\tau)$. 
    \end{proof}
    As toric varieties, the following statement holds from Proposition\ref{prop: closure of stratum}.
    \begin{corollary}\label{cor:stratification-structure2}
        We keep the notation in Proposition \ref{prop:stratification-structure}.  
        The following statements are equivalent:
        \begin{enumerate}
            \item[(i)] $\tau\subset\sigma$
            \item[(ii)] $Z_\sigma\cap\overline{Z_\tau} \neq \emptyset$
        \end{enumerate}
    \end{corollary}
    \begin{proof}
        We can check quickly that one direction follows from Proposition \ref{prop:stratification-structure} (a). 
        Then, we assume that statement (ii) holds. 
        From Proposition \ref{prop: closure of stratum}, $\overline{Z_\tau} = Z\cap\overline{O_\tau}$. 
        This indicates $O_\sigma\cap \overline{O_\tau} \neq\emptyset$. 
        Thus, the statement (i) holds from the argument of toric varieties. 
    \end{proof}
    Fundamentally, one notable point of the superiority of \cite{NO22} is to focus on tropical compactification. 
    Tropical compactification was initially defined in \cite{Tev} and is known for various applications. 
    Interestingly, the following proposition asserts that a proper mock toric variety $Z$ is a tropical compactification of $Z\cap T_N$. 
    This result indicates that the method of \cite{NO22} can be applied to the mock toric varieties.
    \begin{proposition}\label{prop:tropical}
        Let $Z$ be a scheme over $k$ and $(N, \Delta, \iota, \Phi, \{N_\varphi\}_{\varphi\in\Phi},\{\Delta_\varphi\}_{\varphi\in\Phi})$ be a mock toric structure of $Z$. 
        Let $T$ denote $T_N$. 
        Then the following action map $m$ is faithfully flat:
        \[
            m\colon T\times Z\rightarrow X(\Delta)
        \]
        Moreover, if $Z$ is proper over $k$, then $Z$ is a tropical compactification of $Z^\circ$.  
    \end{proposition}
    \begin{proof}
        From condition (4) in Definition \ref{def:mock-toric}, $m$ is surjective. 
        Thus, it is enough to show that the restriction $m|_{T\times Z(\sigma)}\colon T\times Z(\sigma)\rightarrow X(\sigma)$ is flat for any $\sigma\in\Delta$. 
        Let $\varphi\in\Phi$ be an element such that $\sigma\in\Delta_\varphi$, and $\sigma_\varphi$ denote $(q_\varphi)_\RR(\sigma)$. 
        From Lemma \ref{lem:torus-fibration}, there exists the toric isomorphism $\lambda\colon X(\sigma)\rightarrow X(\sigma_\varphi)\times T'$ such that $\pr_1\circ \lambda = (q_\varphi)_*$, where $T'$ denotes an algebraic torus with $\dim(T') = \rank(N_\varphi)$. 
        Let $T_\varphi$ denote $T_{N/N_\varphi}$, $m_\varphi$ denote an action map $T_\varphi\times X(\sigma_\varphi)\rightarrow X(\sigma_\varphi)$, and $m'$ denote the multiplicative morphism $T'\times T'\rightarrow T'$. 
        We remark that $\lambda$ induces the following commutative diagram:
        \begin{equation*}
            \begin{tikzcd} 
                T\times X(\sigma)\ar[r, "m"]\ar[d, "\lambda\times\lambda"]& X(\sigma)\ar[d, "\lambda"]\\
                (T_\varphi\times X(\sigma_\varphi))\times (T'\times T')\ar[r, "m_\varphi\times m'"]& X(\sigma_\varphi)\times T'
            \end{tikzcd}
        \end{equation*}
        Moreover, we can check that $\lambda|_{T}\colon T\rightarrow T_\varphi\times T'$ is isomorphic. 
        From the condition (5) in Definition \ref{def:mock-toric}, $\iota_\varphi(Z(\sigma))$ is an open subset of $X(\sigma_\varphi)$. 
        Let $V$ denote $\iota_\varphi(Z(\sigma))$, $j\colon V\rightarrow Z(\sigma)$ denote the inverse morphism of $\iota_\varphi$, and $p$ denote the composition $\pr_2\circ\lambda\circ\iota\colon 
        Z\rightarrow T'$. 
        Let $\theta$ denote the morphism from $T'\times Z(\sigma)\rightarrow T'\times V$ defined as follows:
        \[
            T'\times Z(\sigma)\ni(t', z)\mapsto (t'\cdot p(z), \iota_\varphi(z))\in T'\times V
        \]
        Then we claim the following statements:
        \begin{itemize}
            \item[(1)] The morphism $\theta$ is an isomorphism. 
            \item[(2)] The following diagram is commutative:
            \begin{equation*}
                \begin{tikzcd} 
                    T\times Z(\sigma)\ar[r, "\mathrm{id}_{T}\times\iota"]\ar[d, "\lambda\times\mathrm{id}_{Z(\sigma)}"]& T\times X(\sigma)\ar[r, "m"]& X(\sigma)\ar[dd, "\lambda"]\\
                    T_\varphi\times T'\times Z(\sigma)\ar[d, "\mathrm{id}_{T_\varphi}\times\theta"]
                    & & \\
                    T_\varphi\times T'\times V\ar[r, "l"]&(T_\varphi\times T')\times (X(\sigma_\varphi)\times T')\ar[r, "m_\varphi\times m'"]&X(\sigma_\varphi)\times T', 
                \end{tikzcd}
            \end{equation*}
            where $l$ denotes a morphism defined as follows:
            \[
                T_\varphi\times T'\times V\ni(t, t', v)\mapsto((t, t'), (v, 1))\in (T_\varphi\times T')\times(X(\sigma_\varphi)\times T')
            \]
        \end{itemize}
        Indeed, we can check that the inverse morphism of $\theta$ 
        is defined as follows:
        \[
            T'\times V\ni(t', v)\mapsto(t'\cdot {p(j(v))}^{-1}, j(v))\in T'\times Z(\sigma)
        \]
        Moreover, we can check the commutativity of the above diagram because of the definition of $\theta$ and the above remark. 

        Then the composition $(m_\varphi\times m')\circ l$ is equal to the restriction $(m_\varphi\times\mathrm{id}_{T'})|_{(T_\varphi\times V)\times T'}\colon$ $(T_\varphi\times V)\times T'\rightarrow X(\sigma_\varphi)\times T'$. 
        Thus, $(m_\varphi\times m')\circ l$ is flat because $V$ is an open subset of $X(\sigma_\varphi)$ and $m_\varphi$ is flat. 
        In particular, the action map $m\colon T\times Z(\sigma)\rightarrow X(\sigma)$ is flat from the above claim. 
    \end{proof}
    The following proposition shows that the smoothness of the ambient space induces the smoothness of Z.  
    \begin{proposition}\label{prop:fan-varphi-smooth}
        Let $Z$ be a scheme over $k$ and $(N, \Delta, \iota, \Phi, \{N_\varphi\}_{\varphi\in\Phi},\{\Delta_\varphi\}_{\varphi\in\Phi})$ be a mock toric structure of $Z$. 
        We assume that $\Delta$ is unimodular. 
        Then the following statements follow: 
        \begin{enumerate}
            \item[(a)] For any $\varphi\in\Phi$, $\Delta(\varphi)$ is unimodular. 
            \item[(b)] A scheme $Z$ is smooth scheme over $k$.
        \end{enumerate}
        \begin{proof}
            We prove the statements from (a) to (b).
            \begin{enumerate}
                \item[(a)] Let $\sigma\in\Delta$ be a cone and $\varphi\in\Phi$ be an element such that $\sigma\in\Delta_\varphi$. 
                Let $\sigma_\varphi$ denote $(q_\varphi)_\RR(\sigma)$. 
                From the condition (1), Proposition \ref{prop: first prop} (b), and Lemma \ref{lem:take-section}, there exists a sub lattice $N_1$ of $N$ such that $\langle\sigma\rangle\cap N \subset N_1$ and $N_1\oplus N_\varphi = N$. 
                Because $N_1/(\langle\sigma\rangle\cap N)$ is a torsion free, there exists a sub lattice $N_2$ of $N_1$ such that $(\langle\sigma\rangle\cap N)\oplus N_2 = N_1$. 
                Thus, we have $(\langle\sigma\rangle\cap N)\oplus N_2 \oplus N_\varphi = N$. 
                Because $\sigma$ is unimodular, there exists $v_1, \ldots, v_r\in\sigma$ such that $\{v_i\}_{1\leq i\leq r}$ is a basis of $\langle\sigma\rangle\cap N$ and $\{\RR_{\geq 0}v_i\}_{1\leq i\leq r}$ are all rays of $\sigma$, where $r$ is $\dim(\sigma)$. 
                From Lemma \ref{lem:torus-fibration} (c), $\{\RR_{\geq 0}q_\varphi(v_i)\}_{1\leq i\leq r}$ is the set of all rays of $\sigma_\varphi$. 
                Moreover, from the above disjoint sum, $\{q_\varphi(v_i)\}$ is a basis of  $\langle\sigma_\varphi\rangle\cap (N/N_\varphi)$. 
                Therefore, $\sigma_\varphi$ is unimodular. 
                \item[(b)] From (a), for any $\varphi\in\Phi$, $X(\Delta(\varphi))$ is smooth over $k$. 
                Thus, from condition (5) in Definition \ref{def:mock-toric}, $Z\cap X(\Delta_\varphi)$ is smooth over $k$. 
                Therefore, $Z$ is smooth over $k$. 
            \end{enumerate}
        \end{proof}
    \end{proposition}

\section{Induced mock toric structure}
    Let $Z$ be a scheme over $k$ and $(N$, $\Delta$, $\iota$, $\Phi$, $\{N_\varphi\}_{\varphi\in\Phi}$, $\{\Delta_\varphi\}_{\varphi\in\Phi})$ be a mock toric structure of $Z$. 
    In this section, we construct other mock toric varieties from $Z$. 
    To be precise, we construct two mock toric structures for these schemes:
    \begin{itemize}
        \item[(1)] A scheme $Z\times_{X(\Delta)}\overline{O_\sigma}$, where $\sigma\in\Delta$. 
        This is a scheme theoretic closure of $Z_\sigma$ in $Z$ from Proposition \ref{prop: closure of stratum}
        \item[(2)] The pullback of $Z$ along a dominant toric morphism $X(\Delta')\rightarrow X(\Delta)$
    \end{itemize}
    \subsection{Induced mock toric structure for orbit closure}
        Let $Z$ be a scheme over $k$ and $(N$, $\Delta$, $\iota$, $\Phi$, $\{N_\varphi\}_{\varphi\in\Phi}$, $\{\Delta_\varphi\}_{\varphi\in\Phi})$ be a mock toric structure of $Z$. 
        Let $\sigma\in\Delta$ be a cone. 
        From Proposition \ref{prop: closure of stratum}, $\overline{Z_\sigma}$ is a closed subscheme of toric variety $\overline{O_\sigma}$. 
        In this subsection, we construct a mock toric structure of $\overline{Z_\sigma}$. 
    \begin{definition}\label{def: mock induced for orbit}
        We use the above notation. 
        Let $N[\sigma]$ denote the quotient lattice $N/(\langle\sigma\rangle\cap N)$ and $\pi^\sigma$ denote the quotient map $N\rightarrow N[\sigma]$. 
        Let $\Delta^{\sigma}$ denote the following sub fan of $\Delta$: 
        \[
            \Delta^{\sigma} = \{\gamma\in\Delta\mid \exists\tau\in\Delta\ \ \mathrm{s.t. }\ \ \gamma\cup\sigma\subset\tau\}
        \]
        , and $\Delta[\sigma]$ denote the following set:
        \[
            \Delta[\sigma] = \{(\pi^\sigma)_\RR(\tau)\mid\sigma\subset\tau\in\Delta\}  
        \]
        From Proposition \ref{prop:orbit-toric}, $\Delta[\sigma]$ is a strongly convex rational polyhedral fan, and we regard $Z\cap\overline{O_\sigma}$ as closed subscheme of $X(\Delta[\sigma])$ from the following diagram whose all squares in the following diagram are Cartesian squares:
        \begin{equation*}
            \begin{tikzcd} 
                Z\cap\overline{O_\sigma}\ar[d, hook]\ar[r, hook, "\iota"]& \overline{O_\sigma}\ar[dr, "\simeq"]\ar[d, hook]&\\
                Z\cap X(\Delta^\sigma)\ar[r, hook, "\iota"]\ar[d]& X(\Delta^\sigma)\ar[d]\ar[r, "(\pi^\sigma)_*"]&X(\Delta[\sigma])\\
                Z\ar[r, hook, "\iota"]& X(\Delta) &
            \end{tikzcd}
        \end{equation*}
        Let $\iota^\sigma$ denote a closed immersion from $Z\cap\overline{O_\sigma}$ to $X(\Delta[\sigma])$ in the above diagram. 
        Let $\Phi^\sigma$ denote a subset of $\Phi$ as follows:
        \[
            \Phi^\sigma = \{\varphi\in\Phi\mid \sigma\in\Delta_\varphi\}
        \]
        For $\varphi\in\Phi^\sigma$, let $N[\sigma]_\varphi$ denote $\pi^\sigma(N_\varphi)$ and $\Delta[\sigma]_\varphi$ denote the following set: 
        \[
            \Delta[\sigma]_\varphi = \{\pi^\sigma_\RR(\tau)\in\Delta[\sigma]\mid \tau\in \Delta_\varphi \ \ \mathrm{and}\ \ \sigma\subset\tau\}
        \]
    \end{definition}
    In the following proposition, we check that $N[\sigma]_\varphi$ is a lattice and $\Delta[\sigma]_\varphi$ is a subfan of $\Delta[\sigma]$ for any $\varphi\in\Phi^\sigma$. 
    \begin{proposition}
        We keep the notation in Definition \ref{def: mock induced for orbit} and let $\varphi\in\Phi^\sigma$ be an element.  
        Then $N[\sigma]_\varphi$ is a lattice and $\Delta[\sigma]_\varphi$ is a subfan of $\Delta[\sigma]$. 
    \end{proposition}
    \begin{proof}
        A module $N/(\langle\sigma\rangle\cap N)$ is torsion-free, so that subgroup $\pi^\sigma(N_\varphi)$ of $N[\sigma]$ is torsion free too. 
        Thus, we only need to show that for any $\tau'\in \Delta[\sigma]_\varphi$, all faces of $\tau'$ are in $\Delta[\sigma]_\varphi$. 
        There exists $\tau\in\Delta$ such that $(\pi^\sigma)_\RR(\tau) = \tau'$, $\tau\in\Delta_\varphi$, and $\sigma\subset\tau$. 
        Because $\Delta[\sigma]$ is a fan, for any face $\gamma'\preceq\tau'$, there exists $\gamma\in\Delta$ such that $(\pi^\sigma)_\RR(\gamma) = \gamma'$ and $\sigma\subset\gamma$. 
        From  Corollary \ref{cor: cor-orbit-toric}(a), we have $\gamma\subset\tau$. 
        Because $\Delta_\varphi$ is a fan, $\gamma\in\Delta_\varphi$, in particular, $\gamma'\in\Delta[\sigma]_\varphi$. 
    \end{proof}
    The following proposition shows that $\overline{Z_\sigma}$ has a mock toric structure. 
    \begin{proposition}\label{prop:orbit-structure}
        We keep the notation in Definition \ref{def: mock induced for orbit}.  
        Then $(N[\sigma], \Delta[\sigma], $ $\iota^\sigma, \Phi^\sigma,$ $\{N[\sigma]_\varphi\}_{\varphi\in\Phi^\sigma},\ \ \ $ $\{\Delta[\sigma]_\varphi\}_{\varphi\in\Phi^\sigma})$ is a mock toric structure of $Z\cap\overline{O_\sigma}$. 
    \end{proposition}
    \begin{proof}
        We check that this tuple satisfies the conditions in Definition \ref{def:mock-toric} from (1) to (5) in order. 
        \begin{enumerate}
            \item[(1)] We show that for any $\varphi\in\Phi^\sigma$ and $\tau'\in\Delta[\sigma]_\varphi$, $N[\sigma]/((\langle\tau'\rangle\cap N[\sigma])+N[\sigma]_\varphi)$ is a torsion-free module. 
            Let $\tau\in\Delta_\varphi$ be a cone such that $(\pi^\sigma)_\RR(\tau) = \tau'$ and $\sigma\subset\tau$. 
            Before the proof, we remark that $\pi^\sigma(\langle\tau\rangle\cap N) = \langle\tau'\rangle \cap N[\sigma]$. 
            Indeed, for any $v'\in \langle\tau'\rangle \cap N[\sigma]$, there exists $v\in \langle\tau\rangle$ and $w\in N$ such that $v' = (\pi^\sigma)_\RR(v) = (\pi^\sigma)_\RR(w)$. 
            Thus, $v-w\in \langle\sigma\rangle$, in particular, $w\in\langle\tau\rangle$ because $\sigma\subset\tau$. 
            Therefore, $w\in\langle\tau\rangle\cap N$. 
            
            From the following equation, $N[\sigma]/((\langle\tau'\rangle\cap N[\sigma])+N[\sigma]_\varphi)$ is isomorphic to $N/((\langle\tau\rangle\cap N) + N_\varphi)$:
            \begin{align*}
                (\langle\tau'\rangle\cap N[\sigma])+N[\sigma]_\varphi &= \pi^\sigma(\langle\tau\rangle\cap N) + \pi^\sigma(N_\varphi)\\
                &= \pi^\sigma((\langle\tau\rangle\cap N)+ N_\varphi)
            \end{align*}
            Thus, $N[\sigma]/((\langle\tau'\rangle\cap N[\sigma])+N[\sigma]_\varphi)$ is a torsion-free module from the condition (1) in Definition \ref{def:mock-toric}. 
            \item[(2)] We show that the following equation holds:
            \[
                \Delta[\sigma] = \bigcup_{\varphi\in\Phi^\sigma}\Delta[\sigma]_\varphi.
            \]
            For $\tau'\in\Delta[\sigma]$, let $\tau\in\Delta$ be a cone such that $\sigma\subset\tau$ and $(\pi^\sigma)_\RR(\tau) = \tau'$. 
            From the condition (2) in Definition \ref{def:mock-toric}, there exists $\varphi\in\Phi$ such that $\tau\in\Delta_\varphi$. 
            Because $\Delta_\varphi$ is a subfan of $\Delta$, we have $\sigma\in\Delta_\varphi$. 
            Thus, $\varphi\in\Phi^\sigma$ and $\tau' = (\pi^\sigma)_\RR(\tau)\in\Delta[\sigma]_\varphi$. 
            \item[(3)] For $\varphi\in \Phi^\sigma$, let $q^\sigma_\varphi$ denote the quotient morphism 
            $N[\sigma]\rightarrow N[\sigma]/N[\sigma]_{\varphi}$. 
            We show that the map $(q^\sigma_\varphi)_\RR|_{\supp(\Delta[\sigma]_\varphi)}$ is injective. 
            Let $\pi^\sigma_\varphi\colon N/N_\varphi\rightarrow N/((\langle\sigma\rangle\cap N) + N_\varphi) = N[\sigma]/N[\sigma]_\varphi$ be the natural quotient map, and we remark that the following diagram is commutative:
            \begin{equation*}
                \begin{tikzcd} 
                    N\ar[r, "\pi^\sigma"]\ar[d, "q_\varphi"]& N[\sigma]\ar[d, "q^\sigma_\varphi"]\\
                    N/N_\varphi\ar[r, "\pi^\sigma_\varphi"]& N[\sigma]/N[\sigma]_\varphi
                \end{tikzcd}
            \end{equation*}
            Let $x'$ and $y'\in\supp(\Delta[\sigma]_\varphi)$ be elements. 
            Then there exists $\tau'$ and $\gamma'\in\Delta[\sigma]_\varphi$ such that $x'\in\tau'$ and $y'\in\gamma'$. 
            From the definition of $\Delta[\sigma]_\varphi$, there exists $\tau$ and $\gamma\in\Delta_\varphi$ such that $(\pi^\sigma)_\RR(\tau) = \tau'$, $(\pi^\sigma)_\RR(\gamma) = \gamma'$, and $\sigma\subset\tau\cap\gamma$. 
            In particular, there exists $x\in\tau$ and $y\in\gamma$ such that $(\pi^\sigma)_\RR(x) = x'$ and $(\pi^\sigma)_\RR(y) = y'$. 
            
            From now on, we assume that $(q^\sigma_\varphi)_\RR(x') = (q^\sigma_\varphi)_\RR(y')$ and we show that $x'=y'$. 
            From the assumption, $x - y\in \langle\sigma\rangle + (N_\varphi)_\RR$. 
            Thus, there exists $z, w\in\sigma$ such that $(x - y) - (z - w)\in (N_\varphi)_\RR$. 
            Hence, $(q_\varphi)_\RR(x + w) = (q_\varphi)_\RR(y + z)$. 
            Because $x\in\tau$, $y\in\gamma$ and $\sigma\subset\tau\cap\gamma$, we have $x+w\in\tau$ and $y+z\in\gamma$. 
            Thus, $x + w = y + z$ from the fact $x + w, y + z\in\supp(\Delta_\varphi)$ and the condition (3) in Definition \ref{def:mock-toric}. 
            Therefore, from the following equation, we have $x' = y'$:
            \begin{align*}
                x' &= (\pi^\sigma)_\RR(x)\\
                &= (\pi^\sigma)_\RR(x + w)\\
                &= (\pi^\sigma)_\RR(y + z)\\
                &= (\pi^\sigma)_\RR(y)\\
                &= y'
            \end{align*}
            \item[(4)] For any $\tau'\in\Delta[\sigma]$, we show that $O_{\tau'}\cap(Z\cap\overline{O_\sigma})\neq\emptyset$. 
            Let $\tau\in\Delta$ be a cone such that $\sigma\subset\tau$ and $\pi^\sigma_\RR(\tau) = \tau'$. 
            Then $O_{\tau'}$ is equal to $O_\tau$, so that $O_{\tau'}\cap(Z\cap\overline{O_\sigma})\neq\emptyset$ from the condition (4) in Definition \ref{def:mock-toric}. 
            \item[(5)] Let $\Delta[\sigma](\varphi)$ denote the following strongly convex rational polyhedral fan in $(N[\sigma]/N[\sigma]_\varphi)\otimes\RR$: 
            \[
                \Delta[\sigma](\varphi) = 
                \{
                    (q^\sigma_\varphi)_\RR(\tau')\mid \tau'\in\Delta[\sigma]_{\varphi}
                \}
            \]
            Then, we show that the following composition 
            $\iota^\sigma_\varphi$ is an open immersion: 
            \begin{equation*}
                \begin{tikzcd} 
                    (Z\cap \overline{O_\sigma})\cap X(\Delta[\sigma]_\varphi)\ar[r, "\iota^\sigma"]& X(\Delta[\sigma]_\varphi)\ar[r, "(q^\sigma_\varphi)_*"]& X(\Delta[\sigma](\varphi)), 
                \end{tikzcd}
            \end{equation*}
            where the former one is a closed immersion restricted to $(Z\cap \overline{O_\sigma})\cap X(\Delta[\sigma]_\varphi)$. 

            However, the proof of the above statement is very complicated, so we go through 4 steps. 
            \begin{enumerate}
                \item[Step.1] The codomain of $\iota^\sigma_\varphi$ is a toric variety, so it is enough to show that for any open affine sub toric variety $V\subset X(\Delta[\sigma](\varphi))$, the restriction $\iota^\sigma_\varphi|_{(\iota^\sigma_\varphi)^{-1}(V)}$ is an open immersion. 
                In this step, we show that for any $\tau''\in\Delta[\sigma](\varphi)$, there unique exists $\tau\in\Delta_\varphi$ such that $\sigma\subset\tau$ and $\tau'' = (q^\sigma_\varphi)_\RR\circ(\pi^\sigma)_\RR(\tau)$ and all squares of the following diagram are Cartesian squares: 
                \begin{equation*}
                    \begin{tikzcd} 
                        Z\cap(\overline{O_\sigma}\cap X(\tau))\ar[dd]\ar[r, "\iota"]&\overline{O_\sigma}\cap X(\tau)\ar[d]\\
                        \ &X(\tau)\ar[d, "\pi^\sigma_*"]\\
                        (Z\cap \overline{O_\sigma})\cap X((\pi^\sigma)_\RR(\tau))\ar[d]\ar[r, "\iota^\sigma"]& X((\pi^\sigma)_\RR(\tau))\ar[d]\ar[r, "(q^\sigma_\varphi)_*"]&X(\tau'')\ar[d]\\
                        (Z\cap \overline{O_\sigma})\cap X(\Delta[\sigma]_\varphi)\ar[r, "\iota^\sigma"]& X(\Delta[\sigma]_\varphi)\ar[r, "(q^\sigma_\varphi)_*"]&X(\Delta[\sigma](\varphi))
                    \end{tikzcd}
                \end{equation*}
                Indeed, from the argument in (3), there unique exists $\tau'\in\Delta[\sigma]_\varphi$ such that $(q^\sigma_\varphi)_\RR(\tau') = \tau''$. 
                From  Corollary \ref{cor: cor-orbit-toric}(a), there exists unique $\tau\in\Delta$ such that $\sigma\subset\tau$, $(\pi^\sigma)_\RR(\tau) = \tau'$, and $\tau\in\Delta_\varphi$. 
                Therefore $\tau$ is a unique cone in $\Delta_\varphi$ such that $\sigma\subset\tau$ and $\tau'' = (q^\sigma_\varphi)_\RR\circ(\pi^\sigma)_\RR(\tau)$.
                
                From the injectivity of the restriction $(q^\sigma_\varphi)_\RR|_{\supp(\Delta[\sigma]_\varphi)}$, the lower right square of the above diagram is a Cartesian square. 
                All squares on the left side are Cartesian squares because they are given only by a pullback from a closed immersion $\iota^\sigma\colon Z\cap\overline{O_\sigma}\hookrightarrow X(\Delta[\sigma])$. 
                \item[Step.2] We use the notation in Step.1 and proof of confirming the condition (5). 
                In this step, we show that the following diagram is commutative:
                \begin{equation*}
                    \begin{tikzcd} 
                        \overline{O_\sigma}\cap X(\tau)\ar[d]\ar[r]&\overline{O_{\sigma_\varphi}}\cap X(\tau_\varphi)\ar[d]\\
                        X(\tau)\ar[d, "\pi^\sigma_*"]\ar[r, "(q_\varphi)_*"]&X(\tau_\varphi)\ar[d, "(\pi^\sigma_\varphi)_*"]\\
                        X((\pi^\sigma)_\RR(\tau))\ar[r, "(q^\sigma_\varphi)_*"]&X(\tau'')\\
                    \end{tikzcd}
                \end{equation*}
                where $\sigma_\varphi$ denote $(q_\varphi)_\RR(\sigma)$ and $\tau_\varphi$ denote $(q_\varphi)_\RR(\tau)$. 
                
                Indeed, the upper square is commutative from Lemma \ref{lem:torus-fibration}(e), and the lower one is commutative because $q^\sigma_\varphi\circ\pi^\sigma = \pi^\sigma_\varphi\circ q_\varphi$. 
                \item[Step.3] In this step, we show that the right vertical morphism in the diagram in Step.2 from $\overline{O_{\sigma_\varphi}}\cap X(\tau_\varphi)$ to $X(\tau'')$ is isomorphic. 
                From the definition of $\pi^\sigma_\varphi$, we have $\ker(\pi^\sigma_\varphi) = ((\langle\sigma\rangle\cap N) + N_\varphi)/N_\varphi$.  
                Thus, $\ker(\pi^\sigma_\varphi) = q_\varphi(\langle\sigma\rangle\cap N)$. 
                
                Now, we claim that $q_\varphi(\langle\sigma\rangle\cap N) = \langle \sigma_\varphi\rangle\cap N/N_\varphi$. 
                One direction is obvious, so we only check that $\langle \sigma_\varphi\rangle\cap N/N_\varphi\subset q_\varphi(\langle\sigma\rangle\cap N)$. 
                Indeed, for any $x\in \langle (q_\varphi)_\RR(\sigma)\rangle\cap N/N_\varphi$, there exists $y\in\langle\sigma\rangle$ such that $(q_\varphi)_\RR(y) = x$. 
                Moreover, from the condition (1) in Definition \ref{def:mock-toric}, Proposition \ref{prop: first prop}, and Lemma \ref{lem:take-section}, there exists a section $s$ of $q_\varphi$ such that $s\circ q_\varphi|_{\langle\sigma\rangle\cap N}$ is an identity map of a lattice $\langle\sigma\rangle\cap N$. 
                Thus, from the construction of $s$, $s_\RR(x) = y$. 
                
                On the other hand, $x\in N/N_\varphi$, so that $s_\RR(x) = s(x) = y\in N$. 
                Thus, $x = q_\varphi(y)\in (q_\varphi)_\RR(\langle\sigma\rangle\cap N)$. 
                
                We already checked that $\ker(\pi^\sigma_\varphi) = \langle (q_\varphi)_\RR(\sigma)\rangle\cap N/N_\varphi$. 
                Moreover, We recall that $(\pi^\sigma_\varphi)_\RR(\tau_\varphi) = \tau''$. 
                Hence, from the proof of Proposition \ref{prop:orbit-toric}(d), the morphism from $\overline{O_{\sigma_\varphi}}\cap X(\tau_\varphi)$ to $X(\tau'')$ in the diagram in Step.2 is an isomorphism. 
                \item[Step.4] In this step, we show that the lower morphism from $(Z\cap \overline{O_\sigma})\cap X((\pi^\sigma)_\RR(\tau))$ to $X(\tau'')$ in the following diagram is an open immersion: 
                \begin{equation*}
                    \begin{tikzcd} 
                        Z\cap(\overline{O_\sigma}\cap X(\tau))\ar[dd]\ar[r, "\iota"]&\overline{O_\sigma}\cap X(\tau)\ar[d]\ar[r]&\overline{O_{\sigma_\varphi}}\cap X(\tau_\varphi)\ar[d]\\
                        \ &X(\tau)\ar[d, "\pi^\sigma_*"]\ar[r, "(q_\varphi)_*"]&X(\tau_\varphi)\ar[d, "(\pi^\sigma_\varphi)_*"]\\
                        (Z\cap \overline{O_\sigma})\cap X((\pi^\sigma)_\RR(\tau))\ar[r, "\iota^\sigma"]& X((\pi^\sigma)_\RR(\tau))\ar[r, "(q^\sigma_\varphi)_*"]&X(\tau'')
                    \end{tikzcd}
                \end{equation*}
                From Proposition \ref{prop:orbit-toric}(d), the vertical middle morphism from $\overline{O_\sigma}\cap X(\tau)$ to $X(\pi^\sigma(\tau))$ is isomorphic. 
                Hence, from Step.1, the left vertical morphism from $Z\cap(\overline{O_\sigma}\cap X(\tau))$ to $(Z\cap \overline{O_\sigma})\cap X((\pi^\sigma)_\RR(\tau))$ is isomorphic too. 
                Thus, it is enough to show that the morphism from $Z\cap(\overline{O_\sigma}\cap X(\tau))$ to $X(\tau'')$ in the above diagram is an open immersion. 
                Moreover, it is enough to show that the upper morphism from $Z\cap(\overline{O_\sigma}\cap X(\tau))$ to $\overline{O_{\sigma_\varphi}}\cap X(\tau_\varphi)$ is an open immersion because the right vertical morphism from $\overline{O_{\sigma_\varphi}}\cap X(\tau_\varphi)$ to $X(\tau'')$ is isomorphic from Step.3. 
                
                By the way, we already claimed that this morphism is an open immersion in the proof of Proposition \ref{prop: closure of stratum} because $(\overline{O_\sigma}\cap X(\tau))\times_{X(\tau)} Z(\tau)= Z\cap(\overline{O_\sigma}\cap X(\tau))$. 
                Therefore, the proof is completed. 
            \end{enumerate}
        \end{enumerate}
    \end{proof}
    \subsection{The mock toric structure induced by dominant toric morphisms}
        Let $Z$ be a scheme over $k$ and $(N$, $\Delta$, $\iota$, $\Phi$, $\{N_\varphi\}_{\varphi\in\Phi}$, $\{\Delta_\varphi\}_{\varphi\in\Phi})$ be a mock toric structure of $Z$. 
    
        Let $\pi \colon N'\rightarrow N$ be a surjective group morphism of lattices of finite rank. 
        Let $\Delta'$ be a strongly convex rational polyhedral fan in $N'_\RR$, and we assume that $\pi$ is compatible with the fans $\Delta'$ and $\Delta$. 
        Let $Z'$ denote the pull back of the following diagram and $\iota'\colon Z'\rightarrow X(\Delta')$ be a closed immersion:
        \begin{equation*}
            \begin{tikzcd} 
                Z'\ar[r, hook,"\iota'"]\ar[d]& X(\Delta')\ar[d, "\pi_*"]\\
                Z\ar[r, hook, "\iota"]& X(\Delta)
            \end{tikzcd}
        \end{equation*}
        In this subsection, we construct a mock toric structure of $Z'$. 
    \begin{definition}\label{def: mock induced for dominant}
        We use the above notation. 
        Let $s\colon N\hookrightarrow N'$ be a section of $\pi$. 
        For $\varphi\in\Phi$, let $N'_\varphi$ denote $s(N_\varphi)$ and $\Delta'_\varphi$ denote the subset of $\Delta'$ defined as follows:
        \[
            \Delta'_\varphi = \{\tau\in\Delta'\mid \pi_\RR(\tau)\subset\supp(\Delta_\varphi)\}
        \]
        We can check that $\Delta'_\varphi$ is a sub fan of $\Delta'$. 
    \end{definition} 
    The following proposition shows that $Z'$ has a mock toric structure. 
    \begin{proposition}\label{prop:relative-induced-structure}
        We keep the notation in Definition \ref{def: mock induced for dominant}.  
        Then $(N',$ $\Delta',$ $\iota',$ $\Phi,$ $\{N'_\varphi\}_{\varphi\in\Phi},$ $\{\Delta'_\varphi\}_{\varphi\in\Phi})$ is a mock toric structure of $Z'$. 
        We call that $Z'$ is a mock toric variety induced by $Z$ along $\pi_*\colon X(\Delta')\rightarrow X(\Delta)$ and $s$. 
    \end{proposition}
    \begin{proof}
        We check that this tuple satisfies the conditions in Definition \ref{def:mock-toric} from (1) to (5) in order. 
        \begin{enumerate}
            \item[(1)]  For any $\varphi\in\Phi$ and $\tau\in\Delta'_\varphi$, we show that $N'/((\langle\tau'\rangle\cap N')+s(N_\varphi))$ is a torsion-free module.  
            Let $x'\in N'$ be an element and $n\in\ZZ_{>0}$ be a positive integer. 
            We assume that $nx'\in (\langle\tau\rangle\cap N')+s(N_\varphi)$. 
            Then there exists $y'\in \langle\tau\rangle\cap N'$ and $z\in N_\varphi$ such that $nx' = y' + s(z)$. 
            This shows that $n\pi(x') = \pi(y') + z$.
            Because $\pi(y')\in\langle\sigma\rangle\cap N$ and $z\in N_\varphi$, there exists $y\in\langle\sigma\rangle\cap N$ and $w\in N_\varphi$ such that $\pi(x') = y + w$ from the condition (1) in Definition \ref{def:mock-toric}. 
            Moreover, from Proposition \ref{prop: first prop}(b), $z = n\cdot w$. 
            This shows that $nx' = y' + ns(w)$ and $x' - s(w)\in\langle\tau\rangle\cap N'$. 
            Thus, $x' = (x' - s(w)) + s(w)\in (\langle\tau\rangle\cap N')+s(N_\varphi)$, and the statement follows. 
            \item[(2)] We show the following equation holds:
            \[
                \Delta' = \bigcup_{\varphi\in\Phi}\Delta'_\varphi.
            \]
            For any $\tau\in\Delta'$, there exists $\sigma\in\Delta$ such that $\pi_\RR(\tau)\subset\sigma$. 
            Thus, from the condition (2) in Definition \ref{def:mock-toric}, there exists $\varphi\in\Phi$ such that $\sigma\in\Delta_\varphi$. 
            Therefore, from the definition of $\Delta'_\varphi$, we have $\tau\in\Delta'_\varphi$. 
            \item[(3)]  For $\varphi\in \Phi$, let $q'_\varphi$ denote the quotient morphism $N'\rightarrow N'/s(N_{\varphi})$. 
            We show that the map $(q'_\varphi)_\RR|_{\supp(\Delta'_\varphi)}$ is injective. 
            Let $x'$ and $y'\in\supp(\Delta'_\varphi)$ be elements, and assume that $(q'_\varphi)_\RR(x') = (q'_\varphi)_\RR(y')$. 
            We show that $x' = y'$. 
            From the assumption, there exists $z\in (N_\varphi)_\RR$ such that $x' - y' = s_\RR(z)$. 
            In particular, $\pi_\RR(x') - \pi_\RR(y') = z$. 
            From the definition of $\Delta'_\varphi$, $\pi_\RR(x')$ and $\pi_\RR(y')\in\supp(\Delta_\varphi)$, so that $\pi_\RR(x') = \pi_\RR(y')$ from the equation $(q_\varphi)_\RR(\pi_\RR(x')) = (q_\varphi)_\RR(\pi_\RR(y'))$ and the condition (3) in Definition \ref{def:mock-toric}. 
            Thus, it follows that $z = 0$ and $x' = y'$. 
            \item[(4)] For any $\tau\in\Delta'$, we show that $Z'\cap O_\tau\neq\emptyset$. 
            There exists $\sigma\in\Delta$ such that $\sigma$ is the smallest cone in $\Delta$ which contains $\pi_\RR(\tau)$. 
            In particular, $\pi_*(O_\tau)\subset O_\sigma$. 
            Because $\pi$ is surjective, $\pi_*(O_\tau) = O_\sigma$. 
            From the condition (4) in Definition \ref{def:mock-toric}, $Z\cap O_\sigma\neq\emptyset$. 
            Thus, from the definition of $Z'$, $Z'\cap O_\tau\neq\emptyset$ too. 
            \item[(5)] Let $\Delta'(\varphi)$ denote the following strongly convex rational polyhedral fan in $(N'/s(N_\varphi))\otimes\RR$: 
            \[
                \Delta'(\varphi) = \{(q'_\varphi)_\RR(\tau)\mid \tau\in\Delta'_{\varphi}\}
            \]
            Then we show that the following composition $\iota'_\varphi$ is an open immersion: 
            \begin{equation*}
                \begin{tikzcd} 
                    Z'\cap X(\Delta'_\varphi)\ar[r, "\iota'"]&X(\Delta'_\varphi)\ar[r, "(q'_\varphi)_*"]&X(\Delta'(\varphi))
                \end{tikzcd}
            \end{equation*}
            where the former is a closed immersion restricted to $Z'\cap X(\Delta'_\varphi)$. 
            The codomain of $\iota'_\varphi$ is a toric variety, so it is enough to show that for any open affine sub toric variety $V\subset X(\Delta'(\varphi))$, the restriction $\iota'_\varphi|_{(\iota'_\varphi)^{-1}(V)}$ is an open immersion. 
            In particular, from the definition of $\Delta'(\varphi)$, it is enough to show that the following morphism is an open immersion for any $\tau\in\Delta'_\varphi$: 
            \begin{equation*}
                \begin{tikzcd} 
                    Z'\cap X(\tau)\ar[r, "\iota'"]& X(\tau)\ar[r, "(q'_\varphi)_*"]&X(\tau_\varphi)
                \end{tikzcd}
            \end{equation*}
            where $\tau_\varphi$ denotes $(q'_\varphi)_\RR(\tau)$. 
            Let $\sigma\in\Delta_\varphi$ be a cone such as $\pi_\RR(\tau)\subset\sigma$ and $\sigma_\varphi$ denote $(q_\varphi)_\RR(\sigma)$. 
            There exists a unique lattice morphism $\pi_\varphi\colon N'/s(N_\varphi)\rightarrow N/N_\varphi$ such that $\pi_\varphi\circ q'_\varphi = q_\varphi \circ \pi$. 
            Moreover, from Lemma \ref{lem:relative-torus-fibration}(a) and (b), the right square of the following diagram is commutative and a Cartesian square: 
            \begin{equation*}
                \begin{tikzcd} 
                    Z'\cap X(\tau)\ar[r,"\iota'"]\ar[d]& X(\tau)\ar[r, "(q'_\varphi)_*"]\ar[d, "\pi_*"]& X(\tau_\varphi)\ar[d, "(\pi_\varphi)_*"]\\
                    Z\cap X(\sigma)\ar[r, "\iota"]& X(\sigma)\ar[r, "(q_\varphi)_*"]& X(\sigma_\varphi)
                \end{tikzcd}
            \end{equation*}
            where $(\pi_\varphi)_*$ is a toric morphism induced by $\pi_\varphi$. 
            On the other hand, the left square of the above diagram is a Cartesian square from the definition of $Z'$, so the largest square of the above diagram is a Cartesian square. 
            From the condition (5) in Definition \ref{def:mock-toric}, the lower morphism from $Z\cap X(\sigma)$ to $X(\sigma_\varphi)$ in the above diagram is an open immersion, so that $\iota'_\varphi$ is an open immersion too. 
        \end{enumerate}
    \end{proof}
    \begin{corollary}
    \label{cor: dominant morphism mock version}
        We keep the notation in Definition \ref{def: mock induced for dominant}.  
        Let $\mu$ denote the restriction  $(\pi_*)|_{Z'}\colon Z'\rightarrow Z$. 
        Then, $\mu$ is a dominant morphism. 
    \end{corollary}
    \begin{proof}
        From the definition of $Z'$, there exists the following Cartesian diagram: 
        \begin{equation*}
            \begin{tikzcd} 
                Z'^\circ\ar[r, hook,"\iota'"]\ar[d]& T_{N'}\ar[d, "\pi_*"]\\
                Z^\circ\ar[r, hook, "\iota"]& T_N
            \end{tikzcd}
        \end{equation*}
        Because $\pi_*\colon T_{N'} \rightarrow T_N$ is an algebraic torus fibration, the restriction $\mu|_{Z'^\circ}$ is so. 
        In particular, $\mu$ is dominant. 
    \end{proof}
    \subsection{Valuations of mock toric varieties}
    Let $Z$ be a scheme over $k$ and $(N$, $\Delta$, $\iota$, $\Phi$, $\{N_\varphi\}_{\varphi\in\Phi}$, $\{\Delta_\varphi\}_{\varphi\in\Phi})$ be a mock toric structure of $Z$. 
    We can regard a lattice point in $\supp(\Delta)$ as a torus invariant discrete valuation on $X(\Delta)$. 
    In this subsection, we try to induce discrete valuations on $Z$ from $\supp(\Delta)\cap N$. 
    In general, a valuation ring does not induce a valuation ring of a closed subscheme. 
    However, $Z$ is covered by open subschemes of toric varieties, and we can induce valuations on $Z$ from valuations on $X(\Delta)$. 

    Let $W$ be a variety over $k$ and $\Val_k(W)$ be a set of all discrete valuations on $W$, which is trivial on $k$. 
    We define a map $\val_Z\colon \supp(\Delta)\cap N\rightarrow \Val_k(Z)$ as follows: 
    For any $w\in \supp(\Delta)\cap N$, from the condition (4) in Definition \ref{def:mock-toric}, there exists $\varphi\in\Phi$ such that $w\in \supp(\Delta_\varphi)$.  
    Then $q_\varphi(w)$ can be regarded as a discrete valuation on $X(\Delta(\varphi))$, which is trivial on $k$. 
    The morphism $\iota_\varphi\colon Z\cap X(\Delta_\varphi)\rightarrow X(\Delta(\varphi))$ is an open immersion, so we define $\val_Z(w)$ as $(\iota_\varphi^{-1})_{\natural}(q_\varphi(w))$.  
    \begin{proposition}\label{prop:property-of-valuations}
    For the above construction of $\val_Z$, the following statements follow: 
        \begin{itemize}
            \item[(a)] The definition of $\val_Z$ is well-defined, i.e., the definition is independent of the choice of $\varphi\in\Phi$. 
            \item[(b)] For any $v\in\supp(\Delta)\cap N$, there unique exists $\sigma\in\Delta$ such that $v\in \sigma^\circ$. 
            Then the valuation ring associated with $\val_Z(v)$ dominates the local ring of the generic point of $Z_\sigma$. 
            \item[(c)] The map $\val_Z$ is injective.  
        \end{itemize}
    \end{proposition}
    \begin{proof}
        We prove these statements from (a) to (c) in order. 
        \begin{enumerate}
            \item[(a)] We divide the proof of (a) into three steps. 
            \begin{enumerate}
                \item[Step.1] Let $\Ray(\Delta)$ denote a subset of $\supp(\Delta)\cap N$ as follows:
                \[
                    \Ray(\Delta) = \{w\in \supp(\Delta)\cap N\mid \RR_{\geq0}\cdot w \in\Delta\}
                \]
                In this step, we show that for any $v\in\Ray(\Delta)$, the definition of $\val_Z(v)$ is well-defined. 
                Let $v\in\Ray(\Delta)$ be an element. 
                If $v = 0$, for any $\varphi$, we have $v = 0\in\supp(\Delta_\varphi)$. 
                However, $(\iota_\varphi^{-1})_{\natural}(q_\varphi(v))$ is a trivial valuation on $Z$, so we may assume that $v\neq 0$. 
                Let $\gamma$ be a 1-dimensional cone generated by $v$, so $\gamma\in\Delta$ from the definition of $\Ray(\Delta)$. 
                Then from Proposition \ref{prop:basic-property1}(d), $\dim(Z_\gamma) = \dim(Z) - 1$. 
                Besides that, from Proposition \ref{prop:stratification-structure}(c), $\overline{Z_\sigma}$ is a prime Weil divisor of a normal variety $Z$. 
                For any $\varphi\in\Phi$ such that $\gamma\in\Delta_\varphi$, let $\gamma_\varphi$ denote $(q_\varphi)_\RR(\gamma)$. 
                Then $q_\varphi(v)$ is a discrete valuation of the local ring of the generic point of $O_{\gamma_\varphi}$. 
                Thus, because of the proof of Proposition \ref{prop:basic-property1}(c), $(\iota_\varphi^{-1})_{\natural}(q_\varphi(v))$ is a discrete valuation of the local ring of the generic point of $\overline{Z_\gamma}$.  
                Therefore, it is enough to show that for any $\varphi\in\Phi$ such that $\gamma\in\Delta_\varphi$, the following subgroup of $\ZZ$ is independent of the choice of $\varphi\in\Phi$:
                \[
                    \{q_\varphi(v)(g)\mid g\in K(X(\Delta(\varphi)))^*\}
                \]
                Because $q_\varphi(v)$ is a torus invariant valuation on $X(\Delta(\varphi))$, the above set is equal to $q_\varphi(v)(M_\varphi)$, where $M_\varphi$ denotes the dual lattice of $N/N_\varphi$. 
                Let $(q_\varphi)^*$ denote a morphism $M_\varphi\rightarrow M$ induced by $q_\varphi$. 
                Then for any $\omega'\in M_\varphi$, the following equation holds:
                \begin{align*}
                    \langle q_\varphi(v), \omega'\rangle &= \langle v, (q_\varphi)^*(\omega')\rangle
                \end{align*}
                This indicates that $q_\varphi(v)(M_\varphi)\subset v(M)$. 
                On the other hand, from Lemma \ref{lem:take-section}, there exists a section $s$ of $q_\varphi\colon N\rightarrow N/N_\varphi$ such that an image of $s$ contains $v$. 
                In particular, $s\circ q_\varphi(v) = v$. 
                Thus, for any $\omega\in M$, the following equation holds: 
                \begin{align*}
                    \langle v, \omega\rangle &= \langle s\circ q_\varphi(v), \omega\rangle\\
                    &= \langle q_\varphi(v), s^*(\omega) \rangle
                \end{align*}
                This shows that $v(M)\subset q_\varphi(v)(M_\varphi)$. 
                Hence, $v(M) = q_\varphi(v)(M_\varphi)$, so that the definition of $\val_Z(v)$ is independent of the choice of $\varphi\in\Phi$. 
                \item[Step.2] Let $\Delta'$ be a refinement of $\Delta$, $\lambda$ denote an identity map of $N$, $Z'$ denote a mock toric variety induced by $Z$ along $\lambda_*\colon X(\Delta')\rightarrow X(\Delta)$ and $\lambda^{-1}$, $\iota'$ denote the closed immersion $Z'\rightarrow X(\Delta')$, and $\mu$ denote a morphism $Z'\rightarrow Z$. 
                In this step, we show that for any $w \in\Ray(\Delta')$, the definition of $\val_Z(w)$ is well-defined. 

                Let $0\neq w\in\Ray(\Delta')$ be an element, $\gamma'$ denote a 1-dimensional cone generated by $w$, and $\varphi\in\Phi$ be an element such that $w = \lambda(w)\in\supp(\Delta_\varphi)$. 
                Then $\gamma'\in\Delta'$ and $\gamma'\in\Delta'_\varphi$. 
                From the proof of Proposition \ref{prop:relative-induced-structure}, the following diagram is commutative:
                \begin{equation*}
                    \begin{tikzcd} 
                        {Z'}^\circ\ar[r,"\iota'"]\ar[d, "\mu"]& T_N \ar[r, "(q_\varphi)_*"]\ar[d, "\lambda_*"]& T_{N/N_\varphi}\ar[d, "(\lambda_\varphi)_*"]\\
                        Z^\circ\ar[r, "\iota"]& T_N\ar[r, "(q_\varphi)_*"]& T_{N/N_\varphi}
                    \end{tikzcd}
                \end{equation*}
                where $(\lambda_\varphi)_*$ is a toric morphism induced by an identitiy map $\lambda_\varphi\colon N/N_\varphi \rightarrow N/N_\varphi$ and $\mu$ denote $\lambda_*|_{Z'^\circ}$. 
                Because the toric morphism $\lambda_*$ is an identity morphism of $T_N$, $\mu$ is an identity morphism, and the induced map $\mu_\natural\colon \Val_{k}(Z')\rightarrow \Val_{k}(Z)$ is an identity map too. 
                From the above diagram, Lemma \ref{lem:val-val}, and the equation $w = \lambda(w)$, we have $(\iota^{-1}_\varphi)_\natural (q_\varphi(v)) = \mu_\natural(\val_{Z'}(w))$. 
                Thus, the definition of $\val_Z(w)$ is independent of the choice of $\varphi\in\Phi$ such that $w\in\supp(\Delta_\varphi)$. 
                \item[Step.3] In this step, for general $w\in\supp(\Delta)\cap N$, the definition of $\val_Z(w)$ is well-defined. 
                Let $w\in\supp(\Delta)\cap N$ be an element. 
                Then there exists a refinement $\Delta'$ of $\Delta$ such that $w\in\Ray(\Delta')$. 
                Thus, from the argument in Step.2, the definition of $\val_Z(w)$ is well-defined. 
            \end{enumerate}
            \item[(b)] From the proof of Proposition \ref{prop:basic-property1}(c) and the definition of $\val_Z$, we can check the statement easily. 
            \item[(c)] Let $v$ and $v'\in\supp(\Delta)\cap N$ be elements such that $\val_Z(v) = \val_Z(v')$. 
            Then, from (b), there exists $\sigma\in\Delta$ such that $\val_Z(v)$ is a valuation ring that dominates a local ring of generic point of $Z_\sigma$. 
            In particular, $v$ and $v'\in\sigma^\circ$. 
            Let $\varphi\in\Phi$ be an element such that $\sigma\in\Delta_\varphi$. 
            Then from the definition of $\val_Z$, $(\iota^{-1}_\varphi)_\natural (q_\varphi(v)) = (\iota^{-1}_\varphi)_\natural (q_\varphi(v'))$, in particular, $q_\varphi(v) = q_\varphi(v')$. 
            Therefore, from the condition (3) in Definition \ref{def:mock-toric}, we have $v = v'$. 
        \end{enumerate}
    \end{proof}
    For any $\sigma\in\Delta$, $Z(\sigma)$ is an affine scheme from Proposition \ref{prop:stratification-structure}(b). 
    The normal affine schme $Z(\sigma)$ contains an affine scheme $Z^\circ$ as an open subscheme. 
    The following proposition shows that the global section ring of $Z(\sigma)$ is characterized by the valuations $\{\val_Z(v)\mid v\in\sigma\cap N\}$ like as affine toric varieties. 
    \begin{proposition}\label{prop:characterization-of-affine mock toric}
        Let $\sigma\in\Delta$ be a cone. 
        We can regard $\Gamma(Z(\sigma), \OO_Z)$ as a subring of $\Gamma(Z^\circ, \OO_Z) = k[Z^\circ]$. 
        Then the following equation holds:
        \[
                \Gamma(Z(\sigma), \OO_Z) = \{f\in k[Z^\circ]\mid \val_Z(w)(f) \geq 0. \quad (\forall w\in \sigma\cap N)\}
        \]  
    \end{proposition}
    \begin{proof}
        First, we show that the right-hand side of the above equation contains $\Gamma(Z(\sigma), \OO_Z)$. 
        For any $v\in\sigma\cap N$, there exists a face $\tau\preceq\sigma$ such that $v\in\tau^\circ$. 
        From Proposition \ref{prop:property-of-valuations}(b), The valuation ring of $\val_Z(v)$ dominates the local ring of the generic point of $Z_\tau$. 
        Moreover, from Proposition \ref{prop:stratification-structure}(a), $Z_\tau\subset Z(\sigma)$. 
        Thus, for any $f\in\Gamma(Z(\sigma), \OO_Z)$, we have $\val_Z(v)(f)\geq 0$. 

        Next, we show that $\Gamma(Z(\sigma), \OO_Z)$ contains the right-hand side of the above equation. 
        Let $f\in k[Z^\circ]$ be a polynomial such that $f$ is contained in the right-hand side of the above equation and let $W$ denote $Z(\sigma)\setminus Z^\circ$. 
        For showing $f\in\Gamma(Z(\sigma), \OO_Z)$, it is enough to show that $f\in R$ for all local rings $R$ of the generic points of prime Weil divisors of $Z$ in $W$, because $Z(\sigma)$ is a normal affine scheme of finite type over $k$. 
        From the argument of the toric varieties and Proposition \ref{prop:basic-property1}(d) and Proposition \ref{prop:stratification-structure}(a), the set of all prime Weil divisors in $W$ is the following set:
        \[
            \{\overline{Z_\gamma}\mid \gamma\preceq\sigma, \dim(\gamma) = 1\}
        \]
        Thus, from Proposition \ref{prop:property-of-valuations}(b), for each prime Weil divisor $D$ of $Z$ in $W$, there exists $v_D\in\sigma\cap N$ such that $\val_Z(v_D)$ is a valuation of the local ring $R_D$ of the generic point of $D$. 
        Therefore, from the assumption of $f$, we have $f\in R_D$. 
    \end{proof}
	Let $\omega\in M$ be an element. 
	The following lemma shows that we can compute $\val_Z(v)(\iota^*(\chi^\omega))$ briefly for any $v\in \supp(\Delta)\cap N$. 
    \begin{lemma}\label{lem: supplement-of-valuation1}
        Let $v\in\supp(\Delta)\cap N$ and $\omega\in M$ be elements. 
        Then $\val_Z(v)(\iota^*(\chi^\omega)) = \langle v, \omega\rangle$. 
    \end{lemma}
    \begin{proof}
        There exists $\varphi\in\Phi$ and $\sigma\in\Delta_\varphi$ such that $v\in \sigma$. 
        From Lemma \ref{lem:take-section}, there exists a section $s$ of $q_\varphi\colon N\rightarrow N/N_\varphi$ such that $\langle\sigma\rangle\cap N\subset s(N/N_\varphi)$. 
        Let $s^*$ denotes the morphism $M\rightarrow M_\varphi$ induced by $s$, and $\omega'\in M_\varphi$ denotes $s^*(\omega)$. 
        Then $\omega - q^*_\varphi(\omega')\in\ker(s^*)$ holds. 
        Let $v'\in \sigma\cap N$ be an element. 
        Then there exists $w'\in N/N_\varphi$ such that $s(w') = v'$. 
        Thus, there exists the following equation:
        \begin{align*}
            \langle v', \omega-q^*_\varphi(\omega')\rangle &= \langle s(w'), \omega-q^*_\varphi(\omega')\rangle\\
            &= \langle w', s^*(\omega-q^*_\varphi(\omega'))\rangle\\
            &= 0
        \end{align*}
        Hence, the equation above indicates that $\chi^{\omega-q^*_\varphi(\omega')}\in \Gamma(X(\sigma), \OO_{X(\Delta)})^*$. 
        In particular, $\iota^*(\chi^{\omega-q^*_\varphi(\omega')})\in \Gamma(Z(\sigma), \OO_Z)^*$. 
        Thus, the following equation holds: 
        \begin{align*}
            \val_Z(v)(\iota^*(\chi^\omega)) &= \val_Z(v)(\iota^*(\chi^{q^*_\varphi(\omega')}))\\
            &= \val_Z(v)(\iota^*_\varphi(\chi^{\omega'}))\\
            &=\langle q_\varphi(v), \omega')\rangle\\
            &=\langle v, q^*_\varphi(\omega')\rangle\\
            &= \langle v, \omega\rangle
        \end{align*}
    \end{proof}

    The base change of the field extension preserves the toric structure.  
    The following proposition shows that the mock toric structure is preserved by the base change of the field extension, too. 
    \begin{proposition}\label{prop:ext'n}
        Let $L/K$ be a field extension. 
        Let $Z$ be a scheme over $K$ and $(N$, $\Delta$, $\iota$, $\Phi$, $\{N_\varphi\}_{\varphi\in\Phi}$, $\{\Delta_\varphi\}_{\varphi\in\Phi})$ be a mock toric structure of $Z$. 
        Let $Z_L$ denote a scheme over $L$ defined as the following Cartesian diagram: 
        \begin{equation*}
            \begin{tikzcd} 
                Z_L\ar[r, hook, "\iota_L"]\ar[d, "\alpha"]& X_L(\Delta)\ar[d]\ar[r]& \Spec(L)\ar[d]\\
                Z\ar[r, hook, "\iota"]& X_K(\Delta)\ar[r]& \Spec({K})
            \end{tikzcd}
        \end{equation*}
        where $\alpha$ denote the morphism $Z_L\rightarrow Z$ and $\iota_L$ denote the closed immersion $Z_L\rightarrow X_L(\Delta)$. 
        
        Then the following statements hold: 
        \begin{enumerate}
            \item[(a)] The tuple $(N$, $\Delta$, $\iota_L$, $\Phi$, $\{N_\varphi\}_{\varphi\in\Phi}$, $\{\Delta_\varphi\}_{\varphi\in\Phi})$ is a mock toric structure of $Z_L$. 
            \item[(b)] A morphism $\alpha$ is a dominant morphism. 
            \item[(c)] The following diagram is commutative: 
            \begin{equation*}
                \begin{tikzcd} 
                    \supp(\Delta)\cap N\ar[dr,"\val_{Z}"]\ar[d, "\val_{Z_L}"]& \\
                    \Val_{L}(Z_L)\ar[r, "\alpha_\natural"]& \Val_{K}(Z)
                \end{tikzcd}
            \end{equation*}
        \end{enumerate}
    \end{proposition}
    \begin{proof}
        We prove these statements from (a) to (c) in order. 
        \begin{enumerate}
            \item[(a)] Because closed and open immersions are stable under a base change, we can check it easily.  
            \item[(b)] Because $\Spec(L)\rightarrow \Spec(K)$ is surjective, $\alpha$ is surjective too. 
            \item[(c)] There exists the following commutative diagram:
                \begin{equation*}
                    \begin{tikzcd} 
                        Z_L\cap X_L(\Delta_\varphi)\ar[r,"\iota_L"]\ar[d, "\alpha"]& X_L(\Delta_\varphi)\ar[r, "(q_\varphi)_{*, L}"]\ar[d, "\beta"]& X_L(\Delta(\varphi))\ar[d, "\beta_\varphi"]\\
                        Z\cap X_K(\Delta_\varphi)\ar[r,"\iota'"]& X_K(\Delta_\varphi)\ar[r, "(q_\varphi)_*"]& X_K(\Delta(\varphi))\\
                    \end{tikzcd}
                \end{equation*}
                where $\beta$ and $\beta_\varphi$ are morphisms induced by a base change along $\Spec(L)\rightarrow \Spec(K)$. 
                We remark that $\beta$ and $\beta_\varphi$ are dominant. 
                Let $v\in \supp(\Delta)\cap N$ be an element and $v_L$ denote a torus invariant valuation of $T_{N, L}$ which is trivial on $L$ associated with $v\in N$. 
                Then $\beta_\natural(v_L) = v$. 
                Thus, the statement follows from Lemma \ref{lem:val-val}. 
            \end{enumerate}
    \end{proof}
    \subsection{Extended valuation}
    We use the notation in the subsection 4.2. 
    In subsection 4.3, we define the valuation $\val_Z(v)$ for a mock toric variety $Z$ and any $v\in \supp(\Delta)\cap N$. 
    Let $Z'\rightarrow Z$ be a morphism of mock toric varieties where the mock toric structure of $Z'$ is induced by that of $Z$ along a dominant toric morphism $X(\Delta')\rightarrow X(\Delta)$. 
    Similarly, we can define the valuation $\val_{Z'}(v')$ for any $v'\in \supp(\Delta')\cap N'$. 
    
    In this subsection, we show that $Z'$ has more valuations induced by $Z$. 
    In the case of toric varieties, the above statement can be explained as follows:
    We assume that $N = 0$ and $\Delta' = \{0\}$. 
    Then we have $Z = \Spec(k)$ and $Z' = T_{N'}$. 
    Although $\supp(\Delta')$ is a point, we know there exist infinite torus invariant valuations on $T_{N'}$, which correspond to the lattice points in $N'$. 
    Thus, this subsection aims to define valuations associated with lattice points outside $\supp(\Delta')\cap N'$. 
    \begin{definition}\label{def: extend mock toric}
        We use the following notation:
            \begin{itemize}
                \item Let $Z$ be a scheme over $k$ and $(N$, $\Delta$, $\iota$, $\Phi$, $\{N_\varphi\}_{\varphi\in\Phi}$, $\{\Delta_\varphi\}_{\varphi\in\Phi})$ be a mock toric structure of $Z$. 
                \item Let $N'$ be a lattice of finite rank, $\pi\colon N'\rightarrow N$ be a surjective map, and $s$ be a section of $\pi$. 
                Let $M'$ denote the dual lattice of $N'$. 
                \item Let $N_1$ denote $N'/s(N)$ and $p$ denote the quotient map $N'\rightarrow N_1$. 
                We remark that $\pi\times p\colon N'\rightarrow N\oplus N_1$ is an isomorphism. 
                \item Let $M_1$ denote the dual lattice of $N_1$. 
                We remark that $\pi^*\oplus p^*\colon M\oplus M_1\rightarrow M'$ is an isomorphism.  
                \item Let $p_\varphi$ and $\pi_\varphi$ denote unique group morphisms such that the following diagrams commute: 
                \begin{equation*}
                    \begin{tikzcd}
                        N'\ar[r, "p"]\ar[d, "q'_\varphi"]& N_1\\
                        N'/s(N_\varphi)\ar[ur, "p_\varphi"]
                    \end{tikzcd}
                \end{equation*}
                \begin{equation*}
                    \begin{tikzcd}
                        N'\ar[r, "\pi"]\ar[d, "q'_\varphi"]& N\ar[d, "q_\varphi"]\\
                        N'/s(N_\varphi)\ar[r, "\pi_\varphi"]& N/N_\varphi
                    \end{tikzcd}
                \end{equation*}
                \item Let $\Delta_{Z, \pi}$ denote a strongly convex rational polyhedral cone $\{\{0_{N'_\RR}\}\}$ in $N'_\RR$.  
                \item Let $Z_\pi$ denote the mock toric variety induced by $Z$ along $\pi_*\colon X(\Delta_{Z, \pi})\rightarrow X(\Delta)$ and $s$.
                \item Let $\iota'$ denote the closed immersion $Z_\pi\rightarrow X(\Delta_{Z, \pi})$ and $\iota'_\varphi$ denote a composition $(q'_\varphi)_*\circ \iota'$.
                \item Let $\epsilon$ denote $Z_\pi\rightarrow Z$ induced by $\pi_*$. 
                We remark that $\epsilon$ is a dominant morphism from Corollary \ref{cor: dominant morphism mock version}. 
                \item Let $\epsilon^*$ denote a ring morphism $k[Z^\circ]\rightarrow k[Z_\pi]$ induced by $\epsilon$. 
                \item Let $p_0$ denote the composition $p_*\circ \iota'\colon Z_\pi\rightarrow T_{N_1}$. 
                Let $p^*_0$ denote a ring morphism $k[M_1]\rightarrow k[Z_\pi]$ induced by $p_0$. 
            \end{itemize}
    \end{definition}
    \begin{proposition}\label{prop:extended-valuation}
        For the notation in Definition \ref{def: extend mock toric}, We define a map $\val_{Z, \pi}\colon (\pi_\RR)^{-1}(\supp(\Delta))\cap N'\rightarrow \Val_k(Z_\pi)$ as follows: 
        For $w\in (\pi_\RR)^{-1}(\supp(\Delta))\cap N'$, from the condition (4) in Definition \ref{def:mock-toric}, there exists $\varphi\in\Phi$ such that $w\in (\pi_\RR)^{-1}(\supp(\Delta_\varphi))$.  
        Then $q'_\varphi(w)$ can be regarded as a discrete valuation on $X(\Delta_{Z, \pi}(\varphi)) = T_{N'/s(N_\varphi)}$ which is trivial on $k$. 
        The morphism $\iota'_\varphi$ is an open immersion, so we define an $\val_Z(w)$ as $({\iota'}_\varphi^{-1})_{\natural}(q'_\varphi(w))$. 

        Then the following statements hold: 
        \begin{enumerate}
            \item[(a)] The definition of $\val_{Z, \pi}$ is well-defined. 
            \item[(b)] Let $\Delta'$ denote a strongly convex rational polyhedral fan in $N'_\RR$. 
            We assume that $\pi$ is compatible with the fans $\Delta'$ and $\Delta$. 
            Let $Z'$ denote the mock toric variety induced by $Z$ along $\pi_*\colon X(\Delta')\rightarrow X(\Delta)$ and $s$. 
            Then $\val_{Z'} = \val_{Z, \pi}|_{\supp(\Delta')\cap N'}$. 
            \item[(c)] 
            For any $\chi'\in k[Z_\pi]^*$, there exists $\eta\in k[Z^\circ]^*$ and $\omega\in M_1$ such that $\chi' = \epsilon^*(\eta){p^*_0}(\chi^\omega)$. 
            \item[(d)] For any $v'\in (\pi_\RR)^{-1}(\supp(\Delta))\cap N'$, $f\in k[Z^\circ]$, and $g\in k[M_1]$, the following equations hold: 
            \begin{align*}
                \val_{Z, \pi}(v')(\epsilon^*(f)) &= \val_{Z}(\pi(v'))(f)\\
                \val_{Z, \pi}(v')(p^*_0(g)) &= p(v')(g)\\
            \end{align*}
        \end{enumerate}
    \end{proposition}
    \begin{proof}
        We prove the statements from (a) to (d) in order. 
        \begin{enumerate}
            \item[(a)] Let $\Delta''$ denote a strongly convex rational polyhedral fan in $N'_\RR$ such that $\supp(\Delta'') = (\pi_\RR)^{-1}(\supp(\Delta))$ and $\pi$ is compatible with the fans $\Delta''$ and $\Delta$. 
            Let $W$ denote a mock toric variety induced by $Z$ along $\pi_*\colon X(\Delta'')\rightarrow X(\Delta)$ and $s$. 
            Let $w\in\pi^{-1}_\RR(\supp(\Delta))\cap N'_+$ and $\varphi\in\Phi$ be an element such that $\pi(w)\in\supp(\Delta_\varphi)$. 
            We show that $w\in\supp(\Delta''_\varphi)$. 
            Indeed, there exists $\tau\in\Delta''$ such that $w\in\tau^\circ$ because $\supp(\Delta'') = (\pi_\RR)^{-1}(\supp(\Delta))$. 
            From the assumption of $\pi$, there exists $\sigma\in\Delta$ such that $\pi_\RR(\tau)\subset \sigma$. 
            On the other hand, there exists $\gamma\in\Delta_\varphi$ such that $\pi(w)\in\gamma$. 
            Thus, $\sigma\cap\gamma\in \Delta_\varphi$ and $\sigma\cap\gamma$ is a face of $\sigma$. 
            Because $w\in\tau^\circ$, $\pi_\RR(\tau)\subset \sigma\cap\gamma$. 
            Hence, $\tau\in\Delta''_\varphi$. 
            Thus, $({\iota'_\varphi}^{-1})_\natural(q'_\varphi(w)) = \val_W(w)$ from the definition of $\val_W$. 
            This equation shows that the definition of the $\val_{Z, \pi}(w)$ is independent of the choice of $\varphi\in\Phi$. 
            \item[(b)] As the proof of (a), we can check that the statement holds. 
            \item[(c)] There exists the following diagram whose all squares are Cartesian squares: 
            \begin{equation*}
                \begin{tikzcd}
                    Z_\pi\ar[r, "\iota'"]\ar[d, "\epsilon"]&T_{N'}\ar[r, "p_*"]\ar[d, "\pi_*"]&T_{N_1}\ar[d]\\
                    Z^\circ\ar[r, "\iota"]&T_N\ar[r]&\Spec(k)
                \end{tikzcd}
            \end{equation*}
            Thus, $\epsilon^*\otimes{p^*_0}\colon k[Z^\circ]\otimes_k k[M_1]\rightarrow k[Z_\pi]$ is isomorphic. 
            Because $k[Z]$ is an integral domain, a unit in $k[Z_\pi]$ can be decomposed into a product of that in $k[Z^\circ]$ and that in $k[M_1]$. 
            \item[(d)] For any $\varphi\in\Phi$, there exists the following commutative diagram: 
            \begin{equation*}
                \begin{tikzcd}
                    Z_\pi\ar[r, "\iota'"]\ar[d, "\epsilon"]&T_{N'}\ar[r, "(q'_\varphi)_*"]\ar[d, "\pi_*"]&T_{N'/s(N_\varphi)}\ar[r, "(p_\varphi)_*"]\ar[d, "(\pi_\varphi)_*"]&T_{N_1}\ar[d]\\
                    Z^\circ\ar[r, "\iota"]&T_N\ar[r, "(q_\varphi)_*"]&T_{N/N_\varphi}\ar[r]&\Spec(k)
                \end{tikzcd}
            \end{equation*} 
            Hence, from Lemma \ref{lem:val-val}, and the definition of $\val_{Z, \pi}$ and $\val_Z$, we have $\epsilon_\natural(\val_{Z, \pi}(v')) = \val_Z(\pi(v'))$. 
            Thus, $\val_{Z, \pi}(v')(\epsilon^*(f)) = \val_{Z}(\pi(v'))(f)$. 
            On the other hand, because $p^*(g) = (q'_\varphi)^*(p_\varphi)^*(g)$, the following equation holds:
            \begin{align*}
                \val_{Z, \pi}(v')(p^*_0(g))&= \val_{Z, \pi}(v')({\iota'_\varphi}^*({p_\varphi}^*(g)))\\
                &= ((q'_\varphi)(v'))({p_\varphi}^*(g))\\
                &= p(v')(g)
            \end{align*}
        \end{enumerate}
    \end{proof}
    As Lemma \ref{lem: supplement-of-valuation1}, the following corollary holds. 
    \begin{corollary}\label{cor: supplement-of-valuation2}
        We keep the notation of Proposition \ref{prop:extended-valuation}.  
        Let $w\in (\pi_\RR)^{-1}(\supp(\Delta))\cap N'$ be an element and $\omega'\in M'$ be an element. 
        Then $\val_{Z, \pi}(w)({\iota'}^*(\chi^{\omega'})) = \langle w, \omega'\rangle$ holds. 
    \end{corollary}
    \begin{proof}
        There exists $\omega\in M$ and $\omega_1\in M_1$ such that $\pi^*(\omega) + p^*(\omega_1) = \omega'$. 
        In particular, $\chi^{\omega'} = \pi^*(\chi^\omega)p^*(\chi^{\omega_1})$ and ${\iota'}^*(\chi^{\omega'}) = \epsilon^*(\iota^*(\chi^\omega))p_0^{*}(\chi^{\omega_1})$ hold. 
        From Proposition \ref{prop:extended-valuation}(d) and Lemma \ref{lem: supplement-of-valuation1}, the following two equations hold for any $w\in (\pi_\RR)^{-1}(\supp(\Delta))\cap N'$: 
        \begin{align*}
            \val_{Z, \pi}(w)(\epsilon^*(\iota^*(\chi^\omega)))&= \val_Z(\pi(w))(\iota^*(\chi^\omega))\\
            &= \langle \pi(w), \omega\rangle\\
            &= \langle w, \pi^*(\omega)\rangle
        \end{align*}
        \begin{align*}
            \val_{Z, \pi}(w)(p_0^{*}(\chi^{\omega_1}))&= p(w)(\chi^{\omega_1})\\
            &= \langle p(w), \omega_1\rangle\\
            &= \langle w, p^*(\omega_1)\rangle
        \end{align*}
        Thus, $\val_{Z, \pi}(w)({\iota'}^*(\chi^{\omega'})) = \langle w, \omega'\rangle$ holds. 
    \end{proof}
    \subsection{Mock toric varieties over $\A^1_k$}
    In this article, we construct the scheme over a valuation ring. 
    To construct a scheme over a valuation ring, we first construct a scheme over $\A^1_k$ and subsequently carry out a base change. 
    From now on, we consider the mock toric varieties over $\A^1_k$ in this subsection. 
    The mock toric variety over $\A^1_k$ means the mock toric variety induced by a toric variety which has a dominant toric morphism to $\A^1_k$. 
    \begin{proposition}\label{prop:relative-mock}
        We use the following notation: 
        \begin{itemize} 
            \item Let $Z$ be a scheme over $k$ and $(N$, $\Delta$, $\iota$, $\Phi$, $\{N_\varphi\}_{\varphi\in\Phi}$, $\{\Delta_\varphi\}_{\varphi\in\Phi})$ be a mock toric structure of $Z$.
            \item Let $N'$ be a lattice of finite rank, $\pi$ be a surjective morphism $\pi\colon N'\rightarrow N$, and $\pi^1$ denote $\pi\times \mathrm{id}_\ZZ\colon N'\oplus \ZZ \rightarrow N\oplus \ZZ$. 
            \item Let $s$ be a section of $\pi$ and $s^1$ denote a section of $\pi^1$ such that $s^1 = s\times \id_\ZZ\colon N\oplus\ZZ\rightarrow N'\oplus\ZZ$. 
            \item For $\varphi\in\Phi$, let $q_\varphi$ denote the quotient map $N\rightarrow N/N_\varphi$ and $q'_\varphi$ denote the quotient map $N'\rightarrow N'/s(N_\varphi)$. 
            \item Let $\pr_1$ and $\pr_2$ denote the first and second projection of $N\oplus\ZZ$ or $N'\oplus\ZZ$. 
            \item Let $j$ denote a section of $\pr_1\colon N\oplus\ZZ\rightarrow N$ such that $\pr_2\circ j = 0$. 
            \item Let $j'$ denote a section of $\pr_1\colon N'\oplus\ZZ\rightarrow N'$ such that $\pr_2\circ j' = 0$. 
            \item Let $Z^1$ denote a mock toric variety induced by $Z$ along $(\pr_1)_*\colon X(\Delta\times \Delta_!)\rightarrow X(\Delta)$ and $j$. 
            \item Let $\Delta'$ be a strongly convex rational polyhedral fan in $(N'\oplus \ZZ)_\RR$. 
            We assume that $\pi^1$ is compatible with the fans $\Delta'$ and $\Delta\times \Delta_!$.
            \item Let $W$ denote a  mock toric variety induced by $Z^1$ along $(\pi^1)_*\colon X(\Delta')\rightarrow X(\Delta\times\Delta_!)$ and $s^1$. 
            \item Let $\iota''$ denote the closed immersion $W\rightarrow X(\Delta')$
            \item Let $\Delta'_0$ denote the following sub fan of $\Delta'$:
            \[
                \Delta'_0 = \{\sigma\in\Delta'\mid (\pr_2)_\RR(\sigma) = \{0\}\}
            \]
            \item Let $\Delta'_{(0)}$ denote the following set:
            \[
                \Delta'_{(0)} = \{(\pr_1)_\RR(\sigma)\mid \sigma\in\Delta'_0\}
            \]
            From Lemma \ref{lem: relative lemma}(a), $\Delta'_{(0)}$ is a strongly convex rational polyhedral fan. 
        \end{itemize}
        Then the following statements follow:
        \begin{enumerate}
            \item[(a)] The morphism $\pi$ is compatible with the fans $\Delta'_{(0)}$ and $\Delta$. 
            Let $Y$ denote a mock toric variety induced by $Z$ along $\pi_*\colon X(\Delta'_{(0)})\rightarrow X(\Delta)$ and $s$. 
            \item[(b)] The toric morphism $(\pr_1)_*\times (\pr_2)_*\colon X(\Delta'_0)\rightarrow X(\Delta'_{(0)})\times {\Gm^1}_{,k}$ is isomorphic from Lemma \ref{lem: relative lemma}(b). 
            This isomorphism induces an isomorphism $p\colon W\cap X(\Delta'_0)\rightarrow Y\times {\Gm^1}_{,k}$. 
            Let $\sigma\in\Delta_0$ be a cone and $\sigma_0$ denote $(\pr_1)_\RR(\sigma)$. 
            Then, this isomorphism induces the following three Cartesian diagrams: 
            \begin{equation*}
                \begin{tikzcd}
                    W(\sigma)\ar[r]\ar[d, hook]&Y(\sigma_0)\times {\Gm^1}_{,k}\ar[d, hook]\\
                    W\cap X(\Delta'_0)\ar[r, "p"]&Y\times{\Gm^1}_{,k}
                \end{tikzcd}
            \end{equation*}
            \begin{equation*}
                \begin{tikzcd}
                    \overline{W_\sigma}\cap X(\Delta'_0)\ar[r]\ar[d, hook]&\overline{Y_{\sigma_0}}\times {\Gm^1}_{,k}\ar[d, hook]\\
                    W\cap X(\Delta'_0)\ar[r, "p"]&Y\times{\Gm^1}_{,k}
                \end{tikzcd}
            \end{equation*}
            \begin{equation*}
                \begin{tikzcd}
                    W_\sigma \ar[r]\ar[d]&Y_{\sigma_0}\times {\Gm^1}_{,k}\ar[d]\\
                    W\cap X(\Delta'_0)\ar[r, "p"]&Y\times{\Gm^1}_{,k}
                \end{tikzcd}
            \end{equation*}
            where the upper morphisms in all diagrams are the restriction of $p$, and the vertical morphisms are open immersions, closed immersions, and locally closed immersions from top to bottom, respectively. 
            \item[(c)] Let $\Psi$ denote the composition of the following maps:
            \[
                X_{k(t)}(\Delta'_{(0)})\rightarrow X(\Delta'_{(0)})\times {\Gm^1}_{,k}\rightarrow X(\Delta'_0)\hookrightarrow X(\Delta')
            \]
            , where the first morphism is a base change of the generic fiber $\Spec(k(t))\rightarrow {\Gm^1}_{,k}$ along $(\pr_2)_*\colon X(\Delta'_{(0)})\times{\Gm^1}_{,k}\rightarrow {\Gm^1}_{,k}$, the second one is the inverse morphism of $p$, and the third one is an open immersion.
            Let $\psi$ denote a morphism $Y_{k(t)}\rightarrow W$ defined by the pullback of $\Psi$ along the closed immersion $W\rightarrow X(\Delta')$. 
            Then, $\psi$ is a dominant morphism. 
            Moreover, the large square of the following diagram is a Cartesian diagram: 
            \begin{equation*}
                \begin{tikzcd}
                    Y_{k(t)}\ar[r]\ar[d]&Y \times {\Gm^1}_{,k}\ar[r, "p^{-1}"]\ar[rd, "\pr_2"]&W\cap X(\Delta'_0)\ar[r, hook]\ar[d, "(\pr_2)_*\circ \iota''"]&W\ar[d, "(\pr_2)_*\circ\iota''"]\\
                    \Spec(k(t))\ar[rr]&\ &{\Gm^1}_{,k}\ar[r, hook]&\A^1_k
                \end{tikzcd}
            \end{equation*}
            \item[(d)] We keep the notation in (b) and (c). 
            The morphism $\psi$ induces the following three Cartesian diagrams:
            \begin{equation*}
                \begin{tikzcd}
                    Y_{k(t)}(\sigma_0)\ar[r]\ar[d, hook]&W(\sigma)\ar[d, hook]\\
                    Y_{k(t)}\ar[r, "\psi"]&W
                \end{tikzcd}
            \end{equation*}
            \begin{equation*}
                \begin{tikzcd}
                    \overline{Y_{\sigma_0, k(t)}}\ar[r]\ar[d, hook]&\overline{W_\sigma}\ar[d, hook]\\
                    Y_{k(t)}\ar[r, "\psi"]&W
                \end{tikzcd}
            \end{equation*}
            \begin{equation*}
                \begin{tikzcd}
                    Y_{\sigma_0, k(t)}\ar[r]\ar[d, hook]&W_\sigma\ar[d, hook]\\
                    Y_{k(t)}\ar[r, "\psi"]&W
                \end{tikzcd}
            \end{equation*}
            where the upper morphisms in all diagrams are the restriction of $\psi$, and the vertical morphisms are open immersions, closed immersions, and locally closed immersions from top to bottom, respectively. 
            \item[(e)] Let $\varphi\in \Phi$ be an element. 
            Then there exists a one-to-one correspondence of cones of $\Delta'_0\cap\Delta'_\varphi$ and those of $\Delta'_{(0), \varphi}$ by $(\pr_1)_\RR$. 
            In particular, $\pr_1\colon N'\oplus \ZZ\rightarrow N'$ is compatible with the fans $\Delta'_0\cap\Delta'_\varphi$ and  $\Delta'_{(0), \varphi}$.
            \item[(f)] We keep the notation in (e).  
            Let $\Delta'(\varphi)_0$ denote the following sub fan of $\Delta'(\varphi)$: 
            \[
                \Delta'(\varphi)_0 = \{\sigma_\varphi\in\Delta'(\varphi)\mid (\pr_2)_\RR(\sigma_\varphi) = \{0\}\}
            \]
            Then $\Delta'(\varphi)_0$ is equal to the following set:
            \[
                \{(q'_\varphi\times\mathrm{id}_\ZZ)_\RR(\sigma)\mid \sigma\in\Delta'_0\cap\Delta'_\varphi\}
            \]
            In particular, $q'_\varphi\times\mathrm{id}_\ZZ$ is compatible with the fans $\Delta'_0\cap\Delta'_\varphi$ and $\Delta'(\varphi)_0$. 
            \item[(g)] We keep the notation in (f).  
            Let $\Delta'(\varphi)_{(0)}$ denote the following set: 
            \[
                \{(\pr_1)_\RR(\sigma_\varphi)\mid \sigma_\varphi\in\Delta'(\varphi)_{0} \}
            \]
            Then $\Delta'(\varphi)_{(0)} = \Delta'_{(0)}(\varphi)$. 
            Moreover, $\Delta'(\varphi)_{(0)}$ is a strongly convex rational polyhedral fan in $(N'/s(N_\varphi))_\RR$. 
            Furthermore, $\pr_1\colon N'/s(N_\varphi)\oplus \ZZ\rightarrow N'/s(N_\varphi)$ is compatible with the fans $\Delta'(\varphi)_0$ and $\Delta'_{(0)}(\varphi)$. 
            \item[(h)] The toric morphism $(\pr_1)_*\times (\pr_2)_*\colon X(\Delta'(\varphi)_0)\rightarrow X(\Delta'_{(0)}(\varphi))\times {\Gm^1}_{,k}$ is isomorphic from (f), (g), and Lemma \ref{lem: relative lemma}(b). 
            Let $p_\varphi$ denote this isomorphism and $\Psi_\varphi$ denote the composition of the following maps:
            \[
                X_{k(t)}(\Delta'_{(0)}(\varphi))\rightarrow X(\Delta'_{(0)}(\varphi))\times {\Gm^1}_{,k}\rightarrow X(\Delta'(\varphi)_0)\hookrightarrow X(\Delta'(\varphi)) 
            \]
            , where the first morphism is a base change of the generic fiber $\Spec(k(t))\rightarrow {\Gm^1}_{,k}$ along $(\pr_2)_*\colon X(\Delta'_{(0)}(\varphi))\times{\Gm^1}_{,k}\rightarrow {\Gm^1}_{,k}$, the second one is the inverse morphism of $p_\varphi$, and the third one is an open immersion. 
            Then the following diagram is commutative:
            \begin{equation*}
                \begin{tikzcd}
                    X_{k(t)}(\Delta'_{(0), \varphi})\ar[r, "\Psi"]\ar[d, "(q'_\varphi)_{*, k(t)}"]& X(\Delta'_\varphi)\ar[d, "(q'_\varphi\times\mathrm{id}_\ZZ)_*"]\\
                    X_{k(t)}(\Delta'_{(0)}(\varphi))\ar[r, "\Psi_\varphi"]&X(\Delta'(\varphi))
                \end{tikzcd}
            \end{equation*}
            \item[(i)] 
            We assume $\Delta'$ is $\{\{0\}\}$. 
            Then we can check that $W$ is $Z^1_{\pi^1}$ and $Y$ is $Z_\pi$.  
            Let $\mu$ denote $\pr_1\circ p$ and $p_0$ denote $\pr_2\circ p$.  
            From (b), $\mu\times p_0\colon Z^1_{\pi^1}\rightarrow Z_{\pi}\times{\Gm^1}_{,k}$ induces the isomorphism $\mu^*\times p^*_0\colon k[Z_\pi]\otimes k[t, t^{-1}]\rightarrow k[Z^1_{\pi^1}]$. 
            Thus, for any $\chi'\in k[Z^1_{\pi^1}]^*$, there exists $\chi\in k[Z_\pi]^*$ and an integer $n$ such that $\chi' = \mu^*(\chi)p^*_0(t^n)$. 
            \item[(j)] We keep the notation in (i).  
            Let $f\in k[{Z}_{\pi}]$, $g\in k[t, t^{-1}]$, and $v'\in ({\pi^1_\RR})^{-1}(\supp(\Delta\times\Delta_!))\cap
            (N'\oplus\ZZ)$ be elements.  
            Then $\val_{Z^1, \pi^1}(v')(\mu^*(f) p^*_0(g)) = \val_{Z, \pi}(\pr_1(v'))(f) + \pr_2(v')(g)$. 
            \item[(k)] The following diagram is commutative:
            \begin{equation*}
                \begin{tikzcd}
                    \supp(\Delta'_{(0)})\cap N'\ar[r, "j'"]\ar[d, "\val_{Y_{k(t)}}"]& \supp(\Delta')\cap 
                    (N'\oplus\ZZ)\ar[d, "\val_{W}"]\\
                    \Val_{k(t)}(Y_{k(t)})\ar[r, "\psi_\natural"]&\Val_k(W)
                \end{tikzcd}
            \end{equation*}
        \end{enumerate}
    \end{proposition}
    \begin{proof}
        We prove these statements from (a) to (k) in order. 
        \begin{enumerate}
            \item[(a)] Let $\sigma\in\Delta'_{0}$ be a cone and $\sigma_0$ denote $(\pr_1)_\RR(\sigma)$. 
            Then $\pi_\RR(\sigma_0) = (\pr_1)_\RR(\pi^1_\RR(\sigma))$. 
            From the definition of $\Delta'_{0}$, we have $\pi^1_\RR(\sigma)\subset \supp(\Delta)\times\{0\}$. 
            Because $\pi^1$ is compatible with the fans $\Delta'$ and $\Delta\times\Delta_!$, there exists $\tau\in\Delta$ such that $\pi^1_\RR(\sigma)\subset \tau\times\{0\}$. 
            Thus, $\pi_\RR(\sigma_0)\subset \tau$. 
            \item[(b)]
            There exists the following diagram of toric morphisms:
            \begin{equation*}
                \begin{tikzcd}
                    X(\Delta'_0)\ar[r, "\pi^1_*"]\ar[d, "(\pr_1)_*"]&X(\Delta\times\{\{0\}\})\ar[r, "(\pr_2)_*"]\ar[d, "(\pr_1)_*"]&{\Gm^1}_{,k}\ar[d]\\
                    X(\Delta'_{(0)})\ar[r, "\pi_*"]&X(\Delta)\ar[r]&\Spec(k)
                \end{tikzcd}
            \end{equation*}
            In particular, all small squares in the above diagram are Cartesian squares. 
            Thus, from the definition of the mock toric structure of $Y$ and $W$, all small squares in the following diagram are Cartesian squares:
            \begin{equation*}
                \begin{tikzcd}
                    W\cap X(\Delta'_0)\ar[r]\ar[d, "\mu"]&Z^1\cap X(\Delta\times\{\{0\}\})\ar[r]\ar[d]&{\Gm^1}_{,k}\ar[d]\\
                    Y\ar[r]&Z\ar[r]&\Spec(k)
                \end{tikzcd}
            \end{equation*}
            where $\mu$ denote $\pr_1\circ p$. 
            The latter part of the statement holds from Proposition \ref{prop: closure of stratum} and Lemma \ref{lem: relative lemma}(b). 
            \item[(c)] We show that $\psi$ is dominant. 
            Indeed, $\psi$ can be decomposed in the upper dominant morphisms in the following diagram from (b): 
            \begin{equation*}
                \begin{tikzcd}
                    Y_{k(t)}\ar[r]\ar[d]& Y\times{\Gm^1}_{,k}\ar[r, "p^{-1}"]\ar[d]& W\cap X(\Delta'_0)\ar[r, hook]\ar[d]& W\ar[d]\\
                    X_{k(t)}(\Delta'_{(0)})\ar[r]& X(\Delta'_{(0)})\times{\Gm^1}_{,k}\ar[r]& X(\Delta'_0)\ar[r, hook]& X(\Delta')
                \end{tikzcd}
            \end{equation*}
            We remark that all squares in the above diagram are Cartesian squares from (b). 
            Thus, from Lemma \ref{lem: relative lemma}(c), the latter part of the statement follows. 
            \item[(d)] We can check it from Proposition \ref{prop: closure of stratum} and Lemma \ref{lem: relative lemma}(d). 
            \item[(e)] Let $\sigma\in\Delta_0$ be a cone and $\sigma_0$ denote $(\pr_1)_\RR(\sigma)$. 
            From Lemma \ref{lem: injective fan}, $(\pr_1)_\RR$ induces a one-to-one correspondence of cones in $\Delta'_0$ and those in $\Delta'_{(0)}$. 
            Thus, we need to check that $\sigma\in\Delta'_0\cap\Delta'_\varphi$ if and only if $\sigma_0\in\Delta'_{(0), \varphi}$. 
            We recall that the mock toric structure of $W$ and $Y$ are defined by $(\pi^1)_*$ and $s^1$, and $\pi_*$ and $s$ respectively.  
            Moreover, there exists the following diagram:
            \begin{equation*}
                \begin{tikzcd}
                    N'\oplus\ZZ\ar[r, "\pr_1"]\ar[d, "\pi^1"]&N'\ar[d, "\pi"]\\
                    N\oplus\ZZ\ar[r, "\pr_1"]&N
                \end{tikzcd}
            \end{equation*}
            Thus, $\pi_\RR(\sigma_0)\subset \supp(\Delta_\varphi)$ if and only if $\pi^1_\RR(\sigma)\subset\supp((\Delta\times\Delta_!)_\varphi) = \supp(\Delta_\varphi)\times\RR_{\geq 0}$. 
            \item[(f)] We recall that $\Delta'(\varphi)$ is the following set:
            \[
                \Delta'(\varphi) = \{(q'_\varphi\times\id_\ZZ)_\RR(\sigma)\mid \sigma\in\Delta'_\varphi\}
            \]
            Thus, from the following commutative diagram, the statement holds:
            \begin{equation*}
                \begin{tikzcd}
                    N'\oplus\ZZ\ar[r, "q'_\varphi\times\id"]\ar[rd, "\pr_2"]& N'/s(N_\varphi)\oplus\ZZ\ar[d, "\pr_2"]\\
                    \ &\ZZ
                \end{tikzcd}
            \end{equation*}
            \item[(g)] We recall that $\Delta'_{(0)}(\varphi) = \{(q'_\varphi)_\RR(\tau)\mid \tau\in\Delta'_{(0),\varphi}\}$. 
            From (e), (f), and the equation $q'_\varphi\circ\pr_1 = \pr_1\circ(q'_\varphi\times\id_\ZZ)$, we have $\Delta'(\varphi)_{(0)} = \Delta'_{(0)}(\varphi)$. 
            Moreover, we remark that $\Delta'_{(0)}(\varphi)$ is strongly convex rational polyhedral fan from Proposition\ref{prop: first prop}(c). 
            \item[(h)] There exists the following commutative diagram: 
            \begin{equation*}
                \begin{tikzcd}
                    X_{k(t)}(\Delta'_{(0), \varphi})\ar[r]\ar[d, "(q'_\varphi)_{*, k(t)}"]& X(\Delta'_{(0), \varphi})\times{\Gm^1}_{,k}\ar[r, "p^{-1}"]\ar[d, "(q'_\varphi)_*\times\id_{{\Gm^1}_{,k}}"]& X(\Delta'_0\cap \Delta'_\varphi)\ar[r, hook]\ar[d, "(q'_\varphi\times\mathrm{id}_\ZZ)_*"]& X(\Delta'_\varphi)\ar[d, "(q'_\varphi\times\mathrm{id}_\ZZ)_*"]\\
                    X_{k(t)}(\Delta'_{(0)}(\varphi))\ar[r]& X(\Delta'_{(0)}(\varphi))\times{\Gm^1}_{,k}\ar[r, "p_\varphi^{-1}"]& X(\Delta'(\varphi)_0)\ar[r, hook]& X(\Delta'(\varphi))
                \end{tikzcd}
            \end{equation*}
            Thus, from the definition of $\Psi$ and $\Psi_\varphi$, the statement holds. 
            \item[(i)] From the mock toric structure of $Z^1_{\pi^1}$ and $Z_\pi$, we can recognize $Z^1_{\pi^1}$ as a mock toric variety induced by $Z_\pi$ along $(\pr_1)_*\colon T_{N'\oplus\ZZ}\rightarrow T_{N'}$ and $j'$. 
            We remark that we can identify with the second projection $\pr_2$ of $N'\oplus\ZZ$ and the quotient map $N'\oplus\ZZ\rightarrow (N'\oplus\ZZ)/j'(N')$. 
            Thus, the statement follows from Proposition \ref{prop:extended-valuation}(c). 
            \item[(j)] Let $\Delta^\flat$ be a strongly convex rational polyhedral fan in $N'_\RR$ such that $\supp(\Delta^\flat) = (\pi_\RR)^{-1}(\supp(\Delta))$ and $\pi$ is compatible with the fans $\Delta^\flat$ and $\Delta$. 
            Let $Y^\flat$ denote a mock toric variety by $Z$ along $\pi_*\colon X(\Delta^\flat)\rightarrow X(\Delta)$ and $s$. 
            Then from Proposition \ref{prop:extended-valuation}(b), $\val_{Y^\flat} = \val_{Z, \pi}$. 
            On the other hand, we can recognize $Z^1_{\pi^1}$ as a mock toric variety induced by $Y^\flat$ along $(\pr_1)_*\colon T_{N'\oplus\ZZ}\rightarrow X(\Delta^\flat)$ and $j'$. 
            We remark that ${(\pr_1)_\RR}^{-1}(\supp(\Delta^\flat)) = {(\pi\circ\pr_1)_\RR}^{-1}(\supp(\Delta))\supset {(\pi^1)_\RR}^{-1}(\supp(\Delta\times\Delta_!))$ so that $\val_{Y^\flat, \pr_1}(v') = \val_{Z^1, \pi^1}(v')$ for any $v'\in ({(\pi_1)_\RR}^{-1}(\supp(\Delta\times\Delta_!)))\cap (N'\oplus\ZZ)$. 
            Thus, from the remark in the proof of (i) and Proposition \ref{prop:extended-valuation}(d), $\val_{Z^1, \pi^1}(v')(\mu^*(f) p^*_0(g)) = \val_{Z, \pi}(\pr_1(v'))(f) + \pr_2(v')(g) = \val_{Y^\flat}(\pr_1(v'))(f) + \pr_2(v')(g)$ for any $v'\in ({(\pi_1)_\RR}^{-1}(\supp(\Delta\times\Delta_!)))\cap (N'\oplus\ZZ)$.   
            \item[(k)] Let $v'\in\supp(\Delta'_{(0)})\cap N'$ be an element. 
            We regard $v'_{k(t)}$ as a torus invariant valuation on $T_{N', k(t)}$, which is trivial on $k(t)$. 
            Thus, we remark that $\Psi_\natural(v')$ is equal to $(v', 0) = j'(v')$ as a torus invariant valuation on $T_{N'\oplus\ZZ}$ which is trivial on $k$ from the construction of $\Psi$. 
            Let $\varphi\in\Phi$ be an element such that $\pi(v')\in\supp(\Delta_\varphi)$. 
            We remark that $(v', 0)\in \supp(\Delta'_\varphi)$. 
            From (h), there exists the following diagram:
            \begin{equation*}
                \begin{tikzcd}
                    Y_{k(t)}\cap X_{k(t)}(\Delta'_{(0), \varphi})\ar[r]\ar[d, "\psi"]&X_{k(t)}(\Delta'_{(0), \varphi})\ar[r, "(q'_\varphi)_{*, k(t)}"]\ar[d, "\Psi"]&X_{k(t)}(\Delta'_{(0)}(\varphi))\ar[d, "\Psi_\varphi"]\\
                    W\cap X(\Delta'_\varphi)\ar[r]&X(\Delta'_\varphi)\ar[r, "(q'_\varphi\times\id_\ZZ)_*"]&X(\Delta'(\varphi))
                \end{tikzcd}
            \end{equation*}
            where two left horizontal morphisms are closed immersions, and we can check that both horizontal composition morphisms are open immersion from the condition (5) in Definition \ref{def:mock-toric}. 
            Hence, from Lemma \ref{lem:val-val}, and the definition of $\val_{Y_{k(t)}}$ and $\val_{W}$, we have $\psi_\natural(\val_{Y_{k(t)}}(v')) = \val_{W}(\Psi_\natural(v'))$.
            Therefore, from the above remark, the statement follows. 
        \end{enumerate}
    \end{proof}
\section{Example of mock toric varieties}
		We defined mock toric variety in Section 3. 
		In this section, we give some new examples of mock toric varieties. 
		This new example is given by ``base point free'' hyperplane arrangements. 
		In particular, we give a new non-toric mock toric structure of projective spaces and show that every Del Pezzo surface has a mock toric structure. 
	    In this section, We use the following notation: 
    \begin{itemize}
        \item Let $n$ and $d$ be positive integers, and $k$ be a field. 
        \item Let $\mathcal{B} = \{f_0, f_1, \ldots, f_d\}$ denote a finite subset of $\Gamma(\PP^n_k, \OO_{\PP^n_k}(1))\setminus\{0\}$. 
        We assume that $\mathcal{B}$ generates $\Gamma(\PP^n_k, \OO_{\PP^n_k}(1))$ as a $k-$vector space. 
        We regard $f_i$ as a homogeneous polynomial with $\deg(f_i) = 1$. 
        \item Let $\iota_0$ be a rational map $\PP^n_k\dashrightarrow \PP^d_k$ defined by $\PP^n_k\ni a\mapsto [f_0(a):f_1(a):\cdots:f_d(a)]\in\PP^d_k$. 
        Then $\iota_0$ is a closed embedding by the assumption of $\mathcal{B}$. 
        \item Let $\mathcal{V}$ denote the following set: 
        $$\mathcal{V} = \{V\subset \Gamma(\PP^n_k, \OO_{\PP^n_k}(1))\mid V = \sum_{i; f_i\in V} k\cdot f_i\}$$
        \item For $V\in\mathcal{V}$, let $\rho(V)$ denote a set $\{0\leq i\leq d\mid f_i\in V\}$. 
        \item Let $\{e^0, e^1, \ldots, e^d\}$ denote a canonical basis of $\ZZ^{d+1}$, and $\mathbf{1}\in \ZZ^{d+1}$ be $\sum_{0\leq i\leq d}e_i$. 
        \item Let $N$ denote $\ZZ^{d+1}/\ZZ\mathbf{1}$, $p$ denote the quotient morphism $\ZZ^{d+1}\rightarrow N$, and $e_i\in N$ denote $p(e^i)$ for $0\leq i\leq d$. 
        \item Let $\Delta_0$ denote a complete strongly convex rational polyhedral fan in $N_\RR$ such that all rays of $\Delta_0$ are $\{e_i\}_{0\leq i\leq d}$. We regard $\PP^n_k$ as a toric variety $X(\Delta_0)$.  
        \item For $V\in \mathcal{V}$, let $e_{V}$ denote $\sum_{i\in \rho(V)}e_i\in N$. 
        \item Let $\mathcal{C}$ denote the following set: 
        \begin{multline*}
            \mathcal{C} = \{(V_1, V_2, \ldots, V_s)\mid s\in\ZZ_{>0},\\
            0\subsetneq V_1\subsetneq V_2\subsetneq\cdots\subsetneq V_s\subsetneq\Gamma(\PP^n_k, \OO_{\PP^n_k}(1)), V_i\in\mathcal{V}, 1\leq\forall i\leq s\}
        \end{multline*}
        
        \item For $c = (V_1, V_2, \ldots, V_s)\in\mathcal{C}$, $\sigma_c$ denote a strongly convex rational polyhedral cone in $N_\RR$ generated by $\{e_{V_i}\}_{1\leq j\leq s}$. Let $\Delta(\mathcal{B})$ denote $\{\sigma_c\mid c\in\mathcal{C}\}$. 
    \end{itemize}
    From \cite{MS15}, we mention a basic and important result for $\Delta(\mathcal{B})$ in the following proposition.
    \begin{proposition}\cite[Theorem 4.2.6]{MS15}
        For the above notation, $\Delta(\mathcal{B})$ is a strongly convex rational simplicial fan in $N_\RR$. 
    \end{proposition}
    The following proposition is also important for defining a new mock toric structure. 
    \begin{proposition}
        An identity morphism $\id_N$ is compatible with the fans $\Delta(\mathcal{B})$ and $\Delta_0$. 
    \end{proposition}
    \begin{proof}
        Let $c = (V_1, V_2, \ldots, V_s)\in\mathcal{C}$ be an element. 
        Because $\mathcal{B}$ generates $\Gamma(\PP^n_k, \OO_{\PP^n_k}(1))$ and $V_s\neq \Gamma(\PP^n_k, \OO_{\PP^n_k}(1))$, $\rho(V_s)\neq\{0, 1, \ldots, d\}$. 
        Thus, from the definition of $\sigma_c$, $\sigma_c$ is contained in $\sum_{i\in\rho(V_s)} \RR_{\geq 0}e_i$. 
        This cone is in $\Delta_0$, and hence the statement follows. 
    \end{proof}
    From now on, we will define some data from $\mathcal{B}$. 
    \begin{definition}
        We will use the following notation:
        \begin{itemize}
            \item Let $Z$ denote the strictly transform of $\PP^n_k$ along a toric morphism $X(\Delta(\mathcal{B}))\rightarrow X(\Delta_0)$. 
            In other words, $Z$ is a scheme theoretic closure of $\iota_0(\PP^n_k)\cap T_N$ in $X(\Delta(\mathcal{B}))$. 
            \item Let $\iota$ denote a closed immersion from $Z$ to $X(\Delta(\mathcal{B}))$. 
            \item Let $\Phi$ denote the following subset of $\{0, \dots, d\}$: $$\Phi = \{\{i_0, i_1, \ldots, i_{n}\}\subset \{0, \ldots, d\}\mid \Gamma(\PP^n_k, \OO_{\PP^n_k}(1)) = \sum_{0\leq j\leq n}k\cdot f_{i_j}\}$$
            \item For $\varphi = \{i_0, i_1, \ldots, i_{n}\}\in \Phi$, let $N_\varphi$ denote $\sum_{j\notin \varphi}\ZZ e_j$. 
            \item For $\varphi\in\Phi$, let $\mathcal{C}_\varphi$ denote the following subset of $\mathcal{C}$: 
            $$\mathcal{C}_\varphi = \{(V_1, V_2, \ldots, V_s)\in\mathcal{C}\mid |\varphi\cap\rho(V_j)| = \dim(V_j)\quad (1\leq\forall j\leq s)\}$$
            Let $\Delta(\mathcal{B})_\varphi$ denote $\{\sigma_c\in\Delta(\mathcal{B})\mid c\in\mathcal{C}_\varphi\}$. 
            We remark that $\Delta(\mathcal{B})_\varphi$ is a sub fan of $\Delta(\mathcal{B})$ from the definition of $\mathcal{C}_\varphi$ and simplicialness of $\Delta(\mathcal{B})$. 
        \end{itemize}
    \end{definition}
    The following proposition shows that there are many non-toric examples of mock toric varieties. 
    \begin{proposition}\label{prop: hyper-example}
        From the above notation, $(N, \Delta(\mathcal{B}), \iota, \Phi, \{N_\varphi\}_{\varphi\in\Phi}, \{\Delta(\mathcal{B})_\varphi\}_{\varphi\in\Phi})$ is a mock toric structure of $Z$. 
        Moreover, $Z$ is proper over $k$. 
    \end{proposition}
    \begin{proof}
         We check that this tuple satisfies the conditions in Definition \ref{def:mock-toric} from (1) to (5) in order:
         \begin{enumerate}
             \item[(1)] We can check the condition (1) easily from the definition of $\Delta(\mathcal{B})$ and $\{N_\varphi\}_{\varphi\in\Phi}$. 
             \item[(2)] It is enough to show that $\mathcal{C} = \cup_{\varphi\in\Phi}\mathcal{C}_\varphi$. 
             Let $c = (V_1, V_2, \ldots, V_s)\in \mathcal{C}$ be an element. 
             Then there exists $\{f_{i_0}, \ldots, f_{i_n}\}\subset \mathcal{B}$ be a basis of $\Gamma(\PP^n_k, \OO_{\PP^n_k}(1))$ such that $\{f_{i_0}, \ldots, f_{i_{\dim(V_j) -1}}\}$ is a basis of $V_j$ for any $0\leq j\leq s$. 
             Let $\varphi$ denote $\{i_0, i_1, \ldots, i_n\}$. 
             In particular, $\varphi\in\Phi$ and $c\in\mathcal{C}_\varphi$ holds. 
             \item[(3)] Let $\varphi\in\Phi$ be an element. 
             Let $\mathcal{B}(\varphi)$ denote $\{f_i\mid i\in\varphi\}$. 
             From the definition of $\varphi$, $\mathcal{B}(\varphi)$ generates $\Gamma(\PP^n_k, \OO_{\PP^n_k}(1))$. 
             Let $\iota_0(\varphi), \mathcal{V}(\varphi), N(\varphi), \Delta_0(\varphi), \mathcal{C}(\varphi), \Delta(\mathcal{B}(\varphi))$ denote the notation above when we replace $\mathcal{B}$ with $\mathcal{B}(\varphi)$. 
             A projection $\ZZ^{d+1}\rightarrow \ZZ^{n + 1}$ whose kernel is generated by $\{e^j\}_{j\notin\varphi}$ induces a surjective morphism $N\rightarrow N(\varphi)$. 
             We can check that the kernel of this morphism is $N_\varphi$. 
             Now, we identify $N(\varphi)$ with $N/N_\varphi$. 
             Let $c\in(V_1, \ldots, V_s)\in\mathcal{C}_\varphi$ be an element. 
             From the definition of $\mathcal{C}_\varphi$, $V_j\in\mathcal{V}(\varphi)$ for any $1\leq j\leq s$. 
             Thus, $\mathcal{C}_\varphi = \mathcal{C}(\varphi)$. 
             Let $\Delta(\mathcal{B})(\varphi)$ denote $\{(q_\varphi)_\RR(\sigma_c)\mid c\in\mathcal{C}_\varphi\}$. 
             From the fact that $\mathcal{C}_\varphi = \mathcal{C}(\varphi)$, $(q_\varphi)_\RR$ induces a one-to-one correspondence of cones in $\Delta(\mathcal{B})_\varphi$ and those in $\Delta(\mathcal{B}(\varphi))$. 
             In particular, $\Delta(\mathcal{B})(\varphi) = \Delta(\mathcal{B}(\varphi))$ over the identification with $N/N_\varphi$ and $N(\varphi)$. 
             We can check that $(q_\varphi)_\RR|_{\sigma_c}$ is injective for any $c\in\mathcal{C}_\varphi$. 
             Therefore, $(q_\varphi)_\RR|_{\supp(\Delta(\mathcal{B})_\varphi)}$ is injective from the above one-to-one correspondence. 
             \item[(4)] First, we show the properness of $Z$. 
             From \cite[Theorem. 4.1.11]{MS15}, the support of $\mathrm{trop(Z^\circ)}$ is $\supp(\Delta(\mathcal{B}))$. 
             Hence, from \cite[Theorem. 6.4.7(i)]{MS15}, $Z$ is proper over $k$. 
             Next, we check the condition (4). 
             Indeed, from \cite[Theorem. 6.4.7(ii)]{MS15}, $Z\cap O_{\sigma_c} \neq \emptyset$ for any $c\in\mathcal{C}$. 
             \item[(5)] Let $\varphi = \{i_0, \ldots, i_n\}\in\Phi$ and $c = (V_1, \ldots, V_s)\in \mathcal{C}_\varphi$ be elements. 
             We may change the index of elements in $\varphi$ such that $\varphi\cap\rho(V_j) = \{i_0, i_1, \ldots, i_{\dim(V_j) - 1}\}$ for any $1 \leq j\leq s$. 
             Let $\lambda(j)$ denote $i_{\dim(V_j) - 1}$ for $1\leq j\leq s$, $\lambda(s + 1)$ denote $i_n$, $V_{s+1}$ denote $\Gamma(\PP^n_{k}, \OO_{\PP^n_k}(1))$, and $V_0$ denote $\{0\}$. 
             Then $\sigma^\vee_c$ is generated by the following elements:
             \[
                \{\omega_{\lambda(j)} - \omega_{\lambda(j + 1)}\}_{1\leq j\leq s}\cup\bigcup_{1\leq j\leq s+1}\{\omega_l - \omega_{\lambda(j)}, \omega_{\lambda(j)} - \omega_l\mid l\in \rho(V_j)\setminus\rho(V_{j-1})\}
             \]
             Let $\sigma_\varphi$ denote $(q_\varphi)_\RR(\sigma_c)\subset (N/N_\varphi)_\RR$, $M_\varphi$ denote the dual lattice of $N/N_\varphi$, and $q^*_\varphi$ denote the morphism $M_\varphi\rightarrow M$ induced by $q_\varphi$. 
             Then $k[\sigma^\vee_c\cap M]$ is isomorphic to the following ring as $k[\sigma^\vee_\varphi\cap M_\varphi]$-algebra: 
             \[
                k[\sigma^\vee_\varphi\cap M_\varphi]\otimes_k k[\bigcup_{1\leq j\leq s+1}\{(\frac{\chi^{\omega_l}}{\chi^{\omega_{\lambda(j)}}})^{\pm}\}_{l\notin\varphi, l\in\rho(V_j)\setminus\rho(V_{j-1})}]
             \]
             Let $I_c$ denote the kernel of $\iota^*\colon k[\sigma^\vee_c\cap M]\rightarrow \Gamma(Z(\sigma_c), \OO_Z)$ and we take  $l\notin\varphi$ and $1\leq j\leq s + 1$ such that $l\in \rho(V_{j})\setminus \rho(V_{j-1})$. 
             Because $f_l\in V_j$ and $\{f_{i_0}, \ldots, f_{\lambda(j)}\}$ is a basis of $V_{j}$, there exists $0\neq (a^{(l)}_m)_{0\leq m\leq\dim(V_j) - 1}\in k^{\dim(V_j)}$ such that $f_l = \sum_{0\leq m\leq\dim(V_j) - 1}a^{(l)}_m f_{i_m}$. 
             Let $h_l\in k[\sigma^\vee_c\cap M]$ denote the following element for $l\notin\varphi$:  
             \[
                \frac{\chi^{\omega_l}}{\chi^{\omega_{\lambda(j)}}} - (\sum_{0\leq m\leq\dim(V_j) - 1}a^{(l)}_m\frac{\chi^{\omega_{i_m}}}{\chi^{\omega_{\lambda(j)}}}), 
             \]
             where $j$ is an integer such that $1\leq j\leq s+1$ and $l\in \rho(V_{j})\setminus \rho(V_{j-1})$. 
             Let $g_l\in k[\sigma^\vee_\varphi\cap M_\varphi]$ denote the second term of $h_l$. 
             Then we can check that $h_l\in I_c$ for any $l\notin\varphi$. 
             Let $g\in k[\sigma^\vee_\varphi\cap M_\varphi]$ denote $\prod_{l\notin\varphi}g_l$. 
             We remark that $g$ is not a nilpotent element in $k[\sigma^\vee_\varphi\cap M_\varphi]$. 
             Then there exists the following commutative diagram: 
            \begin{equation*}
                \begin{tikzcd} 
                    k[\sigma^\vee_\varphi\cap M_\varphi]\ar[r, "q^*_\varphi"]\ar[d]& k[\sigma^\vee_c\cap M]\ar[d, "\iota^*"]\\
                    k[\sigma^\vee_\varphi\cap M_\varphi]_g\ar[r]& \Gamma(Z(\sigma_c), \OO_Z)
               \end{tikzcd}
            \end{equation*}
            where the left morphism is a localization map, and the bottom morphism is induced by the universality of a localization. 
            For each $l\notin\varphi$, $h_l = \frac{\chi^{\omega_l}}{\chi^{\omega_{\lambda(j)}}} - g_l\in I_c$ where $1\leq j\leq s+1$ is an integer such that $l\in\rho(V_j)\setminus \rho(V_{j - 1})$. 
            Thus, the bottom morphism is surjective because of the ring structure of $k[\sigma^\vee_c\cap M]$ as a  $k[\sigma^\vee_\varphi\cap M_\varphi]$-algebra. 
            The both ring $k[\sigma^\vee_\varphi\cap M_\varphi]_g$ and $\Gamma(Z(\sigma_c), \OO_Z)$ are $n$-dimentional integral domains because $Z(\sigma_c)$ is a scheme theoretic closure of an integral scheme $Z^\circ$ in $X(\sigma_c)$. 
            Thus, the bottom morphism is isomorphic. 
            Therefore, $\iota_\varphi$ is an open immersion. 
         \end{enumerate}
         In (4), we already have shown that $Z$ is proper over $k$. 
    \end{proof}
    We already know a projective space is a toric variety. 
    The following proposition shows that there are many non-toric mock toric structures of a projective space: 
    \begin{proposition}\label{prop:basic-example-hyperplane}
        For the notation above, we assume that $\{f_{i_0}, \ldots, f_{i_n}\}$ is a basis of $\Gamma(\PP^n_k, \OO_{\PP^n_k})$ for any integers $0\leq i_0 < \ldots < i_n\leq d$. 
        Then we remark that $\Phi$ is the following set:
        \[
            \Phi = \{\varphi\subset\{0, \ldots, d\}\mid |\varphi| = n + 1\}
        \]
        Let $\Delta_1$ denote a sub fan of $\Delta_0$, which consists of all cones whose dimension is less than $n + 1$. 
        From the assumption of $\mathcal{B}$, $\iota_0$ passes through $X(\Delta_1)$. 
        Let $\iota_1$ denote the closed immersion from $\PP^n_k$ to $X(\Delta_1)$ and $\Delta_{1, \varphi}$ denote $\{\sigma\in\Delta_1\mid \sigma\subset \sum_{i\in\varphi}\RR_{\geq 0}e_i\}$ for $\varphi\in\Phi$. 
        Then $(N, \Delta_1, \iota_1, \Phi, \{N_\varphi\}_{\varphi\in\Phi}, \{\Delta_{1, \varphi}\}_{\varphi\in\Phi}\})$ is a mock toric structure of $\PP^n_k$. 
    \end{proposition}
    \begin{proof}
        We check that this tuple satisfies the conditions in Definition \ref{def:mock-toric} from (1) to (5) in order. 
        \begin{enumerate}
            \item[(1)] We can check the condition (3) from the definition of $\{N_\varphi\}_{\varphi\in\Phi}$ and $\{\Delta_{1, \varphi}\}_{\varphi\in\Phi}$. 
            \item[(2)] We can check the condition (2) from the definition of $\Delta_1$ and $\{\Delta_{1, \varphi}\}_{\varphi\in\Phi}$. 
            \item[(3)] We can check the condition (3) from the definition of $\{N_\varphi\}_{\varphi\in\Phi}$ and $\{\Delta_{1, \varphi}\}_{\varphi\in\Phi}$. 
            \item[(4)] From the assumption of $\mathcal{B}$, $\PP^n_k\setminus T_N$ is a simple normal crossing divisor of $\PP^n_k$.  
            Thus, the condition (4) holds. 
            \item[(5)] We fix $\varphi\in\Phi$. 
            Let $\Delta_1(\varphi)$ denote $\{(q_\varphi)_\RR(\sigma)\mid \sigma\in\Delta_{1, \varphi}\}$. 
            We can check that $X(\Delta_1(\varphi))$ is isomorphic to $\PP^n_k$ and the rational map $\PP^d_k\dashrightarrow X(\Delta_1(\varphi))$ which induced by $q_\varphi\colon N\rightarrow N/N_\varphi$ is only a linear projection. 
            Let $P_\varphi$ denote this linear projection. 
            From the assumption of $\mathcal{B}$, $\iota(\PP^n_k)$ doesn't intersect the indeterminant locus of $P_\varphi$. 
            Moreover, the restriction $P_\varphi|_{\iota_0(\PP^n_k)}$ is an isomorphism. 
            Thus, the composition $(q_\varphi)_*\circ \iota_1|_{\PP^n_k\cap X(\Delta_{1, \varphi})}$ is an open immersion.  
        \end{enumerate}
    \end{proof}
    As a corollary of \ref{prop:basic-example-hyperplane}, we show that Del Pezzo surfaces have a mock toric structure. 
    \begin{corollary}\label{cor: Del-Pezzo-model}
        Let $\{p_1, \ldots, p_m\}\subset \PP^2_k(k)$ be finite $k-$rational points on $\PP^2_k$, and $X$ denote $\mathrm{Bl}_{\{p_1, \ldots, p_m\}}\PP^2_k$. 
        Then, $X$ has a mock toric structure. 
        In particular, Del Pezzo surfaces have a mock toric structure. 
    \end{corollary}
    \begin{proof}
        Let $\{L^{0, i}, L^{1, i}\}_{1\leq i\leq m}$ be a family of lines in $\PP^2_k$ such that the following two conditions hold: 
        \begin{itemize}
            \item[(a)] For any $1\leq i\leq m$, $L^{0, i}\cap L^{1, i} = \{p_i\}$. 
            \item[(b)] These $2m$ lines are all distinct and all three lines of $2m$ lines has no common point. 
        \end{itemize}
        Thus, from Proposition \ref{prop:basic-example-hyperplane}, $\PP^2_k$ has a mock toric structure associated with the hyperplane arrangement. 
        In particular, $p_1, \ldots, p_m$ are strata of this mock toric structure. 
        Thus, from Proposition \ref{prop:relative-induced-structure}, $\mathrm{Bl}_{p_1, \ldots, p_m}\PP^2_k$ has a mock toric structure. 
    \end{proof}

\section{Hypersurfaces of mock toric varieties}
        Let $Z$ be a scheme over $k$ and $(N$, $\Delta$, $\iota$, $\Phi$, $\{N_\varphi\}_{\varphi\in\Phi}$, $\{\Delta_\varphi\}_{\varphi\in\Phi})$ be a mock toric structure of $Z$. 
        Let $f$ denote a function in $k[Z^\circ]$ and $H$ denote the hypersurface in $Z^\circ$ defined by $f = 0$. 
        We have an interest in the structure of the scheme theoretic closure $\overline{H}$ of $H$ in $Z$. 
        When $Z$ is a toric variety, its structure is important to compute the stable birational volume of $H$. 
        Although $H$ and $\overline{H}$ are birational, it is important to take a compactification of $H$ when we calculate it (cf.\cite{NO22}). 
        For expanding on the previous point, the following two conditions are important for any $\sigma\in\Delta$:
        \begin{itemize}
            \item $Z_\sigma\not\subset\overline{H}$
            \item $Z_\sigma\cap \overline{H}$ is smooth over $k$. 
        \end{itemize}
        In this subsection, we give the criterion of the first condition.
        Moreover, by considering an appropriate subdivision $\Delta'$ of $\Delta$ and the induced mock toric variety $Z'\subset X(\Delta')$, we observe that the first condition is satisfied for all cones in $\Delta'$. 
        For checking these aspects, the valuations introduced in subsection 4.3 are valuable. 
        In the latter part of this section, we explore the inheritance of the first condition for the mock toric varieties defined by $Z$, as discussed in Section 4, even when assuming $Z$ and $f$ satisfy the first condition.
        
        In this section, we use the following notation:
        \begin{itemize}
            \item Let $M$ denote the dual lattice of $N$ and $k[M]$ denote a global section ring of an algebraic torus $T_N$. 
            \item For $\varphi\in\Phi$, let $M_\varphi$ denote the dual lattice of $N/N_\varphi$ and $q_\varphi^*\colon M_\varphi\hookrightarrow M$ be a group morphism induced by $q_\varphi\colon N\rightarrow N/N_\varphi$. 
            \item For $\omega\in M$, let $\chi^\omega$ denote a torus invariant monomial in $k[M]$ associated with $\omega$. 
            \item Let $\langle \cdot, \cdot\rangle\colon N\times M\rightarrow \ZZ$ be a natural pairing. 
            We identify elements in $N$ as torus invariant valuations so that we can check that $\langle v, \omega\rangle = v(\chi^\omega)$. 
            \item Let $k[Z^\circ]$ denote a global section ring $\Gamma(Z^\circ, \mathscr{O}_Z)$. 
            \item For $f\in k[Z^\circ]$, let $H^\circ_{Z, f}$ denote a closed subscheme of $Z^\circ$ defined by $f = 0$. 
            \item For $f\in k[Z^\circ]$, let $H_{Z, f}$ denote a scheme theoretic closure of $H^\circ_{Z, f}$ in $Z$. 
        \end{itemize}
        \subsection{Fineness and non-degeneracy of $f$ for $Z$}
        In this subsection, we consider two conditions of $f\in k[Z^\circ]$ for $\Delta$. 
        In particular, for one condition, we provide a criterion using the valuations defined in subsection 4.3.
        
        We prepare the notation for the hypersurfaces of the toric varieties. 
        The second condition in the following definition is called \textbf{Newton non-degenerate}. 
        \begin{definition}\label{def:recall-fine,non-dege}
            Let $N$ be a lattice of a finite rank, $\Delta$ be a strongly convex rational polyhedral fan in $N_\RR$, $M$ be the dual lattice of $N$, $f$ be a polynomial in $k[M]$, $H^\circ_{X(\Delta), f}$ denote a closed subscheme of $T_N$ defined by $f = 0$, and $H_{X(\Delta), f}$ denote a scheme theoretic closure of $H^\circ_{X(\Delta), f}$ in $X(\Delta)$.

            Then we define the following conditions about $f\in k[M]$: 
            \begin{enumerate}
                \item[(1)] If $O_\sigma\not\subset H_{X(\Delta), f}$ for all $\sigma\in\Delta$, then we call that $f$ is \textbf{fine} for $\Delta$. 
                \item[(2)] If $f$ is fine for $\Delta$ and $H_{X(\Delta)}\cap O_\sigma$ is smooth over $k$ for all $\sigma\in\Delta$, then we call that $f$ is \textbf{non-degenerate} for $\Delta$. 
            \end{enumerate}
        \end{definition}
        Mock toric varieties are covered by some open subsets of toric varieties. 
        The following proposition indicates that hypersurfaces of the mock toric variety are also covered by some open subsets of hypersurfaces of toric varieties.  
        \begin{proposition}\label{prop:fundamental hypersurface}
            Let $\sigma\in\Delta$ be a cone, $f\in k[Z^\circ]$ be a function, and $\varphi\in\Phi$ be an element such that $\sigma\in\Delta_\varphi$. 
            Then the following statements follow:
            \begin{enumerate}
                \item[(a)] There exists $g_\varphi\in k[M_\varphi]$ such that $f/\iota^*_\varphi(g_\varphi)\in k[Z^\circ]^*$. 
                \item[(b)] For such $g_\varphi\in k[M_\varphi]$, the following equation holds:
                \[
                    H_{X(\Delta(\varphi)),g_\varphi} \times_{X(\Delta(\varphi))} (Z\cap X(\Delta_\varphi)) = H_{Z, f}\cap (Z\cap X(\Delta_\varphi))
                \]
            \end{enumerate}
        \end{proposition}
        \begin{proof}
            We prove the statements from (a) to (b). 
            \begin{enumerate}
                \item[(a)] From the condition (5) in Definition \ref{def:mock-toric}, $\iota_\varphi\colon Z^\circ\hookrightarrow T_{N/N_\varphi}$ is an open immersion of affine schemes. 
                Moreover, $k[M_\varphi]$ is a UFD, and $k[Z^\circ]$ is a Noetherian normal integral scheme from Proposition\ref{prop:stratification-structure}(b), so the statement follows. 
                \item[(b)] There exists the following diagram whose two small squares are Cartesian squares:
                \begin{equation*}
                    \begin{tikzcd} 
                        H^\circ_{Z, f}\ar[r,"\iota'"]\ar[d]& Z^\circ\ar[d, "\iota_\varphi"]\ar[r]& Z\ar[d, "\iota_\varphi"]\cap X(\Delta_\varphi)\\
                        H^\circ_{X(\Delta(\varphi)), g_\varphi}\ar[r]& T_{N/N_\varphi}\ar[r]& X(\Delta(\varphi))
                    \end{tikzcd}
                \end{equation*}
                Thus, from Lemma \ref{lem:open-image} and the condition (5) in Definition \ref{def:mock-toric}, the statement holds.
            \end{enumerate}
        \end{proof}
        In toric varieties, the values $v(\chi^\omega)$ has additivity for $v\in N$ where $\omega\in M$. 
        The following proposition shows that the values $val_Z(-)(\chi)$ also have additivity for lattice points in each cone in $\Delta$, where $\chi\in k[Z^\circ]^*$.  
        \begin{proposition}\label{prop: units correspondence}
            Let $\chi\in k[Z^\circ]^*$ be a unit, $\sigma\in\Delta$ be a cone and $\varphi\in\Phi$ be an element such that $\sigma\in\Delta_\varphi$. 
            Then there exists $\omega\in M_\varphi$ such that $\val_Z(v)(\chi) = \val_Z(v)(\iota^*_\varphi(\chi^\omega))$ for any $v \in \sigma\cap N$. 
        \end{proposition}
        \begin{proof}
            Because $k[M_\varphi]$ is a UFD, there exists $a_\varphi, b_\varphi\in k[M_\varphi]$ such that these are co-prime and $\chi=\iota^*_\varphi(a_\varphi)/\iota^*_\varphi(b_\varphi)$. 
            In particular, $\iota^*_\varphi(a_\varphi)$ and $\iota^*_\varphi(b_\varphi)$ are units in $k[Z^\circ]^*$. 
            Thus, $H_{Z, \iota^*(a_\varphi)} = \emptyset$, so that $H_{X(\Delta(\varphi)), a_\varphi}\cap \iota_\varphi(Z\cap X(\Delta_\varphi)) = \emptyset$ from Proposition \ref{prop:fundamental hypersurface}(b). 
            Hence, $a_\varphi$  is fine for $\Delta(\varphi)$ from Proposition \ref{prop:basic-property1}(c). 
            Therefore, from Lemma \ref{lem: orbit and function}, there exists $\omega_1\in M_\varphi$ such that $(q_\varphi (v))(a_\varphi\chi^{\omega_1})= 0$ for any $v\in\sigma\cap N$. 
            In particular, $\val_Z(v)(\iota^*_\varphi(a_\varphi)) = \val_Z(v)(\iota^*_\varphi(\chi^{-\omega_1}))$ for any $v\in\sigma\cap N$. 
            The same argument follows for $b_\varphi$, so the statement holds.
        \end{proof}
        The following proposition shows the criterion for the conditions on the scheme theoretic closure of the hypersurface of $Z^\circ$ in $Z$, as described in the initial part of the section. 
        \begin{proposition}\label{prop: condition of fine}
            Let $f\in k[Z^\circ]$ be a function and $\sigma\in\Delta$ be a cone. 
            Then the following statements are equivalent:  
            \begin{enumerate}
                \item[(a)] There exists $\chi\in k[Z^\circ]^*$ such that $\val_Z(v)(\chi f) = 0$ for any $v\in \sigma\cap N$. 
                \item[(b)] 
                $Z_\sigma \not\subset H_{Z, f}$. 
            \end{enumerate}
            In particular, if these conditions hold, an ideal associated with a closed subscheme $Z(\sigma)\cap H_{Z, f}$ of an affine scheme $Z(\sigma)$ is generated by $\chi f$. 
        \end{proposition}
        \begin{proof}
            Let $\varphi\in\Phi$ be an element such that $\sigma\in\Delta_\varphi$ and $\sigma_\varphi$ denote $(q_\varphi)_\RR(\sigma)$. 
            From Proposition \ref{prop:fundamental hypersurface}(a) and (b) and Proposition \ref{prop: units correspondence}, there exists $g_\varphi\in k[M_\varphi]$ such that $\val_Z(v)(f) = (q_\varphi(v))(g_\varphi)$ for any $v\in \sigma\cap N$, and the following equation holds:
            \[
                H_{X(\Delta(\varphi)), g_\varphi}\times_{X(\Delta_\varphi)} Z(\sigma) = H_{Z, f}\cap Z(\sigma)
            \]
            Thus, from Proposition \ref{prop: units correspondence}, the condition of (a) is equivalent to the following condition (a'):
            \begin{enumerate}
                \item[(a')] There exists $\omega\in M_\varphi$ such that $q_\varphi(v)(\chi^\omega g_\varphi) = 0$ for any $v\in \sigma\cap N$. 
            \end{enumerate}
            On the other hand, from the above equation and Proposition \ref{prop:basic-property1}(c), the condition of (b) is equivalent to the following condition (b'):
            \begin{enumerate}
                \item[(b')] 
                $O_{\sigma_\varphi}\not\subset H_{X(\Delta(\varphi)), g_\varphi}$ 
            \end{enumerate}
            Therefore, from Lemma \ref{lem: orbit and function}, The condition of (a') and that of (b') are equivalent, so that of (a) and that of (b) are so.
            
            In this case, let $I_f$ denote an ideal of $k[Z^\circ]$ generated by $f$. 
            We show that $I_f\cap \Gamma(Z(\sigma), \OO_Z)$ is generated by $\chi f$. 
            If this claim holds, then the statement of the last part holds. 
            Let $g\in I_f\cap \Gamma(Z(\sigma), \OO_Z)$ be a function. 
            Then there exists $h\in k[Z^\circ]$ such that $g = (\chi f)h$. 
            Because $\val_Z(v)(\chi f) = 0$ and $\val_Z(v)(g)\geq 0$ for any $v\in\sigma\cap N$ from Proposition \ref{prop:characterization-of-affine mock toric}, $\val_Z(v)(h)\geq 0$ for any $v\in \sigma\cap N$ from the additivity of valuations. 
            Thus, $h\in \Gamma(Z(\sigma), \OO_Z)$ from Proposition \ref{prop:characterization-of-affine mock toric}  and this shows that $I_f\cap \Gamma(Z(\sigma), \OO_Z)$ is generated by $\chi f$. 
        \end{proof}
        As defined in Definition\ref{def:recall-fine,non-dege}, we define two properties of $f\in k[Z^\circ]$ for $\Delta$ in the following definition. 
        \begin{definition}\label{def:fine and non-degenerate}
            Let $f\in k[Z^\circ]$ be a function. 
            If $Z_\sigma\not\subset H_{Z, f}$ for any $\sigma\in\Delta$, then we call that $f$ is \textbf{fine} for $\Delta$. 
            If $f$ is fine for $\Delta$ and $H_{Z, f}\cap Z_\sigma$ is smooth over $k$ for any $\sigma\in\Delta$, then we call that $f$ is \textbf{non-degenerate} for $\Delta$. 
        \end{definition}
        \subsection{Valuation functions}
        For $f\in k[Z^\circ]$ and $v\in \supp(\Delta)\cap N$, we can compute an integer $\val_{Z}(v)(f)$. 
        In this subsection, we define a one map from $\supp(\Delta)$ to $\RR$ whose values is equal to $\val_{Z}(v)(f)$ at any $v\in \supp(\Delta)\cap N$. 
        This extended map is useful to construct canonical refinement $\Delta_f$ of $\Delta$ such that $f$ is fine for $\Delta_f$. 
        
        At first, we define it when $Z$ is affine toric variety. 
        \begin{definition}\label{def:Recall-val function for polynomial}
            We keep the notation in Definition \ref{def:recall-fine,non-dege}.  
            Let $\sigma\in\Delta$ be a cone. 
            \begin{enumerate}
                \item Let $\val^\sigma_f$ denote a function from $\sigma$ to $\RR$ which satisfies the following three conditions 
                \begin{enumerate}
                    \item[(i)] A function $\val^\sigma_f$ is continuous. 
                    \item[(ii)] A function $\val^\sigma_f$ is preserved by the scalar product of the non-negative real numbers, i.e., $\val^\sigma_f(rv) = r\val^\sigma_f(v)$ for any $v\in\sigma$ and $r\in\RR_{\geq 0}$. 
                    \item[(iii)] 
                    For any $v\in \sigma\cap N$, $\val^\sigma_f(v) = v(f)$. 
                \end{enumerate}
                We remark that such a function exists and is unique from Lemma \ref{lem: uniqueness of the function} and Lemma \ref{lem: property of val function}(c). 
                \item Let $\Omega^\sigma_f$ be a following subset of $M_\Q$: 
                \[
                    \Omega^\sigma_f = \{\omega\in M_\Q\mid \val^\sigma_f(v)\leq \langle v, \omega\rangle\quad\forall v\in\sigma\}
                \]
                \item For $\omega\in\Omega^\sigma_f$, let $C^\sigma_f(\omega)$ denote a following subset of $\sigma$:
                \[
                    C^\sigma_f(\omega) = \{v\in\sigma\mid \val^\sigma_f(v) = \langle v, \omega\rangle\}
                \]
            \end{enumerate}
        \end{definition}
        From now on, we define the map $\val_f\colon \supp(\Delta)\rightarrow \RR$ as follows and consider the basic property of $\val_f$.   
        \begin{proposition}\label{prop: mock val function}
            The following statements follow:
            \begin{enumerate}
                \item[(a)] There unique exists a function $\val_f\colon \supp(\Delta)\rightarrow \RR$ such that it satisfies the following three conditions:
                \begin{enumerate}
                    \item[(1)] A function $\val_f$ is continuous. 
                    \item[(2)] A function $\val_f$ preserves the scalar product of non-negative real numbers, i.e. $\val_f(rv) = r\val_f(v)$ for any $v\in\supp(\Delta)$ and $r\in\RR_{\geq 0}$.
                    \item[(3)] For any $v\in \supp(\Delta)\cap N$, $\val_f(v) = \val_Z(v)(f)$.
                \end{enumerate}
                \item[(b)] Let $f$ and $g\in k[Z^\circ]$ be functions. 
                Then $\val_{fg} = \val_{f} + \val_{g}$. 
                \item[(c)] Let $\sigma\in\Delta$ be a cone and $\val^\sigma_f$ denote the restriction of $\val_f$ for $\sigma$. 
                Then $\val^\sigma_f$ is an upper convex function. 
                \item[(d)] Let $\sigma\in\Delta$ be a cone and $\chi\in k[Z^\circ]^*$ be a unit. 
                Then $\val^\sigma_\chi$ is a linear function over $\sigma$. 
            \end{enumerate}
        \end{proposition}
        \begin{proof}
            We prove the statements from (a) to (d) in order. 
            \begin{enumerate}
                \item[(a)] The uniqueness of $\val_f$ follows from Lemma \ref{lem: uniqueness of the function}. 
                First, we show that $\val_f|_{\sigma}$ exists for any $\sigma\in\Delta$. 
                Let $\varphi\in\Phi$ be an element such that $\sigma\in\Delta_\varphi$ and $\sigma_\varphi$ denote $(q_\varphi)_\RR(\sigma)$. 
                From the proof of Proposition \ref{prop: condition of fine}, there exists $g_\varphi\in k[M_\varphi]$ such that $\val_Z(v)(f) = (q_\varphi(v))(g_\varphi)$ for any $v\in\sigma\cap N$. 
                Let $G_\sigma$ denote $\val^{\sigma_\varphi}_{g_\varphi}\circ (q_\varphi)_\RR|_{\sigma}$. 
                From Lemma \ref{lem: property of val function}(c), $G_\sigma$ is a continuous map and is preserved by the scalar product of non-negative real numbers because $(q_\varphi)_\RR$ is an $\RR-$linear map. 
                Moreover, $G_\sigma(v) = \val_Z(v)(f)$ for any $v\in\sigma\cap N$. 
                The definition of $G_\sigma$ is independent of the choice of $\varphi\in\Phi$ such that $\sigma\in\Delta_\varphi$ because of Lemma \ref{lem: uniqueness of the function}. 
                Thus, $\val_f|_\sigma$ exists. 
                
                In addition to this, $\{G_\sigma\}_{\sigma\in\Delta}$ can be glued from Lemma \ref{lem: uniqueness of the function} and this function is continuous because each $\sigma$ is a closed subset of $\supp(\Delta)$ and $|\Delta|$ is finite. 
                \item[(b)] The functions $\val_f$ and $\val_g$ are continuous and preserved by the scalar product of non-negative real numbers so that $\val_{f} + \val_{g}$ has the same conditions. 
                Moreover, for any $v\in\supp(\Delta)\cap N$, the following equations hold, so that the statement holds from Lemma \ref{lem: uniqueness of the function}:
                \begin{align*}
                    \val_{fg}(v) &= \val_Z(v)(fg)\\
                    &= \val_Z(v)(f) + \val_Z(v)(g) \\
                    &= \val_{f}(v) + \val_{g}(v)\\
                    &= (\val_{f} + \val_{g})(v)
                \end{align*}
                \item[(c)] We use the notation in the proof of (a).  
                We have already shown  that $\val^\sigma_f = \val^{\sigma_\varphi}_{g_\varphi}\circ(q_\varphi)_\RR|_{\sigma}$. 
                Thus, from Lemma \ref{lem: property of val function}(d), $\val^\sigma_f$ is an upper convex function because $(q_\varphi)_\RR$ is an $\RR$-linear map. 
                \item[(d)] Let $\varphi\in\Phi$ be an element such that $\sigma\in\Delta_\varphi$. 
                From Proposition \ref{prop: units correspondence}, there exists $\omega\in M_\varphi$ such that $\val_\chi(v) = \langle q_\varphi(v), \omega\rangle$ for any $v\in\sigma\cap N$. 
                Thus, $\val_{\chi}(v) = \langle(q_\varphi)_\RR(v), \omega\rangle$ for any $v\in\sigma$ from Lemma \ref{lem: uniqueness of the function}. 
                Therefore, $\val^\sigma_\chi$ is a linear function over $\sigma$. 
            \end{enumerate}
        \end{proof}
        In the following proposition, we construct the refinement $\Delta_f$ of $\Delta$ from $\val_f$. 
        \begin{proposition}\label{prop: refinement along function}
            Let $f\in k[Z^\circ]$ be a function and $\sigma\in\Delta$ be a cone. 
            \begin{itemize}
                \item Let $\Omega^\sigma_f$ denote the following set:
                \[
                    \Omega^\sigma_f = \{(\chi, b)\in k[Z^\circ]^*\times\ZZ_{>0}\mid val_f(v) \leq \frac{1}{b}\val_\chi(v)\quad(\forall v\in\sigma)\}
                \]
                \item For $(\chi, b)\in \Omega^\sigma_f$, let $C^\sigma_f(\chi, b)$ denote the following set:
                \[
                    C^\sigma_f(\chi, b) = \{v\in \sigma\mid val_f(v) = \frac{1}{b}\val_\chi(v)\}
                \]
            \end{itemize}
            Then the following statements follow:
            \begin{enumerate}
                \item[(a)] Let $\sigma\in\Delta$ be a cone and $\varphi\in\Phi$ be an element such that $\sigma\in\Delta_\varphi$. 
                Let $\sigma_\varphi$ denote $(q_\varphi)_\RR(\sigma)$. 
                Then there exists $g_\varphi\in k[M_\varphi]$ such that $\val^\sigma_f(v) = \val^{\sigma_\varphi}_{g_\varphi}((q_\varphi)_\RR(v))$ for any $v\in\sigma$. 
                \item[(b)] Let $g_\varphi\in k[M_\varphi]$ be a function that satisfies the condition in (a). 
                For any $(\chi, b)\in\Omega^\sigma_f$, there exists $\omega\in\Omega^{\sigma_\varphi}_{g_\varphi}$ such that the following equation holds:
                \[
                    (q_\varphi)_\RR(C^\sigma_f(\chi, b)) = C^{\sigma_\varphi}_{g_\varphi}(\omega)
                \]
                Conversely, for any $\omega\in\Omega^{\sigma_\varphi}_{g_\varphi}$, there exists $(\chi, b)\in\Omega^\sigma_f$ such that the following equation holds:
                \[
                    C^{\sigma_\varphi}_{g_\varphi}(\omega) = (q_\varphi)_\RR(C^\sigma_f(\chi, b))
                \]
                \item[(c)] The set $\{C^\sigma_f(\chi, b)\}_{(\chi, b)\in\Omega^\sigma_f}$ is a strongly convex rational polyhedral fan and a refinement of $\sigma$.
                \item[(d)] Let $\tau$ be a face of $\sigma$. 
                Then $\Omega^\sigma_f\subset\Omega^\tau_f$. 
                Moreover, for any $(\chi, b)\in\Omega^\sigma_f$, we have $C^\sigma_f(\chi, b)\cap\tau = C^\tau_f(\chi, b)$. 
                In particular, $C^\tau_f(\chi, b)\preceq C^\sigma_f(\chi, b)$.  
                \item[(e)] The following set $\Delta_f$ is a strongly convex rational polyhedral fan in $N_\RR$ and a refinement of $\Delta$:
                \[
                    \Delta_f = \bigcup_{\sigma\in\Delta}\{C^\sigma_f(\chi, b)\}_{(\chi, b)\in\Omega^\sigma_f}
                \]
                \item[(f)] For any $(\chi, b)\in\Omega^\sigma_f$, there exists $\chi'\in k[Z^\circ]^*$ such that $\val_{f}(v) = \val_{\chi'}(v)$ for any $v\in C^\sigma_f(\chi, b)$. 
            \end{enumerate}
        \end{proposition}
        \begin{proof}
            We prove the statements from (a) to (f) in order. 
            \begin{enumerate}
                \item[(a)] The argument about this statement was done in the proof of Proposition \ref{prop: mock val function}(a). 
                \item[(b)] We remark that the restriction $(q_\varphi)_\RR|_\sigma$ is bijective from the condition (3) in Definition \ref{def:mock-toric}. 
                Let $\chi\in k[Z^\circ]^*$ be a unit  and $b\in\ZZ_{>0}$ be a positive integer. 
                Then from Proposition \ref{prop: units correspondence}, there exists $\omega\in M_\varphi$ such that $\val^\sigma_\chi(v) = \val^{\sigma_\varphi}_{\chi^\omega}\circ (q_\varphi)_\RR(v)$ for any $v\in \sigma$.  
                In particular,  $\frac{1}{b}\val^\sigma_\chi(v) = \langle(q_\varphi)_\RR(v), 
                \frac{\omega}{b}\rangle$ for any $v\in \sigma$. 
                Thus, if $(\chi, b)\in\Omega^\sigma_f$, then $\frac{\omega}{b}\in\Omega^{\sigma_\varphi}_{g_\varphi}$ and $(q_\varphi)_\RR(C^\sigma_f(\chi, b)) = C^{\sigma_\varphi}_{g_\varphi}(\frac{\omega}{b})$. 

                Conversely, let $\omega'\in (M_\varphi)_\Q$ be an element. 
                Then there exists $\omega\in M_\varphi$ and $b\in\ZZ_{>0}$ such that $\omega' = \frac{\omega}{b}\in (M_\varphi)_\Q$. 
                Let $\chi\in k[Z^\circ]^*$ denote $\iota^*_\varphi(\chi^\omega)$. 
                We can check that $\frac{1}{b}\val^\sigma_\chi(v) = \langle(q_\varphi)_\RR(v), \omega'\rangle$ for any $v\in \sigma$.
                Thus, if $\omega'\in \Omega^{\sigma_\varphi}_{g_\varphi}$, then $(\chi, b)\in\Omega^\sigma_f$ and $C^{\sigma_\varphi}_{g_\varphi}(\omega) = (q_\varphi)_\RR(C^\sigma_f(\chi, b))$. 
                \item[(c)] We already have shown that $(q_\varphi)|_{\langle\sigma\rangle\cap N}\colon \langle\sigma\rangle\cap N \rightarrow \langle\sigma_\varphi\rangle\cap (N/N_\varphi)$ is isomorphic in Proposition\ref{prop: first prop}(d). 
                Thus, from (b) and Lemma \ref{lem: property of val function}(b), $\{C^\sigma_f(\chi, b)\}_{(\chi, b)\in\Omega^\sigma_f}$ is a rational polyhedral convex fan in $N_\RR$ and it is a refinement of $\sigma$. 
                Because $\sigma$ is strongly convex, this fan is also strongly convex. 
                \item[(d)] From the definition of $\Omega^\sigma_f$ and $\Omega^\tau_f$, we can check easily that $\Omega^\sigma_f\subset \Omega^\tau_f$ and $C^\sigma_f(\chi, b)\cap\tau = C^\tau_f(\chi, b)$ for any $(\chi, b)\in\Omega^\sigma_f$. 
                Let $(\chi, b)\in\Omega^\sigma_f$ and $\omega\in \sigma^\vee$ be elements such that $\sigma\cap \omega^\perp = \tau$. 
                Because $C^\sigma_f(\chi, b)\subset \sigma$, we have $\omega\in(C^\sigma_f(\chi, b))^\vee$ too. 
                Thus, $C^\sigma_f(\chi, b)\cap \omega^\perp = C^\sigma_f(\chi, b)\cap \tau$ is a face of $C^\sigma_f(\chi, b)$. 
                \item[(e)] For each $\sigma\in\Delta$, we have already shown that $\{C^\sigma_f(\chi, b)\}_{(\chi, b)\in\Omega^\sigma_f}$is a strongly convex rational polyhedral fan in $N_\RR$ and it is a refinement of $\sigma$.  
                Thus, it is enough to show that for any two cones in $\Delta_f$, the intersection of those is a common face of both cones. 
                Let $\sigma$ and $\tau\in\Delta$ be cones, and $(\chi_1, b_1)\in\Omega^\sigma_f$ and $(\chi_2, b_2)\in\Omega^\tau_f$ be elements. 
                Then from (d), $\{(\chi_1, b_1), (\chi_2, b_2)\}\subset \Omega^{\sigma\cap\tau}_f$. 
                Moreover, $C^\sigma_f(\chi_1, b_1)\cap \tau = C^{\sigma\cap\tau}_f(\chi_1, b_1)$ and $C^\tau_f(\chi_2, b_2)\cap \sigma = C^{\sigma\cap\tau}_f(\chi_2, b_2)$ from (d). 
                Furthermore, $C^{\sigma\cap\tau}_f(\chi_1, b_1)\preceq C^\sigma_f(\chi_1, b_1)$ and $C^{\sigma\cap\tau}_f(\chi_2, b_2)\preceq C^\tau_f(\chi_2, b_2)$. 
                Thus, $C^\sigma_f(\chi_1, b_1)\cap C^\tau_f(\chi_2, b_2) = C^{\sigma\cap\tau}_f(\chi_1, b_1)\cap C^{\sigma\cap\tau}_f(\chi_2, b_2)$. 
                The cone $C^{\sigma\cap\tau}_f(\chi_1, b_1)\cap C^{\sigma\cap\tau}_f(\chi_2, b_2)$ is a common face of $C^{\sigma\cap\tau}_f(\chi_1, b_1)$ and $C^{\sigma\cap\tau}_f(\chi_2, b_2)$ from (b). 
                Therefore, this cone is a common face of $C^\sigma_f(\chi_1, b_1)$ and $C^\tau_f(\chi_2, b_2)$. 
                \item[(f)]  We use the notation in the proof of (a) and (b).  
                Then there exists $\omega\in \Omega^{\sigma_\varphi}_{g_\varphi}$ such that $(q_\varphi)_\RR(C^\sigma_f(\chi, b)) = C^{\sigma_\varphi}_{g_\varphi}(\omega)$. 
                Thus, from Lemma \ref{lem: property of val function}(g), there exists $\omega_0\in M_\varphi$ such that $\val^{\sigma_\varphi}_{g_\varphi}(v') = \val^{\sigma_\varphi}_{\chi^{\omega_0}}(v')$ for any $v'\in C^{\sigma_\varphi}_{g_\varphi}(\omega)$. 
                Let $\chi'$ denote $\iota^*_\varphi(\chi^{\omega_0})\in k[Z^\circ]^*$. 
                Then $\val^\sigma_{\chi'} = \val^{\sigma_\varphi}_{\chi^{\omega_0}}\circ (q_\varphi)_\RR|_\sigma$. 
                Hence, from the definition of $g_\varphi$, $\val^\sigma_f(v) = \val^\sigma_{\chi'}(v)$ for any $v\in C^\sigma_f(\chi, b)$. 
            \end{enumerate}
        \end{proof}
        The following corollary indicates that for any $f\in k[Z^\circ]$, there exists a refinement $\Delta'$ of $\Delta$ such that $f$ is fine for $\Delta'$. 
        \begin{corollary}\label{cor: example of fine fan}
            We keep the notation in Proposition \ref{prop: refinement along function}.  
            Let $\Delta'$ be a refinement of $\Delta_f$. 
            Then $f$ is fine for $\Delta'$. 
        \end{corollary}
        \begin{proof}
            We can check it easily from Proposition\ref{prop: condition of fine} and Corollary \ref{prop: refinement along function}(e) and (f).
        \end{proof}
        \subsection{The inheritance of fineness and non-degeneracy}
        In this subsection, we assume that $f$ is fine for $\Delta$, and we consider what holds under the assumption. 

        In the previous subsection, we analyze $H_{Z, f}$ through the intermediary of the hypersurfaces in open subschemes of toric varieties. 
        The following lemma shows that if $f$ is fine for $\Delta$, the fineness also holds for these hypersurfaces of toric varieties.   
        \begin{lemma}\label{lem:fine for varphi}
            We keep the notation in Proposition \ref{prop: refinement along function}. 
            We assume that $f$ is fine for $\Delta$.  
            Let $\varphi\in\Phi$ be an element and $g_\varphi\in k[M_\varphi]$ be a function such that $f/\iota^*_\varphi(g_\varphi)\in k[Z^\circ]^*$.  
            Then $g_\varphi$ is fine for $\Delta(\varphi)$. 
        \end{lemma}
        \begin{proof}
            From Proposition \ref{prop:basic-property1}(c) and Proposition \ref{prop:fundamental hypersurface}(b), 
            $H_{X(\Delta(\varphi)), g_\varphi}$ does not contain any torus orbits of $X(\Delta(\varphi))$. 
        \end{proof} 
        If $f$ is fine for $\Delta$, then $H_{Z, f}\cap Z_\sigma$ is a hypersurface of $Z_\sigma$ for any $\sigma\in\Delta$. 
        In the following proposition, this hypersurface also has fineness. 
        \begin{proposition}\label{prop: heredity-finess-for-orbit}
            Let $\sigma\in\Delta$ be a cone and $f\in k[Z^\circ]$ be a function such that $f$ is fine for $\Delta$. 
            From Proposition \ref{prop: condition of fine}, let $\chi\in k[Z^\circ]^*$ be a unit such that $\val_Z(v)(\chi f) = 0$ for any $v\in\sigma\cap N$. 
            Let $p^\sigma$ denote the quotient morphism $\Gamma(Z(\sigma), \OO_Z)\rightarrow \Gamma(Z_\sigma, \OO_{\overline{Z_\sigma}})$ associated with the closed immersion $Z_\sigma\rightarrow Z(\sigma)$, and $f^\sigma$ denote $p^\sigma(\chi f)$. 
            Then the following statements follow:
            \begin{enumerate}
                \item[(a)] As a closed subschme of $Z_\sigma$, $H_{Z, f}\cap Z_\sigma = H^\circ_{\overline{Z_\sigma}, f^\sigma}$. 
                \item[(b)] We use the notation in Definition \ref{def: mock induced for orbit}. Then $f^\sigma$ is fine for $\Delta[\sigma]$. 
                \item[(c)] As a closed subschme of $\overline{Z_\sigma}$, $H_{Z, f}\cap \overline{Z_\sigma} = H_{\overline{Z_\sigma}, f^\sigma}$. 
                \item[(d)]Let ${}_f\Delta$ denote the following subset of $\Delta$:
                \[
                    {}_f\Delta = \{\tau\in\Delta\mid H_{Z, f}\cap Z_\tau\neq \emptyset\}
                \]
                Then ${}_f\Delta$ is a subfan of $\Delta$. 
                \item[(e)] If $f$ is non-degenerate for $\Delta$, then $f^\sigma$ is non-degenerate for $\Delta[\sigma]$ too. 
            \end{enumerate}
        \end{proposition}
        \begin{proof}
            We prove the statements from (a) to (e) in order. 
            \begin{enumerate}
                \item[(a)] From Proposition \ref{prop: condition of fine}, the ideal of $\Gamma(Z(\sigma), \OO_Z)$ associated with a closed subscheme $H_{Z, f}\cap Z(\sigma)\subset Z(\sigma)$ is generated by $\chi f$. Thus, the statement follows. 
                \item[(b)] From (a), we remark that $H_{\overline{Z_\sigma}, f^\sigma}\subset H_{Z, f}\cap \overline{Z_\sigma}$. 
                Let $\tau'\in\Delta[\sigma]$ be a cone. 
                Then there exists $\tau\in\Delta$ such that $\sigma\subset\tau$ and $(\pi^\sigma)_\RR(\tau) = \tau'$. 
                Let $Z^\sigma$ denote $\overline{Z_\sigma} = Z\cap \overline{O_\sigma}$. 
                Then $(Z^\sigma)_{\tau'} = Z_\tau$. 
                Thus, $H_{\overline{Z_\sigma}, f^\sigma}$ does not contain $(Z^\sigma)_{\tau'}$ from the assumption of $f$, the above remark, and Proposition \ref{prop:stratification-structure}(a). 
                Therefore, $f^\sigma$ is fine for $\Delta[\sigma]$. 
                \item[(c)] Let $\varphi\in\Phi$ be an element such that $\sigma\in\Delta_\varphi$. 
                From Proposition \ref{prop:fundamental hypersurface}(a) and Lemma \ref{lem:fine for varphi}, there exists $g_\varphi\in k[M_\varphi]$ such that $f/\iota^*_\varphi(g_\varphi)\in k[Z^\circ]^*$ and $g_\varphi$ is fine for $\Delta(\varphi)$. 
                Let $\eta\in k[Z^\circ]^*$ denote $f/\iota^*_\varphi(g_\varphi)$. 
                Then $H_{Z, f} = H_{Z, \eta^{-1} f}$ and $\chi f = (\chi\eta)(\eta^{-1} f)$, so we may assume that $\eta = 1$, i.e. $f = \iota^*_\varphi(g_\varphi)$. 
                Moreover, from Proposition \ref{prop: units correspondence}, there exists $\omega\in M_\varphi$ such that $\val_Z(v)(\chi) = \val_Z(v)(\iota^*_\varphi(\chi^\omega))$ for any $v\in\sigma\cap N$. 
                In particular, $\val_Z(v)(\chi^{-1}\iota^*_\varphi(\chi^\omega)) = 0$ for any $v\in\sigma\cap N$. 
                Thus, from Proposition \ref{prop:characterization-of-affine mock toric}, $\chi^{-1}\iota^*_\varphi(\chi^\omega)\in\Gamma(Z(\sigma), \OO_Z)^*$. 
                Thus, we may assume that $\chi = \iota^*_\varphi(\chi^\omega)$ for some $\omega\in M_\varphi$. 
                Especially, $q_\varphi(v)(\chi^\omega g_\varphi) = 0$ for any $v \in \sigma\cap N$ and $\chi f = \iota^*_\varphi(\chi^\omega g_\varphi)$. 

                By the way, let $\sigma_\varphi$ denote $(q_\varphi)_\RR(\sigma)$ and there exists the following commutative diagram:
                \begin{equation*}
                    \begin{tikzcd} 
                        \Gamma(X(\sigma_\varphi), \OO_{X(\Delta(\varphi))})\ar[d, hook, "\iota^*_\varphi"]\ar[r, "p^\sigma_\varphi"]& \Gamma(O_{\sigma_\varphi}, \OO_{\overline{O_{\sigma_\varphi}}})\ar[d, hook, "(\iota^\sigma)^*_\varphi"]\\
                        \Gamma(Z(\sigma), \OO_{Z})\ar[r, "p^\sigma"] & \Gamma(Z_\sigma, \OO_{\overline{Z_\sigma}})
                    \end{tikzcd}
                \end{equation*}
                where horizontal maps are surjective ring morphisms induced by closed immersions, and vertical maps are injective ring morphisms induced by open immersions in the condition (5) in Definition \ref{def:mock-toric}. 
                Let $g^\sigma_\varphi$ denote $p^\sigma_\varphi(\chi^\omega g_\varphi)$ $\in$ $\Gamma(O_{\sigma_\varphi},$ $\OO_{\overline{O_{\sigma_\varphi}}})$. 
                Then, from the above diagram, there exists the following diagram whose all small squares are Cartesian squares:
                \begin{equation*}
                    \begin{tikzcd} 
                        H^\circ_{\overline{Z_\sigma}, f^\sigma}\ar[d, hook]\ar[r]&Z_\sigma \ar[d]\ar[r]& \overline{Z_\sigma}\cap X(\Delta_\varphi)\ar[d, hook, "(\iota^\sigma)_\varphi"]\\
                        H^\circ_{\overline{O_{\sigma_\varphi}}, g^\sigma_\varphi}\ar[r]& O_{\sigma_\varphi}\ar[r] & \overline{O_{\sigma_\varphi}}
                    \end{tikzcd}
                \end{equation*}
                Because $(\iota^\sigma)_\varphi$ is an open immersion, the following equation holds from Lemma \ref{lem:open-image}:
                \[
                    H_{\overline{Z_\sigma}, f^\sigma}\cap X(\Delta_\varphi) = H_{\overline{O_{\sigma_\varphi}}, g^\sigma_\varphi}\times_{\overline{O_{\sigma_\varphi}}} (\overline{Z_\sigma}\cap X(\Delta_\varphi))
                \]
                On the other hand, we recall that $g_\varphi$ is fine for $\Delta(\varphi)$. 
                Thus, from Lemma \ref{lem: heredity of fineness for orbit}(c), $H_{\overline{O_{\sigma_\varphi}}, g^\sigma_\varphi} = H_{X(\Delta(\varphi)), g_\varphi}\cap \overline{O_{\sigma_\varphi}}$.  
                Hence, from Proposition \ref{prop:fundamental hypersurface}(b), the following equation holds:
                \begin{align*}
                    H_{\overline{Z_\sigma}, f^\sigma}\cap X(\Delta_\varphi) &= H_{\overline{O_{\sigma_\varphi}}, g^\sigma_\varphi}\times_{\overline{O_{\sigma_\varphi}}} (\overline{Z_\sigma}\cap X(\Delta_\varphi))\\
                    &= (H_{X(\Delta(\varphi)), g_\varphi}\cap \overline{O_{\sigma_\varphi}})\times_{\overline{O_{\sigma_\varphi}}} (\overline{Z_\sigma}\cap X(\Delta_\varphi))\\
                    &= (H_{X(\Delta(\varphi)), g_\varphi}\times_{X(\Delta(\varphi))} \overline{O_{\sigma_\varphi}})\times_{\overline{O_{\sigma_\varphi}}} (\overline{Z_\sigma}\cap X(\Delta_\varphi))\\
                    &= (H_{X(\Delta(\varphi)), g_\varphi}\times_{X(\Delta(\varphi))} (Z\cap X(\Delta_\varphi)))\times_{(Z\cap X(\Delta_\varphi))} (\overline{Z_\sigma}\cap X(\Delta_\varphi))\\
                    &= (H_{Z, f}\cap X(\Delta_\varphi))\times_{(Z\cap X(\Delta_\varphi))} (\overline{Z_\sigma}\cap X(\Delta_\varphi))\\
                    &= (H_{Z, f}\cap\overline{Z_\sigma})\cap X(\Delta_\varphi)
                \end{align*}
                Because $\overline{Z_\sigma}$ is covering by a family of open subsets $\{X(\Delta_\varphi)\}_{\sigma\in\Delta_\varphi}$ of $X(\Delta)$, we have $H_{Z, f}\cap \overline{Z_\sigma} = H_{\overline{Z_\sigma}, f^\sigma}$. 
                \item[(d)] We assume that $H_{Z, f}\cap Z_\sigma = \emptyset$. 
                Then $f^\sigma\in\Gamma(Z_\sigma, \OO_{\overline{Z_\sigma}})^*$ from (a). 
                Thus, from (c), $H_{Z, f}\cap \overline{Z_\sigma} = \emptyset$. 
                In particular, from Proposition \ref{prop:stratification-structure}(a), $H_{Z, f}\cap Z_\tau = \emptyset$ for any $\tau\in\Delta$ such that $\sigma\subset\tau$. 
                This shows that ${}_f\Delta$ is a subfan of $\Delta$. 
                \item[(e)] From the definition of a mock toric structure of $\overline{Z_\sigma} = Z\cap\overline{O_\sigma}$ in Definition \ref{def: mock induced for orbit}, the strata of $\overline{Z_\sigma}$ can be identified with the strata of $Z$ which is contained in $\overline{Z_\sigma}$. 
                Thus, from (c), $f^\sigma$ is non-degenerate for $\Delta[\sigma]$. 
            \end{enumerate}
        \end{proof}
        The following proposition shows that if $f$ is decomposed in some non-units in $k[Z^\circ]$, then all factors are fine for $\Delta$. 
        \begin{proposition}\label{prop:val func of multiple poly}
            Let $f, f_1, \ldots, f_r\in k[Z^\circ]$ be nonzero functions. 
            We assume that $f = f_1\cdot f_2\cdot \cdots\cdot f_r$. 
            Then the following statement follows:  
            \begin{enumerate}
                \item[(a)] The following statements are equivalent:
                    \begin{enumerate}
                        \item[(1)] A function $f$ is fine for $\Delta$.
                        \item[(2)] For any $1\leq i\leq r$, $f_i$ is fine for  $\Delta$.
                    \end{enumerate}
                \item[(b)] If $f$ is non-degenerate for $\Delta$, then $f_i$ is so for any $1\leq i\leq r$. 
                \item[(c)] We assume that $f$ is non-degenerate for $\Delta$ and $f = f_1\cdot f_2\cdot \cdots\cdot f_r$ is an irreducible decomposition of $f$. 
                Then, as a closed subscheme of $Z$, the following equation holds:
                \[
                    H_{Z, f} = \coprod_{1\leq i\leq r} H_{Z, f_i}
                \]
                Moreover, the above equation is a connected decomposition of $H_{Z, f}$. 
            \end{enumerate} 
        \end{proposition}
        \begin{proof}
            We prove these statements from (a) to (c) in order. 
            \begin{enumerate}
                \item[(a)] To show one direction is easy from Proposition \ref{prop: condition of fine}, so we assume that the condition (1) holds. 
                Let $\sigma\in\Delta$ be a cone and $v_0\in \sigma^\circ$ be an element. 
                From the assumption and Proposition \ref{prop: condition of fine}, there exists $\chi\in k[Z^\circ]^*$ such that $\val^\sigma_f + \val^\sigma_\chi \equiv 0$. 
                We repalce $f$ with $\chi f$ and $f_1$ with $\chi f_1$, and we may assume that $\val^\sigma_f \equiv 0$. 
                From Proposition \ref{prop: refinement along function}(c), for any $1\leq i\leq r$, there exists $(\chi_i, b_i)\in\Omega^\sigma_{f_i}$ such that $v_0\in C^\sigma_{f_i}(\chi_i, b_i)$. 
                Then the following equation and inequality hold for any $v\in\sigma$:
                \begin{align*}
                    \sum_{1\leq i\leq r}\val^\sigma_{f_i}(v)&\leq \sum_{1\leq i\leq r} \frac{1}{b_i}\val^\sigma_{\chi_i}(v)\\
                    \sum_{1\leq i\leq r}\val^\sigma_{f_i}(v_0)&= \sum_{1\leq i\leq r} \frac{1}{b_i}\val^\sigma_{\chi_i}(v_0)\\
                \end{align*} 
                From Proposition \ref{prop: mock val function}(b) and the assumption of $f$, the left-hand sides of the above equation and inequality are equal to $0$, respectively. 
                Moreover, from Proposition \ref{prop: mock val function} (d), $\sum_{1\leq i\leq r}\frac{1}{b_i}\val^\sigma_{\chi_i}$ is a linear function on $\sigma$. 
                Thus, $\sum_{1\leq i\leq r}\frac{1}{b_i}\val^\sigma_{\chi_i} \equiv 0$ because $v_o\in\sigma^\circ$.
                 
                Next, we show that $\bigcap_{1\leq i\leq r}C^\sigma_{f_i}(\chi_i, b_i) = \sigma$ by a contradiction. 
                We assume that $w\in\sigma\setminus\bigcap_{1\leq i\leq r}C^\sigma_{f_i}(\chi_, b_i)$ exists. 
                From the definition of $C^\sigma_{f_i}(\chi_i, b_i)$, we have $\sum_{1\leq i\leq r}\val^\sigma_{f_i}(w)< \sum_{1\leq i\leq r}\frac{1}{b_i}\val^\sigma_{\chi_i}(w) = 0$. 
                However, $\sum_{1\leq i\leq r} val^\sigma_{f_i}(w)= 0$, so it is a contradiction. 
            
                Thus, for any $1\leq i\leq r$, we have $C^\sigma_{f_i}(\chi_i, b_i) = \sigma$. 
                Hence, from Proposition \ref{prop: refinement along function} (f), for any $1\leq i\leq r$, there exists $\chi'_i\in k[Z^\circ]^*$ such that $\val^\sigma_{f_i} = \val^\sigma_{\chi'_i}$. 
                In particular, $\val_Z(v)(f_i\cdot{\chi'_i}^{-1}) = 0$ for any $v\in \sigma\cap N$. 
                \item[(b)] Let $\sigma\in\Delta$ be a cone. 
                From (a) and Proposition \ref{prop: condition of fine}, there exists $\chi_1, \cdots, \chi_r\in k[Z^\circ]^*$ such that $\val^\sigma_{f_i} + \val^\sigma_{\chi_i} \equiv 0$ for any $1\leq i\leq r$. 
                Thus, for any $1\leq i\leq r$, the ideal of $\Gamma(Z(\sigma), \OO_Z)$ associated with the closed subscheme $H_{Z, f_i}\cap Z(\sigma)$ of $Z(\sigma)$ is generated by $\chi_i f_i$. 
                Let $\chi$ denote $\chi_1\chi_2\cdots\chi_r\in k[Z^\circ]^*$. 
                Then from Proposition \ref{prop: mock val function} (b), $\val^\sigma_f + \val^\sigma_\chi \equiv 0$. 
                In particular,  the ideal of $\Gamma(Z(\sigma), \OO_Z)$ associated with the closed subscheme $H_{Z, f}\cap Z(\sigma)$ of $Z(\sigma)$ is generated by $\chi f$. 

                Let $p^\sigma$ denote a quotient ring morphism $\Gamma(Z(\sigma), \OO_Z)\rightarrow 
                \Gamma(Z_\sigma, \OO_{Z_\sigma})$ associated with the closed immersion $Z_\sigma\rightarrow Z(\sigma)$. 
                Let $f^\sigma, f^\sigma_{1}, \cdots, f^\sigma_{r}$ denote $p^\sigma(\chi f)$, $p^\sigma(\chi_1 f_1), \cdots ,p^\sigma(\chi_r f_r)$ respectively. 
                From the construction of $\chi$, we have $f^\sigma = f^\sigma_{1} f^\sigma_{2}\cdots f^\sigma_{r}$. 
                From Proposition \ref{prop: heredity-finess-for-orbit}(a), for any $1\leq i\leq r$, the closed subscheme $H_{Z, f_i}\cap Z_\sigma$ of $Z_\sigma$ is defined by $f^\sigma_i$. 
                Similarly, the closed subscheme $H_{Z, f}\cap Z_\sigma$ of $Z_\sigma$ is defined by $f^\sigma$. 
                From the assumption, $H_{Z, f}\cap Z_\sigma$ is smooth over $k$. 
                We recall that $k[Z_\sigma]$ is a UFD from Proposition\ref{prop:stratification-structure}(d). 
                Thus, $f^\sigma_1, f^\sigma_2, \cdots, f^\sigma_r$ are co-prime, and the following equation holds as a closed subscheme of $Z_\sigma$. 
                \[
                    H_{Z, f}\cap Z_\sigma = \coprod_{1\leq i\leq r} H_{Z, f_i}\cap Z_\sigma
                \]
                In particular, $H_{Z, f_i}\cap Z_\sigma$ is smooth over $k$ for any $1\leq i\leq r$. 
                Therefore, $f_i$ is non-degenerate for $\Delta$ for any $1\leq i\leq r$. 
                \item[(c)] We use the notation of the proof of (b). 
                From the proof of (b), $H_{Z, f_i}\cap H_{Z, f_j} = \emptyset$ for any $1\leq i< j\leq r$. 
                In particular, $\chi_i f$ and $\chi_j f$ generates a unit ideal of $\Gamma(Z(\sigma), \OO_Z)$. 
                Thus, the following equation holds as a closed subscheme of $Z(\sigma)$ for any $\sigma\in\Delta$: 
                \[
                    H_{Z, f}\cap Z(\sigma) = \coprod_{1\leq i\leq r} H_{Z, f_i}\cap Z(\sigma)
                \]
                Hence, the first part of the statement holds. 
                
                For any $1\leq i\leq r$, $H^\circ_{Z, f_i}$ is irreducible. 
                Thus, from the definition of $H_{Z, f_i}$, $H_{Z, f_i}$ is irreducible too. 
                Therefore, the right-hand side of the equation in the statement is a connected decomposition of $H_{Z, f}$.
                \end{enumerate}
        \end{proof}
        The following proposition gives sufficient conditions when $H_{Z, f}$ is smooth over $k$. 
        \begin{proposition}\label{prop:unimodular+nondegenerate=smooth}
            Let $f\in k[Z^\circ]$ be a function.  
            We assume that $f$ is non-degenerate for $\Delta$ and $\Delta$ is unimodular. 
            Then $H_{Z, f}$ is smooth over $k$. 
        \end{proposition}
        \begin{proof}
            Let $\sigma\in\Delta$ be a cone and $\varphi\in\Phi$ be an element such that $\sigma\in\Delta_\varphi$. 
            Let $\sigma_\varphi$ denote $(q_\varphi)_\RR(\sigma)$. 
            Then from Proposition \ref{prop:fundamental hypersurface}(a) and (b), there exists $g_\varphi\in k[M_\varphi]$ such that $H_{X(\sigma_\varphi),g_\varphi} \times_{X(\sigma_\varphi)} Z(\sigma) = H_{Z, f}\cap Z(\sigma)$. 
            We remark that $\iota_\varphi(Z(\sigma))$ is an open subscheme of $X(\sigma_\varphi)$ and $H_{X(\sigma_\varphi),g_\varphi}$ does not contain $O_{\sigma_\varphi}$ from Lemma \ref{lem:fine for varphi}. 
            Moreover, from the proof of Proposition \ref{prop:basic-property1}(c), $(H_{X(\sigma_\varphi),g_\varphi}\cap O_{\sigma_\varphi})\times_{X(\sigma_\varphi)} Z(\sigma) = H_{Z, f}\cap Z_\sigma$. 
            Because $f$ is nondegenerate for $\Delta$, $H_{Z, f}\cap Z_\sigma\cong H_{X(\sigma_\varphi),g_\varphi}\cap O_{\sigma_\varphi}\cap \iota_\varphi(Z(\sigma))$  is smooth over $k$. 
            Thus, from Lemma \ref{lem: for 7-1}, there exists 
            \begin{itemize}
                \item A toric monoid $S$
                \item A morphism $p\colon X(\sigma_\varphi)\rightarrow \Spec(k[S])$ over $\Spec(k)$
                \item An open subset $U$ of $\iota_\varphi(Z(\sigma))$
            \end{itemize}
            such that
            \begin{enumerate}
                \item[(1)] $H_{X(\sigma_\varphi), g_\varphi}\cap O_{\sigma_\varphi}\cap\iota_\varphi(Z(\sigma))\subset U$
                \item[(2)] The restriction $p|_{H_{X(\sigma_\varphi), f}\cap U}\colon H_{X(\sigma_\varphi), f}\cap U\rightarrow \Spec(k[S])$ is smooth.
            \end{enumerate}
            Let $U_\sigma$ denote an open subset ${(\iota_\varphi)}^{-1}(U)$ of $Z(\sigma)$. 
            Then $H_{Z, f}\cap Z_\sigma\subset U_\sigma$ and the restriction $(p\circ \iota_\varphi)|_{H_{Z, f}\cap U_\sigma}\colon {H_{Z, f}\cap U_\sigma}\rightarrow \Spec(k[S])$ is smooth. 
            From the assumption that $\sigma$ is unimodular, $\sigma_\varphi$ is unimodular from Proposition \ref{prop:fan-varphi-smooth}(a). 
            Thus, from the proof of Lemma \ref{lem: for 7-1}, we can take $S$ such that $\Spec(k[S])$ is smooth over $k$. 
            Therefore, $H_{Z, f}\cap U_\sigma$ is smooth over $k$. 
            We recall that $H_{Z, f}\cap Z_\sigma\subset U_\sigma$. 
            Thus, by considering for all $\sigma\in\Delta$, we can check that $H_{Z, f}$ is smooth over $k$. 
        \end{proof}
        We use the notation in the subsection 4.2.
        We have already shown that a toric resolution of $X(\Delta')\rightarrow X(\Delta)$ induces a new mock toric variety $Z'$. 
        The following proposition shows the structure of $H_{Z', f}$ explicitly. 
        \begin{proposition}\label{prop: heredity-finess-for-dominant}
             Let $N'$ be a lattice of finite rank, $\pi\colon N'\rightarrow N$ be a surjection, $s$ be a section of $\pi$, $\Delta$ and $\Delta'$ be strongly convex rational polyhedral fans in $N_\RR$ and $N'_\RR$ respectively, and $f\in k[Z^\circ]$ be a function. 
            We assume that $\pi$ is compatible with the fans $\Delta'$ and $\Delta$.
            Let $Z'$ denote a mock toric variety induced by $Z$ along $\pi_*\colon X(\Delta')\rightarrow X(\Delta)$ and $s$, $\mu$ denote the restriction $\pi_*|_{Z'}\colon Z'\rightarrow Z$, and  $\mu^*\colon k[Z^\circ]\rightarrow k[{Z'}^\circ]$ be a ring morphism induced by $\mu$. 
            Then the following statements hold:
            \begin{enumerate}
                \item[(a)] The following diagram is commutative: 
                    \begin{equation*}
                        \begin{tikzcd} 
                            \supp(\Delta')\cap N'\ar[r, "\pi"]\ar[d, "\val_{Z'}"]& \supp(\Delta)\cap N\ar[d, "\val_{Z}"]\\
                            \Val_k(Z')\ar[r, "\mu_\natural"]& \Val_{k}(Z)
                        \end{tikzcd}
                    \end{equation*}
                \item[(b)] $\val_f\circ\pi_\RR|_{\supp(\Delta')} = \val_{\mu^*(f)}$
                \item[(c)] If $f$ is fine for $\Delta$, then $\mu^*(f)$ is fine for $\Delta'$. 
                \item[(d)] If $f$ is fine for $\Delta$, then $H_{Z, f}\times_{Z}Z' = H_{Z', \mu^*(f)}$.
                \item[(e)] We assume that $f$ is fine for $\Delta$. 
                Then $f$ is non-degenerate for $\Delta$ if and only if $\mu^*(f)$ is non-degenerate for $\Delta'$.
            \end{enumerate}
        \end{proposition}
        \begin{proof}
            We prove the statements from (a) to (e) in order. 
            \begin{enumerate}
                \item[(a)] Let $\tau\in\Delta'$ and $\sigma\in\Delta$ be cones, and $\varphi\in\Phi$ be an element such that $\pi_\RR(\tau)\subset \sigma$ and $\sigma\in\Delta_\varphi$. 
                We recall that $\tau\in\Delta'_\varphi$. 
                Let $\sigma_\varphi$ denote $(q_\varphi)_\RR(\sigma)$ and $\tau_\varphi$ denote $(q'_\varphi)_\RR(\tau)$. 
                We remark that there exists the following commutative diagram:
                \begin{equation*}
                    \begin{tikzcd} 
                        Z'(\tau)\ar[r, "\iota'"]\ar[d, "\mu"]& X(\tau)\ar[r, "(q'_\varphi)_*"]\ar[d, "\pi_*"]& X(\tau_\varphi)\ar[d, "(\pi_\varphi)_*"]\\
                        Z(\sigma)\ar[r, "\iota"]&X(\sigma)\ar[r, "(q_\varphi)_*"]&X(\sigma_\varphi)
                    \end{tikzcd}
                \end{equation*}
                where $\pi_\varphi$ denote a group morphism $N'/s(N_\varphi)\rightarrow N/N_\varphi$ such that $\pi_\varphi\circ q'_\varphi= q_\varphi\circ \pi$. 
                Let $w\in\tau\cap N'$ be an element. 
                Then $\mu_\natural(\val_{Z'}(w)) = \val_Z(\pi(w))$ from Lemma \ref{lem:val-val}. 
                \item[(b)] Because $\pi_\RR$ is a linear map, both functions are continuous maps and are preserved by the scalar product of non-negative real numbers. 
                Moreover, from (a), the following equation holds for any $w\in\supp(\Delta')\cap N'$:
                \begin{align*}
                    \val_f(\pi_\RR(w)) &= \val_Z(\pi(w))(f)\\
                    &= (\mu_\natural(\val_{Z'}(w)))(f)\\
                    &= \val_{Z'}(w)(\mu^*(f))\\
                    &= \val_{\mu^*(f)}(w)
                \end{align*}
                Thus, from Lemma \ref{lem: uniqueness of the function}, the statement holds.
                \item[(c)] 
                Let $\tau\in\Delta'$ and $\sigma\in\Delta$ be cones such that $\pi_\RR(\tau)\subset \sigma$. 
                Then there exists $\chi\in k[Z^\circ]^*$ such that $\val^\sigma_f(v) = \val^\sigma_\chi(v)$ for any $v\in\sigma$ from Proposition \ref{prop: condition of fine}. 
                From (b), $\val_f\circ\pi_\RR = \val_{\mu^*(f)}$ and $\val_\chi\circ\pi_\RR = \val_{\mu^*(\chi)}$. 
                Thus, the following equation holds for any $w\in\tau$:
                \begin{align*}
                    \val_{\mu^*(f)}(w) &= \val_f(\pi_\RR(w))\\
                    &=\val_\chi(\pi_\RR(w))\\
                    &=\val_{\mu^*(\chi)}(w)
                \end{align*}
                Therefore, from Proposition \ref{prop: condition of fine}, $H_{Z', \mu^*(f)}$ does not contain $Z'_\tau$, so $\mu^*(f)$ is fine for $\Delta'$. 
                \item[(d)] From Proposition \ref{prop:fundamental hypersurface}(a) and Lemma \ref{lem:fine for varphi}, there exists $g_\varphi\in k[M_\varphi]$ such that $f/\iota^*_\varphi(g_\varphi)\in k[Z^\circ]^*$ and $g_\varphi$ is fine for $\Delta(\varphi)$. 
                Then from Proposition \ref{prop:fundamental hypersurface}(b), the following equations hold:
                \[
                    H_{Z, f}\cap X(\Delta_\varphi) = H_{X(\Delta(\varphi)), g_\varphi}\times_{X(\Delta(\varphi))}(Z\cap X(\Delta_\varphi))
                \]
                Let $\pi_\varphi^*\colon \Gamma(X(\sigma_\varphi), \OO_{X(\Delta(\varphi))})\rightarrow \Gamma(X(\tau_\varphi), \OO_{X(\Delta'(\varphi))})$ denote a ring morphism induced by a toric morphism $(\pi_\varphi)_*\colon X(\tau_\varphi)\rightarrow X(\sigma_\varphi)$. 
                Let $f'$ denote $\mu^*(f)$ and $g'$ denote $\pi_\varphi^*(g_\varphi)$. 
                Then $f'/(\iota'_\varphi)^*(g')\in k[{Z'}^\circ]^*$ and the following equation holds:
                \[
                    H_{Z', f'}\cap X(\Delta'_\varphi) = H_{X(\Delta'(\varphi)), g'}\times_{X(\Delta'(\varphi))}(Z'\cap X(\Delta'_\varphi))
                \]
                On the other hand, from Lemma \ref{lem: heredity of fineness for dominant}(b), the following equations hold:
                \[
                    H_{X(\Delta'(\varphi)), g'} = H_{X(\Delta(\varphi)), g_\varphi}\times_{X(\Delta(\varphi))}X(\Delta'(\varphi))
                \]
                Thus, the following equation holds:
                \begin{align*}
                    H_{Z', f'}\cap X(\Delta'_\varphi)&= H_{X(\Delta'(\varphi)), g'}\times_{X(\Delta'(\varphi))}
                    (Z'\cap X(\Delta'_\varphi))\\
                    &= (H_{X(\Delta(\varphi)), g_\varphi}\times_{X(\Delta(\varphi))}X(\Delta'(\varphi)))\times_{X(\Delta'(\varphi))}(Z'\cap X(\Delta'_\varphi))\\
                    &= (H_{X(\Delta(\varphi)), g_\varphi}\times_{X(\Delta(\varphi))}(Z\cap X(\Delta_\varphi)))\times_{Z\cap X(\Delta_\varphi)}(Z'\cap X(\Delta'_\varphi))\\
                    &= (H_{Z, f}\cap X(\Delta_\varphi))\times_{Z\cap X(\Delta_\varphi)}(Z'\cap X(\Delta'_\varphi))\\
                    &= (H_{Z, f}\times_Z Z')\times_{Z'}(Z'\cap X(\Delta'_\varphi))\\
                    &= (H_{Z, f}\times_{Z}Z')\cap X(\Delta'_\varphi)\\
                \end{align*}
                Therefore, $H_{Z', f'} = H_{Z, f}\times_{Z}Z'$ because $\{Z'(\tau)\}_{\tau\in\Delta'}$ is an open covering of $Z'$. 
                \item[(e)] We use the notation in the proof of (a).  
                There exists $\sigma\in\Delta$ such that $\pi_\RR(\tau)\cap \sigma^\circ\neq\emptyset$. 
                Thus, we remark that $(\pi_\varphi)_\RR(\tau_\varphi)\cap\sigma^\circ_\varphi\neq\emptyset$ too. 
                Hence, there exists the following commutative diagram whose all small squares are Cartesian squares:
                \begin{equation*}
                    \begin{tikzcd} 
                        Z'_\tau\ar[r, hook, "\iota'"]\ar[d, "\mu"]&O_\tau\ar[r, "(q'_\varphi)_*"]\ar[d, "\pi_*"]& O_{\tau_\varphi}\ar[d, "(\pi_\varphi)_*"]\\
                        Z_\sigma\ar[r, hook, "\iota"]&O_\sigma\ar[r, "(q_\varphi)_*"] & O_{\sigma_\varphi}
                    \end{tikzcd}
                \end{equation*}
                In particular, $(\pi_\varphi)_*\colon O_{\tau_\varphi}\rightarrow O_{\sigma_\varphi}$ is a trivial algebraic torus fibration, so $\mu\colon Z'_\tau\rightarrow Z_\sigma$ is also a  trivial algebraic torus fibration. 
                Therefore, from (d), $H_{Z', \mu^*(f)}\cap Z'_\tau$ is isomorphic to a trivial algebraic torus fibration of $H_{Z, f}\cap Z_\sigma$, so the smoothness of $H_{Z', \mu^*(f)}\cap Z_\tau$ and the smoothness of $H_{Z, f}\cap Z_\sigma$ are equivalent. 
            \end{enumerate}
        \end{proof}
        In \ref{prop:ext'n}, we considered the relation between mock toric varieties and the field extension. 
        From now on, we consider the relation with $H_{Z, f}$ and the field extension. 
        \begin{proposition}\label{prop: heredity-finess-for-basechange}
            We use the notation in Proposition \ref{prop:ext'n}. 
            Let $f\in K[Z^\circ]$ be a function. 
            Then the following statements follow:
            \begin{enumerate}
                \item[(a)] $\val_f = \val_{\alpha^*(f)}$
                \item[(b)] If $f$ is fine for $\Delta$, then $\alpha^*(f)$ is fine for $\Delta$. 
                \item[(c)] If $f$ is fine for $\Delta$, then $H_{Z, f}\times_{Z}{Z_L} = H_{Z_L, \alpha^*(f)}$.
                \item[(d)] If $f$ is non-degenerate for $\Delta$, then $\alpha^*(f)$ is non-degenerate for $\Delta$. 
            \end{enumerate}
        \end{proposition}
        \begin{proof}
            We prove the statements from (a) to (d) in order. 
            \begin{enumerate}
                \item[(a)] Both functions are continuous and are preserved by the scalar product of non-negative real numbers. 
                Let $v\in\supp(\Delta)\cap N$ be an element. 
                Then from Proposition \ref{prop:ext'n}(c), the following equation holds:
                \begin{align*}
                    \val_f(v) &= \val_{Z}(v)(f)\\
                    &= (\alpha_\natural(\val_{Z_L}(v)))(f)\\
                    &= \val_{Z_L}(v)(\alpha^*(f))\\
                    &= \val_{\alpha^*(f)}(v)
                \end{align*}
                Thus, from Lemma \ref{lem: uniqueness of the function}, the statement holds.
                \item[(b)] Let $\sigma\in\Delta$ be a cone. 
                Then from Proposition \ref{prop: condition of fine}, there exists $\chi\in K[Z^\circ]^*$ such that $\val^\sigma_f = \val^\sigma_\chi$. 
                From (a), $\val^\sigma_{\alpha^*(f)} = \val^\sigma_{\alpha^*(\chi)}$. 
                Thus, from Proposition \ref{prop: condition of fine}, $\alpha^*(f)$ is fine for $\Delta$.   
                \item[(c)] Let $\sigma\in\Delta$ be a cone. 
                Then from Proposition \ref{prop: condition of fine}, there exists $\chi\in K[Z^\circ]^*$ such that $\val^\sigma_f = \val^\sigma_\chi$. 
                We remark that  $\val^\sigma_{\alpha^*(f)} = \val^\sigma_{\alpha^*(\chi)}$ too from (a). 
                Thus, the ideal of $\Gamma(Z_L(\sigma), \OO_{Z_L})$ associated with the closed subscheme $H_{Z_L, \alpha^*(f)}\cap Z_L(\sigma)$ of $Z_L(\sigma)$ is generated by $\alpha^*(\chi^{-1}f)$ and the ideal of $\Gamma(Z(\sigma), \OO_{Z})$ associated with the closed subscheme $H_{Z, f}\cap Z(\sigma)$ of $Z(\sigma)$ is generated by $\chi^{-1}f$. 
                This shows that $H_{Z_L, \alpha^*(f)}\cap Z_L(\sigma) = (H_{Z, f}\times_{Z}Z_L)\cap Z_L(\sigma)$, so that the statement holds. 
                \item[(d)] Let $\sigma\in\Delta$ be a cone. 
                From (c), $H_{Z_L, \alpha^*(f)}\cap (Z_L)_\sigma = (H_{Z, f}\cap Z_\sigma)\times_{\Spec({K})}
                \Spec(L)$. 
                From the assumption, $H_{Z, f}\cap Z_\sigma$ is smooth over $K$, so $H_{Z_L, \alpha^*(f)}\cap (Z_L)_\sigma$ is smooth over $L$. 
                Thus, $\alpha^*(f)$ is non-degenerate for $\Delta$. 
            \end{enumerate}
        \end{proof}
        \subsection{Extended valuation function}
        In this subsection, we use the notation in Definition \ref{def: extend mock toric}. 
        
        In Proposition \ref{prop:extended-valuation}, we define $\val_{Z, \pi}\colon\pi^{-1}_\RR(\supp(\Delta))\cap N'\rightarrow \Val_k{Z_\pi}$. 
        Let $f\in k[Z_\pi]$ be a function. 
        First, we define a map $\val_{f, \pi}\colon \pi^{-1}_\RR(\supp(\Delta))\rightarrow \RR$ as in Proposition \ref{prop: mock val function}. 
        \begin{proposition}\label{prop: extended val function}
            Let $f\in k[Z_\pi]$ be a function. 
            Then there unique exists a function $\val_{f, \pi}\colon$ $\pi^{-1}_\RR(\supp(\Delta))\rightarrow \RR$ such that it satisfies the following three conditions:
                \begin{enumerate}
                    \item[(1)] A function $\val_{f, \pi}$ is continuous.
                    \item[(2)] A function $\val_{f, \pi}$ is preserved by the scalar product of non-negative real numbers, i.e., for any $v\in\pi^{-1}_\RR(\supp(\Delta))$ and $r\in\RR_{\geq 0}$, we have $\val_{f, \pi}(rv) = r\val_{f, \pi}(v)$.
                    \item[(3)] For any $v\in \pi^{-1}_\RR(\supp(\Delta))\cap N'$, we have $\val_{f, \pi}(v) = \val_{Z, \pi}(v)(f)$.
                \end{enumerate}
        \end{proposition}
        \begin{proof}
            From Lemma \ref{lem: uniqueness of the function}, we can check the uniqueness of $\val_{f, \pi}$. 
            We show the existence of $\val_{f, \pi}$. 
            Let $\Delta'$ be a strongly convex rational polyhedral fan in $N'_\RR$ such that $\supp(\Delta') = \pi^{-1}_\RR(\supp(\Delta))$ and $\pi$ is compatible with the fans $\Delta'$ and $\Delta$.  
            Let $W$ be a mock toric variety induced by $Z$ along $\pi_*\colon X(\Delta')\rightarrow X(\Delta)$ and $s$. 
            Then from Proposition \ref{prop:extended-valuation}(b), $\val_{Z, \pi}(v') = \val_W(v')$ for any $v'\in \pi^{-1}_\RR(\supp(\Delta))\cap N'$. 
            We remark that $k[Z_\pi] = k[W^\circ]$ and we already have defined the map $\val_f\colon \supp(\Delta')\rightarrow \RR$ by a mock toric variety $W$ in Proposition \ref{prop: mock val function}(a). 
            Thus, $\val_f(v') = \val_{Z, \pi}(v')(f)$ for any $v'\in \pi^{-1}_\RR(\supp(\Delta))\cap N'$. 
            Therefore, $\val_f$ satisfies all conditions in the statement. 
        \end{proof}
        For $A\subset\pi^{-1}_\RR(\supp(\Delta))$, let $\val^A_{f, \pi}$ denote the restriction $\val_{f, \pi}|_{A}$. 
        The following proposition is analogous to Proposition\ref{prop: units correspondence}. 
        \begin{proposition}\label{prop:extend-unit}
            Let $\chi'\in k[Z_\pi]^*$ be a unit, $\sigma\in\Delta$ be a cone, and $\varphi\in\Phi$ be an element such that $\sigma\in\Delta_\varphi$. 
            Let $\sigma^\pi$ denote $\pi_\RR^{-1}(\sigma)$, and $M'_\varphi$ denote the dual lattice of $N'/s(N_\varphi)$.  
            Then there exists $\omega\in M'_\varphi$ such that $\val^{\sigma^\pi}_{\chi', \pi} =\val^{\sigma^\pi}_{{\iota'}^*_\varphi(\chi^\omega), \pi}$. 
            In particular, $\val^{\sigma^\pi}_{\chi, \pi}$ is a linear map on $\sigma^\pi$. 
        \end{proposition}
        \begin{proof} 
            From Proposition \ref{prop:extended-valuation}(c), there exists $\chi\in k[Z^\circ]^*$ and $\omega_1\in M_1$ such that $\chi' = \epsilon^*(\chi)p^*_0(\chi^{\omega_1})$. 
            Let $M_\varphi$ denote the dual lattice of $N/N_\varphi$. 
            From Proposition \ref{prop: units correspondence}, there exists $\omega_2\in M_\varphi$ such that $\val^\sigma_\chi = \val^\sigma_{\iota^*_\varphi(\chi^{\omega_2})}$. 
            Hence, from Proposition \ref{prop:extended-valuation}(d), $\val_{Z, \pi}(v')(\epsilon^*(\chi))
            = \val_{Z, \pi}(v')(\epsilon^*(\iota^*_\varphi(\chi^{\omega_2})))$ for any $v'\in \sigma^\pi\cap N'$. 
            We recall that $\iota_\varphi\circ\epsilon = \pi_\varphi\circ\iota'_\varphi$. 
            Thus, $\val_{Z, \pi}(v')(\chi')= \val_{Z, \pi}(v')({\iota'_\varphi}^*(\pi^*_\varphi(\chi^{\omega_2}))) + \val_{Z, \pi}(v')({\iota'_\varphi}^*(p^*_\varphi(\chi^{\omega_1}))$ for any $v'\in \sigma^\pi\cap N'$. 
            Let $\omega$ denote $p^*_\varphi(\omega_1)+\pi^*_\varphi(\omega_2)\in M'_\varphi$. 
            Therefore, from Lemma\ref{lem: uniqueness of the function}, $\val^{\sigma^\pi}_{\chi, \pi} =  \val^{\sigma^\pi}_{{\iota'_\varphi}^*(\chi^\omega), \pi}$. 
        \end{proof}
        \subsection{Hypersurfaces of mock toric varieties over $\A^1_k$}
        In subsection 4.5, we considered the mock toric varieties over $\A^1_k$. 
        In this subsection, we consider hypersurfaces of them. 
        This analysis is crucial for the construction of strictly toroidal models of the hypersurfaces in mock toric varieties that we conduct in the next Section. 
        \begin{proposition}\label{prop:heredity-finess-for-relative}
            We use the notation in Proposition \ref{prop:relative-mock}.  
            Let $f\in k[W^\circ]$ be a function. 
            We assume that $\charac(k) = 0$. 
            Then the following statements hold:
            \begin{enumerate}
                \item[(a)] $\val_f\circ j'_\RR = \val_{\psi^*(f)}$
                \item[(b)] If $f$ is fine for $\Delta'$, then $\psi^*(f)$ is fine for $\Delta'_{(0)}$. 
                \item[(c)] If $f$ is fine for $\Delta'$, then $H_{W, f}\times_{W}{Y_{k(t)}} = H_{Y_{k(t)}, \psi^*(f)}$. 
                \item[(d)] If $f$ is non-degenerate for $\Delta'$, then $\psi^*(f)$ is non-degenerate for $\Delta'_{(0)}$. 
            \end{enumerate}
        \end{proposition}
        \begin{proof}
            We prove the statements from (a) to (d) in order. 
            \begin{enumerate}
                \item[(a)] Both functions are continuous and are preserved by the scalar product of 
                non-negative real numbers because $j'_\RR$ is a linear function. 
                Let $v'\in\supp(\Delta'_{(0)})\cap N'$ be an element. 
                Then from Proposition \ref{prop:relative-mock}(k), the following equation holds:
                \begin{align*}
                    \val_f\circ j'_\RR(v') &= \val_{W}(j'(v'))(f)\\
                    &= (\psi_\sharp(\val_{Y_{k(t)}}(v')))(f)\\
                    &= \val_{Y_{k(t)}}(v')(\psi^*(f))\\
                    &= \val_{\psi^*(f)}(v')
                \end{align*}
                Therefore, from Lemma \ref{lem: uniqueness of the function}, the statement holds.
                \item[(b)] Let $\tau\in\Delta'_{(0)}$ be a cone. 
                Then there exists $\sigma\in\Delta'_{0}$ such that $(\pr_1)_\RR(\sigma) = \tau$. 
                Moreover, from Proposition \ref{prop: condition of fine}, there exists $\chi\in k[W^\circ]^*$ such that $\val^\sigma_f = \val^\sigma_\chi$. 
                From (a), $\val^\tau_{\psi^*(f)} = \val^\tau_{\psi^*(\chi)}$ because $j'_\RR(\tau) = \sigma$.  
                Therefore, from Proposition \ref{prop: condition of fine}, $\psi^*(f)$ is fine for $\Delta'_{(0)}$.   
                \item[(c)] We use the notation of the proof of (b). 
                From the definition of $\psi$, the following diagram is a Cartesian diagram:
                \begin{equation*}
                    \begin{tikzcd} 
                        Y_{k(t)}(\tau)\ar[r, hook]\ar[d, "\psi"]&X_{k(t)}(\tau)\ar[d, "\Psi"]\\
                        W(\sigma)\ar[r, hook]& X(\sigma)
                    \end{tikzcd}
                \end{equation*}
                Thus, the ideal of $\Gamma(W(\sigma), \OO_{W})$ associated with the closed subscheme $H_{W, f}\cap W(\sigma)$ of $W(\sigma)$ is generated by $\chi^{-1}f$, and the ideal of $\Gamma(Y_{k(t)}(\tau), \OO_{Y_{k(t)}})$ associated with the closed subscheme $H_{Y_{k(t)}, \psi^*(f)}\cap Y_{k(t)}(\tau)$ of $Y_{k(t)}(\tau)$ is generated by $\psi^*(\chi^{-1}f)$. 
                This shows that the following equation holds from Proposition \ref{prop:relative-mock}(d):
                \begin{align*}
                    H_{Y_{k(t)}, \psi^*(f)}\cap Y_{k(t)}(\tau) &= (H_{W, f}\cap W(\sigma))\times_{W(\sigma)}Y_{k(t)}(\tau) \\
                    &= (H_{W, f}\times_W Y_{k(t)})\cap Y_{k(t)}(\tau)
                \end{align*}
                Therefore, the statement holds. 
                \item[(d)] We use the notation in the proof of (c).  
                Then all small squares in the following diagram are Cartesian squares from Proposition \ref{prop:relative-mock}(c) and (d):
                \begin{equation*}
                    \begin{tikzcd} 
                        H_{Y_{k(t)}, \psi^*(f)}\cap (Y_{k(t)})_{\tau}\ar[r, hook]\ar[d]&Y_{k(t)}(\tau)\ar[r]\ar[d, "\psi"]&\Spec(k(t))\ar[d]\\
                        H_{W, f}\cap W_{\sigma}\ar[r, hook]&W(\sigma)\ar[r, "(\pr_2)*"] & \A^1_k
                    \end{tikzcd}
                \end{equation*}
                Thus, from the assumption, $H_{W, f}\cap W_{\sigma}$ is smooth over $k$. 
                Hence, from the generic smoothness, the lower morphism $H_{W, f}\cap W_{\sigma}\rightarrow  \A^1_k$ in the above diagram is smooth over the generic point of $\A^1_k$. 
                Therefore, $H_{Y_{k(t)}, \psi^*(f)} \cap (Y_{k(t)})_{\tau}$ is smooth over $\Spec(k(t))$. 
            \end{enumerate}
        \end{proof}
\section{Strictly toroidal models of hypersurfaces of mock toric varieties}
        At first, we consider the following definition.  
        \begin{definition}[\cite{NO21}]
                Let $\mathscr{X}$ be a flat and separated $\Rt$-scheme of finite type, and let $x\in\mathscr{X}_k$ be a point. 
                We say that $\mathscr{X}$ is strictly toroidal at $x$ if there exist a toric monoid $S$, an open neighbourhood $U$ of $x$ in $\mathscr{X}$, and a smooth morphism of $\Rt$-schemes $U \rightarrow \Spec(\Rt[S]/(\chi^\omega - t^q))$, where $q$ is a positive rational number, $\omega$ is an element of $S$, and $\chi^\omega$ is a torus invariant monomial in $k[S]$ associated with $\omega$ such that $k[S]/(\chi^\omega)$ is reduced. 
                We call that $\mathscr{X}$ is \textbf{strictly toroidal} if it is so at any $x\in \mathscr{X}_k$. 
        \end{definition}
        For example, if the closed fiber $\mathcal{X}_k$ is a simple normal crossing divisor of $\mathcal{X}$, $\mathcal{X}$ is strictly toroidal. 
        A strictly toroidal model of a smooth proper $\Kt$-variety $X$ plays an important role in computing a stable birational volume of $X$(cf. \cite[Lemma 3.3.5]{NO21}). 
        We can construct a strictly toroidal model of a Newton non-degenerate hypersurface of a toric variety combinatorially and compute a stable birational volume of it(\cite[Theorem 3.14]{NO22}). 
        
        In this section, our goal is to consider sufficient conditions for constructing a strictly toroidal model of a hypersurface of a mock toric variety and construct it. 
        We remark that in \cite{NO22}, they consider the hypersurfaces of toric varieties over $\Rt$. 
        However, in this article, we first construct a hypersurface of a mock toric variety over $\A^1_k$ and subsequently carry out a base change from $\A^1_k$ to $\Spec(\Rt)$. 
        
        Let $k$ be an algebraically closed field with $\charac(k) = 0$. 
        In this section, we use the notation in Proposition \ref{prop:relative-mock} and the following notation. 
        \begin{itemize} 
            \item From Proposition \ref{prop:relative-mock}, $\pi$ is compatible with the fans $\Delta'_{(0)}$ and $\Delta$. 
            Let $Y$ denote a mock toric variety induced by $Z$ along $\pi_*\colon X(\Delta'_{(0)})\rightarrow X(\Delta)$ and $s$.
            \item We remark that if $\Delta' = \{\{0_{(N'\oplus\ZZ)_\RR}\}\}$, we have $W = Z^1_{\pi^1}$ and $Y = Z_\pi$. 
            \item Let $M$ and $M'$ be dual lattices of $N$ and $N'$. 
            \item For $\varphi\in\Phi$, let $M_\varphi$ and $M'_\varphi$ denote the dual lattice of $N/N_\varphi$ and $N'/s(N_\varphi)$ respectively. 
            \item Let $\langle \cdot, \cdot\rangle$ denote a natural pairing of $(N, M)$, $(N', M')$, $(N\oplus\ZZ, M\oplus\ZZ)$, or $(N'\oplus\ZZ, M'\oplus\ZZ)$. 
            In particular, the pairings of $(N\oplus\ZZ, M\oplus\ZZ)$, or $(N'\oplus\ZZ, M'\oplus\ZZ)$ are defined as follows: 
            \[
                (N\oplus\ZZ) \times (M\oplus\ZZ)\ni ((v, a), (\omega, b)) \mapsto \langle v, \omega\rangle + ab\in\ZZ
            \]
            \[
                (N'\oplus\ZZ) \times (M'\oplus\ZZ)\ni ((v', a), (\omega', b)) \mapsto \langle v', \omega'\rangle + ab\in\ZZ
            \]
            We identify elements in $N$ as torus invariant valuations on $T_N$ so that we can check that $\langle v, \omega\rangle = v(\chi^\omega)$. 
            \item We use the notation in Proposition \ref{prop:relative-mock}(i).  
            Then there exists isomorphisms $\mu\times p_0\colon Z^1_{\pi^1}\rightarrow Z_\pi\times{\Gm^1}_{,k}$ and $\mu^*\otimes p^*_0\colon k[Z_\pi]\otimes_k k[t, t^{-1}]\rightarrow k[Z^1_{\pi^1}]$. 
            We regard $k[Z_\pi]$ and $k[t, t^{-1}]$ as sub rings of $k[Z^1_{\pi^1}]$ by $\mu^*$ and $p^*_0$, so we sometimes omit $\mu^*$ and $p^*_0$. 
            \item Let $\iota''$ denote the closed immersion $W\rightarrow X(\Delta')$. 
            \item For $\varphi\in\Phi$, let $\iota''_\varphi$ denote $(q'_\varphi\times\id_\ZZ)_*\circ\iota''|_{W\cap X(\Delta'_\varphi)}$. 
            \item Let $\Rt$ be a valuation ring defined as follows:
            \[
                \Rt = \bigcup_{n\in\ZZ_{>0}} k[[t^{\frac{1}{n}}]]
            \]
            \item Let $\Kt$ be a fraction field of R. 
            We remark that $\Kt$ is written as follows:
            \[
                \Kt = \bigcup_{n\in\ZZ_{>0}} k((t^{\frac{1}{n}}))
            \]
            \item Let $\Spec(\Rt)\rightarrow \Spec(k[t])$ be a morphism of affine scheme induced from a $k$-morphism $k[t]\hookrightarrow \Rt$ whose image of $t$ is $t$.  
        \end{itemize}
        The following proposition concerns units in $k[Z^1_{\pi^1}]$. 
        \begin{proposition}\label{prop:fundamental-relative1}
            Let $w \in (\pi^1_\RR)^{-1}(\supp(\Delta\times\Delta_!))$ be an element and $\chi'\in k[Z^1_{\pi^1}]$ be an unit. 
            From Proposition \ref{prop:relative-mock}(i), there exists $\chi\in k[Z_\pi]^*$ and $n\in \ZZ$ such that $\chi' = \chi t^n$. 
            Then $\val_{\chi', \pi^1}(w) = \val_{\chi, \pi}((\pr_1)_\RR(w)) + n(\pr_2)_\RR(w)$. 
        \end{proposition}
        \begin{proof}
            Both functions $\val_{\chi', \pi^1}$ and $\val_{\chi, \pi}\circ (\pr_1)_\RR + n(\pr_2)_\RR$ are continuous and are preserved by the scalar product of non-negative real numbers by Proposition \ref{prop: extended val function}. 
            Thus, it is enough to show that $\val_{\chi', \pi^1}(w) = \val_{\chi, \pi}((\pr_1)_\RR(w)) + n(\pr_2)_\RR(w)$ for any $w\in (\pi^1_\RR)^{-1}(\supp(\Delta\times\Delta_!))\cap (N'\oplus\ZZ)$. 
            Let $w\in (\pi^1_\RR)^{-1}(\supp(\Delta\times\Delta_!))\cap (N'\oplus\ZZ)$ be an element. 
            Then from Proposition \ref{prop:relative-mock}(j), the following equation follows:
            \begin{align*}
                \val_{\chi', \pi^1}(w) &= \val_{Z^1, \pi^1}(w)(\chi')\\
                &= \val_{Z, \pi}(\pr_1(w))(\chi) + \pr_2(w)(t^n)\\
                &= \val_{\chi, \pi}((\pr_1)_\RR(w)) + n(\pr_2)_\RR(w)
            \end{align*}
        \end{proof}
        We note some properties of $\Delta'$, and these properties are needed to construct strictly toroidal models of hypersurfaces of mock toric varieties. 
        \begin{definition}\label{def: type of polytope}
            We define some properties of $\Delta'$ as follows:   
            \begin{itemize}
                \item Let $\Delta'_{\spe}$ denote a subset of $\Delta'$ defined as follows:
                \[
                    \Delta'_{\spe} = \{\sigma\in\Delta'\mid\sigma\cap(N'_\RR\times\{1\})\neq\emptyset\}
                \]
                We remark that the following equation holds:
                \[
                    \Delta'_{\spe} = \{\sigma\in\Delta'\mid(\pr_2)_\RR(\sigma)\neq \{0\}\}
                \]
                \item Let $\Delta'_{\bdd}$ denote a subset of $\Delta'_{\spe}$ defined as follows:
                \[
                    \Delta'_{\bdd} = \{\sigma\in\Delta'_{\spe}\mid\sigma\cap(N'_\RR\times\{1\}) \mathrm{\ \ is\ bounded.}\}
                \]
                \item We call that $\Delta'$ is \textbf{compactly\ arranged} if for every $\sigma_1, \sigma_2\in\Delta'_{\bdd}$, and $\tau\in\Delta'$ such that $\sigma_1\cup \sigma_2\subset \tau$, there exists $\sigma_3\in\Delta'_{\bdd}$ such that $\sigma_1\cup \sigma_2\subset \sigma_3$. 
                \item We call that $\Delta'$ is \textbf{generically\ unimodular} if every $\sigma\in\Delta'$ such that $(\pr_2)_\RR(\sigma) = \{0\}$ is unimodular. 
                \item We call that $\Delta'$ is \textbf{specifically\ reduced} if for every $\gamma\in\Delta'_{\bdd}$ such that $\dim(\gamma) = 1$, we have $\gamma\cap (N'\times\{1\}) \neq \emptyset$. 
            \end{itemize}
        \end{definition}
        The following proposition gives sufficient conditions for some conditions in Definition \ref{def: type of polytope}.  
        \begin{proposition}\label{prop:example of type of polytope}
            We keep the notation in Definition \ref{def: type of polytope}.  
            Then the following statements follow:
            \begin{enumerate}
                \item[(a)] If $\Delta'$ is a simplicial fan, then $\Delta'$ is compactly arranged. 
                \item[(b)] If $\Delta'$ is a unimodular fan, then $\Delta'$ is compactly arranged and generically unimodular. 
            \end{enumerate}
        \end{proposition}
        \begin{proof}
            We show these statements from (a) to (b). 
            \begin{enumerate}
                \item[(a)] Let $\sigma_1, \sigma_2\in\Delta'_{\bdd}$, and $\tau\in\Delta'$ be cones such that $\sigma_1\cup \sigma_2\subset \tau$. 
                Then $\sigma_1$ and $\sigma_2$ are faces of $\tau$. 
                Because $\tau$ is a simplicial cone, there exists a face $\sigma_3$ of $\tau$ such that $\sigma_3$ is generated by $\sigma_1$ and $\sigma_2$. 
                From the definition of $\sigma_3$, we have $\sigma_3\in\Delta'_{\bdd}$ and $\sigma_1\cup \sigma_2\subset \sigma_3$. 
                \item[(b)] The sub fan $\Delta'_{0}$ of $\Delta'$ is unimodular too. 
                Thus, $\Delta'$ is generically unimodular. 
                Moreover, $\Delta'$ is compactly arranged from (a). 
            \end{enumerate}
        \end{proof}
        When $\Delta'$ is specifically reduced, $\Delta'(\varphi)$ has the "same" conditions for any $\varphi\in\Phi$. 
        \begin{proposition}\label{prop:varphi-fan, special}
            We keep the notation in Definition \ref{def: type of polytope}.
            We assume that $\Delta'$ is specifically reduced. 
            Then for any $\varphi\in\Phi$ and any $\gamma\in\Delta'(\varphi)$ with $\dim(\gamma) = 1$ and $\gamma\not\subset (N'/s(N_\varphi))_\RR\times\{0\}$, we have $\gamma\cap ((N'/s(N_\varphi))\times\{1\}) \neq \emptyset$. 
        \end{proposition}
        \begin{proof}
            Let $\varphi\in\Phi$ be an element. 
            From the condition (3) in Definition \ref{def:mock-toric} and Lemma \ref{lem: injective fan}, there is a one-to-one correspondence of the following two sets by $(q'_\varphi\times\id_\ZZ)_\RR$: 
            \[
                \{\gamma'\in\Delta'_\varphi\mid \dim(\gamma') = 1, \gamma'\not\subset N'_\RR\times\{0\}\}
            \]
            \[
                \{\gamma\in\Delta'(\varphi)\mid \dim(\gamma) = 1, \gamma\not\subset (N'/s(N_\varphi))_\RR\times\{0\}\}
            \]
            Let $\gamma\in\Delta'(\varphi)$ be a cone such that $\dim(\gamma) = 1$ and $\gamma\not\subset (N'/s(N_\varphi))\times\{0\}$. 
            Then there exists $\gamma'\in\Delta'_\varphi$ such that $(q'_\varphi\times\id_\ZZ)_\RR(\gamma') = \gamma$, $\dim(\gamma') = 1$, and $\gamma'\not\subset N'_\RR\times\{0\}$ from the above argument. 
            In particular, $\gamma'\in\Delta'_{\bdd}$. 
            Thus, from the assumption, $\gamma'\cap (N'\times\{1\}) \neq\emptyset$. 
            Hence, there exists $v'\in N'$ such that $(v', 1)\in\gamma'$. 
            Because $(q'_\varphi\times\id_\ZZ)_\RR(\gamma') = \gamma$, we have $(q'_\varphi(v'), 1)\in\gamma$. 
            Therefore, $\gamma\cap ((N'/s(N_\varphi))\times\{1\}) \neq \emptyset$. 
        \end{proof}
        The reason for introducing the definition of specifically reduced fans is evident from the following proposition. 
        While we later introduce strictly toroidal schemes over a valuation ring, the definition of specifically reduced fans is crucial for the construction of strictly toroidal models of hypersurfaces of mock toric varieties.
        \begin{proposition}
        \label{prop:reduced+nondegenerate=strictly toroidal} 
            Let $f\in k[W^\circ]$ be a function. 
            We assume that $\Delta'$ is specifically reduced and $f$ is non-degenerate for $\Delta'$. 
            Let $x\in H_{W, f}\times_{\A^1} \{0\}$ be a point. 
            Then there exists 
            \begin{enumerate}
                \item[(1)] An open neighborhood $U\subset W$ of $x$
                \item[(2)] A toric monoid $S$
                \item[(3)] An element $\omega\in S$
                \item[(4)] An $\A^1_k$-morphism $p\colon U\rightarrow \Spec(k[t][S]/(t-\chi^\omega))$
            \end{enumerate}
            such that 
            \begin{enumerate}
                \item[(i)] The restriction $p|_{U\cap H_{W, f}}\colon U\cap H_{W, f}\rightarrow \Spec(k[t][S]/(t-\chi^\omega))$ is a smooth morphism. 
                \item[(ii)] A ring $k[S]/(\chi^\omega)$ is reduced. 
            \end{enumerate}
        \end{proposition}
        \begin{proof}
            Let $\sigma\in\Delta'$ be a cone and $\varphi\in\Phi$ be an element such that $x\in H_{W, f}\cap W_\sigma$ and $\sigma\in\Delta'_\varphi$. 
            From the assumption of $x$, we have $\sigma\not\subset N'_\RR\times\{0\}$. 
            Let $\sigma_\varphi$ denote $(q'_\varphi\times\id_{\ZZ})_\RR(\sigma)$. 
            Then from Proposition \ref{prop:fundamental hypersurface}, there exists $g_\varphi\in k[M'_\varphi\oplus\ZZ]$ such that $H_{X(\sigma_\varphi),g_\varphi} \times_{X(\sigma_\varphi)} W(\sigma) = H_{W, f}\cap W(\sigma) $. 
            We remark that $\iota''_{\varphi}(W(\sigma))$ is an open subscheme of $X(\sigma_\varphi)$ and $H_{X(\Delta'_\varphi), g_\varphi}\cap O_{\sigma_\varphi}\not\subset O_{\sigma_\varphi}$ from Lemma \ref{lem:fine for varphi}. 
            Moreover, from the proof of Proposition \ref{prop:basic-property1}(c), $(H_{X(\sigma_\varphi),g_\varphi}\cap O_{\sigma_\varphi})\times_{X(\sigma_\varphi)} W(\sigma) = H_{W, f}\cap W_\sigma$. 
            Because $f$ is nondegenerate for $\Delta'$, $H_{W, f}\cap W_\sigma\cong H_{X(\sigma_\varphi),g_\varphi}\cap O_{\sigma_\varphi}\cap \iota''_{\varphi}(W(\sigma))$  is smooth over $k$. 
            
            From the assumption and Proposition \ref{prop:varphi-fan, special}, $\sigma_\varphi\subset (N'/s(N_\varphi))_\RR\times\RR_{\geq0}$ and $\sigma_\varphi \not\subset (N'/s(N_\varphi))_\RR\times\{0\}$, and $\gamma\cap ((N'/s(N_\varphi))\times\{1\})\neq\emptyset$ for every 1-dimensional ray $\gamma\preceq\sigma_\varphi$ with $\gamma\cap ((N'/s(N_\varphi))_\RR\times\{1\})\neq\emptyset$. 

            Thus, from Lemma \ref{lem: for 7-4}, there exists 
            \begin{itemize}
                \item A toric monoid $S$
                \item An alement $\omega\in S$
                \item A morphism $\theta\colon X(\sigma_\varphi)\rightarrow \Spec(k[t][S]/(t-\chi^\omega))$ over $\Spec(k[t]) = \A^1$
                \item An open subset $U'$ of $\iota''_{\varphi}(W(\sigma))$
            \end{itemize}
            such that
            \begin{enumerate}
                \item[(1)] $H_{X(\sigma_\varphi), g_\varphi}\cap O_{\sigma_\varphi}\cap \iota''_{\varphi}(W'(\sigma))\subset U'$
                \item[(2)] The restriction $\theta|_{H_{X(\sigma_\varphi), g_\varphi}\cap U'}\colon 
                H_{X(\sigma_\varphi), g_\varphi}\cap U'\rightarrow \Spec(k[t][S]/(t-\chi^\omega))$ is smooth. 
                \item[(3)] A ring $k[S]/(\chi^\omega)$ is reduced. 
            \end{enumerate}
            Let $U$ denote an open subset ${\iota''}_{\varphi}^{-1}(U')$ of $W$. 
            Then $H_{W, f}\cap W_\sigma\subset U$.  
            In particular, $x\in U$. 
            Let $p$ denote a composition $\theta\circ\iota''_{\varphi}|_U$. 
            We remark that $\iota''_{\varphi}$ is an $\A^1_k$-morphism because of the mock toric structure of $W$, so $p$ is an $\A^1_k$-morphism. 
            Moreover, we can check that the restriction $p|_{H_{W, f}\cap U}\colon {H_{W, f}\cap U}\rightarrow \Spec(k[t][S]/(t-\chi^\omega))$ is smooth. 
        \end{proof}
        
        The following lengthy proposition provides sufficient conditions for the construction of a strictly toroidal model of a hypersurface of a mock toric variety, which is one of the aims of this article. 
        \begin{proposition}\label{prop: irreducible computation}
            We create an itemized list for the notation.
            \begin{itemize}
                \item Let $f\in k[W^\circ]$ be a function.  
                We assume that $f$ is fine for $\Delta'$. 
                \item We assume that $\supp(\Delta') = (\pi^1_\RR)^{-1}(\supp(\Delta\times\Delta_!))$.
                \item Let $p'$ denote the following composition:
                \begin{equation*}
                    \begin{tikzcd}
                        W\ar[r, hook,"\iota''"]& X(\Delta')\ar[r, "\pi^1_*"]& X(\Delta\times\Delta_!)\ar[r, "(\pr_2)_*"]& \A^1_k
                    \end{tikzcd}
                \end{equation*}
                \item Let $q'$ denote the restriction $p'|_{H_{W, f}}\colon H_{W, f}\rightarrow \A^1_k$.
                \item Let $p'_{\Rt}\colon \mathscr{W}\rightarrow \Spec(\Rt)$ and $q'_{\Rt}\colon \mathscr{H}\rightarrow \Spec(\Rt)$ denote the base change morphisms of $p'$ and $q'$ along $\Spec(\Rt)\rightarrow\A^1_k$. 
                \item $p'_{\Kt}\colon \WW_{\Kt}\rightarrow \Spec(\Kt)$ and $q'_{\Kt}\colon\HH_{\Kt}\rightarrow \Spec(\Kt)$ denote the base change morphisms of $p'_{\Rt}$ and $q'_{\Rt}$ along $\Spec(\Kt)\rightarrow\Spec(\Rt)$. 
                We remark that natural morphisms $\WW_{\Kt}\hookrightarrow \WW$ and $\HH_{\Kt}\hookrightarrow \HH$ are open immersions.  
                \item Let $\alpha\colon Y_{\Kt}\rightarrow Y_{k(t)}$ be a dominant morphism defined by a base change along a field extension $\Kt/{k(t)}$. 
                We regard $Y_{\Kt}$ as a mock toric over $\Kt$ induced by $Y$ along a field extension $\Kt/{k(t)}$ (cf. Proposition \ref{prop:ext'n}). 
                \item Let $\HH^\circ$ denote a Cartesian product  $H^\circ_{W, f}\times_{{\Gm^1}_{,k}}\Spec(\Kt)$. 
            \end{itemize}
            Then the following statements follow:
            \begin{enumerate}
                \item[(a)] The morphism $p'_{\Rt}$ is flat and $q'_{\Rt}$ is separated and of finite presentation over $\Spec(\Rt)$. 
                \item[(b)] As a closed subscheme of $X(\Delta')\times_{\A^1_k}\Spec(\Kt) = X(\Delta'_{(0)})\times_k \Spec(\Kt)$, $\WW_{\Kt}$ is equal to $Y_{\Kt}$.
                \item[(c)] If $Z$ is proper over $k$, then $q'_{\Rt}$ is proper. 
                \item[(d)] If $H^\circ_{W, f}$ is reduced and every irreducible components of $H^\circ_{W, f}$ dominates $\Gm^1$, then $q'_{\Rt}$ is flat. 
                \item[(e)] The scheme theoretic closure of $\mathscr{H}^\circ$ in $\WW_{\Kt}$ is equal to $\HH_{\Kt}$. 
                \item[(f)] If the assumption of (d) holds, then the scheme theoretic closure of $\HH_{\Kt}$ in $\WW$ is equal to $\HH$. 
                \item[(g)] If $\Delta'$ is generically unimodular and $f$ is non-degenerate for $\Delta'$, then $\HH_{\Kt}$ is smooth over $\Spec(\Kt)$. 
                \item[(h)] Let $y\in \HH_k$ be a point. 
                If $\Delta'$ is specifically reduced and $f$ is non-degenerate for $\Delta'$, then there exists
                \begin{enumerate}
                    \item[(1)] An open neighborhood $U\subset \HH$ of $y$
                    \item[(2)] A toric monoid $S$
                    \item[(3)] An element $\omega\in S$
                    \item[(4)] A $\Spec(\Rt)$-morphism $\beta\colon U\rightarrow \Spec(\Rt[S]/(t-\chi^\omega))$
                \end{enumerate}
                such that 
                \begin{enumerate}
                    \item[(I)] The morphism $\beta$ is a smooth morphism. 
                    \item[(II)] A ring $k[S]/(\chi^\omega)$ is reduced. 
                \end{enumerate}
                \item[(i)] We assume that $f$ is non-degenerate for $\Delta'$ and $f\in k[W^\circ]$ is irreducible. 
                Let $f_1, f_2, \ldots, f_r\in \Kt[Y^\circ_{\Kt}]$ be a prime decomposition of $\alpha^*(\psi^*(f))\in \Kt[Y^\circ_{\Kt}]$. 
                Then through the identification with $\WW_{\Kt}$ and $Y_{\Kt}$, the following equation holds as a closed subscheme of $Y^\circ_{\Kt}$:
                \[
                    \HH^\circ = \coprod_{1\leq i\leq r} H^\circ_{Y_{\Kt}, f_i}
                \]
                Similarly, the following equation as a closed subscheme of $Y_{\Kt}$ too:
                \[
                    \HH_{\Kt} = \coprod_{1\leq i\leq r} H_{Y_{\Kt}, f_i}
                \]
                \item[(j)] We assume that the assumptions of (d), (h), and (i) hold. 
                We use the notation in (i). 
                For $1\leq i\leq r$, let $\HH_i$ denote the scheme theoretic closure of $H_{Y_{\Kt}, f_i}$ in $\WW$. 
                Then $\HH_i$ is flat over $\Spec(\Rt)$ and $(\HH_i)_{\Kt} = H_{Y_{\Kt}, f_i}$ for any $1\leq i\leq r$. 
                Moreover, as a closed subscheme of $\WW$, the following equation holds:
                \[
                    \HH = \coprod_{1\leq i\leq r} \HH_i
                \]
                \item[(k)] We assume that the assumption of (c), (g), and (j) holds. 
                We use the notation in (j). 
                Then $\HH_i\rightarrow\Spec(\Rt)$ is a proper strict toroidal model of a smooth variety $(\HH_i)_{\Kt} = H_{Y_{\Kt}, f_i}$ for any $1\leq i\leq r$.
            \end{enumerate}
        \end{proposition}
        \begin{proof}
            We prove the statements from (a) to (q) in order. 
            \begin{enumerate}
                \item[(a)] From the definition of $W$, the restriction $(\pi^1)_*|_{W}\colon W\rightarrow Z^1$ is dominant from Corollary \ref{cor: dominant morphism mock version}. 
                Because $Z^1$ is isomorphic to $Z\times\A^1$, $p'$ is a dominant morphism. 
                Moreover, from Proposition \ref{prop:basic-property1}(a), $W$ is an integral scheme . 
                Thus, $p'$ is a flat morphism, so $p'_{\Rt}$ is flat too. 

                We show that $q'$ is separated and of finite presentation. 
                Indeed, toric varieties are separated and of finite type over $k$. 
                Thus, the toric morphism $(\pr_2\circ\pi^1)_*\colon X(\Delta')\rightarrow \A^1$ is separated and of finite type over $\A^1_k$, in particular, of finite presentation over $\A^1_k$. 
                Moreover, a closed immersion $H_{W, f}\hookrightarrow X(\Delta')$ is separated and of finite presentation because $X(\Delta')$ is a Noetherian scheme. 
                Thus, $q'$ is separated and of finite presentation, so $q'_{\Rt}$ is so. 
                \item[(b)] From Proposition \ref{prop:relative-mock}(c), we have already shown that the isomorphism $X(\Delta')\times_{\A^1_k}\Spec(k(t)) \cong X(\Delta'_{(0)})\times_{\Spec(k)} \Spec(k(t))$ induces the isomorphism $W\times_{\A^1_k}\Spec(k(t)) \cong Y_{k(t)}$. 
                Thus, by the base change along $\Spec(\Kt)\rightarrow\Spec(k(t))$, the statement holds. 
                \item[(c)] We show that $q'$ is proper. 
                Indeed, there exists the following diagram whose all small squares are Cartesian squares:
                \begin{equation*}
                    \begin{tikzcd} 
                        W\ar[r, hook]\ar[d]& X(\Delta')\ar[d, "\pi^1_*"]&\ \\
                        Z^1\ar[r, hook]\ar[d]& X(\Delta\times\Delta_!)\ar[r, "(\pr_2)_*"]\ar[d, "(\pr_1)_*"]& \A^1_k\ar[d]\\
                        Z\ar[r, hook]& X(\Delta)\ar[r]& \Spec(k)
                    \end{tikzcd}
                \end{equation*}
                Because $Z$ is proper over $k$, $Z^1$ is proper over $\A^1_k$. 
                Moreover, from the assumption of $\Delta'$, $\pi^1_*\colon X(\Delta')\rightarrow X(\Delta\times\Delta_!)$ is proper. 
                Thus, $W$ is proper over $\A^1_k$.  
                Therefore, $p'$ and $q'$ are proper morphisms because $H_{W, f}\hookrightarrow W$ is a closed immersion. 
                Therefore, the statement holds. 
                \item[(d)] We show that $q'$ is flat. 
                We recall that $k$ is algebraically closed, and $\A^1_k$ is integral and regular. 
                Moreover, $H_{W, f}$ is a scheme theoretic closure of $H^\circ_{W, f}$, in particular, $H^\circ_{W, f}$ is a dense open subset of $H_{W, f}$.  
                Thus, $H_{W, f}$ is reduced and for any $y\in H_{W, f}$, there exists an irreducible components $E$ of $H^\circ_{W, f}$ such that $y\in\overline{E}$. 
                Therefore, from the assumption, $q'$ is a flat morphism. 
                In particular, $q'_\Rt$ is also flat. 
                \item[(e)] From (b), We can identify with $\WW_{\Kt}$ and $Y_{\Kt}$. 
                Thus, there exists the following commutative diagram whose all squares are Cartesian squares:
                \begin{equation*}
                    \begin{tikzcd} 
                        Y^\circ_{\Kt}\ar[r, "\alpha"]\ar[d]&Y^\circ_{k(t)}\ar[r, "\psi"]\ar[d]&W^\circ\ar[d]\\
                        \Spec(\Kt)\ar[r]& \Spec(k(t))\ar[r]& {\Gm^1}_{,k}
                    \end{tikzcd}
                \end{equation*}
                We use the notation Proposition \ref{prop:heredity-finess-for-relative} and Proposition \ref{prop: heredity-finess-for-basechange}.  
                Then we have $\HH^\circ = H^\circ_{Y_{\Kt}, \alpha^*\psi^*(f)}$ on the identification with $\WW_{\Kt}$ and $Y_{\Kt}$. 
                In particular, the scheme theoretic closure of $\HH^\circ$ in $\WW_{\Kt}$ is equal to $H_{Y_{\Kt}, \alpha^*\psi^*(f)}$. 
                Moreover, from the assumption of $f$, Proposition \ref{prop:heredity-finess-for-relative}(b) and (c), Proposition \ref{prop: heredity-finess-for-basechange}(b) and (c), and Proposition \ref{prop:relative-mock}(c), there exists the following diagram whose all small squares are Cartesian squares:
                \begin{equation*}
                    \begin{tikzcd} 
                        H_{Y_{\Kt}, \alpha^*\psi^*(f)}\ar[r]\ar[d]&H_{Y_{k(t)}, \psi^*(f)}\ar[r]\ar[d]&H_{W, f}\ar[d]\\
                        Y_{\Kt}\ar[r, "\alpha"]\ar[d]&Y_{k(t)}\ar[r, "\psi"]\ar[d]&W\ar[d]\\
                        \Spec(\Kt)\ar[r]&\Spec(k(t))\ar[r]&\A^1_k
                    \end{tikzcd}
                \end{equation*}
                The above diagram indicates that $H_{Y_{\Kt}, \alpha^*\psi^*(f)}$ is equal to $\HH_{\Kt}$ as the closed subscheme $Y_{\Kt} = \WW_{\Kt}$. 
                Therefore, the statement holds. 
                \item[(f)] From (a) and (d), $\HH$ and $\WW$ are flat over $\Spec(\Rt)$. 
                Thus, from Lemma \ref{lem:Gubler}, $\HH$ is a scheme theoretic closure of $\HH_{\Kt}$ in $\WW$. 
                \item[(g)] We use the notation in (e).  
                Then $\alpha^*\psi^*(f)$ is non-degenerate for $\Delta'_{(0)}$ from Proposition \ref{prop:heredity-finess-for-relative} (d) and Proposition \ref{prop: heredity-finess-for-basechange}(d). 
                In the proof of (e), we have already checked that $H_{Y_{\Kt}, \alpha^*\psi^*(f)} = \HH_{\Kt}$. 
                Thus, from Proposition \ref{prop:unimodular+nondegenerate=smooth}, $\HH_{\Kt}$ is smooth over $\Kt$. 
                \item[(h)] We remark that $H_{W, f}\times_{\A^1}\{0\} = \HH_k$ as a $k$-scheme. 
                Thus, from Proposition \ref{prop:reduced+nondegenerate=strictly toroidal}, there exists 
                \begin{enumerate}
                    \item[(1)] An open neighborhood $U_0\subset W$ of $y$
                    \item[(2)] A toric monoid $S$
                    \item[(3)] An element $\omega\in S$
                    \item[(4)] An $\A^1_k$-morphism $\theta\colon U_0\rightarrow \Spec(k[t][S]/(t-\chi^\omega))$
                \end{enumerate}
                such that 
                \begin{enumerate}
                    \item[(I)] The restriction $\theta|_{U_0\cap H_{W, f}}\colon U_0\cap H_{W, f}\rightarrow \Spec(k[t][S]/(t-\chi^\omega))$ is a smooth morphism. 
                    \item[(II)] A ring $k[S]/(\chi^\omega)$ is reduced. 
                \end{enumerate}
                By the base change along $\Spec(\Rt)\rightarrow\A^1$, let $V$ denote an open subscheme $U_0\times_{\A^1}\Spec(\Rt)$ of $\WW$ and $\beta$ denote a morphism $\theta\times\id_{\Spec(\Rt)}\colon V\rightarrow \Spec(\Rt[S]/(t-\chi^\omega))$ over $\Spec(\Rt)$. 
                Then the restriction $\beta|_{\HH\cap V}\colon \HH\cap V\rightarrow \Spec(\Rt[S]/(t-\chi^\omega))$ is a smooth morphism. 
                Let $U$ denote an open subset $\HH\cap V$ of $\HH$. 
                We remark that $y\in U$. 
                \item[(i)] We use the notation in (e). 
                From the argument in (g), $\alpha^*\psi^*(f)$ is non-degenerate for $\Delta'_{(0)}$. 
                Thus, from Proposition \ref{prop:val func of multiple poly} (b), $f_1, \ldots, f_r$ are non-degenerate for $\Delta'_{(0)}$. 
                In particular, $H^\circ_{Y_{\Kt}, f_i}$ is smooth over $\Kt$ for any $1\leq i\leq r$. 
                Moreover, from the proof of (e), we can identify with $\HH^\circ$ and  $H^\circ_{Y_{\Kt}, \alpha^*\psi^*(f)}$. 
                Thus, the following equation is the connected decomposition of $\HH^\circ$:
                \[
                    \HH^\circ = \coprod_{1\leq i\leq r} H^\circ_{Y_{\Kt}, f_i}
                \]
                From (e) and Proposition \ref{prop:val func of multiple poly} (c), we can identify with  $\HH_{\Kt}$ and $H_{Y_{\Kt}, \alpha^*\psi^*f}$, and the following equation is the connected decomposition of $\HH_{\Kt}$:
                \[
                    \HH_{\Kt} = \coprod_{1\leq i\leq r} H_{Y_{\Kt}, f_i}
                \]
                \item[(j)] We use the notation in (i). 
                From (f) and (i), we remark that the number of connected components of $\HH$ is less than and equal to $r$.  
                In addition to this, we remark that the generic fiber of a connected component of $\HH$ is non-empty and can be decomposed into a disjoint union of some $H_{Y_{\Kt}, f_i}$. 
                Let $E$ be a connected component of $\HH$. 
                From the first remark, $E$ is an open and closed subscheme of $\HH$. 
                In particular, $E$ is flat and locally of finite presentation over $\Spec(\Rt)$ from (a) and (d). 
                Moreover, from (h), $\HH_k$ is reduced, so $E_k$ is reduced too.  
                Thus, from Lemma \ref{lem: stacks exchange}(b), $E_\Kt$ is connected. 
                This shows that there exists a unique $1\leq i\leq r$ such that $E_\Kt = H_{Y_{\Kt}, f_i}$. 
                We recall that $E$ and $\WW$ is flat over $\Spec(\Rt)$ and $E$ is a closed subscheme of $\WW$. 
                Thus, from Lemma \ref{lem:Gubler}, $E = \HH_i$. 
                Hence, for any $1\leq i\leq r$, $\HH_i$ is flat over $\Spec(\Rt)$, $(\HH_i)_{\Kt} = H_{Y_{\Kt}, f_i}$, and $\HH_i$ is a connected component of $\HH$. 
                
                \item[(k)] From (c) and (j), $\HH$ is proper over $\Spec(\Rt)$, and $\HH_i$ is a closed subscheme of $\HH$ for any $1\leq i\leq r$. 
                Thus, $\HH_i$ is proper over $\Spec(\Rt)$ for any $1\leq i\leq r$. 
                From (g) and (i), $\HH_{\Kt}$ is smooth over $\Spec(\Kt)$ and $(\HH_i)_{\Kt}$ is an open subscheme of $\HH_\Kt$ for any $1\leq i\leq r$. 
                Thus, $(\HH_i)_{\Kt}$ is smooth over $\Spec(\Kt)$ for any $1\leq i\leq r$. 
                From (j), $\HH_i$ is flat over $\Spec(\Rt)$ for any $1\leq i\leq r$. 
                Thus, from (h), $\HH_i$ is a proper strictly toroidal model of a smooth $\Kt$-variety $H_{Y_{\Kt}, f_i}$ for any $1\leq i\leq r$. 
            \end{enumerate}
        \end{proof}
\section{Appendix}
        In this section, we prove the lemmas needed in this article. 
        We use the notation in Section 2. 
        The lemmas presented here are all fundamental and well-known, but they are frequently used in the proofs of propositions in the main article. 
        Therefore, they are compiled here as an appendix.
        \subsection{Toric Varieties Part.1}
        In this subsection, we prove some lemmas about toric varieties. 
        While some lemmas have already been published in the literature(\cite{CLS11} and \cite{Ful93}), the author has decided to describe them here because he thought the details of the proofs are crucial. 
        \vskip\baselineskip
        The following lemma is elementary, but it is important for the main article. 
        \begin{lemma}\label{lem: injective fan}
            Let $V, V'$ be $\RR$-linear spaces and $f\colon V\rightarrow V'$ be an $\RR$-linear map. 
            Let $\Delta$ denote a polyhedral convex fan in $V$ and assume that the restriction of $f$ to $\supp(\Delta)$ is injective. 
            Then the following statements follow: 
            \begin{enumerate}
                \item[(a)] For any $\sigma\in\Delta$, $f(\sigma)$ is a polyhedral convex cone. 
                \item[(b)] There exists a one-to-one correspondence with faces of $\sigma$ and those of $f(\sigma)$ 
                by $f$. 
                \item[(c)] The following set $\Delta'$ is a polyhedral convex fan in $V'$:
                \[
                    \Delta' = \{f(\sigma)\mid \sigma\in\Delta\}
                \]
                \item[(d)] If $\Delta$ is a strongly convex polyhedral fan, $\Delta'$ is so. 
            \end{enumerate}
        \end{lemma}
        \begin{proof}
            We prove the statements from (a) to (d) in order. 
            \begin{enumerate}
                \item[(a)] Because $f$ is a linear map, $f(\sigma)$ is a polyhedral convex cone in $V'$. 
                \item[(b)] From the assumption, $\langle\sigma\rangle\cap\ker(f) = 0$, so that there exists a linear map $g_\sigma\colon V'\rightarrow V$ such that $(g_\sigma\circ f)|_{\langle\sigma\rangle}$ is the identity map of $\langle\sigma\rangle$. 
                Let $W$ and $W'$ denote the dual spaces of $V$ and $V'$, respectively. 
                Then for any $\omega\in\sigma^\vee$, we claim that $g^*_\sigma(\omega) \in (f(\sigma))^\vee$ and $f(\sigma\cap\omega^\perp) = f(\sigma)\cap g^*_\sigma(\omega)^\perp$. 
                Indeed, for any $v\in \sigma = g_\sigma(f(\sigma))$, the following equation holds:
                \begin{align*}
                    \langle f(v), g^*_\sigma(\omega)\rangle 
                            &= \langle g_\sigma(f(v)), \omega\rangle \\
                            &= \langle v, \omega\rangle
                \end{align*}
                On the other hand, for any $\omega'\in(f(\sigma))^\vee$, we claim that $f^*(\omega')\in\sigma^\vee$ and $f(\sigma\cap f^*(\omega')^\perp) = f(\sigma)\cap (\omega')^\perp$. 
                Indeed, for any $v\in \sigma$, the following equation holds: 
                \begin{align*}
                    \langle v, f^*(\omega')\rangle &= \langle f(v), \omega'\rangle 
                \end{align*}
                Thus, $f$ induces a one-to-one correspondence faces of $\sigma$ and those of $f(\sigma)$ because the restriction of $f$ for $\langle\sigma\rangle$ is injective. 
                \item[(c)] For any $\sigma\in\Delta$, all faces of $f(\sigma)$ are contained in $\Delta'$ from (b). 
                In addition, for any $\sigma$ and  $\tau\in\Delta$, $f(\sigma)\cap f(\tau) = f(\sigma\cap\tau)$ from the injectivity of $f$. 
                From (b), $f(\sigma\cap\tau)$ is a common face of $f(\sigma)$ and $f(\tau)$, so that $\Delta'$ is a fan. 
                \item[(d)] Because $f$ is an injective linear map, the statement holds. 
            \end{enumerate}
        \end{proof}
        The following elementary lemma is needed to prove Lemma\ref{lem:relative-torus-fibration}. 
        \begin{lemma}\label{lem:take-section}
            Let $N, N'$ be lattices of finite rank and $\pi\colon N\rightarrow N'$ be a surjective morphism. 
            Let $N_0$ be a sublattice of $N$ and assume that $N'/\pi(N_0)$ is torsion free and $N_0\cap \mathrm{ker}(\pi) = \{0\}$. 
            Then there is a section $s$ of $\pi$ such that $s(N')$ contains $N_0$. 
        \end{lemma}
        \begin{proof}
            Let $\pi'\colon N\rightarrow N'/(\pi(N_0))$ be a composition of $\pi$ and the quotient map from $N'$ to $N'/\pi(N_0)$. 
            From the assumption, there is a section $s'$ of $\pi'$. 
            Then $N$ can be decomposed into a direct sum $\mathrm{ker}(\pi')\oplus s'(N'/\pi(N_0))$. 
            By the assumption of $\pi$, $\mathrm{ker}(\pi')$ can be decomposed into a direct sum $\mathrm{ker}(\pi)\oplus N_0$, so that we would take a section $s$ of $\pi$ as the inverse morphism of the following composition:
            \begin{equation*}
                \begin{tikzcd} 
                    N_0\oplus s'(N'/\pi(N_0))\ar[r, hook]& N\ar[r, "\pi"]& N'
                \end{tikzcd}
            \end{equation*}
            Then $N_0\subset s(N')$. 
        \end{proof}
        In this article, there are many cases when toric morphisms are algebraic torus fibrations. 
        Therefore, the author has decided to summarize this morphism here. 
        \begin{lemma}\label{lem:torus-fibration}
            Let $N, N'$ be lattices of finite rank, $\pi\colon N\rightarrow N'$ be a surjective morphism, $s\colon N'\rightarrow N$ be a section of $\pi$, and $\sigma\subset N_\RR$ be a strongly convex rational polyhedral cone. 
            We assume that $\langle\sigma\rangle\cap N\subset s(N')$. 
            Then the following statements follow:  
            \begin{itemize}
                \item[(a)] Let $\sigma'$ denote $\pi_\RR(\sigma)$. 
                Then, $\sigma'$ is a strongly convex rational polyhedral cone in $N'$. 
                \item[(b)] For any $\tau\preceq\sigma$, we have $\pi_\RR(\tau)\preceq\sigma'$. 
                \item[(c)] Conversely, for any $\tau'\preceq\sigma'$, there exists a unique $\tau\preceq\sigma$ such that $\tau' = \pi_\RR(\tau)$. 
                \item[(d)] 
                For any $\tau\preceq\sigma$, let $\tau'$ denote $\pi_\RR(\tau)$. 
                Then, the following diagram of toric morphisms is a Cartesian diagram. 
                \begin{equation*}
                    \begin{tikzcd} 
                        X(\tau)\ar[r, "\pi_*"]\ar[d, hook]& X(\tau')\ar[d, hook]\\
                        X(\sigma)\ar[r, "\pi_*"]& X(\sigma')
                    \end{tikzcd}
                \end{equation*}
                \item[(e)] For the notation in (d), $\pi_*(O_\tau) = O_{\pi_{\tau'}}$ and the following diagram induced by toric morphisms is a Cartesian diagram: 
                \begin{equation*}
                    \begin{tikzcd} 
                        \overline{O_\tau}\ar[r, "\pi_*"]\ar[d, hook]&\overline{O_{\tau'}}\ar[d, hook]\\
                        X(\sigma)\ar[r, "\pi_*"]& X(\sigma')
                    \end{tikzcd}
                \end{equation*}
            \end{itemize}
        \end{lemma}
        \begin{proof}
            We prove these statements from (a) to (e) in order.
            \begin{itemize}
                \item[(a)] From the assumption, the restriction of $\pi_\RR$ to $\langle\sigma\rangle$ is injective. 
                Thus, from Lemma \ref{lem: injective fan}(a) and (d), $\sigma'$ is a strongly convex polyhedral cone. 
                Because $\pi$ is a morphism of lattices, $\sigma'$ is a rational polyhedral cone. 
                \item[(b), (c)] The statements follow from the injectivity of the restriction $\pi_\RR|_{\langle\sigma\rangle}$ and Lemma \ref{lem: injective fan} (b). 
                \item[(d), (e)] Let $\lambda\colon N'\oplus\mathrm{ker}(\pi)\rightarrow N$ be a lattice morphism defined as follows:
                \[
                    N'\oplus \mathrm{ker}(\pi) \ni (v', w)\mapsto s(v') + w\in N
                \]
                Because $s$ is a section of $\pi$, $\lambda$ is isomorphic. 
                In addition, from the assumption of $s$, $\lambda_\RR(\sigma'\times \{0\}) = \sigma$, so that $\lambda$ induces the toric isomorphism $\lambda_*\colon X(\sigma')\times \Gm^{\mathrm{rank}(\mathrm{ker}(\pi))} = X(\sigma'\times \{0\})\rightarrow X(\sigma)$ and there exists the following diagram of toric morphisms:
                \begin{equation*}
                    \begin{tikzcd} 
                        X(\sigma')\times \Gm^{\mathrm{rank}(\mathrm{ker}(\pi))}\ar[r, "\lambda_*"]\ar[rd, "\mathrm{pr}_1"]&X(\sigma)\ar[d, "\pi_*"]\\
                        & X(\sigma')
                    \end{tikzcd}
                \end{equation*}
                Hence, we may replace $N$ with $N'\oplus\ker(\pi)$, $\pi$ with $\pr_1$, and $\sigma$ with $\sigma'\times\{0\}$. 
                Therefore, statements (d) and (e) follow immediately from the diagram above. 
            \end{itemize}
        \end{proof}
        The following lemma is needed for Lemma\ref{lem:relative-torus-fibration}. 
        \begin{lemma}\label{lem:relative-section}
            Let $N, N', L$ be lattices of finite rank, $f\colon N\rightarrow N'$ and $\pi\colon L\rightarrow N$ be surjective morphisms, and $s\colon N'\rightarrow N$ and $u\colon N\rightarrow L$ be sections of $f$ and $\pi$ respectively. 
            Let $L'$ denote a lattice $L/u(\ker(f))$ and $g$ denote the quotient morphism from $L$ to $L'$. 

            Then the following statements follow:
            \begin{itemize}
                \item[(a)] There exist unique morphisms $\pi'\colon L'\rightarrow N'$ and $u'\colon N'\rightarrow L'$ such that the following diagram commutes and $\pi'\circ u' = \mathrm{id}_{N'}$: 
                \begin{equation*}
                    \begin{tikzcd} 
                        N\ar[r, "u"]\ar[d, "f"]& L\ar[r, "\pi"]\ar[d, "g"] &N\ar[d, "f"]\\
                        N'\ar[r, "u'"]& L'\ar[r, "\pi'"]&N'
                    \end{tikzcd}
                \end{equation*}
                \item[(b)] There exists a unique section $t$ of $g$ such that the following diagram commutes: 
                \begin{equation*}
                    \begin{tikzcd} 
                        N'\ar[r, "u'"]\ar[d, "s"]& L'\ar[r, "\pi'"]\ar[d, "t"] &N'\ar[d, "s"]\\
                        N\ar[r, "u"]& L\ar[r, "\pi"]&N
                    \end{tikzcd}
                \end{equation*}
                \item[(c)] Let $\lambda\colon N'\oplus \ker(f)\rightarrow N$ and $\mu\colon L'\oplus \ker(f)\rightarrow L$ be morphisms defined as follows:
                \[
                    N'\oplus \mathrm{ker}(f) \ni (x', x)\mapsto s(x') + x\in N
                \]
                \[
                    L'\oplus \mathrm{ker}(f) \ni (y', x)\mapsto t(y') + u(x)\in L
                \]
                Then $\lambda$ and $\mu$ are isomorphic and the following diagram commutes:
                \begin{equation*}
                    \begin{tikzcd} 
                        L'\oplus \ker(f)\ar[rrr, "(\pi'{,}\mathrm{id})"]\ar[rd, "\mu"]\ar[dd, "\mathrm{pr}_1"]& & & N'\oplus \ker(f)\ar[ld,"\lambda"]\ar[dd, "\mathrm{pr}_1"]\\
                        & L\ar[r, "\pi"]\ar[ld,"g"] & N\ar[rd, "f"] &\\
                        L'\ar[rrr, "\pi'"] & & & N'
                    \end{tikzcd}
                \end{equation*}
            \end{itemize}
        \end{lemma}
        \begin{proof}
            We prove these statements from (a) to (c) in order. 
            \begin{itemize}
                \item[(a)] From the definition of $L'$, $g\circ u(\ker(f)) = 0$ and $f\circ\pi(\ker(g)) = 0$. 
                Thus, $u'$ and $\pi'$ are induced so that $g\circ u = u'\circ f$ and $f\circ \pi = \pi'\circ g$. 
                \item[(b)] From sections $u$ and $s$, $L$ can be decomposed into a direct sum $u(s(N'))\oplus \ker(\pi)\oplus u(\ker(f))$. 
                Thus, we would take a section $t_0$ of $g$ as the inverse morphism of the following isomorphism:
                \begin{equation*}
                    \begin{tikzcd} 
                    u(s(N'))\oplus \ker(\pi)\ar[r, hook]& L\ar[r, "g"]&L'
                    \end{tikzcd}
                \end{equation*}
                Then we claim that $t_0\circ u' = u\circ s$ and $\pi\circ t_0 = s\circ \pi'$. 
                First, because $t_0\circ g|_{u\circ s(N')} = \id_{u\circ s(N')}$, the following equation holds:
                \begin{align*}
                    t_0\circ u' &= t_0\circ u' \circ f\circ s\\
                    &= t_0\circ g\circ u\circ s\\
                    &= u\circ s
                \end{align*}
                Second, for any $y'\in L'$, there exists $x'\in N'$ and $y\in\ker(\pi)$ such that $y' = g(u\circ s(x') + y)$ and $t_0(y') = u\circ s(x') + y$ from the definition of $t_0$. 
                Then the following equation holds:
                \begin{align*}
                    \pi\circ t_0(y') &= \pi(u\circ s(x') + y)\\
                    &= \pi\circ u \circ s(x')\\
                    &= s(x')
                \end{align*}
                On the other hand, the following equation holds:
                \begin{align*}
                    s\circ \pi'(y') &= s\circ \pi'\circ g (u\circ s(x') + y)\\
                    &= s\circ f\circ \pi(u\circ s(x') + y)\\
                    &= s\circ f(s(x'))\\
                    &= s(x')
                \end{align*}
                Thus, $\pi\circ t_0 = s\circ\pi'$. 
                Finally, we show the uniqueness of such $t$. 
                For any $y'\in L'$, there exists $x'\in N'$ and $y\in\ker(\pi)$ such that $y' = g(u\circ s(x') + y)$ from the direct sum decomposition of $L$ above. 
                Then for such $t$,
                \begin{align*}
                    t(g(u\circ s(x'))) &= t\circ g \circ t \circ u'(x')\\
                    &= t\circ u'(x')\\
                    &= u\circ s(x')
                \end{align*}
                On the other hand, for such $t$, we have $t\circ g(y) - y\in\ker(g)$. 
                In addition, the following equation holds:
                \begin{align*}
                    \pi(t\circ g(y) - y) &= \pi(t\circ g(y))\\
                    &= s\circ\pi'\circ g(y)\\
                    &= s\circ f\circ\pi(y)\\
                    &= 0
                \end{align*}
                Thus, $t\circ g(y) - y\in\ker(g)\cap\ker(\pi)$. 
                Because $u$ is a section of $\pi$, we have that $\ker(g)\cap\ker(\pi) = u(\ker(f))\cap \ker(\pi)$ 
                $= 0$.
                In particular, $t\circ g(y) = y$. 
                Therefore, for such $t$, we have $t(y') = u\circ s(x') + y$, so that this shows that $t$ is unique. 
                \item[(c)] Because $s$ is a section of $f$, $\lambda$ is isomorphic. 
                Likewise, because $\ker(f)$ and $\ker(g)$ is isomorphic by the restriction of $u$ to $\ker(f)$ and $t$ is a section of $g$, $\mu$ is ismorphic too. 
                We can check the commutativity of the diagram because $s\circ \pi' = \pi\circ t$. 
            \end{itemize}
        \end{proof}
        When we consider the structure of mock toric varieties induced by a dominant toric morphism, the following lemma is very important. 
        \begin{lemma}\label{lem:relative-torus-fibration}
            Let $N, N', L$ be lattices of finite rank, $f\colon N\rightarrow N'$ and $\pi\colon L\rightarrow N$ be surjective morphisms, $u\colon N\rightarrow L$ be sections of $\pi$. 
            Let $L'$ denote a lattice $L/u(\ker(f))$ and $g$ denote the quotient morphism from $L$ to $L'$.
            Let $\sigma$ and $\tau$ be strongly convex rational polyhedral cones in $N_\RR$ and $L_\RR$, respectively. 
            We assume that $\ker(f)\cap (\langle\sigma\rangle \cap N) = 0$, $N'/f(\langle\sigma\rangle \cap N)$ is torsion free, and $\pi_\RR(\tau)\subset \sigma$. 
            
            Then the following statements follow:
            \begin{itemize}
                \item[(a)] 
                Let $\sigma'$ denote $f_\RR(\sigma)$ and $\tau'$ denote $g_\RR(\tau)$. 
                Then $\sigma'$ and $\tau'$ are strongly convex rational polyhedral cones in $N'_\RR$ and $L'_\RR$, respectively. 
                In addition, $\pi'$ is compatible with the cones $\tau'$ and $\sigma'$. 
                \item[(b)] The following diagram of toric varieties commutes and is a Cartesian diagram. 
                \begin{equation*}
                    \begin{tikzcd} 
                        X(\tau)\ar[r, "g_*"]\ar[d, "\pi_*"]& X(\tau')\ar[d, "\pi'_*"]\\
                        X(\sigma)\ar[r, "f_*"]& X(\sigma')
                    \end{tikzcd}
                \end{equation*}
            \end{itemize}
        \end{lemma}
        \begin{proof}
            We prove these statements from (a) to (b). 
            \begin{itemize}
                \item[(a)] From Lemma \ref{lem:take-section}, there exists a section $s$ of $f$ such that $\langle\sigma\rangle\cap N\subset s(N')$. 
                Moreover, from Lemma \ref{lem:relative-section}(b), there exists a section $t$ of $g$ such that $s\circ \pi' = \pi\circ t$ and $t\circ u' = u\circ s$. 
                In particular, $t(L') = u\circ s(N')\oplus\ker(\pi)$. 
                From the assumption of $\tau$, $\langle\tau\rangle\cap L\subset u(\langle\sigma\rangle \cap N)\oplus\ker(\pi)$, so that $\langle\tau\rangle\cap L\subset t(L')$. 
                Then from Lemma \ref{lem:torus-fibration}(a), the sets $f_\RR(\sigma)$ and $g_\RR(\tau)$ are strongly convex rational polyhedral cones in $N'_\RR$ and $L'_\RR$ respectively. 
                In addition, because $\pi'\circ g = f\circ \pi$ and $\pi_\RR(\tau)\subset \sigma$, we have  $\pi'_\RR(\tau')\subset \sigma'$. 
                Thus, $\pi'$ is compatible with the cones with $\tau'$ and $\sigma'$. 
                \item[(b)] The commutativity of the diagram follows because $\pi'\circ g = f\circ \pi$. 
                From the proof of Lemma \ref{lem:torus-fibration}(d) and (e) and Lemma \ref{lem:relative-section}(c), there exists the following commutative diagram of toric morphisms:
                \begin{equation*}
                    \begin{tikzcd} 
                        X(\tau')\times\Gm^d\ar[rrr, "(\pi_*'{,}\mathrm{id})"]\ar[rd, "\mu_*"]\ar[dd, "\mathrm{pr}_1"]& & & X(\sigma')\times\Gm^d\ar[ld,"\lambda_*"]\ar[dd, "\mathrm{pr}_1"]\\
                        & X(\tau)\ar[r, "\pi_*"]\ar[ld,"g_*"] & X(\sigma)\ar[rd, "f_*"] &\\
                        X(\tau')\ar[rrr, "\pi'_*"] & & & X(\sigma')
                    \end{tikzcd}
                \end{equation*}
                ,where $d$ denotes $\mathrm{rank}(\ker(f))$ and $\lambda\colon N'\oplus \ker(f)\rightarrow N$ and $\mu\colon L'\oplus \ker(f)\rightarrow L$ are isomorphisms defined by the proof of Lemma \ref{lem:torus-fibration}(d) and (e). 
                In particular, the largest diagram is a Cartesian diagram, and $\lambda_*$ and $\mu_*$ are isomorphic.
                Thus, the diagram in the statement is a Cartesian diagram. 
            \end{itemize}
        \end{proof}
        The following lemma is a basic well-known result about the torus orbit decomposition of the toric varieties. 
        \begin{lemma}\cite[Prop. 3.2.6]{CLS11}\label{lem: basic-includ-cone}
            Let $N$ be a lattice of finite rank, $\Delta$ be a strongly convex rational polyhedral fan in $N_\RR$, $M$ be the dual lattice of $N$, $\sigma$ be a cone in $\Delta$. 
            We recall that a global section ring of $X(\sigma)$ is $k[\sigma^\vee\cap M]$. 
            Let $\mathfrak{p}_\sigma$ denote a prime ideal of $k[\sigma^\vee\cap M]$ as follows: 
            \[
                \mathfrak{p}_\sigma = \bigoplus_{\omega\in (\sigma^\vee\cap M)\setminus \sigma^\perp} k\chi^\omega
            \]
            Let $O_\sigma$ denote a closed subvariety of $X(\sigma)$ associated with $\mathfrak{p}$. 
            Through an open immersion $X(\sigma)\hookrightarrow X(\Delta)$, we regard $O_\sigma$ as a locally closed subscheme of $X(\Delta)$. 

            Then the following statements follow: 
            \begin{enumerate}
                \item[(a)] Let $\tau\in\Delta$ be a cone such that $\sigma\not\subset \tau$. 
                Then $O_\sigma\cap X(\tau) = \emptyset$. 
                \item[(b)] $X(\sigma) = O_\sigma \amalg \bigcup_{\tau\precneqq\sigma} X(\tau)$
                \item[(c)] $X(\Delta) = \coprod_{\sigma\in\Delta} O_{\sigma}$
                \item[(d)] $\overline{O_\sigma} = \coprod_{\sigma\preceq\tau} O_{\tau}$
            \end{enumerate}
        \end{lemma}
        \begin{proof}
            We prove the statement from (a) to (d) in order. 
            \begin{enumerate}
                \item[(a)] Let $\gamma$ denote $\sigma\cap \tau\in\Delta$. 
                Because $\sigma\not\subset \tau$, $\gamma\neq\sigma$. 
                Let $\omega\in\sigma^\vee$ be an element such that $\gamma = \sigma\cap \omega^\perp$. 
                We remark that $\omega\notin\sigma^\perp$ because $\gamma \neq \sigma$. 
                In particular, $\chi^\omega\in\mathfrak{p}_\sigma$. 
                On the other hand, $\mathfrak{p}_\sigma k[\gamma^\vee\cap M] = k[\gamma^\vee\cap M]$ because $k[\gamma^\vee\cap M] = k[\sigma^\vee\cap M]_{\chi^\omega}$. 
                Thus, $O_\sigma\times_{X(\sigma)}X(\gamma) = \emptyset$. 
                Therefore, $O_\sigma\cap X(\tau) = \emptyset$ because $X(\gamma) = X(\sigma)\cap X(\tau)$. 
                \item[(b)] From (a), for any $\tau\precneqq\sigma$, $O_\sigma\cap X(\tau)=\emptyset$. 
                Thus, $O_\sigma\cap \bigcup_{\tau\precneqq\sigma} X(\tau) = \emptyset$. 
                It is enough to show that $X(\sigma)\subset O_\sigma \amalg \bigcup_{\tau\precneqq\sigma} X(\tau)$. 
                Let $\mathfrak{p}\subset k[\sigma^\vee\cap M]$ be a prime ideal. 
                We assume that $\mathfrak{p}\subset \mathfrak{p}_\sigma$. 
                From the definition of $\mathfrak{p}_\sigma$, there exists $\omega\in (\sigma^\vee\cap M)\setminus\sigma^\perp$ such that $\chi^\omega\notin\mathfrak{p}$. 
                In particular, $\mathfrak{p}\in D(\chi^\omega)$ where $D(\chi^\omega)$ is a fundamental open subset of an affine scheme $X(\sigma)$ associated with $\chi^\omega\in k[\sigma^\vee\cap M]$. 
                Let $\tau$ be a denote a face $\sigma\cap \omega^\perp$ of $\sigma$.  
                Then $D(\chi^\omega) = X(\tau)$ as an open subset of $X(\sigma)$. 
                Thus, $\mathfrak{p}\in X(\tau)$. 
                Because $\omega\not\in\sigma^\perp$, $\tau \neq \sigma$. 
                Therefore, $X(\sigma)\subset O_\sigma \amalg \bigcup_{\tau\precneqq\sigma} X(\tau)$. 
                \item[(c)] First, we show that $O_\sigma\cap O_\tau = \emptyset$ for any distinct $\sigma, \tau\in \Delta$. 
                Because $\sigma\neq\tau$, $\sigma\not\subset\tau$ or $\tau\not\subset\sigma$ holds. 
                We may assume $\sigma\not\subset\tau$. 
                Then $O_\sigma\cap X(\tau)$ from (a). 
                Because $O_\tau\subset X(\tau)$, $O_\sigma\cap O_\tau = \emptyset$. 
                Next, we show that $X(\Delta)\subset \bigcup_{\sigma\in\Delta} O_\sigma$. 
                Let $x\in X(\Delta)$ be a point. 
                Then there exists minimal cone $\sigma\in\Delta$ such that $x\in X(\sigma)$ because $\Delta$ is a fan. 
                If $x\notin O_\sigma$, there exists $\tau\precneqq\sigma$ such that $x\in X(\tau)$. 
                However, it is a contradiction to the minimality of $\sigma$. 
                Thus, $x\in O_\sigma$. 
                \item[(d)] 
                Let $\tau\in\Delta$ be a cone. 
                First, we assume that $\sigma\not\subset \tau$. 
                Then $O_\sigma\cap X(\tau) = \emptyset$ from (a). 
                Thus, $\overline{O_\sigma}\cap X(\tau)=\emptyset$. 
                Second, we assume that $\sigma\subset \tau$. 
                Let $\mathfrak{q}$ denote $\mathfrak{p}_\sigma\cap k[M\cap\tau^\vee]$. 
                From the definition of $\mathfrak{p}_\sigma$, $\mathfrak{q} = \bigoplus_{\omega\in (\tau^\vee\cap M)\setminus \sigma^\perp} k\chi^\omega$. 
                Thus, $\mathfrak{q}\subset \mathfrak{p}_\tau$. 
                Hence, $O_\tau\subset \overline{O_\sigma}\cap X(\tau)$. 
                Therefore, from (c), the statement holds. 
            \end{enumerate}
        \end{proof}
        The following lemma is a basic well-known result about the relation between a toric morphism and torus orbits. 
        \begin{lemma}\label{lem: basic-toric-morphism-inclusion}
            Let $N$ and $N'$ be a lattice, $\Delta$ and $\Delta'$ be a strongly convex rational polyhedral fan in $N_\RR$ and $N'_\RR$ respectively. 
            Let $\pi$ be a lattice morphism from $N'$ to $N$, and we assume that $\pi$ is compatible with the fans $\Delta'$ and $\Delta$. 
            Then the following statements hold:
            \begin{enumerate}
                \item[(a)] Let $\tau\in\Delta'$ be a cone and $\sigma\in\Delta$ be a minimal cone  such that $\pi_\RR(\tau)\subset \sigma$. 
                Then $\pi_*(O_\tau)\subset O_\sigma$. 
                \item[(b)] We use the notation of (a). 
                If $\pi$ is surjective, $\pi_*(O_\tau) = O_\sigma$. 
                \item[(c)] Let $\sigma\in\Delta$ be a cone and $U$ be an open subscheme $X(\Delta')\times_{X(\Delta)}X(\sigma)$ of $X(\Delta')$. 
                Then $U = \bigcup_{\tau\in\Delta', \pi_\RR(\tau)\subset\sigma} X(\tau)$. 
            \end{enumerate}
        \end{lemma}
        \begin{proof}
            Let $M$ and $M'$ be the dual lattices of $N$ and $N'$, respectively. 
            Let $\pi^*$ denote a lattice morphism $M'\rightarrow M$ or a ring morphism $k[M']\rightarrow k[M]$ induced by $\pi$. 
            We prove the statements from (a) to (c) in order. 
            \begin{enumerate}
                \item[(a)] From the assumption, $\pi$ induces the ring morphism $\pi^*\colon k[\sigma^\vee\cap M]\rightarrow k[\tau^\vee\cap M']$. 
                Thus, it is enough to show that $\mathfrak{p}_\sigma\subset (\pi^*)^{-1}(\mathfrak{p}_\tau)$. 
                Let $\omega\in(\sigma^\vee\cap M)\setminus\sigma^\perp$ be an element and $v'\in \tau^\circ\cap N'$ be an element. 
                Then $\pi(v')\in\sigma^\circ\cap N$ from the minimally of $\sigma$. 
                Because $\omega\notin\sigma^\perp$ and $\pi(v')\in\sigma^\circ\cap N$, we can check that $\langle v', \pi^*(\omega)\rangle = \langle \pi(v'), \omega\rangle >0$. 
                Thus, $\pi^*(\omega)\notin \tau^\perp$. 
                Therefore, $\mathfrak{p}_\sigma\subset (\pi^*)^{-1}(\mathfrak{p}_\tau)$. 
                \item[(b)] Because $\pi$ is surjective, $\pi^*\colon k[\sigma^\vee\cap M]\rightarrow k[\tau^\vee\cap M']$ is injective. 
                Let $f = \sum_{\omega\in\sigma^\vee\cap M}a_\omega\chi^\omega\in (\pi^*)^{-1}(\mathfrak{p}_\tau)$ be a Laurant polynomial. 
                From the injectivity of $\pi^*$, $\sum_{\omega\in\sigma^\vee\cap M}a_\omega\chi^{\pi^*(\omega)}$ is a $T_{N'}$-invariant decomposition of $\pi^*(f)$. 
                In particular, for any $\omega\in \sigma^\vee\cap M$ such that $a_\omega\neq 0$, $\pi^*(\omega)\notin\tau^\perp$ because $\pi^*(f)\in\mathfrak{p}_\tau$. 
                Let $\omega\in \sigma^\vee\cap M$ be an element such that $a_\omega\neq 0$. 
                Then there exists $v'\in \tau\cap N'$ such that $\langle v', \pi^*(\omega)\rangle > 0$. 
                Thus, $\langle \pi(v'), \omega\rangle = \langle v', \pi^*(\omega)\rangle > 0$. 
                Because $\pi_\RR(\tau)\subset \sigma$, $\pi(v')\in\sigma\cap N$. 
                Hence, $\omega\notin \sigma^\perp$ and $\chi^\omega\in \mathfrak{p}_\sigma$. 
                Therefore, $f\in\mathfrak{p}_\sigma$ and $\mathfrak{p}_\sigma = (\pi^*)^{-1}(\mathfrak{p}_\tau)$ from (a). 

                The equation $\mathfrak{p}_\sigma = (\pi^*)^{-1}(\mathfrak{p}_\tau)$ induces an injective ring morphism $\pi^*\colon k[\sigma^\perp \cap M]\rightarrow k[\tau^\perp \cap M']$. 
                We can check that this ring morphism is faithfully flat because lattice morphism $\pi^*|_{M\cap \sigma^\perp}\colon M\cap \sigma^\perp\rightarrow M'\cap \tau^\perp$ is split injection. 
                Thus, $\pi_*(O_\tau) = O_\sigma$.
                \item[(c)] Let $\tau\in\Delta'$ be a cone. 
                First, we assume that $\pi_\RR(\tau)\subset \sigma$. 
                Then from (a), $\pi_*(O_\tau)\subset O_\sigma$ so that $O_\tau\subset U$ 
                
                Second, we assume that $\pi_\RR(\tau)\not\subset \sigma$. 
                From the assumption of $\pi$, there exists a minimal cone $\gamma\in\Delta$ such that $\pi_\RR(\tau)\subset \gamma$. 
                From the assumption of $\tau$, $\gamma\not\subset\sigma$. 
                Thus, from Lemma\ref{lem: basic-includ-cone}(a), $O_\gamma\cap X(\sigma) = \emptyset$. 
                From (a), $\pi_*(O_\tau)\subset O_\gamma$ so that $O_\tau\cap U = \emptyset$. 

                In conclusion, from Lemma\ref{lem: basic-includ-cone}(b) and (c), the following equation holds:
                \begin{align*}
                    U &= \bigcup_{\pi_\RR(\tau)\subset \sigma, \tau\in\Delta'} O_\tau\\
                    &= \bigcup_{\pi_\RR(\tau)\subset \sigma, \tau\in\Delta'}X(\tau)
                \end{align*}
            \end{enumerate}
        \end{proof}
        The following lemma is a basic well-known result. 
        We used it when we consider the mock toric structure of the closure of a stratum of a mock toric variety. 
        \begin{proposition}\cite[Prop.3.2.7]{CLS11}\label{prop:orbit-toric}
            Let $N$ be a lattice of finite rank and $\Delta$ be a strongly convex rational polyhedral fan in $N_\RR$.
            For $\sigma\in\Delta$, let $\Delta^{\sigma}$ denote the following sub fan of $\Delta$. 
            \[
                \Delta^{\sigma} = \{\gamma\in\Delta\mid \exists\tau\in\Delta\ \ \mathrm{s.t. }\ \ \gamma\cup\sigma\subset\tau\}
            \]
            Let $\pi^\sigma$ denote the quotient map $N\rightarrow N/(\langle\sigma\rangle\cap N)$ and $\Delta[\sigma]$ denote the following set:
            \[
                \Delta[\sigma] = \{(\pi^\sigma)_\RR(\tau)\mid\sigma\subset\tau\in\Delta\}  
            \]
            Then the following statements follow:
            \begin{itemize}
                \item[(a)] The set $\Delta[\sigma]$ is a strongly convex rational polyhedral fan in $(N/(\langle\sigma\rangle\cap N))_\RR$. 
                \item[(b)] The morphism $\pi^\sigma$ is compatible with the fans $\Delta^\sigma$ and $\Delta[\sigma]$. 
                \item[(c)] A closed immersion $\overline{O_\sigma}\hookrightarrow X(\Delta)$ passes through $X(\Delta^\sigma)\subset X(\Delta)$. 
                \item[(d)] The following composition is isomorphic:
                \begin{equation*}
                    \begin{tikzcd} 
                        \overline{O_\sigma}\ar[r, hook]& X(\Delta^\sigma)\ar[r, "(\pi^\sigma)_*"]& X(\Delta[\sigma])
                    \end{tikzcd}
                \end{equation*}
                , where the former morphism is a closed immersion, and the latter one is a toric morphism induced by $\pi^\sigma$. 
            \end{itemize}
        \end{proposition}
        \begin{proof}
            We prove these statements from (a) to (e) in order. 
            \begin{itemize}
                \item[(a)] 
                Let $\tau\in\Delta$ be a cone such that $\sigma\preceq\tau$. 
                Because $\pi^\sigma$ is a lattice morphism, $(\pi^\sigma)_\RR(\tau)$ is a rational polyhedral convex cone. 
                Let $x$ and $y\in \tau$ be elements such that $(\pi^\sigma)_\RR(x) + (\pi^\sigma)_\RR(y) = 0$. 
                Then there exists $z$ and $w\in\sigma$ such that $x + y = z - w$ because $\ker((\pi^\sigma)_\RR) = \langle\sigma\rangle$. 
                Thus, $x$ and $y\in\sigma$ because $x, y, w\in\tau$, $\sigma\preceq\tau$, and $x + y + w\in\sigma$. 
                Therefore, $(\pi^\sigma)_\RR(\tau)$ is strongly convex. 
                
                Next, we show that $\Delta[\sigma]$ is a fan. 
                Let $M$ and $M^\sigma$ be dual lattices of $N$ and $N/(\langle\sigma\rangle\cap N)$ respectively, and $(\pi^\sigma)^*\colon M^\sigma\rightarrow M$ be a morphism induced by $\pi^\sigma$. 
                Let $\tau\in\Delta$ be a cone such that $\sigma\preceq\tau$. 
                Then we can check that for any  $\omega\in(\pi^\sigma)_\RR(\tau)^\vee$, $(\pi^\sigma)^*_\RR(\omega)\in\tau^\vee$ and $(\pi^\sigma)_\RR(\tau\cap (\pi^\sigma)^*_\RR(\omega)^\perp) = (\pi^\sigma)_\RR(\tau)\cap\omega^\perp$. 
                In fact, for any $v\in\tau$, the following equation holds: 
                \begin{align*}
                    \langle v, (\pi^\sigma)^*_\RR(\omega)\rangle &= \langle (\pi^\sigma)_\RR(v), \omega\rangle 
                \end{align*}
                In addition to this, we can check that $(\pi^\sigma)^*_\RR(M^\sigma_\RR) = \langle\sigma\rangle^\perp$.  
                This implies that for any $\omega\in(\pi^\sigma)_\RR(\tau)^\vee$, $\sigma\subset\tau\cap(\pi^\sigma)^*_\RR(\omega)^\perp$. 
                Thus, $(\pi^\sigma)_\RR(\tau\cap (\pi^\sigma)^*_\RR(\omega)^\perp)\in\Delta[\sigma]$. 
                Therefore, all faces of $(\pi^\sigma)_\RR(\tau)$ are contained in $\Delta[\sigma]$. 
                
                Thirdly, we show that for any $\tau$ and  $\gamma\in\Delta$ such that $\sigma\subset\tau\cap\gamma$, we have that $(\pi^\sigma)_\RR(\tau)\cap(\pi^\sigma)_\RR(\gamma) = (\pi^\sigma)_\RR(\tau\cap\gamma)$. 
                Because one direction is easy, we only show that $(\pi^\sigma)_\RR(\tau)\cap(\pi^\sigma)_\RR(\gamma) \subset (\pi^\sigma)_\RR(\tau\cap\gamma)$. 
                Let $x\in\tau$ and $y\in\gamma$ be elements such that $(\pi^\sigma)_\RR(x) = (\pi^\sigma)_\RR(y)$. 
                Then there exist $z$ and $w\in\sigma$ such that $x - y = z - w$. 
                The equation $x + w = y + z$ indicates that $x + w$ and $y + z\in\tau\cap\gamma$ because $\sigma\subset \tau\cap\gamma$. 
                Thus, $(\pi^\sigma)_\RR(x) = (\pi^\sigma)_\RR(x + w)\in (\pi^\sigma)_\RR(\tau\cap\gamma)$. 
                This implies that $(\pi^\sigma)_\RR(\tau)\cap(\pi^\sigma)_\RR(\gamma) \subset (\pi^\sigma)_\RR(\tau\cap\gamma)$. 
                
                Finally, we show that for any $\tau$ and  $\gamma\in\Delta$ such that $\sigma\subset\gamma\subset\tau$, we have that $(\pi^\sigma)_\RR(\gamma)\preceq(\pi^\sigma)_\RR(\tau)$. 
                Let $\omega\in\tau^\vee$ be an element such that $\gamma = \tau\cap \omega^\perp$. 
                Because $\langle\sigma\rangle\subset\langle\gamma\rangle\subset\omega^\perp$, there is an element $\omega'\in M^\sigma_\RR$ such that $(\pi^\sigma)^*_\RR(\omega') = \omega$. 
                Then for any $v\in\tau$, the following equation holds: 
                \begin{align*}
                    \langle (\pi^\sigma)_\RR(v), \omega'\rangle &= \langle v, (\pi^\sigma)^*_\RR(\omega')\rangle\\
                    &= \langle v, \omega\rangle
                \end{align*}
                Thus $\omega'\in(\pi^\sigma_\RR(\tau))^\vee$ and $(\pi^\sigma)_\RR(\gamma) = (\pi^\sigma)_\RR(\tau)\cap (\omega')^\perp$.
                Therefore,  $(\pi^\sigma)_\RR(\gamma)\preceq (\pi^\sigma)_\RR(\tau)$. 
                \item[(b)] 
                For $\gamma\in\Delta^\sigma$, there exists $\tau\in\Delta$ such that $\gamma\cup\sigma\subset\tau$. 
                Then $(\pi^\sigma)_\RR(\gamma)\subset(\pi^\sigma)_\RR(\tau)$ and $(\pi^\sigma)_\RR(\tau)\in\Delta[\sigma]$, so that $\pi^\sigma$ is compatible with the fans $\Delta^\sigma$ and $\Delta[\sigma]$.
                
                \item[(c)] From Lemma\ref{lem: basic-includ-cone}(d), $\overline{O_\sigma}$ is decomposed into torus orbits $O_\tau$ where $\tau\in\Delta$ is a cone which contains $\sigma$. 
                Thus, the following inclusion holds:
                \[
                    \overline{O_\sigma}\subset \bigcup_{\sigma\subset\tau\in\Delta} X(\tau) = X(\Delta^\sigma)
                \]
                \item[(d)] Let $\tau\in\Delta$ be a cone such that $\sigma\subset\tau$.  
                First, we show that the following diagram is commutative and a Cartesian diagram:
                \begin{equation*}
                    \begin{tikzcd} 
                        X(\tau)\ar[d]\ar[r, "(\pi^\sigma)_*"]& X((\pi^\sigma)_\RR(\tau))\ar[d]\\
                        X(\Delta^\sigma)\ar[r, "(\pi^\sigma)_*"] & X(\Delta[\sigma])
                    \end{tikzcd}
                \end{equation*}
                In fact, for any $\gamma\in\Delta^\sigma$, we can show that if $(\pi^\sigma)_\RR(\gamma)\subset (\pi^\sigma)_\RR(\tau)$, then $\gamma\subset\tau$ as follows: 
                Let $x\in\gamma$ and $y\in\tau$ be elements such that $(\pi^\sigma)_\RR(x) = (\pi^\sigma)_\RR(y)$. 
                Then there exists $z$ and $w\in \sigma$ such that $x - y = z - w$. 
                From the definition of $\Delta^\sigma$, there exists $\beta\in\Delta$ such that $\sigma\cup\gamma\subset\beta$. 
                The equation $x + w = y + z$ indicates that $x + w\in \tau$ because $\sigma\subset\tau$. 
                Because $x$ and $w\in\beta$, and $\beta\cap\tau\preceq\beta$, we have $x$ and $w\in\beta\cap\tau$. 
                In particular, $x\in\tau$. 
                Thus, the diagram above is a Cartesian diagram from Lemma\ref{lem: basic-toric-morphism-inclusion}(c). 

                Next, we show that for any $\tau\in\Delta$ such that $\sigma\subset\tau$, the upper composition from $\overline{O_\sigma}\cap X(\tau)$ to $X((\pi^\sigma)_\RR(\tau))$ of the following diagram whose all squares are Cartesian squares is an isomorphism:
                \begin{equation*}
                    \begin{tikzcd} 
                        \overline{O_\sigma}\cap X(\tau)\ar[d]\ar[r, hook]& X(\tau)\ar[d]\ar[r, "(\pi^\sigma)_*"] &X((\pi^\sigma)_\RR(\tau))\ar[d]\\
                        \overline{O_\sigma}\ar[r, hook]& X(\Delta^\sigma)\ar[r, "(\pi^\sigma)_*"] &X(\Delta[\sigma])
                    \end{tikzcd}
                \end{equation*}
                If this claim holds, then the statement (c) follows because $X(\Delta[\sigma])$ is covered by $\{X((\pi^\sigma)_\RR(\tau))\}_{(\sigma\subset)\ \tau\in\Delta}$. 
                
                Once the argument comes to affine toric varieties, we only check it algebraically. 
                We use the notation in the proof of statement (a), and we can show that $(\pi^\sigma)^*_\RR((\pi^\sigma)_\RR(\tau)^\vee) = \langle\sigma\rangle^\perp\cap\tau^\vee$. 
                Moreover, $(\pi^\sigma)^*((\pi^\sigma)_\RR(\tau)^\vee\cap M^\sigma) = \langle\sigma\rangle^\perp\cap\tau^\vee\cap M$ because $\pi^\sigma(M^\sigma) = \langle\sigma\rangle^\perp\cap M$. 
                By the way, the ring morphism associated with $(\pi^\sigma)_*$ is a natural inclusion $k[(\pi^\sigma)_\RR(\tau)^\vee\cap M^\sigma]\hookrightarrow k[\tau^\vee\cap M]$. 
                On the other hand, the ring morphism associated with the closed immersion $\overline{O_\sigma}\cap X(\tau)\rightarrow X(\tau)$ is $k[\tau^\vee\cap M]\rightarrow k[\tau^\vee\cap \langle\sigma\rangle^\perp\cap M]$ whose kernel is generated by $\{\chi^\omega\}_{\omega\in (\tau^\vee\cap M)\setminus \langle\sigma\rangle^\perp}$. 
                Therefore, the morphism $\overline{O_\sigma}\cap X(\tau)\rightarrow X((\pi^\sigma)_\RR(\tau))$ is an isomorphism. 
            \end{itemize}
        \end{proof}
        \begin{corollary}\label{cor: cor-orbit-toric}
            We keep the notation in Proposition \ref{prop:orbit-toric}, let $\sigma, \tau,$ and $\gamma$ be a cone in $\Delta$ such that $\sigma\subset\tau\cap\gamma$. 
            the following statements follow:
            \begin{enumerate}
                \item[(a)] If $(\pi^\sigma)_\RR(\gamma)\subset(\pi^\sigma)_\RR(\tau)$, then $\gamma\subset\tau$. 
                \item[(b)] $(\pi^\sigma)_\RR(\gamma)\cap(\pi^\sigma)_\RR(\tau) = (\pi^\sigma)_\RR(\gamma\cap\tau)$. 
            \end{enumerate}
        \end{corollary}
        \begin{proof}
            Both statements follow from the proof of Proposition \ref{prop:orbit-toric}. 
        \end{proof}
        The following lemma concerns toric varieties, which have a dominant toric morphism to $\A^1_k$. 
        \begin{lemma}\label{lem: relative lemma}
            Let $N$ be a lattice of finite rank and $\Delta$ be a strongly convex rational polyhedral fan in $(N\oplus\ZZ)_\RR$. 
            We assume that $\pr_2\colon N\oplus\ZZ\rightarrow \ZZ$ is compatible with the fans $\Delta$ and $\Delta_!$. 
            Let $\Delta_0$ denote the following sub fan of $\Delta$:
            \[
                \Delta_0 = \{\sigma\in\Delta\mid (\pr_2)_\RR(\sigma) = \{0\}\}
            \]
            Let $\Delta_{(0)}$ denote the following set:
            \[
                \Delta_{(0)} = \{(\pr_1)_\RR(\sigma)\mid \sigma\in\Delta_0\}
            \]
            Then the following statements follow:
            \begin{enumerate}
                \item[(a)] The set $\Delta_{(0)}$ is a strongly convex rational polyhedral fan in $N_\RR$. 
                \item[(b)] A toric morphism $(\pr_1)_*\times(\pr_2)_*\colon X(\Delta_0)\rightarrow X(\Delta_{(0)})\times {\Gm^1}_{,k}$ is an isomorphism. 
                Let $\sigma\in\Delta_0$ be a cone and $\sigma_0$ denote $(\pr_1)_\RR(\sigma)$. 
                Then, the above isomorphism induces the following three Cartesian diagrams:
                \begin{equation*}
                    \begin{tikzcd}
                        X(\sigma)\ar[r]\ar[d, hook]&X(\sigma_0)\times {\Gm^1}_{,k}\ar[d, hook]\\
                        X(\Delta_0)\ar[r, "p"]&X(\Delta_{(0)})\times{\Gm^1}_{,k}
                    \end{tikzcd}
                \end{equation*}
                \begin{equation*}
                    \begin{tikzcd}
                        \overline{O_\sigma}\cap X(\Delta_0)\ar[r]\ar[d, hook]&\overline{O_{\sigma_0}}\times {\Gm^1}_{,k}\ar[d, hook]\\
                        X(\Delta_0)\ar[r, "p"]&X(\Delta_{(0)})\times{\Gm^1}_{,k}
                    \end{tikzcd}
                \end{equation*}
                \begin{equation*}
                    \begin{tikzcd}
                        O_\sigma \ar[r]\ar[d]&O_{\sigma_0}\times {\Gm^1}_{,k}\ar[d]\\
                        X(\Delta_0)\ar[r, "p"]&X(\Delta_{(0)})\times{\Gm^1}_{,k}
                    \end{tikzcd}
                \end{equation*}
                where $p$ denote $(\pr_1)_*\times(\pr_2)_*$.
                \item[(c)] Let $\Psi$ denote the composition of the following maps:
                \[
                    X_{k(t)}(\Delta_{(0)})\rightarrow X(\Delta_{(0)})\times {\Gm^1}_{,k}\rightarrow X(\Delta_0)\hookrightarrow X(\Delta), 
                \]
                where the first morphism is a base change of the generic fiber $\Spec(k(t))\rightarrow \Gm^1$ along $(\pr_2)_*\colon X(\Delta_{(0)})\times{\Gm^1}_{,k}\rightarrow \Gm^1$, the second one is the inverse morphism of the morphism in (b), and the third one is an open immersion. 
                Then, the large square of the following diagram is a Cartesian square:
                \begin{equation*}
                    \begin{tikzcd}
                        X_{k(t)}(\Delta_{(0)})\ar[r]\ar[d]&X(\Delta_{(0)})\times {\Gm^1}_{,k}\ar[r, "\sim"]\ar[rd, "\pr_2"]&X(\Delta_0)\ar[r, hook]\ar[d]&X(\Delta)\ar[d, "(\pr_2)_*"]\\
                        \Spec(k(t))\ar[rr]&\ &\Gm^1\ar[r, hook]&\A^1
                    \end{tikzcd}
                \end{equation*}
                \item[(d)] A morphism $\Psi$ induces the following three Cartesian diagrams:
                \begin{equation*}
                    \begin{tikzcd}
                        X_{k(t)}(\sigma_0)\ar[r]\ar[d, hook]&X(\sigma)\ar[d, hook]\\
                        X_{k(t)}(\Delta_{(0)})\ar[r, "\Psi"]&X(\Delta)
                    \end{tikzcd}
                \end{equation*}
                \begin{equation*}
                    \begin{tikzcd}
                        \overline{O_{\sigma_0, k(t)}}\ar[r]\ar[d, hook]&\overline{O_\sigma}\ar[d, hook]\\
                        X_{k(t)}(\Delta_{(0)})\ar[r, "\Psi"]&X(\Delta)
                    \end{tikzcd}
                \end{equation*}
                \begin{equation*}
                    \begin{tikzcd}
                        O_{\sigma_0, k(t)}\ar[r]\ar[d, hook]&O_\sigma\ar[d, hook]\\
                        X_{k(t)}(\Delta_{(0)})\ar[r, "\Psi"]&X(\Delta)
                    \end{tikzcd}
                \end{equation*}
            \end{enumerate}
        \end{lemma}
        \begin{proof}
            We prove the statements from (a) to (d) in order. 
            \begin{enumerate}
                \item[(a)] Because $(\pr_1)_\RR|_{\supp(\Delta_0)}$ is injective, $\Delta_{(0)}$ is a polyhedral convex fan in $N_\RR$ from Lemma \ref{lem: injective fan}(c). 
                Because $\Delta$ is a rational convex fan and $\pr_1\colon (N\oplus\ZZ)\rightarrow N$ is a morphism of lattices, 
                $\Delta_{(0)}$ is a rational convex fan. 
                Moreover, $\Delta_{(0)}$ is strongly convex from Lemma \ref{lem: injective fan}(d). 
                \item[(b)] We can regard $\Delta_0$ as a direct product of fans $\Delta_{(0)}$ and $\{\{0\}\}$, so $(\pr_1)_*\times(\pr_2)_*$ is isomorphic. 
                The latter part of the statement follows from this identification.
                \item[(c)] The right square is a Cartesian square, and the morphism $X(\Delta_{(0)})\times {\Gm^1}_{,k}\rightarrow X(\Delta_0)$ is an isomorphism from (b). 
                Thus, the statement holds. 
                \item[(d)] Let $\theta$ denote a morphism $X_{k(t)}(\Delta_{(0)})\rightarrow X(\Delta_{(0)})\times{\Gm^1}_{,k}$ in the diagram in (c). 
                Then there exists the following Cartesian diagram:
                \begin{equation*}
                    \begin{tikzcd}
                        X_{k(t)}(\sigma_0)\ar[r]\ar[d, hook]&X(\sigma_0)\times{\Gm^1}_{,k}\ar[d, hook]\\
                        X_{k(t)}(\Delta_{(0)})\ar[r, "\theta"]&X(\Delta_{(0)})\times{\Gm^1}_{,k}
                    \end{tikzcd}
                \end{equation*}
                Thus, from (b), all small squares in the following diagram are Cartesian squares: 
                \begin{equation*}
                    \begin{tikzcd}
                        X_{k(t)}(\sigma_0)\ar[r]\ar[d, hook]&X(\sigma_0)\times{\Gm^1}_{,k}\ar[d, hook]\ar[r, "\sim"]&X(\sigma)\ar[d, hook]\ar[r]&X(\sigma)\ar[d, hook]\\
                        X_{k(t)}(\Delta_{(0)})\ar[r, "\theta"]&X(\Delta_{(0)})\times{\Gm^1}_{,k}\ar[r, "p^{-1}"]&X(\Delta_{0})\ar[r, hook]&X(\Delta)
                    \end{tikzcd}
                \end{equation*}
                where $p$ denotes  $(\pr_1)_*\times(\pr_2)_*$. 
                Thus, the first diagram in the statement is a Cartesian diagram. 
                Similarly, the second one and the third one in the statement are Cartesian diagrams. 
            \end{enumerate}
        \end{proof}
        \subsection{Toric varieties Part 2}
        Let $N$ be a lattice of finite rank, $M$ be the dual lattice of $N$, $\sigma$ be a strongly convex rational polyhedral cone in $N_\RR$, $f$ be a Laurent polynomial in $k[M]$, $H^\circ_{X(\sigma), f}$ denote a hypersurface in $T_N$ defined by $f = 0$, and $H_{X(\sigma), f}$ denote the scheme theoretic closure of $H^\circ_{X(\sigma), f}$ in $X(\sigma)$. 
        In this subsection, we consider how to determine whether $O_\sigma\subset H_{X(\sigma), f}$ or not. 
        In fact, this condition can be determined by torus invariant valuations(cf. Lemma\ref{lem: orbit and function}). 
        The method of determination itself is likely to be a well-known fact, as the proof is not particularly challenging. 
        However, the author was unable to find the literature on it. 

        A summary of this subsection follows: 
        In Definition \ref{def:val cone} and Lemma \ref{lem: property of val cone1}, we construct a cone $C^\sigma_f$ and consider its property. 
        The cone $C^\sigma_f$ is important to define the function $\val^\sigma_f$ at Definition\ref{def:val function for polynomial} and construct the refinement of $\sigma$ along $f$ at Lemma\ref{lem: refinement of the toric fan along function}. 
        Next, we show Lemma\ref{lem: orbit and function}. 
        For the proof, the refinement of $\sigma$ along $f$ plays an important role. 
        After showing Lemma\ref{lem: orbit and function}, we define two new properties of $f$ for a strongly convex rational polyhedral fan $\Delta$ at Definition \ref{def: fine toric fan for function}. 
        \vskip\baselineskip
        For $v\in\sigma\cap N$, we can define $v(f)$. 
        At first, we consider the configuation of $\{(v, v(f))\}_{v\in \sigma\cap N}$ in $(N\oplus\ZZ)_\RR$ by constructing a new cone $C^\sigma_f$. 
        \begin{definition}\label{def:val cone}
            Let $N$ be a lattice of finite rank, $M$ be the dual lattice of $N$, and $\sigma$ be a rational polyhedral convex cone in $N_\RR$. 
            Let $f\in k[M]$ be a nonzero Laurant polynomial as follows:
            \[
                f = \sum_{\omega\in M} d_\omega\chi^\omega, 
            \]
            where $d_\omega\in k$ for $\omega\in M$. 
            Let $D^\sigma_f$ be a rational polyhedral convex cone in  $(M\oplus\ZZ)_\RR$ generated by the following set:
            \[
                \{(\omega', 0)\in(M\oplus\ZZ)_\RR\mid\omega'\in\sigma^\vee\} \cup \{(\omega, 1)\in(M\oplus\ZZ)_\RR\mid d_\omega \neq 0\}
            \]
            Let $C^\sigma_f$ denote a dual cone of $D^\sigma_f$ in $(N\oplus\ZZ)_\RR$. 
            We remark that $C^\sigma_f$ is also a rational polyhedral convex cone in $(N\oplus\ZZ)_\RR$. 
        \end{definition}
        From Lemma\ref{lem: property of val cone1}(c), we can compute the configuation of $\{(v, v(f))\}_{v\in \sigma\cap N}$ in $(N\oplus\ZZ)_\RR$. 
        In fact, they are on the boundary of a cone. 
        Moreover, the cone $C^\sigma_f$ constructs the refinement of $\sigma$. 
        This refinement plays an important role in showing Lemma\ref{lem: orbit and function}. 
        \begin{lemma}\label{lem: property of val cone1}
            We keep the notation in Definition \ref{def:val cone}, the following statements follow:
            \begin{enumerate}
                \item[(a)] For any $(v, a)\in C^\sigma_f$, we have that $v\in\sigma$. 
                \item[(b)] Let $v\in\sigma$ be an element. 
                Then the following subset of $\RR$ is non-empty and has the minimum value:
                \[
                    \{a\in\RR\mid (v, a)\in C^\sigma_f\}
                \]
                From now on, let $s(v)$ denote the minimum value of the set above. 
                \item[(c)] For any $v\in\sigma\cap N$, we have $v(f) + s(v) = 0$. 
                \item[(d)] Let $\tau\preceq C^\sigma_f$ be a face. 
                We assume that $(0, 1)\notin \tau$. 
                Then for any $(v, a)\in \tau$, we have $a = s(v)$. 
                \item[(e)] For any $v\in\sigma$, there exists a face $\tau\preceq C^\sigma_f$ such that $(v, s(v))\in\tau$ and $(0, 1)\notin\tau$. 
                \item[(f)] For any $r\in\RR$ and $v\in\sigma$, we have $r\cdot s(v) = s(r\cdot v)$. 
                \item[(g)] Let $\Delta^{\sigma, f}$ denote the following set:
                \[
                    \Delta^{\sigma, f} = \{\tau\preceq C^\sigma_f\mid (0, 1) \notin\tau\}
                \]
                We remark that $\Delta^{\sigma, f}$ is a rational polyhedral fan in $N_\RR\oplus \RR$. 
                Let $\pi$ denote a first projection from $N\oplus\ZZ$ to $N$. 
                Then the following subset $\Delta^\sigma_f$ is a rational polyhedral fan in $N_\RR$ and a refinement of a fan which consists of all faces of $\sigma$:
                \[
                    \Delta^\sigma_f = \{\pi_\RR(\tau)\mid \tau\in\Delta^{\sigma, f}\}
                \]
                \item[(h)] Let $E^\sigma_f$ denote a convex cone in $(N\oplus\ZZ)_\RR$ generated by the following set:
                \[
                    \{(v, s(v))\mid v\in\sigma\}\cup\{(0, 1)\}
                \]
                Then $E^\sigma_f$ is $C^\sigma_f$. 
                \item[(i)] The cone $C^\sigma_f$ is strongly convex if and only if $\Delta^\sigma_f$ is a strongly convex fan 
                in $N_\RR$.  
                In particular, if $\sigma$ is strongly convex, $\Delta^\sigma_f$ is a strongly convex fan in $N_\RR$. 
            \end{enumerate}
        \end{lemma}
        \begin{proof}
            We prove the statements from (a) to (i) in order. 
            \begin{enumerate}
                \item[(a)] Let $(v, a)\in C^\sigma_f$ be an element. 
                For any $\omega'\in\sigma^\vee$, we have $(\omega', 0)\in D^\sigma_f$. 
                Thus, $\langle v, \omega'\rangle = \langle (v, a), (\omega', 0)\rangle\geq 0$. 
                In particular, $v\in\sigma$. 
                \item[(b)] Let $a_v\in\RR$ denote the following number:
                \[
                    a_v = \max_{d_\omega\neq 0}\{-\langle v, \omega\rangle\}
                \]
                Then we can check that $(v, a_v)\in C^\sigma_f$ and $a\geq a_v$ for any $(v, a)\in C^\sigma_f$ because of the definition of $C^\sigma_f$. 
                \item[(c)] Let $v\in\sigma\cap N$ be an element. 
                Because $v$ is a torus invariant valuation, $v(f) = \min_{d_\omega\neq 0}\{\langle v, \omega\rangle\}$. 
                Thus, $v(f) + s(v) = 0$ from the proof of (b). 
                \item[(d)] Let $(v, a)\in\tau$ be an element. 
                From (b), $a\geq s(v)$. 
                Moreover, $(v, s(v)),$ and $(0, a - s(v))\in C^\sigma_f$. 
                Because $\tau$ is a face of $C^\sigma_f$, we have that $(v, s(v))$ and $(0, a - s(v))\in\tau$. 
                Thus, from the assumption of $\tau$, we have that $a = s(v)$. 
                \item[(e)] Let $\tau\preceq C^\sigma_f$ be a face such that $(v, s(v))\in\tau^\circ$. 
                If $\tau$ contains $(0, 1)$, then there exists $\epsilon>0$ such that $(v, s(v) - \epsilon)\in\tau$ because $(v, s(v))\in\tau^\circ$. 
                However, it is a contradiction to the definition of $s(v)$ from (b). 
                Thus, we have that $(0, 1)\notin\tau$. 
                \item[(f)] From the proof of (b), we can check it easily. 
                \item[(g)] We remark that $\pi_\RR|_{\supp(\Delta^{\sigma, f})}$ is injective from (d). 
                Then, from Lemma \ref{lem: injective fan}(c), $\Delta^{\sigma, f}$ is a polyhedral convex fan. 
                A morphism $\pi$ is a lattice morphism, and $C^\sigma_f$ is a rational polyhedral fan so that $\Delta^{\sigma}_f$ is a rational polyhedral convex fan.  
                From (a) and (e), we have that $\supp(\Delta^{\sigma}_f) = \sigma$. 
                \item[(h)] We can check easily that $E^\sigma_f\subset C^\sigma_f$ from (b). 
                Let $(v, a)\in C^\sigma_f$ be an element.
                From (a), $v\in\sigma$. 
                Moreover, from (b), $a - s(v)\geq 0$. 
                Thus, $(v, a) = (v, s(v)) + (0, a - s(v))\in E^\sigma_f$. 
                \item[(i)] We remark that if $\sigma$ is strongly convex, then $D^\sigma_f$ is a full cone from its construction, so that $C^\sigma_f$ is strongly convex. 
                Because $\pi_\RR|_{\supp(\Delta^{\sigma, f})}$ is injective, the strong convexity of $\Delta^{\sigma, f}$ is equivalent to that of $\Delta^\sigma_f$ from Lemma \ref{lem: injective fan}(d). 
                Let $\tau$ denote the minimal face of $C^\sigma_f$. 
                We remark that $\tau$ is a linear subspace of $N_\RR\oplus\RR$. 
                Because $C^\sigma_f$ does not contain $(0, -1)$, $\tau$ does not contain $(0, 1)$. 
                Thus, $\tau\in \Delta^{\sigma, f}$ and the statement follows. 
            \end{enumerate}
        \end{proof}
        The following lemma is cited many times in the main argument.
        \begin{lemma}\label{lem: uniqueness of the function}
            Let $N$ be a lattice of finite rank and $\sigma$ be a rational polyhedral convex cone in $N_\RR$. 
            Let $s$ and $t$ be maps from $\sigma$ to $\RR$. 
            We assume that $s$ and $t$ satisfy the following 3 conditions:
            \begin{enumerate}
                \item These two maps are continuous. 
                \item For any $r\in\RR$ and $v\in\sigma$, $r\cdot s(v) = s(r\cdot v)$ and $r\cdot t(v) = t(r\cdot v)$. 
                \item For any $v\in\sigma\cap N$, we have that $s(v) = t(v)$. 
            \end{enumerate}
            Then we have $s = t$.
        \end{lemma}
        \begin{proof}
            From the conditions (2) and (3), $s(v) = t(v)$ for any $v\in\sigma\cap N_\Q$. 
            Because $\sigma$ is a rational polyhedral convex cone, $\sigma\cap N_\Q$ is a dense subset of $\sigma$. 
            Thus, $s = t$ because $\RR$ is a Hausdorff space and continuity of $s$ and $t$. 
        \end{proof}
        For any $v\in \sigma$, we define the virtual valuation $\val^\sigma_f(v)$. 
        \begin{definition}\label{def:val function for polynomial}
            We keep the notation in Definition \ref{def:val cone}. 
            \begin{enumerate}
                \item Let $\val^\sigma_f$ denote a function from $\sigma$ to $\RR$ defined as $-s$. 
                \item Let $\Omega^\sigma_f$ be the following subset of $M_\Q$: 
                \[
                    \Omega^\sigma_f = \{\omega\in M_\Q\mid \val^\sigma_f(v)\leq \langle v, \omega\rangle\quad\forall v\in\sigma\}
                \]
                \item For $\omega\in\Omega^\sigma_f$, let $C^\sigma_f(\omega)$ denote the following subset of $\sigma$:
                \[
                    C^\sigma_f(\omega) = \{v\in\sigma\mid \val^\sigma_f(v) = \langle v, \omega\rangle\}
                \]
            \end{enumerate}
        \end{definition}
        From the following lemma, $\{C^\sigma_f(\omega)\}_{\omega\in\Omega^\sigma_f}$ is a rational polyhedral convex fan in $N_\RR$ and this is a refinement of the fan which consists of all faces of $\sigma$. 
        \begin{lemma}\label{lem: property of val function}
            We keep the notation in Definition \ref{def:val cone} and Definition \ref{def:val function for polynomial}, the following statements follow:
            \begin{enumerate}
                \item[(a)] For any $\omega\in\Omega^\sigma_f$, there exists $\tau\in\Delta^{\sigma, f}$ such that $C^\sigma_f(\omega) = \pi_\RR(\tau)$. 
                In particular, $C^\sigma_f(\omega)$ is a rational polyhedral convex cone in $N_\RR$. 
                \item[(b)] Conversely, for any $\tau\in\Delta^{\sigma, f}$, there exists $\omega\in\Omega^\sigma_f$ such that $\pi_\RR(\tau) = C^\sigma_f(\omega)$. 
                In particular, $\{C^\sigma_f(\omega)\}_{\omega\in\Omega^\sigma_f}$ is a rational polyhedral convex fan in $N_\RR$ and this is a refinement of the fan which consists of all faces of $\sigma$. 
                \item[(c)] The function $\val^\sigma_f$ is a continuous map from $\sigma$ to $\RR$. 
                In particular, $\val^\sigma_f$ is a map that satisfies conditions (1) and (2) in Lemma \ref{lem: uniqueness of the function}. 
                Moreover, $\val^\sigma_f(v) = v(f)$ for any $v\in \sigma\cap N$. 
                \item[(d)] For any $\sigma\in\Delta$, $\val^\sigma_f$ is an upper convex function. 
                \item[(e)] Let $\tau$ be a face of $\sigma$.  
                Then $val^\sigma_f|_\tau = val^\tau_f$. 
                \item[(f)] Let $\tau\preceq\sigma$ be a face and $\omega\in\Omega^\sigma_f$ be an element. 
                Then $\omega\in\Omega^\tau_f$ and $C^\sigma_f(\omega)\cap\tau = C^\tau_f(\omega)$. 
                In particular, $C^\tau_f(\omega)\preceq C^\sigma_f(\omega)$. 
                \item[(g)] For any $\omega\in\Omega^\sigma_f$, there exists $\omega_0\in M$ such that $\val^\sigma_f(v) = \langle v, \omega_0\rangle$ for any $v\in C^\sigma_f(\omega)$. 
            \end{enumerate}
        \end{lemma}
        \begin{proof}
            We prove the statements from (a) to (g) in order. 
            \begin{enumerate}
                \item[(a)] Let  $\omega\in\Omega^\sigma_f$ be an element. 
                From the definition of $\Omega^\sigma_f$ and Lemma \ref{lem: property of val cone1}(h), $(\omega, 1)\in D^\sigma_f$. 
                Let $\tau$ denote $C^\sigma_f\cap (\omega, 1)^\perp$. 
                We remark that $\tau\preceq C^\sigma_f$ and $(0, 1)\notin\tau$. 
                Thus $C^\sigma_f(\omega) = \pi_\RR(\tau)$.  
                Indeed, for any $v\in\sigma$, the equation $\langle (v, s(v)), (\omega, 1)\rangle = 0$ is equivalent to the equation $\val^\sigma_f(v) = \langle v, \omega\rangle$ because of the definition of $\val^\sigma_f$. 
                \item[(b)] Let $\tau\in\Delta^{\sigma, f}$ be a cone. 
                Then there exists $(\omega, b)\in D^\sigma_f$ such that $\omega\in M_\Q$, $b\in\Q_{\geq 0}$ and $\tau = C^\sigma_f\cap (\omega, b)^\perp$. 
                If $b = 0$, then $(0, 1)\in\tau$. 
                However, it is a contradiction to the definition of $\Delta^{\sigma, f}$. 
                Thus, we may assume that $b = 1$. 
                From the definition of $\Omega^\sigma_f$ and Lemma \ref{lem: property of val cone1} (h), we have $\omega\in\Omega^\sigma_f$. 
                Thus, from the same argument in the proof of (a), 
                $\pi_\RR(\tau) = C^\sigma_f(\omega)$. 
                Therefore, from Lemma \ref{lem: property of val cone1}(g), $\Delta^\sigma_f = \{C^\sigma_f(\omega)\}_{\omega\in\Omega^\sigma_f}$ and $\{C^\sigma_f(\omega)\}_{\omega\in\Omega^\sigma_f}$ is a rational polyhedral convex fan which is a refinement of $\sigma$. 
                \item[(c)] From (b), the following equation holds:
                \[
                    \sigma = \bigcup_{\omega\in\Omega^\sigma_f} C^\sigma_f(\omega)
                \]
                For each $\omega\in\Omega^\sigma_f$, $\val^\sigma_f$ is a linear map on $C^\sigma_f(\omega)$. 
                In addition to this, $C^\sigma_f(\omega)$ is a closed subset of $\sigma$. 
                Thus, $\val^\sigma_f$ is a continuous map. 
                From Lemma \ref{lem: property of val cone1}(f), $\val^\sigma_f$ satisfies conditions (2) in Lemma \ref{lem: uniqueness of the function}. 
                From \ref{lem: property of val cone1}(c), the latter part of the statement follows. 
                \item[(d)] From the definition of $\val^\sigma_f$, it is enough to show that $s$ is a lower convex function. 
                From the convexity of  $C^\sigma_f$ and \ref{lem: property of val cone1}(b), $s$ is a lower convex function. 
                \item[(e)] From (c), both functions satisfy conditions (1) and (2) in Lemma \ref{lem: uniqueness of the function}. 
                Thus, it is enough to show that $val^\sigma_f(v) = val^\tau_f(v)$ for any $v\in\tau\cap N$. 
                It follows from Lemma \ref{lem: property of val cone1} (c). 
                \item[(f)] From (e), $\Omega^\sigma_f\subset \Omega^\tau_f$ and $C^\sigma_f(\omega)\cap\tau = C^\tau_f(\omega)$ for any $\omega\in \Omega^\sigma_f$. 
                Let $\omega_\tau\in\sigma^\vee$ be an element such that $\tau = \sigma\cap\omega_\tau^\perp$. 
                Then $\omega_\tau\in (C^\sigma_f(\omega))^\vee$ and $C^\sigma_f(\omega)\cap \omega_\tau^\perp \preceq C^\sigma_f(\omega)$. 
                Thus, $C^\sigma_f(\omega)\cap \omega_\tau^\perp = C^\sigma_f(\omega)\cap \tau = C^\tau_f(\omega)$ is a face of $C^\sigma_f(\omega)$. 
                \item[(g)] Let $v\in\langle C^\sigma_f(\omega)\rangle\cap N$ be an element. 
                Then there exists $v_1$ and $v_2\in C^\sigma_f(\omega)\cap N$ such that $v = v_1 - v_2$ because $C^\sigma_f(\omega)$ is a rational polyhedral cone. 
                From the definition of $C^\sigma_f(\omega)$ and (c), the following equation holds: 
                \begin{align*}
                    \langle v, \omega\rangle &= \langle v_1, \omega\rangle- \langle v_2, \omega\rangle\\
                    &= \val^\sigma_f(v_1) - \val^\sigma_f(v_2)\\
                    &= v_1(f) - v_2(f)\in\ZZ
                \end{align*} 
                Let $M'$ be the dual lattice of $\langle C^\sigma_f(\omega)\rangle\cap N$ and $\iota^*$ denote a group morphism $M\rightarrow M'$ induced by an inclusion $\iota\colon \langle C^\sigma_f(\omega)\rangle\cap N\hookrightarrow N$. 
                Then we can identify $\langle -,\omega\rangle|_{\langle C^\sigma_f(\omega)\rangle\cap N}$ as an element in $M'$.
                We remark that $\iota^*$ is surjective because $N/(\langle C^\sigma_f(\omega)\rangle\cap N)$ is torsion free. 
                Thus, there exists $\omega_0\in M$ such that $\iota^*(\omega_0) = \langle -,\omega\rangle|_{\langle C^\sigma_f(\omega)\rangle\cap N}$. 
                Hence, the following equation holds for any $v'\in C^\sigma_f(\omega)\cap N$:
                \begin{align*}
                    \val^\sigma_f(v') &=\langle v', \omega\rangle\\
                    &= \langle v', \iota^*(\omega_0)\rangle\\
                    &= \langle v', \omega_0\rangle
                \end{align*}
                Thus, from Lemma \ref{lem: uniqueness of the function}, $\val^\sigma_f(v') = \langle v', \omega_0\rangle$ for any $v'\in C^\sigma_f(\omega)$. 
            \end{enumerate}
        \end{proof}
        For a strongly convex rational polyhedral fan $\Delta$, we can construct a refinement of $\Delta$ along $f$.  
        \begin{lemma}\label{lem: refinement of the toric fan along function}
            Let $N$ be a lattice of finite rank, $\Delta$ be a strongly convex rational polyhedral fan in $N_\RR$, and $f\in k[M]$ be a polynomial. 
            Then the following set $\Delta_f$ is a strongly convex rational polyhedral fan and a refinement of $\Delta$: 
            \[
                \Delta_f = \bigcup_{\sigma\in\Delta}\{C^\sigma_f(\omega)\}_{\omega\in\Omega^\sigma_f}
            \]
        \end{lemma}
        \begin{proof}
            For any $\sigma\in\Delta$, $\{C^\sigma_f(\omega)\}_{\omega\in\Omega^\sigma_f}$ is a strongly convex polyhedral fan in $N_\RR$ from Lemma \ref{lem: property of val cone1} (i) and Lemma \ref{lem: property of val function} (b). 
            Thus, all faces of a cone in $\Delta_f$ are contained in $\Delta_f$. 
            Let $\sigma$ and $\tau\in\Delta$ be cones, $\omega_1\in\Omega^\sigma_f$ and $\omega_2\in \Omega^\tau_f$ be elements. 
            We show that $C^\sigma_f(\omega_1)\cap C^\tau_f(\omega_2)$ is a common face of $C^\sigma_f(\omega_1)$ and $C^\tau_f(\omega_2)$. 
            Indeed, $C^\sigma_f(\omega_1)\cap \tau = C^{\sigma\cap\tau}_f(\omega_1)$,  $C^\tau_f(\omega_2)\cap \sigma = C^{\sigma\cap\tau}_f(\omega_2)$, and $\{\omega_1, \omega_2\}\subset \Omega^{\sigma\cap\tau}_f$ from Lemma \ref{lem: property of val function}(f). 
            Thus, $C^\sigma_f(\omega_1)\cap C^\tau_f(\omega_2) = C^{\sigma\cap\tau}_f(\omega_1)\cap C^{\sigma\cap\tau}_f(\omega_2)$. 
            Moreover, the cone $C^{\sigma\cap\tau}_f(\omega_1)\cap C^{\sigma\cap\tau}_f(\omega_2)$ is a common face of $C^{\sigma\cap\tau}_f(\omega_1)$ and $C^{\sigma\cap\tau}_f(\omega_2)$ from Lemma \ref{lem: property of val function}(b). 
            Thus, from Lemma \ref{lem: property of val function}(f), this cone is a common face of $C^\sigma_f(\omega_1)$ and $C^\tau_f(\omega_2)$. 
        \end{proof}
        From the definition of $\val^\sigma_f$, we can check that $\val^\sigma_f$ is a piecewise linear function on $\sigma$. 
        From now on, let's consider the case when $\val^\sigma_f$ is a linear map on $\sigma$. 
        The following lemma reveals that the linearity of $\val^\sigma_f$ is also preserved for the factors of $f$.
        \begin{lemma}\label{lem: val-func-of-multiple-poly}
            We keep the notation in Definition \ref{def:val cone} and Definition \ref{def:val function for polynomial}. 
            Let $f, f_1, \ldots, f_r\in k[M]$ be nonzero Laurant polynomials. 
            We assume that $f = f_1\cdot f_2\cdot \cdots\cdot f_r$ and $\sigma$ is strongly convex.
            Then the following statements are equivalent:
            \begin{enumerate}
                \item[(a)] There exists $\omega\in M$ such that $v(f\cdot\chi^\omega) = 0$ for any $v\in\sigma\cap N$. 
                \item[(b)] For any $1\leq i\leq r$, there exists $\omega_i\in M$ such that $v(f_i\cdot\chi^{\omega_i}) = 0$ for any $v\in\sigma\cap N$. 
            \end{enumerate}
        \end{lemma}
        \begin{proof}
            One of the directions is obvious, so we show the statement of (b) holds when the statement of (a) holds. 
            If it is necessary, we replace $f$ with $\chi^\omega f$ and $f_1$ with $\chi^\omega f_1$ and we may assume that $\omega = 0$.
            Then $val^\sigma_f \equiv 0$ because of Lemma \ref{lem: uniqueness of the function} and Lemma \ref{lem: property of val function}(c). 
            Let $v_0\in \sigma^\circ \cap N$ be an element. 
            Then from Lemma \ref{lem: property of val function}(b), for any $1\leq i\leq r$, there exists $\omega_i\in \Omega^\sigma_{f_i}$ such that $v_0\in C^\sigma_{f_i}(\omega_i)$. 
            From Lemma \ref{lem: uniqueness of the function} and Lemma \ref{lem: property of val function}(c), $val^\sigma_f = \sum_{1\leq i\leq r} val^\sigma_{f_i}$. 
            Then the following equation and inequality holds:
            \begin{align*}
                0 =\sum_{1\leq i\leq r}\val^\sigma_{f_i}(v)&\leq \langle v, \omega'\rangle\quad(v\in\sigma)\\
                0 =\sum_{1\leq i\leq r}\val^\sigma_{f_i}(v_0)&= \langle v_0, \omega'\rangle\\, 
            \end{align*}
            where $\omega'$ denotes $\sum_{1\leq i\leq r}\omega_i\in M_\Q$. 
            These equation and inequality show that $\omega'\in\sigma^\vee$ and $\omega'\in\sigma^\perp$ because $v_0\in
            \sigma^\circ$. 
            Thus, $\langle v, \omega'\rangle = 0$ for any $v\in\sigma$. 
            
            We show that $\bigcap_{1\leq i\leq r}C^\sigma_{f_i}(\omega_i) = \sigma$ by contradiction. 
            We assume that $w\in\sigma\setminus\bigcap_{1\leq i\leq r}C^\sigma_{f_i}(\omega_i)$ exists. 
            Then from the definition of $C^\sigma_{f_i}(\omega_i)$, $\sum_{1\leq i\leq r}\val^\sigma_{f_i}(w)<\langle w, \omega'\rangle = 0$. 
            However, $\sum_{1\leq i\leq r} val^\sigma_{f_i}(w) = \val^\sigma_f(w) = 0$, so it is a contradiction. 
            
            Thus, for any $1\leq i\leq r$, $C^\sigma_{f_i}(\omega_i) = \sigma$. 
            Hence, from Lemma \ref{lem: property of val function}(g), for any $1\leq i\leq r$, there exists $\omega'_i\in M$ such that $\val^\sigma_{f_i}(v) = \langle v, \omega'_i\rangle$ for any $v\in\sigma$. 
            In particular, $v(f_i\cdot\chi^{-\omega'_i}) = 0$ for any $v\in \sigma\cap N$. 
        \end{proof}
        The problem that was stated at the beginning of the subsection can be characterized by the linearity of $\val^\sigma_f$ on $\sigma$ from Lemma\ref{lem: orbit and function}. 
        \begin{lemma}\label{lem: orbit and function}
            We keep the notation in Definition \ref{def:val cone}.  
            We assume that $\sigma$ is strongly convex. 
            Let $H_0$ be a closed subscheme of $T_N$ defined by $f = 0$ and $H$ be a scheme theoretic closure of $H_0$ in $X(\sigma)$. 
            Then the following statements are equivalent:
            \begin{enumerate}
                \item[(a)] A closed subscheme $H$ does not contain $O_\sigma$
                \item[(b)] There exists $\omega\in M$ such that $v(f\cdot\chi^\omega) = 0$ for any $v\in\sigma\cap N$.
            \end{enumerate}
            In this case, the ideal of $\Gamma(X(\sigma), \OO_{X(\sigma)})$ associated with a closed subscheme $H$ of $X(\sigma)$ is generated by $f\cdot\chi^\omega$. 
        \end{lemma}
        \begin{proof}
            Let $k[\sigma^\vee\cap M]$ denote $\Gamma(X(\sigma), \OO_{X(\Delta)})$, $I_f\subset k[M]$ denote an ideal generated by $f$, and $\mathfrak{p}_\sigma\subset k[\sigma^\vee\cap M]$ denote the prime ideal associated to $O_\sigma\subset X(\sigma)$. 
            
            First, we assume that the condition of (a) holds. 
            Then there exists $g\in I_f\cap k[\sigma^\vee\cap M]$ such that $g\notin\mathfrak{p}_\sigma$. 
            In particular, for  $v_0\in\sigma^\circ\cap N$, we have $v_0(g) = 0$. 
            Moreover, for any $v\in\sigma\cap N$, $v(g)\geq 0$ because $g\in k[\sigma^\vee\cap M]$. 
            Therefore, $\val^\sigma_g \equiv 0$ from Lemma \ref{lem: property of val function}(d). 
            On the other hand, there exists $h\in k[M]$ such that $g = f h$ because $g\in I_f$. 
            Thus, from Lemma \ref{lem: val-func-of-multiple-poly}, the condition of (b) holds.

            Next, we assume that the condition of (b) holds. 
            From the assumption, $f\cdot\chi^\omega\in I_f\cap k[\sigma^\vee\cap M]$. 
            Thus,  $f\cdot\chi^\omega\notin\mathfrak{p}_\sigma$ and the condition of (a) holds. 
            
            In this case, for any $g\in I_f\cap k[\sigma^\vee\cap M]$, there exists $h\in k[M]$ such that $g = (f\cdot\chi^\omega)h$. 
            Because, $v(f\cdot \chi^\omega) = 0$ and $v(g)\geq 0$ for any $v\in\sigma\cap N$. 
            Hence, for any $v\in \sigma\cap N$, $v(h)\geq 0$ from the additivity of valuations. 
            Thus, $h\in k[\sigma^\vee\cap M]$ and this shows that $I_f\cap k[\sigma^\vee\cap M]$ is generated by $f\cdot\chi^\omega$. 
        \end{proof}
        After showing Lemma\ref{lem: orbit and function}, we define two properties of $f$ for $\Delta$. 
        These two properties are important for hypersurfaces of mock toric varieties. 
        \begin{definition}\label{def: fine toric fan for function}
            Let $N$ be a lattice of a finite rank, $\Delta$ be a strongly convex rational polyhedral fan in $N_\RR$, $f\in k[M]$ be a Laurent polynomial, $H^\circ_{X(\Delta), f}$ denote the closed subscheme of $T_N$ defined by $f = 0$ and $H_{X(\Delta), f}$ denote the scheme theoretic closure of $H^\circ_{X(\Delta), f}$ in $X(\Delta)$. 
            If for any $\sigma\in\Delta$, $H_{X(\Delta),f}$ does not contain $O_\sigma$, then we call that $f$ is fine for $\Delta$. 
        \end{definition}
        \begin{definition}\label{def: toric nondegenerate}
            We keep the notation in Definition \ref{def: fine toric fan for function}, and we assume that $f$ is fine for $\Delta$. 
            If $H_{X(\Delta), f}\cap O_\sigma$ is smooth over $k$ for any $\sigma\in\Delta$, we call that $f$ is non-degenerate for $\Delta$. 
        \end{definition}

        For the notation above, the fineness or non-degeneracy of $f$ is inherited for the closure of the orbit.
        
        \begin{lemma}\label{lem: heredity of fineness for orbit}
            We keep the notation in Definition \ref{def: fine toric fan for function}, and we assume that $f$ is fine for $\Delta$. 
            Let $\sigma\in\Delta$ be a cone. 
            Then there exists $\omega\in M$ such that $v(\chi^\omega f) = 0$ for any $v\in\sigma\cap N$ from Lemma \ref{lem: orbit and function}.
            Let $p^\sigma\colon\Gamma(X(\sigma), \mathscr{O}_{X(\Delta)})\rightarrow \Gamma(O_\sigma, \OO_{O_\sigma})$ be a quotient map associated with the closed immersion $O_\sigma\hookrightarrow X(\sigma)$ and let  $f^\sigma$ denote $p^\sigma(\chi^\omega f)$. 
            Then the following statements follow:
            \begin{enumerate}
                \item[(a)] As a closed subscheme of $O_\sigma$, $H_{X(\Delta), f}\cap O_\sigma = H^\circ_{\overline{O_\sigma}, f^\sigma}$. 
                \item[(b)] We use the notation in Proposition \ref{prop:orbit-toric}, $f^\sigma$ is fine for $\Delta[\sigma]$. 
                \item[(c)] As a closed subscheme of $\overline{O_\sigma}$, $H_{X(\Delta), f}\cap \overline{O_\sigma} = H_{\overline{O_\sigma}, f^\sigma}$. 
                \item[(d)] If $f$ is non-degenerate for $\Delta$, then $f^\sigma$ is non-degenerate for $\Delta[\sigma]$. 
            \end{enumerate}
        \end{lemma}
        \begin{proof}
            We show the statements from (a) to (d) in order. 
            \begin{enumerate}
                \item[(a)] From Lemma \ref{lem: orbit and function}, $\chi^\omega f$ is a generator of an ideal of $k[X(\sigma)]$ associated with a closed subscheme $H_{X(\sigma), f}$ of $X(\sigma)$. 
                Thus, $f^\sigma$ is a generator of an ideal of $k[O_\sigma]$ associated with a closed subscheme $H_{X(\sigma), f}\cap O_\sigma$ of $O_\sigma$. 
                \item[(b)] From the definition of $H_{\overline{O_\sigma}, f^\sigma}$ and (a), $H_{\overline{O_\sigma}, f^\sigma}\subset H_{X(\Delta), f}\cap \overline{O_\sigma}$. 
                Thus, $H_{\overline{O_\sigma}, f^\sigma}$ does not contain $O_\tau$ for any $\sigma\subset \tau\in\Delta$ because $f$ is fine for $\Delta$. 
                \item[(c)] The closed subscheme $\overline{O_\sigma}\subset X(\Delta)$ is covered by $\{X(\tau)\}_{\sigma\subset\tau\in\Delta}$. 
                Thus, for each such $\tau$, it is enough to show that the scheme theoretic closure of $H^\circ_{\overline{O_\sigma}, f^\sigma}$ in $\overline{O_\sigma}\cap X(\tau)$ is equal to $H_{X(\Delta), f}\cap\overline{O_\sigma}\cap X(\tau)$ from Lemma \ref{lem:covering-image}. 
                For such $\tau$, there exists the following commutative diagram: 
                \begin{equation*}
                    \begin{tikzcd} 
                        \Gamma(X(\tau), \OO_{X(\Delta)})\ar[d, hook]\ar[r, "p^\sigma"]& \Gamma(\overline{O_\sigma}\cap X(\tau), \OO_{\overline{O_\sigma}})\ar[d, hook]\\
                        \Gamma(X(\sigma), \OO_{X(\Delta)})\ar[r, "p^\sigma"] & \Gamma(O_\sigma, \OO_{\overline{O_\sigma}})
                    \end{tikzcd}
                \end{equation*}
                , where horizontal maps are induced by two closed immersions $O_\sigma\hookrightarrow X(\sigma)$ and $\overline{O_\sigma}\cap X(\tau)\hookrightarrow X(\tau)$. 
                From Lemma \ref{lem: orbit and function}, there exists $\omega'\in M$ such that $w(\chi^{\omega'}f) = 0$ for any $w\in \tau\cap N$ because $f$ is fine for $\Delta$. 
                Because $\sigma\subset\tau$, $v(\chi^{\omega-\omega'}) = 0$ for any $v\in\sigma\cap N$. 
                Thus, $\chi^{\omega-\omega'} \in \Gamma(X(\sigma), \OO_{X(\Delta)})^*$. 
                Let $\eta$ denote $p^\sigma(\chi^{\omega-\omega'})$ and $f^\tau$ denote $p^\sigma(\chi^{\omega'}f)$. 
                Then $f^\tau\in\Gamma(\overline{O_\sigma}\cap X(\tau), \OO_{\overline{O_\sigma}})$ and $f^\sigma = \eta f^\tau$. 
                In particular, $\eta\in\Gamma(O_\sigma, \OO_{\overline{O_\sigma}})^*$. 
                Thus, $H^\circ_{\overline{O_\sigma}, f^\sigma} = H^\circ_{\overline{O_\sigma}, f^\tau}$. 
                
                On the other hand, the ideal of $\Gamma(\overline{O_\sigma}\cap X(\tau), \OO_{\overline{O_\sigma}})$ associated with a closed subscheme $H_{X(\Delta), f}\cap\overline{O_\sigma}\cap X(\tau)\subset \overline{O_\sigma}\cap X(\tau)$ is generated by $f^\tau$ because the ideal of $\Gamma(X(\tau), \OO_{X(\Delta)})$ associated with a closed subscheme $H_{X(\Delta), f}\cap X(\tau)\subset X(\tau)$ is generated by $\chi^{\omega'}f$. 
                Let $\tau'$ denote $\pi^\sigma_\RR(\tau)$ in the notation of Proposition \ref{prop:orbit-toric}. 
                We remark that $X(\tau') = \overline{O_\sigma}\cap X(\tau)$. 
                If $\val^{\tau'}_{f^\tau}(w')>0$ for some $w'\in {\tau'}^\circ\cap N[\sigma]$, then $H_{X(\Delta), f}\cap\overline{O_\sigma}\cap X(\tau)$ should contain $O_\tau$. 
                However, it is a contradiction to the assumption that $f$ is fine for $\Delta$. 
                Thus, $\val^{\tau'}_{f^\tau}\equiv 0$, so that the ideal of $\Gamma(\overline{O_\sigma}\cap X(\tau), \OO_{\overline{O_\sigma}})$ associated with the scheme theoretic closure of $H^\circ_{\overline{O_\sigma}, f^\tau}$ in $\overline{O_\sigma}\cap X(\tau)$ is generated by $f^\tau$ too from Lemma \ref{lem: orbit and function}. 
                Therefore, the scheme theoretic closure of $H^\circ_{\overline{O_\sigma}, f^\sigma}$ in $\overline{O_\sigma}\cap X(\tau)$ is equal to $H_{X(\Delta), f}\cap\overline{O_\sigma}\cap X(\tau)$ as closed subschemes of $\overline{O_\sigma}\cap X(\tau)$. 
                \item[(d)] We use the notation in the proof of (c), and we have that $H_{\overline{O_\sigma}, f^\sigma}\cap O_{\tau'} = H_{X(\Delta), f}\cap \overline{O_\sigma}\cap O_{\tau'} = H_{X(\Delta), f}\cap O_\tau$. 
                Thus, from the assumption, $H_{\overline{O_\sigma}, f^\sigma}\cap O_{\tau'}$ is smooth over $k$. 
            \end{enumerate}
        \end{proof}
        From the following lemma, we can check that the fineness or non-degeneracy of $f$ is inherited for the toric resolution.
        \begin{lemma}\label{lem: heredity of fineness for dominant}
            Let $N$ and $N'$ be lattices of finite rank, $\alpha\colon N'\rightarrow N$ be a surjective morphism, $\Delta$ and $\Delta'$ be strongly convex rational polyhedral fans in $N_\RR$ and $N'_\RR$ respectively. 
            We assume that $\alpha$ is compatible with the fans $\Delta'$ and $\Delta$. 
            Let $f\in k[M]$ be a polynomial and let $\alpha^*$ denote a ring morphism $k[M]\rightarrow k[M']$ induced by $\alpha$. 
            We assume that $f$ is fine for $\Delta$. 
            Then the following statements follow:
            \begin{enumerate}
                \item[(a)] The polynomial $\alpha^*(f)$ is fine for $\Delta'$. 
                \item[(b)] As a closed subscheme of $X(\Delta')$, $H_{X(\Delta), f}\times_{X(\Delta)}X(\Delta') = H_{X(\Delta'), \alpha^*(f)}$. 
                \item[(c)] If $f$ is non-degenerate for $\Delta$, $\alpha^*(f)$ is non-degenerate for $\Delta'$. 
            \end{enumerate}
        \end{lemma}
        \begin{proof}
            We show the statements from (a) to (c) in order. 
            \begin{enumerate}
                \item[(a)] Let $\tau\in\Delta'$ and $\sigma\in\Delta$ be cones such that $\alpha_\RR(\tau)\subset \sigma$. 
                Then there exists $\omega\in M$ such that $v(\chi^\omega f) = 0$ for any $v\in\sigma\cap N$ from Lemma \ref{lem: orbit and function}. 
                We remark that for any $w\in \tau\cap N'$, $w(\alpha^*(\chi^\omega f)) = \alpha(w)(\chi^\omega f)$.  
                Thus, $w(\chi^{\alpha^*(\omega)}\alpha^*(f)) = 0$ for any $w\in\tau\cap N'$ because $\alpha(w)\in \sigma\cap N$. 
                Therefore, from Lemma \ref{lem: orbit and function}, $\alpha^*(f)$ is fine for $\Delta'$. 
                \item[(b)] We use the notation in the proof of (a), and the ideal of $k[X(\tau)]$ associated with the closed subscheme $H_{X(\Delta'), \alpha^*(f)}\cap X(\tau)$ of $X(\tau)$ is generated by $\chi^{\alpha^*(\omega)}\alpha^*(f)$ from Lemma \ref{lem: orbit and function}. 
                Thus, the following diagram is a Cartesian diagram:
                \begin{equation*}
                    \begin{tikzcd} 
                        H_{X(\Delta'), \alpha^*(f)}\cap X(\tau)\ar[d, "\alpha_*"]\ar[r, hook]& X(\tau)\ar[d, "\alpha_*"]\\
                        H_{X(\Delta), f}\cap X(\sigma)\ar[r, hook] & X(\sigma)
                    \end{tikzcd}
                \end{equation*}
                Then the statement follows from Lemma \ref{lem:covering-image}. 
                \item[(c)] We use the notation in the proof of (a).
                There exists $\sigma\in\Delta$ such that $\alpha_\RR(\tau)\cap\sigma^\circ \neq\emptyset$. 
                We remark that $\alpha_*(O_\tau) = O_\sigma$ because $\alpha$ is surjective. 
                Thus, there exists the following Cartesian diagram from the diagram in the proof of (b): 
                \begin{equation*}
                    \begin{tikzcd} 
                        H_{X(\Delta'), \alpha^*(f)}\cap O_\tau\ar[d, "\alpha_*"]\ar[r, hook]& O_\tau\ar[d, "\alpha_*"]\\
                        H_{X(\Delta), f}\cap O_\sigma\ar[r, hook] & O_\sigma
                    \end{tikzcd}
                \end{equation*}
                Because $\alpha_*\colon O_\tau\rightarrow O_\sigma$ is a trivial algebraic torus fibration, the smoothness of $H_{X(\Delta'), \alpha^*(f)}\cap O_\tau$ and the smoothness of $H_{X(\Delta), f}\cap O_\sigma$ are equivalent. 
            \end{enumerate}
        \end{proof}
        In general, $H_{X(\Delta), f}$ contains some torus orbits of $X(\Delta)$. 
        The following lemma indicates that a toric resolution $X(\Delta_f)\rightarrow X(\Delta)$ removes this obstacle.
        \begin{lemma}\label{lem:example of toric fine fan}
            We use the notation in Lemma \ref{lem: refinement of the toric fan along function}. 
            Let $\Delta'$ be a refinement of $\Delta_f$. 
            Then $f$ is fine for $\Delta'$. 
            Moreover, if $f$ is non-degenerate for $\Delta_f$, $f$ is non-degenerate for $\Delta'$ too. 
        \end{lemma}
        \begin{proof}
            From Lemma \ref{lem: heredity of fineness for dominant}(a) and (c), it is enough to show that $f$ is fine for $\Delta_f$. 
            Let $\tau\in\Delta_f$ be a cone. 
            Then there exists $\sigma\in\Delta$ and $\omega'\in \Omega^\sigma_f$ such that $\tau = C^\sigma_f(\omega')$.  
            From Lemma \ref{lem: property of val function}(g), there exists $\omega\in M$ such that $\val^\sigma_f(v) = \langle v, \omega\rangle$ for any $v\in\tau$. 
            In particular, $v(\chi^{-\omega}f) = 0$ for any $v\in\tau\cap N$. 
            Thus, from Lemma \ref{lem: orbit and function}, $H_{X(\tau), f}$ does not contain $O_\tau$. 
            Therefore, $f$ is fine for $\Delta_f$. 
        \end{proof}
        \subsection{Toric varieties Part. 3}
        In Definition \ref{def: toric nondegenerate}, the non-degeneracy of $f$ is determined by orbit-wise. 
        In this subsection, we connect the "smoothness" of $H_{X(\Delta), f}$ and this non-degeneracy of $f$. 
        We remark that this "smoothness" does not indicate true smoothness.
        \begin{definition}
            Let $S$ be a monoid. 
            If there exists a lattice $N$ of finite rank and a strongly convex rational polyhedral cone $\sigma$ in $N_\RR$ such that $S$ is isomorphic to $\sigma^\vee\cap M$, where $M$ is the dual lattice, then we call that $S$ is a \textbf{toric monoid}. 
            In this situation, let $k[S]$ denote a $k$-algebra generated by $S$. 
            For $\omega\in S$, let $\chi^\omega$ denote an element of $k[S]$ associated with $\omega$. 
        \end{definition}
        The "smoothness" mentioned at the beginning of this subsection is explained in the following lemma. 
        \begin{lemma}\label{lem: for 7-1}
            Let $N$ denote a lattice of finite rank, $\sigma$ be a strongly convex rational polyhedral cone in $N_\RR$, $f\in k[M]$ be a Laurent polynomial, and $V$ be a non-empty open subset of $X(\sigma)$. 
            We assume that $H_{X(\sigma), f}$ does not contain $O_\sigma$ and $H_{X(\sigma), f}\cap O_\sigma\cap V$ is smooth over $k$. 
            Then there exist 
            \begin{itemize}
                \item A toric monoid $S$
                \item A morphism $p\colon X(\sigma)\rightarrow \Spec(k[S])$ over $\Spec(k)$
                \item A open subset $U$ of $V$
            \end{itemize}
            such that
            \begin{enumerate}
                \item[(1)] $H_{X(\sigma), f}\cap O_\sigma\cap V\subset U$.
                \item[(2)] The restriction $p|_{H_{X(\sigma), f}\cap U}\colon H_{X(\sigma), f}\cap U\rightarrow \Spec(k[S])$ is smooth. 
            \end{enumerate}
        \end{lemma}
        \begin{proof}
            From Lemma\ref{lem: orbit and function}, there exists $\omega\in M$ such that $\chi^\omega f\in k[\sigma^\vee\cap M]$ and $H_{X(\sigma), f} = H_{X(\sigma), \chi^\omega f}$. 
            Thus, we may replace $f$ with $\chi^\omega f$ and assume that $f\in k[\sigma^\vee\cap M]$. 
            There exists a sublattice $N_0$ of $N$ such that $(\langle\sigma\rangle\cap N)\oplus N_0 = N$. 
            Let $\pi$ denote the quotient morphism $N\rightarrow N/N_0$ and $s$ denote a section of $\pi$ such that $s\circ\pi(v) = v$ for any $v\in\langle\sigma\rangle\cap N$. 
            Let $\sigma'$ denote $\pi_\RR(\sigma)$. 
            Then from Lemma \ref{lem:torus-fibration}, $\sigma'$ is a full and strongly convex rational polyhedral cone in $(N/N_0)_\RR$ and the toric morphism $\pi_*\colon X(\sigma)\rightarrow X(\sigma')$ is a trivial algebraic torus fibration. 
            Because $\sigma'$ is a full cone in $(N/N_0)_\RR$, a closed orbit $O_{\sigma'}\subset X(\sigma')$ is a closed point of $X(\sigma')$. 
            Let $x$ denote the closed point $O_{\sigma'}\in X(\sigma')$. 
            Then $(\pi_*)^{-1}(x) = O_\sigma$. 
            In particular, $(H_{X(\sigma), f})_{\kappa(x)}\cap V_{\kappa(x)} = H_{X(\sigma), f}\cap O_\sigma\cap V$. 
            Moreover, $H_{X(\sigma), f}\cap O_\sigma\cap V\neq O_\sigma\cap V$ because $H_{X(\sigma), f}$ does not contain $O_\sigma$. 
            We remark that $X(\sigma')$ is an affine scheme of finite type over $k$. 
            Thus, from  Corollary \ref{cor:fundamental smoothness of small open}, there exists an open subset $U$ of $V$ such that 
            \begin{itemize}
                \item  $H_{X(\sigma), f}\cap O_\sigma\cap V\subset U$.
                \item The restriction $\pi_*|_{H_{X(\sigma), f}\cap U}\colon H_{X(\sigma), f}\cap U\rightarrow X(\sigma')$ is smooth. 
            \end{itemize}
            Therefore, the statement holds. 
        \end{proof}
        The following elementary lemma is needed in Lemma \ref{lem: for 7-3}. 
        \begin{lemma}\label{lem:for 7-2}
            Let $N$ be a lattice of finite rank and $N'$ be a sublattice of $N\oplus\ZZ$. 
            We assume that $(N\oplus\ZZ)/N'$ is tosion free and there exists $v\in N$ such that $(v, 1)\in N'$. 
            Then there exists a sub lattice $N''$ of $N$ such that $N'\oplus (N''\times\{0\}) = N\oplus\ZZ$. 
        \end{lemma}
        \begin{proof}
            Let $\iota$ denote a canonical inclusion $N\rightarrow N\oplus\ZZ$ such that $\iota(w) = (w, 0)$ for any $w\in N$. 
            Let $N_1$ be an inverse image of $N'$ along $\iota$. 
            We claim that $(N_1\times\{0\})\oplus \ZZ(v, 1) = N'$. 
            Indeed, for any $(w, a)\in N'$, we have $(av, a)\in \ZZ(v, 1)$ and $(w-av, 0)\in (N_1\times\{0\})$. 
            Moreover, $N/N_1$ is torsion free because $(N\oplus\ZZ)/N'$ is torsion free. 
            Thus, there exists a sub lattice $N_2$ of $N$ such that $N_1\oplus N_2 = N$. 
            Therefore, $(N_1\times\{0\})\oplus(N_2\times\{0\})\oplus\ZZ(v, 1) = N\oplus\ZZ$, so let $N''$ denote $N_2$. 
        \end{proof}

        Let $X(\Delta)\rightarrow \A^1$ be a toric morphism. 
        The following lemma indicates how to check the reduced-ness of the fiber of the closed orbit of $\A^1_k$ combinatorially. 
        \begin{lemma}\label{lem: for 7-3}
            Let $\sigma$ be a strongly convex rational polyhedral cone in $(N\oplus\ZZ)_\RR$. 
            We assume that $\sigma\subset N_\RR\times\RR_{\geq0}$ and $\sigma \not\subset N_\RR\times\{0\}$, and for every 1-dimensional ray $\gamma\preceq\sigma$ with $\gamma\cap (N_\RR\times\{1\})\neq\emptyset$, we have $\gamma\cap (N\times\{1\})\neq\emptyset$. 
            Then the following statements follow:
            \begin{enumerate}
                \item[(a)] There exists a sublattice $N'$ of $N$ such that $N/N'$ is torsion free and the restriction $(\pi, \mathrm{id})|_{\langle\sigma\rangle\cap (N\oplus\ZZ)}\colon \langle\sigma\rangle\cap (N\oplus\ZZ) \rightarrow N/N'\oplus\ZZ$ is isomorphic, where $\pi\colon N\rightarrow N/N'$ be the quotient map and $(\pi, \mathrm{id}_\ZZ)\colon N\oplus\ZZ\rightarrow N/N'\oplus\ZZ$ is a direct product morphism $\pi$ and ${\mathrm{id}}_\ZZ$. 
                \item[(b)] For such $N'$ in (a), let $\sigma'$ denote $(\pi, \mathrm{id}_\ZZ)_\RR(\sigma)$. 
                Let $F$ denote $(\pr_2)_*^{-1}(0)$ which is the fiber of the toric morphism $(\pr_2)_*\colon X(\sigma')\rightarrow\A^1$ at a closed orbit of $\A^1_k$. 
                Then $\sigma'$ is a strongly convex rational polyhedral cone in $(N/N'\oplus\ZZ)_\RR$ and a second projection $\pr_2\colon (N/N'\oplus\ZZ)\rightarrow\ZZ$ is compatible with the cones $\sigma'$ and $[0, \infty)$. 
                Moreover, $F$ is reduced. 
            \end{enumerate}
        \end{lemma}
        \begin{proof}
            We prove the statements from (a) to (b). 
            \begin{enumerate}
                \item[(a)] From the assumption of $\sigma$, there exists a 1-dimensional ray $\gamma\preceq\sigma$ such that $\gamma\not\subset N_\RR\times \{0\}$. 
                Because $\gamma\cap (N\times\{1\})\neq \emptyset$, there exists $v\in N$ such that $(v, 1)\in\sigma$.  
                Thus, from Lemma \ref{lem:for 7-2}, there exists a sublattice $N_0$ of $N$ such that $(\langle\sigma\rangle\cap(N\oplus\ZZ))\oplus (N_0\times\{0\}) = N\oplus\ZZ$. 
                Therefore, the statement holds. 
                \item[(b)] From Lemma \ref{lem:torus-fibration}(a), $\sigma'$ is a full and strongly convex rational polyhedral cone in $(N/N'\oplus\ZZ)_\RR$. 
                Because $\sigma\subset N_\RR\times\RR_{\geq 0}$, we have $\sigma'\subset (N/N')_\RR\times \RR_{\geq 0}$. 
                Thus, a second projection $\pr_2\colon (N/N'\oplus\ZZ)\rightarrow\ZZ$ is compatible with the cones $\sigma'$ and $[0, \infty)$. 
                From Lemma \ref{lem:torus-fibration}(b) and (c), there is a one-to-one correspondence with 1-dimensional rays of $\sigma$ and those of $\sigma'$. 
                Thus, from the assumption of $\sigma$ and $N'$, for every 1-dimensional ray $\gamma'\preceq\sigma'$ with $\gamma'\cap (N/N')_\RR\times\{1\}\neq\emptyset$, we have $\gamma'\cap (N/N')\times\{1\}\neq\emptyset$. 
                
                We show that $F$ is reduced. 
                Let $t\in k[\ZZ]$ be a torus invariant monomial associated with $1\in\ZZ$. 
                The irreducible component of the fiber of the toric morphism $X(\sigma')\rightarrow\A^1$ at a closed orbit of $\A^1_k$ is an orbit closure of $X(\sigma')$ associated with a 1-dimensional ray of $\sigma'$ such that it intersects $(N/N')_\RR\times\{1\}$. 
                Let $\gamma'\preceq\sigma'$ be a 1-dimensional ray with $\gamma'\cap (N/N')_\RR\times\{1\}\neq\emptyset$ and $\eta'\in X(\sigma)$ denote the generic point of $O_{\gamma'}$. 
                Let $v'\in(N/N')\oplus\ZZ$ denote a minimal generator of $\gamma'$. 
                Then $v'$ is a unique torus invariant valuation of $X(\sigma')$ such that the valuation ring associated with $v'$ is $\OO_{X(\sigma'), \eta}$ and the image of $v'$ is $\ZZ$. 
                Now, we show that $v'(t) = 1$. 
                Indeed, there exists $v''\in N/N'$ such that $v' = (v'', 1)$ from the above argument. 
                Because the second component of $v'$ is $1$, we have $v'(t) = 1$. 
                In particular, $\OO_{F, \eta}$ is reduced. 
                Thus, $F$ has a property $(R_0)$. 
                Moreover, $X(\sigma')$ is a Cohen-Macauray and an integral scheme. 
                Hence, $F$ is a Cohen-Macauray scheme. 
                In particular, $F$ has a property $(S_1)$. 
                Therefore, $F$ is reduced. 
            \end{enumerate}
        \end{proof}
        When we compute the stable birational volume, it is important to construct a strictly toroidal scheme (cf. Definition\ref{NONO21}). 
        The following lemma indicates that the sufficient conditions for the reduced-ness of the fiber of the closed orbit can be given by the combinatorial setting. 
        \begin{lemma}\label{lem: for 7-4}
            We keep the notation and the assumption in Lemma \ref{lem: for 7-3}. 
            We identify with $k[M\oplus \ZZ]$ and $k[M]\otimes_k k[t, t^{-1}]$, where $t$ is a torus invariant monomial associated with $1\in\ZZ$. 
            Let $f\in k[M\oplus \ZZ]$ be a polynomial and $V$ be an open subscheme of $X(\sigma)$. 
            We assume that $H_{X(\sigma), f}$ does not contain $O_\sigma$ and $H_{X(\sigma), f}\cap O_\sigma\cap V$ is smooth over $k$. 
            We regard $X(\sigma)$ is a scheme over $\Spec(k[t]) = \A^1$ by the toric morphism $(\pr_2)_*\colon X(\sigma)\rightarrow \A^1$. 
            
            Then there exist 
            \begin{itemize}
                \item A toric monoid $S$
                \item An alement $\omega\in S$
                \item A morphism $p\colon X(\sigma)\rightarrow \Spec(k[t][S]/(t-\chi^\omega))$ over $\Spec(k[t])$
                \item An open subset $U$ of $V$
            \end{itemize}
            such that
            \begin{enumerate}
                \item[(1)] $H_{X(\sigma), f}\cap O_\sigma\cap V\subset U$
                \item[(2)] The restriction $p|_{H_{X(\sigma), f}\cap U}\colon H_{X(\sigma), f}\cap U\rightarrow \Spec(k[t][S]/(t-\chi^\omega))$ is smooth. 
                \item[(3)] A ring $k[S]/(\chi^\omega)$ is reduced. 
            \end{enumerate}
        \end{lemma}
        \begin{proof}
            From Lemma\ref{lem: orbit and function}, there exists $\eta\in M\oplus\ZZ$ such that $\chi^\eta f\in k[\sigma^\vee\cap (M\oplus\ZZ)]$ and $H_{X(\sigma), f} = H_{X(\sigma), \chi^\eta f}$. 
            Thus, we may replace $f$ with $\chi^\eta f$ and assume that $f\in k[\sigma^\vee\cap (M\oplus\ZZ)]$. 
            We use the notation in Lemma \ref{lem: for 7-3}.  
            Let $S$ be a toric monoid associated with $X(\sigma')$, $g\colon X(\sigma')\rightarrow \Spec(k[S])$ be a natural isomorphism, and $\omega\in S$ be an element associated with $t$. We remark that $(\pi, \mathrm{id})_*\colon X(\sigma)\rightarrow X(\sigma')$ is a trivial algebraic torus fibration.  
            We regard $k[S]$ as a $k[t]$-algebra by a $k$-morphism $q\colon k[t]\rightarrow k[S]$ such that $q(t) = \chi^\omega$. 
            Then there exists the following commutative diagram:
            \begin{equation*}
                \begin{tikzcd}
                    X(\sigma')\ar[r, "g"]\ar[d, "(\pr_2)_*"]& \Spec(k[S])\ar[ld, "q^*"]\\
                    \Spec(k[t])
                \end{tikzcd}
            \end{equation*}
            where $q^*$ is a morphism induced by $q$. 

            Let $r$ denote $q\otimes\mathrm{id}_{k[S]}\colon k[t][S]\rightarrow k[S]$. 
            We regard $k[t][S]$ as a $k[t]$-algebra by a $k$-morphism $s\colon k[t]\rightarrow k[t][S]$ such that $s(t) = t$. 
            We remark that $r$ is a $k[t]$-morphism. 
            Let $r_0$ denote the quotient map $k[t][S]/(t - \chi^\omega)\rightarrow k[S]$ induced by $r$ and $s_0$ denote a morphism $k[t]\rightarrow k[t][S]/(t - \chi^\omega)$ induced by $s$. 
            Because of the definition of $r$, $r_0$ is an isomorphism. 
            Thus, there exists the following commutative diagram:
            \begin{equation*}
                \begin{tikzcd}
                    X(\sigma')\ar[r, "g"]\ar[d, "(\pr_2)_*"]& \Spec(k[S])\ar[ld, "q^*"]\ar[d, "r^*_0"]\\
                    \Spec(k[t])& \Spec(k[t][S]/(t-\chi^\omega))\ar[l, "s^*_0"]
                \end{tikzcd}
            \end{equation*}
            where $r^*_0$ and $s^*_0$ are morphisms induced by $r_0$ and $s_0$. 
            Let $p$ be a composition $r^*_0\circ g\circ(\pi, \mathrm{id}_\ZZ)_*\colon X(\sigma)\rightarrow \Spec(k[t][S]/(t-\chi^\omega))$. 
            Then from the proof of Lemma \ref{lem: for 7-1}, there exists an open subset of $U$ of $V$ such that $H_{X(\sigma), f}\cap O_\sigma\cap V\subset U$ and the restriction $p|_{H_{X(\sigma), f}\cap U}\colon H_{X(\sigma), f}\cap U\rightarrow \Spec(k[t][S]/(t-\chi^\omega))$ is smooth. 
            Moreover, $k[S]/(\chi^\omega)$ is isomorphic to a global section ring of the fiber of $(\pr_2)_*\colon X(\sigma')\rightarrow\A^1$ at a closed orbit of $\A^1_k$. 
            Thus, from Lemma \ref{lem: for 7-3}, $k[S]/(\chi^\omega)$ is reduced. 
        \end{proof}
        
        \subsection{Scheme theoretic image of quasi-compact morphism}
        In this subsection, we prove some lemmas about the scheme theoretic image which we use in this article. 
        \begin{lemma}\label{lem:open-image}
            Let $f\colon X\rightarrow Y$ be a quasi-compact morphism of schemes, and $Z$ be a scheme theoretic image of $f$, and $V$ be an open subscheme of $Y$. 
            Then $Z\cap V$ is a scheme theoretic image of $f|_{f^{-1}(V)}\colon f^{-1}(V)\rightarrow V$. 
        \end{lemma}
        \begin{proof}
            Because $f$ is quasi-compact, $\mathscr{I} = \ker(\OO_Y\rightarrow f_*\OO_X)$ is a quasi-coherent ideal sheaf associated with the closed subscheme $Z$ of $Y$. 
            Because $\mathscr{I}|_V = \ker(\OO_{V}\rightarrow (f|_{f^{-1}(V)})_*\OO_{f^{-1}(V)})$ and $\mathscr{I}|_V$ is a quasi-coherent ideal sheaf associated with the closed subscheme $Z\cap V$ of $V$, $Z\cap V$ is a scheme theoretic image of $f|_{f^{-1}(V)}$. 
        \end{proof}
        \begin{lemma}\label{lem:covering-image}
            Let $f\colon X\rightarrow Y$ be a quasi-compact morphism of schemes, $Z$ be a closed subscheme of $Y$, and $\{V_i\}_{i\in I}$ be an open covering of $Y$. 
            Then, the following statements are equivalent. 
            \begin{itemize}
                \item[(a)] A closed subscheme $Z$ of $Y$ is a scheme theoretic image of $f$. 
                \item[(b)] For any $i\in I$, a closed subscheme $Z\cap V_i$ of $V_i$ is a scheme theoretic image of $f|_{f^{-1}(V_i)}$. 
            \end{itemize}
        \end{lemma}
        \begin{proof}
            When the statement (a) holds, the statement (b) holds from Lemma \ref{lem:open-image}. 
            Then we assume the statement (b). 
            Let $\mathscr{I}_Z\subset \OO_Y$ be a quasi-coherent ideal sheaf associated with $Z$ of $Y$. 
            From the assumption, for $i\in I$, $\mathscr{I}_Z|_{V_i} = \ker(\OO_{V_i}\rightarrow (f|_{f^{-1}(V_i)})_*\OO_{f^{-1}(V_i)})$.
            Thus, $\mathscr{I}_Z = \ker(\OO_Y\rightarrow f_*\OO_X)$ because $\{V_i\}_{i\in I}$ is an open covering of $Y$. 
            Therefore, $Z$ is a scheme theoretic image of $f$. 
        \end{proof}
        \subsection{Some lemmas about valuation}
        \begin{lemma}\label{lem:val-val}
            Let $X, X', Y, Y', Z$, and $Z'$ be integral schemes, $\iota\colon Z\rightarrow X$, $q\colon X\rightarrow Y$, $\iota'\colon Z'\rightarrow X'$, $q'\colon X'\rightarrow Y'$, $g\colon X'\rightarrow X$, $h\colon Y'\rightarrow Y$, and $f\colon Z'\rightarrow Z$ are morphisms such that the following diagram is commutative: 
            \begin{equation*}
                \begin{tikzcd}
                    Z'\ar[r, "\iota'"]\ar[d, "f"]& X'\ar[r, "q'"]\ar[d, "g"]& Y'\ar[d, "h"]\\
                    Z\ar[r, "\iota"]& X\ar[r, "q"]& Y
                \end{tikzcd}
            \end{equation*}
            Let $r$ denote $q\circ\iota$ and $r'$ denote $q'\circ\iota'$.
            We assume that $f, g, h, q,$ and $q'$ are dominant and $r$ and $r'$ are open immersions. 
            Then the following equation follows as a map from $\Val(X')$ to $\Val(Z)$: 
            \[
                f_\natural\circ({r'}^{-1})_\natural\circ(q')_\natural = (r^{-1})_\natural\circ q_\natural\circ g_\natural
            \]
        \end{lemma}
        \begin{proof}
            From the commutativity of the above diagram, $h\circ q' = q\circ g$ and $f\circ (r')^{-1} = r^{-1}\circ h$. 
            Thus, we can check that the statement follows. 
        \end{proof}
        \subsection{Some lemmas about valuation ring}
        We use the notation as follows:
        \begin{itemize}
            \item Let $k$ be an algebraically closed field with $\charac(k) = 0$. 
            \item Let $t$ be an indeterminate invariant.
            \item Let $l$ be a positive integer and $R_l$ be a valuation ring defined as follows:
            \[
                R_l = k[[t^{\frac{1}{l}}]]
            \]
            \item Let $l$ be a positive integer and $K_l$ be a fraction field of $R_l$. 
            We remark that $K$ is written as follows:
            \[
                K_l = k((t^{\frac{1}{l}}))
            \]
            \item Let $\Rt$ be a valuation ring defined as follows:
            \[
                \Rt = \bigcup_{n\in\ZZ_{>0}} R_n
            \]
            \item Let $\Kt$ be a fraction field of $\Rt$. 
            A field $\Kt$ is a field of Puiseux series over $k$. 
            We remark that $\Kt$ is written as follows:
            \[
                \Kt = \bigcup_{n\in\ZZ_{>0}} K_n
            \]
        \end{itemize}
        In this subsection, we prove some lemmas about schemes over $\Spec(\Rt)$. 
        \vskip\baselineskip
        The following lemma is referenced in \cite[Remark. 4.6]{G13}; however, for the sake of thoroughness, a detailed proof is provided here.
        \begin{lemma}\label{lem:Gubler}
        \cite[Remark. 4.6]{G13}
            Let $\mathfrak{X}$ be a flat scheme over $\Spec(\Rt)$, $X$ denote a base change of $\mathfrak{X}$ along $\Spec(\Kt)\rightarrow \Spec(\Rt)$, and $Y$ be a closed subscheme of $X$. 
            Then there exists a unique closed subscheme $\mathfrak{Y}$ of $\mathfrak{X}$ such that the following conditions hold:
            \begin{enumerate}
                \item[(1)] A scheme $\mathfrak{Y}$ is flat over $\Spec(\Rt)$. 
                \item[(2)] A base change of $\mathfrak{Y}$ along $\Spec(\Kt)\rightarrow \Spec(\Rt)$ is equal to $Y$ as a closed subscheme of $X$. 
            \end{enumerate}
            In fact, such $\mathfrak{Y}$ is a scheme theoretic closure of $Y$ in $\mathfrak{X}$. 
        \end{lemma}
        \begin{proof}
            We remark that the open immersion $X\hookrightarrow \mathfrak{X}$ is quasi-compact because $\Spec(\Kt)\rightarrow\Spec(\Rt)$ is quasi-compact. 
            Thus, the composition $Y\hookrightarrow X\rightarrow \mathfrak{X}$ is quasi-compact too.
            We already know that the statement holds in the case when $\mathfrak{X}$ is an affine scheme from \cite[Proposition. 4.4]{G13}. 

            First, we show the existence of $\mathfrak{Y}$. 
            Let $\mathfrak{Y}_0$ be a scheme theoretic closure of $Y$ in $\mathfrak{X}$. 
            Thus, for any affine open subset $U\subset \mathfrak{X}$, the scheme theoretic closure of $U\cap Y$ in $U$ is $U\cap\mathfrak{Y}_0$ from Lemma \ref{lem:open-image}. 
            For such $U$, $U$ is flat over $\Spec(\Rt)$, so that $(U\cap\mathfrak{Y}_0)_{\Kt} = U\cap Y$ and $U\cap\mathfrak{Y}_0$ is flat over $\Spec(\Rt)$ from \cite[Proposition. 4.4]{G13}. 
            By considering one affine open covering of $\mathfrak{X}$, $\mathfrak{Y}_0$ is flat over $\Spec(\Rt)$ and the generic fiber of $\mathfrak{Y}_0$ is equal to $Y$ as a closed subscheme of $X$. 

            Second, we show the uniqueness of $\mathfrak{Y}$. 
            Let $\mathfrak{Y}_1$ be a closed subscheme of $\mathfrak{X}$ such that $\mathfrak{Y}_1$ is flat over $\Spec(\Rt)$ and the generic fiber of $\mathfrak{Y}_1$ is equal to $Y$ as a closed subscheme of $X$. 
            Let $U$ be an affine open subset of $\mathfrak{X}$. 
            Then $U\cap \mathfrak{Y}_1$ is flat over $\Spec(\Rt)$ and the generic fiber of $U\cap \mathfrak{Y}_1$ is equal to $U\cap Y$ as a closed subscheme of $U\times_\Rt \Kt$.
            Because $U$ is flat over $\Spec(\Rt)$, $U\cap \mathfrak{Y}_1$ is the scheme theoretic closure of $U\cap Y$ in $U$ 
            from \cite[Proposition. 4.4]{G13}. 
            Thus, by considering one affine open covering of $\mathfrak{X}$, $\mathfrak{Y}_1$ is the scheme theoretic closure of $Y$ in $\mathfrak{X}$ from Lemma \ref{lem:covering-image}.
        \end{proof}
        The following lemma is referenced in \cite[\href{https://stacks.math.columbia.edu/tag/055C}{Tag 055C}]{SP}. 
        We remark that if $S$ in the following lemma is replaced with $\Rt$, then the same proof does not hold. 
        \begin{lemma}\label{lem: DVR-connected}
        \cite[\href{https://stacks.math.columbia.edu/tag/055C}{Tag 055C}]{SP}
            Let $S$ be a DVR, $u$ be its uniformizer, $Q$ be the fraction field of $S$, and $\kappa$ be the residue field of $S$. 
            Let $B$ be a flat $S$-algebra and $x\in B\otimes_S Q$ be an idempotent element. 
            We remark that extension morphism $B\rightarrow B\otimes_S Q$ is injective, and $u$ is a non-zero divisor of $B$ because $B$ is flat over $S$. 
            We assume that $B\otimes_S \kappa$ is reduced. 
            Then we have $x\in B$.
        \end{lemma}
        \begin{proof}
            Let $a\in B$ be an element and $n\in\ZZ_{\geq 0}$ be a non-negative integer such that $x = \frac{a}{u^n}$ in $B\otimes_S Q$. 
            We can take $n$ minimal. 
            Because $x$ is an idempotent element, $\frac{a}{u^n} = \frac{a^2}{u^{2n}}$  in $B\otimes_S Q$. 
            Then $a^2 = au^n$  in $B$ because $u$ is a non-zero divisor of $B$. 
            If $n \geq 1$, this shows that $a + (u)B$ is a nilpotent element in $B\otimes_S \kappa$. 
            However, from the assumption, $a\in (u)B$ and this is a contradiction to the minimality of $n$. 
            Thus, we can take $n$ as $0$, so that $x\in B$.  
        \end{proof}
        It is not trivial whether Lemma \ref{lem: DVR-connected} holds when $S = \Rt$. 
        In fact, if $B$ is flat and of finite presentation over $\Rt$, then Lemma \ref{lem: DVR-connected} holds. 
        This claim is proven in Lemma \ref{lem: stacks exchange}. 
        The following lemma serves as a preparation for the proof of Lemma \ref{lem: stacks exchange}. 
        \begin{lemma}\label{lem:R-of finite presentation}
            Let $A$ be a flat and of finite presentation ring over $\Rt$ and $x\in A\otimes_\Rt \Kt$ be an element. 
            Then there exists a positive integer $l$ and a sub $R_l$-algebra $C$ of $A$ such that
            \begin{itemize}
                \item A $R_l$-algebra $C$ is flat over $R_l$.
                \item All morphisms in the following diagram are injective:
                \begin{equation*}
                    \begin{tikzcd}
                        C\ar[r]\ar[d, hook] & C\otimes_{R_l} K_l\ar[d]\\
                        A\ar[r]&A\otimes_{\Rt} \Kt
                    \end{tikzcd}
                \end{equation*}
                where the right one is induced by the base of $R_l\rightarrow K_l$ along $C\hookrightarrow A$
                \item A subring $C\otimes_{R_l} K_l$ of $A\otimes_{\Rt} \Kt$ contains $x$. 
                \item The inclusion $C\hookrightarrow A$ induce the isomorphism $C\otimes_{R_l}k \cong A\otimes_\Rt{k}$
            \end{itemize}
        \end{lemma}
        \begin{proof}
            From the assumption, there exist $N\in\ZZ_{> 0}$ and $f_1, \ldots, f_r\in \Rt[X_1, \ldots, X_N]$ such that $\Rt[X_1, \ldots, X_N]/(f_1, \ldots, f_r)\cong A$ as an $\Rt$-algebra. 
            Moreover, there exists $g\in \Kt[X_1, \ldots, X_N]$ such that $g + (f_1, \ldots, f_r) = x$ by the identification with $\Kt[X_1, \ldots, X_N]/(f_1, \ldots, f_r)$ and $A\otimes_{\Rt} \Kt$. 
            Thus, from the definition of $\Rt$ and $\Kt$, there exists $l\in\ZZ_{>0}$ such that $f_1, \ldots, f_r\in R_l[X_1, \ldots, X_N]$ and $g\in K_l[X_1, \ldots, X_N]$. 
            Let $C$ denote a $R_l$-algebra $R_l[X_1, \ldots, X_N]/(f_1, \ldots, f_r)$. 
            We remark that $C\otimes_{R_l} \Rt\cong \Rt[X_1, \ldots, X_N]/(f_1, \ldots, f_r)$ as an $\Rt$-algebra. 
            We remark that $\Rt\otimes_{R_l}K_l = \Rt[t^{-\frac{1}{l}}] = \Kt$. 

            First, we show that $C$ is flat over $R_l$. 
            Because $R_l$ is a DVR, it is enough to show that $C$ is a torsion-free $R_l$-module. 
            Let $f\colon C\rightarrow C$ denote an endomorphism of $C$ defined by the scalar product of $t^\frac{1}{l}$. 
            Then, it is enough to show that $f$ is injective. 
            Because $\Rt$ is a faithfully flat $R_l$-module, it is enough to show that $f\otimes\id_{\Rt}$, which is induced by $f$, is injective. 
            From the assumption, $\Rt[X_1, \ldots, X_N]/(f_1, \ldots, f_r)$ is flat over $\Rt$. 
            Thus, $f\otimes\id_{\Rt}$ is injective. 

            Second, the extension morphism $C\rightarrow \Rt[X_1, \ldots, X_N]/(f_1, \ldots, f_r)$ is injective because $C$ is flat over $R_l$. 
            Thus, we can regard $C$ as an $R_l$-sub algebra of $A$ on the identification with $\Rt[X_1, \ldots, X_N]/(f_1, \ldots, f_r)$ and $A$. 
            We can easily check that this inclusion $C\hookrightarrow A$ induces an isomorphism $C\otimes_{R_l}k\cong A\otimes_{\Rt}k$ because both rings are isomorphic to $k[X_1, \ldots, X_N]/(\overline{f_1}, \ldots, \overline{f_r})$, where $\overline{f_i}\in k[X_1, \ldots, X_N]$ is a polynomial whose each coefficient is replaced by constant terms of coefficients of $f_i$. 
            
            Third, in the above diagram, we show that all morphisms are injective. 
            Indeed, the left one is injective from the above argument. 
            The upper one is injective because $C$ is flat over $R_l$. 
            The right one is injective because $K_l$ is flat over $R_l$. 
            At this point, we remark that $A\otimes_\Rt \Kt = A\otimes_\Rt \Rt\otimes_{R_l}K_l = A\otimes_{R_l}K_l$. 
            The lower one is injective because $A$ is flat over $\Rt$. 

            Finally, we show that $x\in C\otimes_{R_l} K_l$. 
            Indeed, $g + (f_1, \ldots, f_r)\in C\otimes_{R_l} K_l$ and this element go to $x$ along the right morphism in the above diagram. 
        \end{proof}
        From now on, we prove the following lemma, which is derived from the Lemma\ref{lem: DVR-connected} by replacing the condition that $S$ is a DVR with the condition that $S = \Rt$ and adding one condition. 
        \begin{lemma}\label{lem: stacks exchange}
            Let $\eta$ denote a generic point of $\Spec(\Rt)$. 
            Let $X$ be a flat and locally of finite presentation scheme over $\Spec(\Rt)$. 
            For an open subscheme $U\subset X$, let $U_\eta$ denote a generic fiber of $U$. 
            We assume that $X$ is a connected scheme and the closed fiber $X_k$ is reduced.
            Then the following statements follow:
            \begin{enumerate}
                \item[(a)] Let $U$ be an open subscheme of $X$. 
                Then, the following ring morphism induced by an open immersion $U_\eta\rightarrow U$ is injective. 
                \[
                    \Gamma(U, \OO_{X})\rightarrow \Gamma(U_\eta, \OO_{X})
                \]
                \item[(b)] A scheme $X_\eta$ is connected. 
            \end{enumerate}
        \end{lemma}
        \begin{proof}
            We prove the statements from (a) to (b). 
            \begin{enumerate}
                \item[(a)] When $U$ is an affine open subscheme of $X$, the statement holds. 
                Indeed, let $A$ be a ring such that $\Spec(A) = X$. 
                Then $A$ is flat over $\Rt$. 
                Thus, an extension map $A\rightarrow A\otimes_{\Rt}\Kt$ is injective. 
                In the general case, $U$ can be covered by open affine subschemes of $X$. 
                Because $\OO_X$ is a sheaf over $X$, the statement holds. 
                \item[(b)] We show that the statement holds by a contradiction. 
                We assume that $X_\eta$ isn't connected. 
                Then there exists an idempotent element $e\in\Gamma(X_\eta, \OO_X)$ such that $e\neq 0$ and $e\neq1$. 
                Let $U$ be an open affine subscheme of $X$ and $A$ be a ring such that $U = \Spec(A)$ and $A$ is of finite presentation over $\Rt$. 
                We remark that $U_\eta = \Spec(A[\frac{1}{t}])$, so that $e|_{U_\eta}\in A\otimes_{\Rt} \Kt$. 
                Then for $A$ and $e|_{U_\eta}\in A\otimes_{\Rt} \Kt$, there exists $l\in\ZZ_{>0}$ and a sub $R_l$-algebra $C$ of $A$ such that $l$ and $C$ satisfy the conditions in Lemma \ref{lem:R-of finite presentation}. 
                Thus, $e|_{U_\eta}\in C\otimes_{R_l}K_l$ and $C\otimes_{R_l}K_l\rightarrow 
                A\otimes_{\Rt}\Kt$ is injective. 
                Hence, $e|_{U_\eta}$ is an idempotent elements of $C\otimes_{R_l}K_l$. 
                Moreover, $C$ is a flat over a DVR $R_l$ and $C\otimes_{R_l}k\cong A\otimes_{\Rt} k$ is reduced from the assumption. 
                Thus, from Lemma \ref{lem: DVR-connected}, we have $e|_{U_\eta}\in C$ and $e|_{U_\eta}\in A$. 
                This shows that $e|_{U_\eta}\in\Gamma(U, \OO_X)$ for any affine open subscheme $U$ of $X$. 
                Hence, from (a), we can check that $e\in\Gamma(X, \OO_X)$. 
                An element $e\in\Gamma(X, \OO_X)$ is a non-trivial idempotent element of $\Gamma(X, \OO_X)$, but it is a contradiction to the assumption that $X$ is connected. 
            \end{enumerate}
        \end{proof}
        \subsection{"Smoothness"}
        In this subsection, we prove the lemmas are needed in Lemma\ref{lem: for 7-1} and \ref{lem: for 7-4}. 
        \begin{lemma}\label{lem: fundamental smoothness}
            We itemize the notation as follows:
            \begin{itemize}
                \item Let $k$ be a field. 
                \item Let $A$ be a ring of finite type over $k$.
                \item Let $X$ denote $\Spec(A)$. 
                \item For $x\in X$, let $\kappa(x)$ denote a residue field of $x$. 
                \item Let $A[X_1, \ldots, X_n]$ denote a global section ring of $X\times_{k}\A^n_k$ and $0\neq f\in A[X_1, \ldots, X_n]$ be a polynomial. 
                \item Let $Y$ be a closed subscheme of $X\times_{k}\A^n_k$ defined by $f = 0$. 
                \item Let $V$ be a non-empty open subset of $X\times_{k}\A^n_k$. 
                \item Let $\pi$ denote the first projection $X\times_{k}\A^n_k\rightarrow X$. 
            \end{itemize}
            We assume that $V_{\kappa(x)}\cap Y_{\kappa(x)} \neq V_{\kappa(x)}$ and $V_{\kappa(x)}\cap Y_{\kappa(x)}$ is smooth over $\kappa(x)$. 
            
            Then there exists a non-empty open subset of $U\subset V$ such that $V_{\kappa(x)}\cap Y_{\kappa(x)}\subset U$ and the restriction $\pi|_{U\cap Y}\colon U\cap Y\rightarrow X$ is a smooth morphism. 
        \end{lemma}
        \begin{proof}
            Let $s$ denote a natural morphism $A[X_1, \ldots, X_n] \rightarrow \kappa(x)[X_1, \ldots, X_n]$ induced by $A\rightarrow \kappa(x)$. 
            Then from the assumption, $s(f)\neq 0$ in $\kappa(x)[X_1, \ldots, X_n]$. 
            We remark that for any $1\leq i\leq n$ and $g\in A[X_1, \ldots, X_n]$, the following equation holds in $\kappa(x)[X_1, \ldots, X_n]$: 
            \[
                \frac{\partial}{\partial X_i} s(g) = s\bigl(\frac{\partial g}{\partial X_i}\bigr)
            \]
            Let $Z$ denote a closed subscheme of $X\times_{k}\A^n_k$ defined by an ideal generated by $\{\frac{\partial f}{\partial X_i} \}_{1\leq i\leq n}$. 
            Now, we show that $V_{\kappa(x)}\cap Y_{\kappa(x)}\cap Z_{\kappa(x)} = \emptyset$ by a contradiction. 
            We assume that $V_{\kappa(x)}\cap Y_{\kappa(x)}\cap Z_{\kappa(x)} \neq\emptyset$. 
            Then $V_{\overline{\kappa(x)}}\cap Y_{\overline{\kappa(x)}}\cap Z_{\overline{\kappa(x)}} \neq\emptyset$. 
            From above remark, this shows that a hypersurface $Y_{\overline{\kappa(x)}}$ of $\A^n_{\overline{\kappa(x)}}$has a singular point in $V_{\overline{\kappa(x)}}\cap Y_{\overline{\kappa(x)}}$, but this is a contradiction to the assumption that $V_{\kappa(x)}\cap Y_{\kappa(x)}$ is smooth over $\kappa(x)$. 
            
            Let $U$ denote $V\cap (X\times_k\A^n_k\setminus Z)$ and $W_i$ denote a fundamental affine open subset of $X\times_k\A^n_k$ defined by $\frac{\partial f}{\partial X_i}$ for $1\leq i\leq n$.  
            Thus, $U$ contains $V_{\kappa(x)}\cap Y_{\kappa(x)}$. 
            Moreover, for any $1\leq i\leq n$, the restriction $\pi\colon Y\cap W_i\rightarrow X$ is a smooth morphism. 
            Thus, the restriction $\pi|_{U\cap Y}\colon U\cap Y\rightarrow X$ is a smooth morphism too.
        \end{proof}
        \begin{corollary}\label{cor:fundamental smoothness of small open}
            We itemize the notation as follows:
            \begin{itemize}
                \item Let $k$ be a field. 
                \item Let $A$ be a ring of finite type over $k$.
                \item Let $X$ denote $\Spec(A)$. 
                \item For $x\in X$, let $\kappa(x)$ denote a residue field of $x$. 
                \item Let $A[X_1^{\pm}, \ldots, X_n^{\pm}]$ denote a global section ring of $X\times_{k}{\Gm^n}_{,k}$ and $0\neq f\in A[X_1^{\pm}, \ldots, X_n^{\pm}]$ be a Laurant polynomial. 
                \item Let $Y$ be a closed subscheme of $X\times_{k}{\Gm^n}_{,k}$ defined by $f = 0$. 
                \item Let $V$ be a non-empty open subset of $X\times_{k}{\Gm^n}_{,k}$. 
                \item Let $\pi$ denote the first projection $X\times_{k}{\Gm^n}_{,k}\rightarrow X$. 
            \end{itemize}
            We assume that $V_{\kappa(x)}\cap Y_{\kappa(x)} \neq V_{\kappa(x)}$ and $V_{\kappa(x)}\cap Y_{\kappa(x)}$ is smooth over $\kappa(x)$. 
            
            Then there exists a non-empty open subset of $U\subset V$ such that $V_{\kappa(x)}\cap Y_{\kappa(x)}\subset U$ and the restriction $\pi|_{U\cap Y}\colon U\cap Y\rightarrow X$ is a smooth morphism. 
        \end{corollary}
        \begin{proof}
            The natural inclusion $A[X_1, \ldots, X_n]\hookrightarrow A[X_1^{\pm}, \ldots, X_n^{\pm}]$ induces an open immersion $X\times_{k}{\Gm^n}_{,k}\hookrightarrow X\times_{k}\A^n_k$. 
            Thus, on the above immersion, there exists $g\in A[X_1, \ldots, X_n]$ such that $Z\cap (X\times_{k}{\Gm^n}_{,k}) = Y$ as a closed subscheme of $X\times_{k}{\Gm^n}_{,k}$, where $Z$ is a closed subscheme of $X\times_{k}\A^n_k$ defined by $g = 0$. 
            Moreover, we can regard $V$ as an open subscheme of $X\times_{k}\A^n_k$. 
            Thus, from Lemma \ref{lem: fundamental smoothness}, the statement holds. 
        \end{proof}
    \bibliographystyle{amsplain}
    \bibliography{yoshino-bib}
\end{document}